\documentclass[12pt, a4paper, titlepage, headsepline, twoside]{book}
\usepackage{graphicx}
\usepackage{style_PM}

\usepackage{xcolor}

\usepackage{amsfonts}       
\usepackage{amsthm}         
\usepackage{mathtools}      
\usepackage{wrapfig}

\usepackage{subfigure}
\graphicspath{{figures/}}

\DeclareRobustCommand{\rchi}{{\mathpalette\irchi\relax}}
\newcommand{\irchi}[2]{\raisebox{\depth}{$#1\chi$}}

\definecolor{Granata}{rgb}{0.64,0,0} 

\newcommand{\mc}[1]{\mathcal{#1}}
\newcommand{\mb}[1]{\mathbb{#1}}
\newcommand{\R}{\mathbb{R}}
\newcommand{\poa}{{\rm PoA}}
\newcommand{\nashe}[1]{{\rm ne}({#1})}
\newcommand{\mnashe}[1]{{\rm mne}({#1})}
\newcommand{\cce}[1]{{\rm cce}({#1})}
\newcommand{\acce}[1]{{{\rm acce}^{\rm opt}}({#1})}
\newcommand{\be}{\begin{equation}}
\newcommand{\ee}{\end{equation}}

\newcommand{\A}[1]{\mathcal{A}^{#1}}
\renewcommand{\a}[1]{a^{#1}}
\newcommand{\aopt}{a^{\rm opt}}
\newcommand{\ami}{a_{-i}}
\newcommand{\intwithzero}[1]{[#1]_0}
\newtheorem{definition}{Definition}

\DeclareSymbolFont{bbold}{U}{bbold}{m}{n}
\DeclareSymbolFontAlphabet{\mathbbold}{bbold}
\newcommand{\vect}[1]{\mathbbold{#1}}
\newcommand{\ones}[1][]{\vect{1}_{#1}}
\usepackage{enumitem}
\setenumerate[1]{label=(\alph{*})}
\usepackage{tikz}
\usepackage{pgfplots}
\usetikzlibrary{intersections,backgrounds,plotmarks}
\tikzset{
    partial ellipse/.style args={#1:#2:#3}{
        insert path={+ (#1:#3) arc (#1:#2:#3)}
    }
}
\newtheorem{theorem}{Theorem}
\newtheorem{assumption}{Assumption}
\newtheorem{remark}{Remark}
\newtheorem{corollary}{Corollary}
\newtheorem{proposition}{Proposition}
\newtheorem{lemma}{Lemma}

\newtheoremstyle{example}{}{}{}{}{\bfseries}{.\smallskip}{3pt}{}
\theoremstyle{example}
\newtheorem{example}{Example}

\def\vspacesteps{1.5mm}
\def\myspace{1.5mm}

\newcommand{\cdotshort}{\!\cdot\!}
\newcommand{\fsv}{f_{\rm SV}}
\newcommand{\fmc}{f_{\rm MC}}
\newcommand{\fgar}{f_{\rm G}}
\newtheoremstyle{break}
  {\topsep}{\topsep}%
  {\itshape}{}%
  {\bfseries}{}%
  {\newline}{}%
\theoremstyle{break}

\newtheoremstyle{named}%
    {}{}{\itshape}{}{\bfseries}{.}{.5em}{\thmnote{#3}}
\theoremstyle{named}
\newtheorem*{namedtheorem}{Theorem}

\usepackage{comment}
\excludecomment{mycomment}
\newcommand{\cut}[1]{}
\begin{mycomment}
\renewcommand{\cut}[1]{#1}
\end{mycomment}
\usepackage{soul}

\newcommand{\poas}{{\rm SPoA}}
\newcommand{\Ir}{\mathcal{I}_R}
\usepackage{nccmath}
\DeclareMathOperator*{\argmax}{arg\,max}
\DeclareMathOperator*{\argmin}{arg\,min}
\DeclareMathOperator{\rank}{rank}
\usepackage{balance}
\newcommand{\pls}{\mc{PLS}}
\newcommand{\pclass}{\mc{P}}
\newcommand{\npclass}{\mc{NP}}
\newcommand{\ppad}{\mc{PPAD}}
\usepackage{algorithm}
\usepackage{caption}
\usepackage{algpseudocode}
\algnewcommand{\LeftComment}[1]{\Statex \(\triangleright\) #1} 

\newcommand{\cdue}{C_1}
\newcommand{\ctre}{C_2}
\newcommand{\cquattro}{C_3}
\renewcommand{\ae}{a^{\rm ne}}
\usepackage{multirow}

\newcommand{\defeq}{\coloneqq}
\newcommand{\zeros}[1][]{\vect{0}_{#1}}
\newcommand\mi{^{-i}}
\newcommand\WE{_\textup{W}}
\newcommand{\VWE}[1]{\bar{#1}_\textup{W}}
\newcommand\NE{_\textup{N}}
\newcommand{\VNE}[1]{\bar{#1}_\textup{N}}%
\newcommand\SO{_\textup{S}}
\newcommand\s{\Sigma}
\newcommand{\di}{d}
\newcommand{\eval}{_{|{z=\sigma(x)}}}
\newcommand{\eqdef}{\eqqcolon}
\newcommand{\minn}[1]{\underset{#1}{\operatorname{min}}\,}
\newcommand{\maxx}[1]{\underset{#1}{\operatorname{max}}\,}
\newcommand{\Dx}{R}
\newcommand\zero{^0}
\newcommand{\GN}{\mc{G}_\N}
\newcommand{\GNS}{(\GN)_{\N=1}^\infty}
\newcommand{\emm}{_{_\N}}

\usepackage[backend=biber,natbib=true,style=alphabetic,sorting=nyt,giveninits=true, defernumbers=false,citetracker=true,maxnames=99,doi=false,isbn=false,url=false, noerroretextools=true]{biblatex}

\usepackage{xpatch}
\newcommand*{\boldnamedo}[1]{%
  \iffieldequalstr{hash}{#1}
    {\bfseries\listbreak}
    {}}
\newbibmacro*{name:bold}{%
  \forlistloop{\boldnamedo}{\boldnames}}

\newcommand*{\boldnames}{}
\forcsvlist{\listadd\boldnames}
  {{30abb2207c22bf1dcd17f45640248687}}

\newcommand*{\makenamesbold}{%
  \xpretobibmacro{name:family}{\begingroup\usebibmacro{name:bold}}{}{}%
  \xpretobibmacro{name:given-family}{\begingroup\usebibmacro{name:bold}}{}{}%
  \xpretobibmacro{name:family-given}{\begingroup\usebibmacro{name:bold}}{}{}%
  \xpretobibmacro{name:delim}{\begingroup\normalfont}{}{}%
  \xapptobibmacro{name:family}{\endgroup}{}{}%
  \xapptobibmacro{name:given-family}{\endgroup}{}{}%
  \xapptobibmacro{name:family-given}{\endgroup}{}{}%
  \xapptobibmacro{name:delim}{\endgroup}{}{}%
}

\usepackage{varioref}

\usepackage[%
colorlinks= true,
raiselinks	= true,
linkcolor	= teal, 
citecolor	= teal,
urlcolor	= ForestGreen,
pdfauthor	= {Paccagnan Dario},
pdftitle	= {PhD Thesis},
pdfkeywords	= {},
pdfsubject	= {PhD Thesis},
plainpages= false]{hyperref}

\usepackage[capitalize,nameinlink,compress]{cleveref}
\crefname{appendix}{Appendix}{Appendices} 
\crefname{figure}{Figure}{Figures} 
\crefname{line}{line}{lines} 
\crefname{claim}{Claim}{Claims} 
\crefname{equation}{}{} 
\crefname{problem}{Problem}{Problems}
\crefname{assumption}{Assumption}{Assumptions}
\crefname{fact}{Fact}{Facts}
\crefname{sas}{Standing Assumptions}{Standing Assumptions}

\usepackage{autonum}

\renewcommand{\ae}{a^{\rm ne}}
\newcommand{\sigmane}{\sigma^{\rm mne} }
\newcommand{\sigmacce}{\sigma^{\rm cce} }
\newcommand{\sigmaacce}{\sigma^{\rm acce} }
\newcommand{\expected}[2]{\mathbb{E}_{#2}\left[#1\right]}
\newcommand{\poane}{{\rm PoA}^{\rm ne}}
\newcommand{\poamne}{{\rm PoA}^{\rm mne}}
\newcommand{\poacce}{{\rm PoA}^{\rm cce}}
\newcommand{\poaacce}{{\rm PoA}^{\rm acce}}
\usepackage{expl3}
\makeatletter
\ExplSyntaxOn
\cs_gset_eq:NN \c_deprecateacro_minus_one \m@ne
\tex_let:D \c_minus_one \c_deprecateacro_minus_one
\ExplSyntaxOff
\makeatother

\usepackage[hyperref=false]{acro}
\acsetup{
   first-long-format = {\emph}, 
only-used=false}

\DeclareAcronym{ne}{short=NE, long=Nash equilibrium, class=abbrev}
\DeclareAcronym{mne}{short=MNE, long=mixed Nash equilibrium, class=abbrev}
\DeclareAcronym{cce}{short=CCE, long=coarse correlated equilibrium, class=abbrev}
\DeclareAcronym{acce}{short=ACCE$^{\,\rm opt}$, long=average coarse correlated equilibrium, class=abbrev}
\DeclareAcronym{poa}{short=$\rm PoA$, long=price of anarchy, class=abbrev2}
\DeclareAcronym{poan}{short=PoA, long=price of anarchy, class=abbrev}
\DeclareAcronym{br}{short=BR, long=best-response dynamics, class=abbrev}
\DeclareAcronym{lp}{short=LP, long=linear program, class=abbrev}
\DeclareAcronym{kkt}{short=KKT, long=Karush-Kuhn-Tucker, class=abbrev}
\DeclareAcronym{mmc}{short=MMC, long=multiagent weighted maximum coverage, class=abbrev}
\DeclareAcronym{gmmc}{short=GMMC, long=general multiagent weighted maximum coverage, class=abbrev}

\newlist{acronyms}{description}{1}
 \setlist[acronyms]{
 labelwidth = 60pt,
 leftmargin = 61pt,
 itemsep= 5pt,
 itemindent = 5pt,}
\DeclareAcroListStyle{mystyle}{list}{ list = acronyms }
\acsetup{ list-style = mystyle,
 list-heading = subsection*}

\renewcommand{\vref}{\cref}
\usepackage{longtable}

\begin{document}
\frontmatter
\begin{center}
\thispagestyle{empty}
\large
Dissertation ETH Zurich No. 25597

\vspace*{2.1cm}

\Large
\textbf{Distributed control and game design:\\[-0.3cm]
{\small From strategic agents to programmable machines}
}

\vspace*{2cm}

\large
A dissertation submitted to attain the degree of\\[0.4cm]
Doctor of Sciences of ETH Zurich\\[0.4cm]
(Dr. sc. ETH Zurich)

\vspace*{1.5cm}

presented by \\[1.2cm]
Dario Paccagnan \\[0.2cm]

Dott. Magistrale, Universit\`a degli studi di Padova, Italy \\
M.Sc. Technical University of Denmark, Denmark\\[0.7cm]
born 19.01.1990 in Treviso \\
citizen of Italy\\[1.2cm]

accepted on the recommendation of \\
Prof. Dr. John Lygeros, examiner \\
Prof. Dr. Andreas Krause, co-examiner \\
Prof. Dr. Jason R. Marden, co-examiner \\[1cm]
2018
\end{center}


\thispagestyle{empty}
\newpage
\thispagestyle{empty}
\vspace*{\fill}

ETH Zurich

IfA - Automatic Control Laboratory

ETL, Physikstrasse 3

8092 Zurich, Switzerland\\

\copyright\ Dario Paccagnan, 2018

All Rights Reserved\\

ISBN 978-3-906916-47-7

DOI\ \,\,10.3929/ethz-b-000314981

\chapter*{}
\thispagestyle{empty}
{\large 
\vspace{-2cm}
\emph{~~Alla mia famiglia}

\vspace{4.5cm}
\hfill{\emph{Fatti non foste a viver come bruti,~~}}\\
\vspace{-1mm}
\hfill{\emph{ma per seguir virtute e canoscenza.}}\\[-3mm]

\hspace{7.3cm}
{ -- Dante Alighieri}
}
\newpage
\thispagestyle{empty}

\chapter*{Acknowledgements}
\addcontentsline{toc}{chapter}{Acknowledgements}
\setcounter{page}{1}
First and foremost, I would like to express my gratitude to my PhD Advisor, Prof. John Lygeros. While this thesis would not be in its current shape without his guidance and support, his contribution is far broader than that. Most of all, I would like to thank him for his unconditional trust, freedom and balance. On one hand, he allowed me to pursue any research direction I found exciting and relevant, on the other he has always been present to counsel me on a range of different topics - technical and non. 
A special thanks goes for always believing in me and for helping me to set and pursue my goals. Looking back at when I moved my first steps at our institute, I can safely say that I have grown not only on the technical side, but all around as a researcher.  Let me keep it simple: thank you for everything!

It was an incredible luck to work with Prof. Jason Marden who agreed to host me in Santa Barbara, without ever meeting with me before. I still remember when I walked in your office for the first time and I told you ``I have ten ideas to discuss with you''. You listened to me extremely carefully for one hour. As a result, few days later we begun working on a completely different topic. I am not sure whether you simply did not like any of the ideas or maybe thought I should have listened to yours first.  Nevertheless, from that point onwards, it has been an amazing journey and our collaboration has been one of the most exciting I have ever had. I would like to thank you for all the time you dedicated me, for you sharp-thinking, for your patience, and, amongst everything, for the enthusiasm you instilled in me: I hope this is just the beginning. 

Thanks to Prof. Andreas Krause for agreeing to serve on my PhD committee and for meeting with me to discuss the content of this thesis, \mbox{regardless of the tight schedule.}
 
I feel truly blessed and honored to have met and interacted with Prof. Maria Elena Valcher. Not only has she given me continuous support during the years, but she has shown me what it means to be an unparalleled advisor. Energy, passion, crystal-clear reasoning, unrestricted support to her students are only a few of her everyday's ingredients. Most importantly, I would like to thank her for being the first person that believed in me, for introducing me to control theory and for allowing me to take off at DTU. Thank you for being a lighthouse.
  
I am deeply indebted with the new Professors of our institute Prof. Florian D\"orfler and Prof. Maryam Kamgarpour. Thank you Florian for ``taking me climbing'', for all the adventures in the mountains we had together, for the advices you gave me when I was only a newbie at IfA, and for putting me in contact with Prof. Marden. Maryam, a sincere thanks goes to you for your support and  collaboration during the early stages of my PhD. I remember vividly how much help and enthusiasm you put in teaching me how to write my ``first article''. If all of this is possible, \mbox{it is in good part also thanks to you.}
  
The atmosphere we breath at IfA is absolutely incredible: a great mix of hard-working, positive energy, and friendship. Sometimes, bent over our papers, we seem to forget how lucky we are. Fortunately (or not), as soon as we step out of the ETH domain, we get reminded of what a privilege this is. I would like to express my greatest gratitude to all those people that made this place what it is. In this regard, there are many colleagues I should thank, but I limit myself to the essential. First of all, I would like to thank Basilio and Francesca for sharing with me a good half of their PhDs, and countless hours on the board conjecturing results that would be proved wrong only a few hours later.  I wish both of you a brilliant career as you truly deserve.
A very special thanks is due to Nicol\`o for sharing innumerable moments inside and outside of the office, for the many jokes we have been laughing of, and for the unmissable extra slice of cake (Lucia's). 
A big thanks goes to the Italian Gang - Basilio, Francesca, Nicol\`o, Giampaolo, Marcello, Saverio - for the countless adventures, and for making me feel as IfA was the extension of my time in high school. Thank you Tony for being such a positive, cheerful and selfless person. Thank you Peyman for the numerous research discussions we had: I feel there is a lot to learn form you.  
Thanks to the dream team - Juan, Marcello, Paul - for all the mountains adventures we shared: I will hardly forget when we set foot on the Dammazwillinge for the first time. I am confident you will not forget either. A special thanks goes to Paul for being a terrific office mate, and for sharing the first years of your PhD with me, often out in the Alps. At this point, I feel I owe a big thanks to our secretaries, Sabrina and Tanja, for making everything run incredibly smooth and for being extremely helpful with really anything. ``If there is a solution, we will find it'' Sabrina told me once, and this pretty much says it all.

At a personal level, I would like to thank all my friends in Zurich and elsewhere, in particular Chiara for the  gazillions hours we spent climbing together, Federica for being a terrific housemate, Martina and Alessandro for hosting me after my return from the US. A line of his own is needed (and probably not enough) to thank Alessio: a very special thought goes to you for all the things we went through together, way too many to recall. I will never forget the 20-hours day on the Allighesi ferrata, our first climbing route in the dolomites, or simply the usual giro di Ca' Tron. Amongst everything, I would like to thank you for teaching me to be ambitious, to fight, to never give up and to endure pain. Last, I want to thank you for always finding time to meet with me (the usual run) when I come back from Zurich, regardless of our busy lives.

At this stage, it feels quite difficult to find the right words to thank those who contributed the most to all of this. Thanks to the bottom of my heart to my family and in particular to my parents and brother. Mum and Dad, you have been a guiding force throughout all these years. I would like to thank you for the curiosity you instilled in me, for giving me a great education and making me understand its importance. I am incredibly grateful for the unconditional support I received, for the freedom you gave me in pursuing my own path, and for teaching me that great things happen only if you dare trying.  
You have set the bar quite high, and if I'll ever be able to be half as good as you have been, I'll consider myself satisfied. A warm thanks goes to my brother for showing me the way, for all the opportunities you gave me, and for being an example of what one can achieve with a ``little'' amount of constant effort. Perhaps I should thank you for those initial english lesson you gave me when I was only a newbie, but I am confident the idea of showing me your beloved ``green english textbook" already made up for that.

Finally, I owe a huge, sincere and heartfelt thanks to Erisa. First and foremost for teaching me (slowly but surely) the meaning of \emph{we}, for your love and patience, and for bearing with me ``doing science'', probably the hardest amongst all. On the latter topic, I am indebted with you for reading through my papers many times, for teaching me that ``things need to get done'' and, more in general, for being a great counsellor. Quite a few things have changed since the first time you ``forced'' me to buy the Zurich-Copenhagen ticket, but I am very excited about all that come next, \emph{together}.
\\
\begin{flushright}
	Dario Paccagnan\\
	Zurich, November 2018
\end{flushright}
\chapter*{Ringraziamenti}
\addcontentsline{toc}{chapter}{Ringraziamenti}
Innanzitutto, vorrei esprimere la mia gratitudine al mio PhD Advisor, il Prof. John Lygeros. Anche se questa tesi non sarebbe nella sua forma attuale senza la sua guida, il suo contributo non si limita certamente a questo. Vorrei ringraziarlo, in particolare, per la fiducia datami, la libert\`a e l'equilibrio. Da un lato, mi ha permesso di intraprendere qualsiasi direzione di ricerca trovassi interessante, dall'altro \`e sempre stato presente per consigliarmi su una serie di argomenti diversi - tecnici e non. 
Devo un ringraziamento speciale al Prof. Lygeros per aver sempre creduto in me e per avermi aiutato a fissare e raggiungere i miei obiettivi. Rivolgendo lo sguardo a quando ho mosso i primi passi nel nostro istituto, posso dire di essere cresciuto non solo dal punto di vista tecnico, ma anche come ricercatore: grazie di tutto!

E' stata una fortuna incredibile lavorare con il Prof. Jason Marden, il quale ha accettato di ospitarmi a Santa Barbara senza avermi incontrato prima. Ricordo bene quando sono entrato nel tuo ufficio per la prima volta dicendo ``Ho dieci idee che vorrei discutere con te''. Mi hai ascoltato con estrema attenzione per pi\`u di un'ora. Pochi giorni dopo, stavamo lavorando su un argomento completamente diverso. Forse non ti sono piaciute le mie idee, o probabilmente avrei dovuto ascoltare prima le tue.  Tuttavia, da quel momento in poi, \`e cominciato un bellissimo viaggio e la nostra collaborazione \`e stata una delle pi\`u entusiasmanti che abbia mai avuto. Vorrei ringraziarti per tutto il tempo che mi hai dedicato, per la tua pazienza e, sopratutto, per l'entusiasmo che mi hai infuso: spero questo sia solo l'inizio. 

Vorrei ringraziare il Prof. Andreas Krause per aver accettato il ruolo di esaminatore esterno e per avermi incontrato per discutere il contenuto di questa tesi.

Sono davvero onorato di aver incontrato e interagito con la Prof. Maria Elena Valcher. Non solo mi ha dato un sostegno continuo nel corso degli anni, ma mi ha anche mostrato cosa significa essere un Advisor senza pari. Energia, passione, ragionamento cristallino, sostegno ai suoi studenti, sono solo alcuni degli ingredienti della sua vita quotidiana. Vorrei ringraziarla, in particolare, per essere la prima persona che ha creduto in me, per avermi introdotto alla teoria del controllo e per avermi spinto ad \mbox{andare al DTU.}
  
Sono profondamente in debito con i nuovi professori del nostro istituto, il Prof. Florian D\"orfler e la Prof. Maryam Kamgarpour. Grazie Florian per "portarmi a scalare", per tutte le avventure in montagna che abbiamo avuto insieme, per i consigli che mi hai dato quando ero solo un novizio all'IfA, e per avermi messo in contatto con il Prof. Marden. Maryam, un sincero ringraziamento va a te per il tuo sostegno e la tua collaborazione durante le prime fasi del mio dottorato. Ricordo molto bene quanto entusiasmo ed energia hai messo nell'insegnarmi a scrivere il mio ``primo articolo''. Se tutto questo \`e possibile, \`e in buona parte anche grazie a te.

L'atmosfera che respiriamo all'IfA \`e assolutamente incredibile: un perfetto mix di lavoro, energia positiva e amicizia. A volte, piegati sui nostri articoli, finiamo per dimenticarci di quanto siamo coccolati. Fortunatamente (o forse no), non appena usciamo dal nostro ufficio, ci accorgiamo immediatamente di che privilegio sia questo. Vorrei esprimere la mia pi\`u grande gratitudine a tutte le persone che hanno reso questo luogo quello che \`e. A questo proposito, mentre dovrei ringraziare molti colleghi, mi limiter\`o solamente all'essenziale. Prima di tutto, vorrei ringraziare Basilio e Francesca per aver condiviso con me una buona met\`a dei loro dottorati, e innumerevoli ore alla lavagna a congetturare risultati che si sarebbero rivelati errati solo poche ore dopo.  Auguro ad entrambi una brillante carriera.
Un ringraziamento speciale va a Nicol\`o per aver condiviso moltissimi momenti dentro e fuori dall'ufficio, per le tante battute di cui abbiamo riso e per l'immancabile fetta di torta in pi\`u (confezionata da Lucia).
Un sentito ringraziamento va all'Italian Gang - Basilio, Francesca, Nicol\`o, Giampaolo, Marcello, Saverio - per le innumerevoli avventure, e per avermi fatto sentire come se IfA fosse la continuazione del liceo. Grazie Tony per essere una persona cos\`\i\ positiva, allegra ed altruista. Grazie Peyman per le numerose discussioni che abbiamo avuto: ho l'impressione che ci sia molto da imparare da te.  
Grazie al Dream Team - Juan, Marcello, Paul - per tutte le avventure in montagna che abbiamo vissuto insieme: non dimenticher\`o mai quando mettemmo piede sul Dammazwillinge per la prima volta. Sono convinto non lo dimenticherete neanche voi. Un ringraziamento speciale va a Paul per essere stato un fantastico compagno di ufficio e per aver condiviso con me, spesso sulle Alpi, i primi anni in IfA. A questo punto, \`e d'obbligo un ringraziamento alle nostre segretarie, Sabrina e Tanja, per aver fatto funzionare tutto in maniera impeccabile: ``se c'\`e una soluzione, la troveremo" ha detto Sabrina pi\`u di una volta.

A livello personale, vorrei ringraziare tutti i miei amici di Zurigo e non solo, in particolare Chiara per le innumerevoli ore passate assieme arrampicando, Federica per essere stata una fantastica compagna di casa, Martina e Alessandro per avermi ospitato dopo il ritorno dagli Stati Uniti. Devo un pensiero molto speciale (e comunque non sufficiente) ad Alessio per tutte le cose che abbiamo vissuto insieme, troppe per ricordarle tutte. Non dimenticher\`o mai la giornata di 20 ore sulla ferrata Allighesi, la nostra prima via d'arrampicata nelle Dolomiti, o semplicemente il solito giro di Ca' Tron. Vorrei ringraziarti sopratutto per avermi insegnato ad essere ambizioso, a combattere, a non arrendermi mai e a sopportare la fatica. Infine, ti devo ringraziare per aver sempre trovato il tempo di incontrarmi al mio ritorno da Zurigo (la solita corsa), indipendentemente dagli innumerevoli impegni.

Giunto a questo punto, risulta difficile trovare le parole giuste per ringraziare coloro che hanno contribuito maggiormente a tutto questo. Grazie di cuore alla mia famiglia, in particolare ai miei genitori ed a mio fratello. Mamma e pap\`a, siete stati una forza trainante in tutti questi anni. Vorrei ringraziarvi per la curiosit\`a che mi avete instillato, per avermi dato un'educazione invidiabile e per avermene fatto capire l'importanza. Sono incredibilmente grato per il sostegno incondizionato che mi avete dato in tutti questi anni, per la libert\`a che ho ricevuto nel seguire la mia strada, e per avermi insegnato che le cose pi\`u belle accadono solo a chi ha il coraggio di provare. Avete posto l'asticella molto in alto, e se mai riuscir\`o a raggiungerne la met\`a, mi considerer\`o soddisfatto. Un caloroso ringraziamento va a Diego per avermi indicato la strada, per tutte le opportunit\`a che mi hai offerto, e per essere un esempio di ci\`o che si pu\`o ottenere con un ``piccolo'' sforzo quotidiano. Forse dovrei ringraziarti per quelle prime lezioni di inglese che mi hai dato quando ero solo un principiante, ma sono sicuro che il piacere di mostrarmi il tuo amato ``libro verde'' ti ha gi\`a ripagato.

Infine, devo un enorme, sincero e sentito grazie ad Erisa. Innanzitutto per avermi insegnato il significato della parola \emph{noi}, per il tuo amore e pazienza, e per sopportarmi quando sono ``impegnato con la scienza'', probabilmente la pi\`u difficile tra tutte le prove. Ti devo un ringraziamento speciale per aver riletto pi\`u e pi\`u volte i miei articoli, per avermi insegnato che ``\`e inutile perdere tempo a cincischiare, bisogna finire il lavoro'' e, pi\`u in generale, per essere un grande consigliere. Molte cose sono cambiate dalla prima volta che mi hai ``costretto'' a comprare il biglietto Zurigo-Copenhagen, e sono elettrizzato per tutto ci\`o che ci aspetta, \emph{insieme}.
\\
\begin{flushright}
	Dario Paccagnan\\
	Zurigo, Novembre 2018
\end{flushright}

\renewcommand\i{^i}
\renewcommand\j{^j}
\chapter*{Abstract}
\addcontentsline{toc}{chapter}{Abstract}

Large scale systems are forecasted to greatly impact our future lives thanks to their wide ranging applications including cooperative robotics, mobility on demand, resource and task allocation, supply chain management, and many more. 
While technological developments have paved the way for the realization of such futuristic systems, we have a limited grasp on how to coordinate the behavior of their individual components to achieve the desired global objective.

With the objective of advancing our understanding, this thesis focus on the analysis and coordination of large scale systems without the need of a centralized authority. 
At a high level, we distinguish these systems depending on wether they are composed of \emph{cooperative} or \emph{non-cooperative} subsystems. 
In regard to the first class, a key challenge is the design of local decision rules for the individual components to guarantee that the collective behavior is desirable with respect to a global objective. Non-cooperative systems, on the other hand, require a more careful thinking in that the designer needs to take into account the self-interested nature of the agents.
 In both cases, the need for distributed protocols stems from the observation that centralized decision making is prohibited due to the scale and privacy requirement associated with typical systems.

In the first part of this thesis, we focus on the coordination of a large number of \emph{non-cooperative} agents. More specifically, we consider strategic decision making problems where each agent's objective is a function of the aggregate behavior of the population. Examples are ubiquitous and include social and traffic networks, demand-response markets, vaccination campaigns, to name just a few. We present two cohesive contributions.
First, we compare the performance of an equilibrium allocation with that of an optimal allocation, that is an allocation where a common welfare function is maximized.  We propose conditions under which all Nash equilibrium allocations are efficient, i.e., are desirable from a macroscopic standpoint. In the journey towards this goal, we prove a novel result bounding the distance between the strategies at a Nash and at a Wardrop equilibrium that might be of independent interest.  
Second, we show how to derive scalable algorithms that guide agents towards an equilibrium allocation, i.e., a stable configuration where no agent has any incentive to deviate. When the corresponding equilibria are efficient, these algorithms attain the global objective and respect the agents' selfish nature.   

In the second part of this thesis, we focus on the coordination of \emph{cooperative} agents. We consider large-scale resource allocation problems, where a number of agents need to be allocated to a set of resources, with the goal of jointly maximizing a given submodular or supermodular set function. Applications include sensor allocation problems, distributed caching, data summarization, and many more.
Since this class of problems is computationally intractable, we aim at deriving tractable algorithms for attaining approximate solutions, ideally with the best possible approximation ratio. We approach the problem from a game-theoretic perspective and ask the following question: how should we design agents' utilities so that any equilibrium configuration recovers a large fraction of the optimum welfare? 
In order to answer this question, we introduce a novel framework providing a tight expression for the worst-case performance (price of anarchy) as a function of the chosen utilities. Leveraging this result, we show how to design utility functions so as to optimize the price of anarchy by means of a tractable linear program. The upshot of our contribution is the design of algorithms that are distributed, efficient, and whose performance is certified to be on par or \mbox{better than that of existing (and centralized) schemes.}

\chapter*{Sommario}
\addcontentsline{toc}{chapter}{Sommario}
I sistemi tecnologici su larga scala promettono di migliorare sensibilmente la qualit\`a della nostra vita futura grazie alle loro numerose applicazioni, tra cui la robotica cooperativa, la mobilit\`a su richiesta, l'allocazione di risorse, la gestione della supply chain.
Nonostante gli sviluppi tecnologici abbiano aperto la strada alla realizzazione di questi sistemi futuristici, abbiamo una conoscenza limitata su come coordinare i singoli componenti per ottenere l'obiettivo macroscopico desiderato.

Questa tesi si concentra sull'analisi e il coordinamento di sistemi su larga scala privi di un'autorit\`a centralizzata, con l'obiettivo di migliorarne la comprensione ed il funzionamento. Ad alto livello, distinguiamo questi sistemi a seconda che essi siano cooperativi o meno. 
Una sfida chiave in relazione ai sistemi cooperativi \`e la progettazione di algoritmi di controllo per le singole componenti che garantiscano il raggiungimento di un predeterminato obiettivo globale. I sistemi non cooperativi, d'altra parte, richiedono una maggiore attenzione in quanto \`e necessario tenere in considerazione la natura egoistica degli agenti. In entrambi i casi, l'utilizzo di protocolli distribuiti \`e reso necessario dalle dimensioni di tali sistemi e dai requisiti di privacy che vi sono associati.
 
 Nella prima parte di questa tesi, ci concentriamo sul coordinamento di sistemi non cooperativi. Pi\`u specificamente, consideriamo problemi strategici in cui l'obiettivo di ciascun agente \`e influenzato del comportamento aggregato della popolazione. Esempi di tali sistemi comprendono i social networks, le reti stradali, i mercati azionari. Nel seguito presentiamo due risultati coesivi.
In primo luogo, confrontiamo la performance di un'allocazione di equilibrio con la performance di un'allocazione ottimale, cio\`e di un'allocazione in cui viene massimizzata una funzione obiettivo comune.  Proponiamo poi condizioni che garantiscono l'efficienza di tutte le allocazioni di equilibrio. Nel percorso verso questo obiettivo, otteniamo un risultato che delimita la distanza tra gli equilibri di Nash e Wardrop e che potrebbe essere di interesse indipendente.  
In secondo luogo, progettiamo algoritmi scalabili che guidano gli agenti verso un'allocazione di equilibrio, cio\`e una configurazione stabile in cui nessun agente ha alcun incentivo a deviare. Quando tali equilibri sono efficienti, questi algoritmi raggiungono l'obiettivo globale e rispettano la natura individualistica degli agenti.

Nella seconda parte di questa tesi, ci concentriamo sul controllo di sistemi cooperativi. In particolare, consideriamo problemi di allocazione delle risorse su larga scala, dove un insieme di risorse deve essere assegnato ad un fissato numero di agenti, con l'obiettivo di massimizzare una funzione obiettivo globale, submodulare o supermodulare. Le applicazioni includono problemi di allocazione dei sensori, caching distribuito, data summarization e molto altro ancora.
Poich\'e questa classe di problemi \`e intrattabile dal punto di vista computazionale, ci prefiggiamo di ricavare soluzioni approssimate con algoritmi efficienti, idealmente con il miglior rapporto di approssimazione possibile. Formuliamo questo problema con il linguaggio della teoria dei giochi e ci poniamo la seguente domanda: come progettare le funzioni obiettivo da assegnare agli agenti in modo che ogni configurazione di equilibrio produca la massima frazione del valore ottimo? 
Per rispondere a questa domanda, introduciamo un nuovo metodo per calcolare in maniera esatta 
la qualit\`a di un equilibrio in relazione alle funzioni obiettivo scelte (price of anarchy). Sfruttando questo risultato, mostriamo come costruire tali funzioni obiettivo in modo da massimizzare la performance dei corrispondenti equilibri grazie ad un programma lineare ausiliario. Il risultato finale \`e la progettazione di algoritmi distribuiti ed efficienti, il cui rapporto di approssimazione \`e alla pari o superiore a quello di molti schemi (centralizzati) comunemente usati.


\tableofcontents

\chapter*{Notation}
\addcontentsline{toc}{chapter}{Notation}
\printacronyms[include-classes=abbrev]

\subsection*{Symbols}
\vspace*{-7mm}
\begin{longtable}[l]{l l}
$\defeq$ & equal by definition\\
$\mb{N}$, $\mb{N}_0$ & set of natural numbers, set of natural numbers including zero\\
$\mb{R}$, $\mb{R}_{>0}$, $\mb{R}_{\ge0}$ & set of real, positive real, non negative real numbers\\
$[p]$, $[p]_0$ & set of integers $\{1,\dots,p\}$, set of integers v$\{0,1,\dots,p\}$\\
$[a,b]$ & interval of real numbers $x\in\mb{R}$ with $a\le x\le b$\\
$\ones[n]$, $\zeros[n]$, $e_i \in\R^n$ & 
 vector of unit entries, vector of zero entries, $i^\text{th}$ canonical vector\\
 $I_n$ & identity matrix $I_n\in\mb{R}^{n\times n}$\\
 $A\succ 0$ $(\succeq0)$ & positive definite (semi-) $A\in\mb{R}^{n\times n}$, i.e., $x^\top A x>0~(\ge0),$ $\forall x\neq 0$\\
 $\|x\|$ & 2-norm of $x\in\mb{R}^n$ \\
 $\|A\|$ & induced 2-norm of $A\in\mb{R}^{n\times n}$,  $\|A\|:=\sup_{x\neq 0} \frac{\|Ax\|}{\|x\|}$ \\ 
 $\lambda_{\textup{min}}(A)$, $\lambda_{\textup{max}}(A)$ & minimum, maximum eigenvalue of the symmetric matrix $A\in\mb{R}^{n\times n}$\\
 $[A]_{ij}=A_{ij},$ & element in position $(i,j)$ of the matrix $A$\\
 $A\otimes B$ & Kronecker product of the matrices $A,B$\\
 $[x^i]_{i=1}^m$ & stacked vector $[x^i]_{i=1}^m\defeq[(x^1)^\top,\ldots ,(x^m)^\top]^\top=[x^1;\ldots;x^m],$ $x^i\in\mb{R}^{n\times 1}$\\
$\Pi_{\mc{X}}(y)$ & metric projection of $y\in\mb{R}^n$ onto $\mc{X}\subseteq \mb{R}^n$, see \cref{def:proj}\\
$f(x)=\mc{O}(g(x))$ & big O notation: $\lim_{x\rightarrow \infty} \frac{f(x)}{g(x)}=0$\\
$f'(x)$& derivative of differentiable $f:\mb{R}\rightarrow \mb{R}$\\
$\nabla_x f(x) \in \R^{n\times m}$ & Jacobian of differentiable $f:\R^n\rightarrow \R^m$,  i.e., $[\nabla_x f(x)]_{ij}\coloneqq \frac{\partial g_j(x)}{\partial x\i}$
\\
$\sum_{i=1}^n \mc{X}^i$ & Minkowski sum of the sets $\{\mc{X}_i\}_{i=1}^n$\\
VI$(\mc{X},F)$ & variational inequality with set $\mc{X}$ and operator $F$, see \cref{def:vi}\\
$\mathcal{U}[a, b]$ & uniform distribution on the real interval $[a, b]$\\
$\ones[\{f(x)\ge0\}]$ & indicator function of the set $\{x\in\mb{R}^n~\text{s.t.}~f(x)\ge0\}$, $f:\mb{R}^n\rightarrow\mb{R}$\\
$|S|$ & cardinality of the (finite) set $S$
\end{longtable}
\vspace*{-2mm}
\subsection*{Reserved symbols}
\vspace*{-1mm}
Part I
\begin{longtable}[l]{l l}
$M$ \hspace*{25mm}& number of players\\
$n$ & dimension of players' strategy vectors\\
$x\i\in\mb{R}^n$ & strategy vector of player $i\in[M]$\\
$x^{-i}\in\mb{R}^{n(M-1)}$ & strategy vector of all players but $i\in[M]$\\
$\mc{X}\i\subseteq\R^n$ & local constraint set  of player $i\in[M]$\\
$\mc{X}$ & product of local constraint sets $\mc{X}\defeq\mc{X}^1\times \dots \times \mc{X}^M$\\
$\mc{C}\subseteq\R^{Mn}$ & coupling constraint set\\
$\sigma(x)\in\R^n$ & average of strategies $\sigma(x)\defeq \frac{1}{M}\sum_{i=1}^M x\i$\\
$J\i(x\i,\sigma(x))$ & cost function of player $i\in[M]$, $J\i:\mb{R}^n\times \mb{R}^n\rightarrow \R$\\
$\mc{G}$ & aggregative game defined in \eqref{eq:GNEP}\\
$x\NE$, $x\WE$ & Nash, Wardrop equilibrium according to \cref{def:NE,def:WE}\\
$J_S(\sigma(x))$ & social cost function, $J_S:\mb{R}^n\rightarrow \R$, see \cref{def:socopt}\\[-2mm]
\end{longtable}
\noindent Part II
\begin{longtable}[l]{l l}
$e$ \hspace*{27mm}& Euler's number\\
$\mc{R}$ & set of resources $\mc{R}=\{r_1,\dots,r_m\}$\\
$v_r$ & value of resource $r\in\mc{R}$, $v_r\in\mb{R}_{\ge 0}$\\
$n$ & number of agents\\
$a_i$ & strategy of player $i\in [n]$\\
$\mc{A}_i$ & strategy set of player $i\in [n]$\\
$\mc{A}$ & product of agents' strategy sets $\mc{A}\defeq \mc{A}_1\times\dots\times\mc{A}_n$\\
$W(a)$& welfare function $W:2^{\mc{R}}\times\dots\times2^{\mc{R}}\rightarrow\mb{R}_{\ge0}$\\
$|a|_r$ & number of agents selecting resource $r$ in allocation $a$\\
$w(j)$ & welfare basis function $w:[n]\rightarrow \R_{\ge0}$\\
$u_i(a)$ & utility function of agent $i\in[n]$, $u_i:\mc{A}\rightarrow\mb{R}$\\
$f(j)$ & distribution rule $f:[n]\rightarrow \R_{\ge0}$\\
$\ae$& pure Nash equilibrium strategy according to \cref{def:nashequilibrium}\\
$\fgar$, $\fsv$, $\fmc$ & distribution rules introduced in \eqref{eq:fgar} and \cref{def:svmc}
\end{longtable}
\subsection*{Complexity classes}
\begin{longtable}[l]{l l}
$\pclass$ \hspace*{26mm}& deterministic polynomial time class\\
$\npclass$ & nondeterministic polynomial time class\\
$\pls$ & polynomial local search class\\
$\ppad$ & polynomial parity arguments on directed graphs class
\end{longtable}

\mainmatter 
\chapter{Overview}
\label{ch:overview}

 Large scale systems have enormous potential for solving many of the current societal challenges. Robotic networks can operate in post-disaster environments and reduce the impact of nuclear, industrial or natural calamities \cite{kuntze2012seneka, kitano1999robocup}. Fleets of autonomous cars are forecasted to revolutionize the future mobility and to reduce traffic congestion as well as pollutant emissions \cite{spieser2014toward}.
 Demand-response schemes have the potential to allow for the integration of a large share of renewable resources \cite{motalleb2016providing}.
On a smaller scale, swarms of ``microbots'' promise groundbreaking results in medicine by means of local drug delivery \cite{servant2015controlled} or microsurgery \cite{ishiyama2002magnetic}. 

While all the above-mentioned systems (and many more) can be thought of as a collection of multiple \emph{subsystems} or \emph{agents}, we distinguish them in two categories depending on wether the corresponding subsystems \emph{are} or \emph{are not} cooperative. An example of cooperative system is that of a drones swarm performing a rescue mission. On the other hand, privately-owned self driving cars are non-cooperative, since each car's objective is that of reaching its destination as swiftly as possible, while respecting the traffic rules. Another example of non-cooperative large scale system is the electricity reserve market, where generators sell their ability to increase or decrease their electricity production to the system operator, whose ultimate objective is that of \mbox{balancing production and consumption.}

\begin{wrapfigure}[12]{r}{0.35\textwidth}
\centering
\vspace*{-0.4cm}
\includegraphics[scale=0.24]{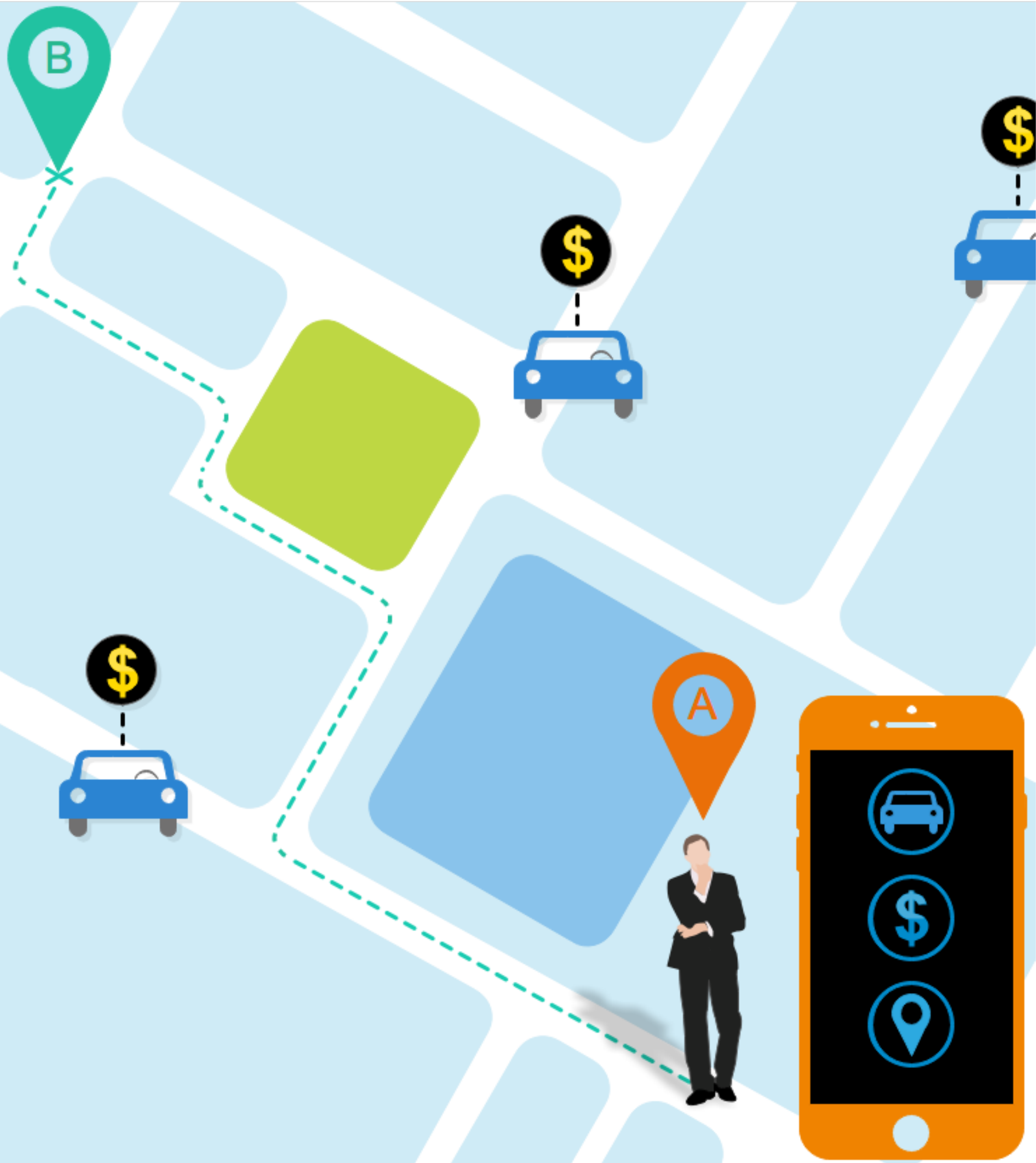}
\vspace*{-1mm}
\caption{}
\label{fig:uber} 
\end{wrapfigure}

One of the \emph{main challenges} in the operation of both types of systems is the design of \emph{local decision rules} for the individual subsystems to guarantee that the collective behavior is desirable with respect to a \emph{global objective} \cite{li2013designing}. With this respect, non-cooperative systems pose an additional layer of difficulty, in that preexisting local objectives might not be aligned with the system-level goal. As a concrete example consider on-demand ridesharing platforms such as Uber, Lyft or Didi, where agents are represented by human drivers hovering around different city neighbors (\cref{fig:uber}). 
While drivers' might position themselves in neighbors that maximize their own profit, the system-operator might have a different goal, e.g., to guarantee a minimum coverage of the city. For this class of systems, it is not sufficient to design local decision rules to be executed from each subsystem, but it is equally important to \emph{incentivize} their adoption.

The spatial distribution, privacy requirements, scale, and quantity of information associated with typical systems do not allow for centralized communication and decision making, but require instead the use of \emph{distributed protocols}.
In addition to the above requirements, designing such protocols is non-trivial due to the presence of heterogenous decision makers and informational constraints. To complete the overview, we note that the quality of a control architecture is usually gauged by several metrics including the satisfaction of the global objective, the robustness to external disturbances, as well as the amount of information propagated through the communication network.

\vspace*{5mm}
\noindent \fbox{\parbox{0.985\textwidth}{
The {\bf goal} of this thesis is to address the challenges previously discussed, with particular attention to the design of local decision rules in relation to their corresponding system-level performance. While in the first part of the thesis we focus on large scale \emph{non-cooperative} system, we exploit some of the insight obtained therein to address, in the second part of the thesis, large scale \emph{cooperative} system and propose novel efficient distributed algorithms.}}

\section{Outline and contributions}

\subsection{Part I: strategic agents}
In the first part of the thesis we focus on large scale systems composed of \emph{strategic} agents. Specifically, we consider the framework of \emph{average aggregative games}, where each agent aims at minimizing a cost function that depends both on his decision and on the average population strategy. Our objective is twofold. First, we wish to understand to what extent selfish decision making reduces the system performance. We do this using the notion of price of anarchy. Second, we aim at the design of scalable and decentralized algorithms that provably converge to a Nash or a Wardrop equilibrium. We achieve this leveraging the theory of variational inequalities. 

\vspace*{\myspace}
{ \noindent \bf Outline.}
\cref{ch:p1introduction} provides an informal introduction to the framework of aggregative games and describes how such games can be used to model applications pertaining to various fields. At the end of the chapter we review the existing literature and connect our work with it. 
In \cref{ch:p1mathpreliminaries} we review the mathematical tools needed throughout \cref{part:1} of this thesis. 
In \cref{ch:p1equilibriaVIandSMON} we formalize the notions of Nash and Wardrop equilibria in the presence of coupling constraints, and use the language of variational inequalities to reformulate these problems. We conclude the chapter studying the monotonicity properties of the variational inequality operator associated with the corresponding equilibrium problems. 
In the first part of \cref{ch:p1distanceWENEandPOA} we study the relation between Nash and Wardrop equilibria with particular attention to the distance between the corresponding strategies. In the second part of this chapter, we leverage these results to bound the performance degradation incurred when moving from a centralized solution to strategic decision making.
In \cref{ch:p1distribtuedalgorithms} we present a best-response algorithm and a gradient-based algorithm that provably converge to a Nash or Wardrop equilibrium.
We conclude \cref{part:1} with \cref{ch:p1applications}, where we demonstrate the results previously obtained to two large scale applications. 

\vspace*{\myspace}
 {\noindent \bf Contributions.}
 The main contributions of \cref{part:1} of this thesis are contained in \cref{ch:p1equilibriaVIandSMON,,ch:p1distanceWENEandPOA,,ch:p1distribtuedalgorithms,,ch:p1applications}, and are detailed in the following.
 \begin{enumerate}
 \item In \cref{ch:p1equilibriaVIandSMON} we introduce the notion of Wardrop equilibrium as a condition on the agents' strategies, rather than a condition on the aggregate behaviour. This allows to address a larger class of equilibrium problems, compared to the existing literature. We then study the relation between Nash and Wardrop equilibrium strategies and show that, in a game with $M$ players, their euclidean distance is upper bounded by $\mathcal{O}({1}/{\sqrt{M}})$ when one of the corresponding variational inequality is strongly monotone (\cref{thm:conv_strategies}). This allows us to provide guarantees on the efficiency of Nash equilibria by studying the efficiency of the corresponding Wardrop equilibria (\cref{thm:lin,,thm:polypoa,,thm:routing}).
 \item In \cref{ch:p1distribtuedalgorithms} we present a best-response and a gradient-based algorithm that allow to compute a Nash or a Wardrop equilibrium in the presence of coupling constraints (\cref{thm:conv_two,,thm:convergence_asp}).
 \item In \cref{ch:p1applications} we apply the theoretical results previously derived to i) coordinate the charging profile of a population of electric vehicles, and ii) to predict the travel time distribution for a road traffic network. The results we obtain both in terms of equilibrium efficiency and algorithmic convergence are novel.
 \end{enumerate}

\subsection{Part II: programmable machines}
In the second part of the thesis we focus on the control of large scale systems composed of multiple \emph{cooperative} subsystems.
We assume that each subsystem (agent) is endowed with computation and communication capabilities, and we aim at achieving a global objective through local coordination of the agents. 
More specifically, we consider a class of combinatorial allocation problems, where each agent selects a subset of resources with the goal of jointly maximizing a given welfare function, additive over the resources. 
Since this class of problems is intractable, we seek distributed algorithms that run in polynomial time and give provable approximation guarantees. 
Rather than directly specifying a decision making process, we adopt the \emph{game design} approach, and assign to each agent a local utility function. The fundamental question we seek to answer in this part of the thesis is how to design local utility functions so that their selfish maximization recovers a large fraction of the optimal welfare.
  
\vspace*{\myspace}
{ \noindent \bf Outline.}
  In \cref{ch:p2introduction} we introduce the problem considered, discuss potential applications as well as related works. 
  In \cref{ch:p2mathpreliminaries} we review the mathematical tools needed for the development of our work.
 In \cref{ch:p2utilitydesign} we formulate the utility design question and tackle it in two steps. First, we provide performance certificates for a given set of utility functions; second, we show how to design utilities that maximize the corresponding worst-case performance.
 In \cref{ch:p2subsupercov} we specialize the results to a class of submodular, supermodular and maximum coverage problems. 
 Finally, in \cref{ch:p2applications} we present two applications: the vehicle-target assignment problem and a coverage problem arising in distributed caching for mobile networks.
 
 \vspace*{\myspace}
 {\noindent \bf Contributions.}
 The main contributions of \cref{part:2} of this thesis are contained in \cref{ch:p2utilitydesign,,ch:p2subsupercov,,ch:p2applications} and are detailed in the following.
 \begin{enumerate}
 \item In \cref{ch:p2utilitydesign} we pose the utility design problem and adopt the notion of price of anarchy as the worst-case performance metric. We show that traditional approaches used to quantify such performance metric are rather conservative and are not suited for the design problem considered here (\cref{thm:smoothnottight}). Motivated by this shortcoming, we propose a novel framework to \emph{compute} (\cref{thm:primalpoa,,thm:dualpoa}) and \emph{optimize} (\cref{thm:optimizepoa}) the price of anarchy as a function of the given utilities. In particular, we show that the utility design problem can be reformulated as a tractable linear program.
The upshot of this contributions is the possibility to apply the game design procedure to a broad class of problems. To the best of our knowledge, this is the first approach that allows to systematically compute and optimize the price of anarchy.

\item In \cref{ch:p2subsupercov} we specialize the previous results to the case of submodular, maximum coverage, and supermodular problems. 
Relative to the submodular case, we obtain a novel and fully explicit expression for the price of anarchy (\cref{thm:wconcavefwdecreas}). We further apply this result to determine the exact price of anarchy for the Shapley value and marginal contribution design methodologies (\cref{cor:SVandMC}). These results are compared with previous (non tight and fragmented) results from the literature, and are placed in the larger context of submodular maximization subject to matroid constraints. Relative to the class of problems considered, we show how optimally-designed utilities provide an approximation ratio superior to the best known ratio $1-c/e$ of \cite{sviridenko2017optimal}.\footnote{The result of \cite{sviridenko2017optimal} improves on the $(1-e^{-c})/c$ of \cite{conforti1984submodular}, where $c$ is the (total) curvature of the welfare function \cite{conforti1984submodular} and $e$ the Euler's number.} Relative to the case of multiagent maximum coverage problems, we obtain a novel analytical expression for the price of anarchy (\cref{thm:poageneralcovering}), and subsume previous results in \cite{gairing2009covering, paccagnan2017arxiv}. Optimally designed utilities achieve a $1-1/e$ approximation, the best possible \cite{feige1998threshold}. We conclude the chapter providing a tight expression for the price of anarchy in the case of supermodular welfare function (\cref{thm:convex}), and show that our result complements \cite{jensen2018, phillips2017design}. Limitedly to this case, we observe that optimally-designed utility functions provide a rather poor approximation ratio.

\item In \cref{ch:p2applications} we test the performance of the proposed algorithms on a task-allocation problem, and on a coverage problem arising in distributed mobile networks. We provide thorough simulation results and show the theoretical and numerical advantages of our approach.
 \end{enumerate}

\section{Publications}
\nocite{part1Paccagnan2018,part2Paccagnan2018,paccagnan18efficiency,gentile2017nash,paccagnan2017arxiv2,paccagnan2017arxiv,gentile2017novel,paccagnan2017risks,burger2017guarantees,paccagnan2017coupl,paccagnan2016aggregative,paccagnan2015ontherange,Paccagnan17book,jorgensen2014IMU}

This thesis contains a subset of the results derived during the author's studies as PhD student at ETH Zurich, all of which have already been published or submitted for publication. The corresponding articles on which this thesis is based are listed below.

\subsection{Part I: strategic agents}
The relations between Nash and Wardrop equilibria presented in \cref{ch:p1distanceWENEandPOA}, the algorithms developed in \cref{ch:p1distribtuedalgorithms} as well as the numerical simulations included in \cref{ch:p1applications} were developed in collaboration with B. Gentile, F. Parise, M. Kamgarpour and J. Lygeros. The results on the equilibrium efficiency featured in \cref{ch:p1distanceWENEandPOA} were derived with the help of F. Parise and J. Lygeros.
\begin{refcontext}[sorting=ydnt]
\makenamesbold
\defbibfilter{aggregative}{ keyword=aggregative }
\printbibliography[filter=aggregative,heading=none,title=none]
\end{refcontext}
%
\subsection{Part II: programmable machines}
The utility design approach presented in \cref{ch:p2introduction}, the characterization and optimization of the price of anarchy presented in \cref{ch:p2utilitydesign,,ch:p2subsupercov} were developed in collaboration with J.R. Marden, with the additional help of R. Chandan limitedly to \cref{thm:smoothnottight}. The approach, the theoretical findings as well as the numerical studies (\cref{ch:p2applications}) are published in the following papers. 
\begin{refcontext}[sorting=ydnt]
\makenamesbold
\defbibfilter{gamedesign}{ keyword=gamedesign }
\printbibliography[filter=gamedesign,heading=none,title=none]
\end{refcontext}
%
\subsection{Other publications}
The following papers were published by the author during his doctoral studies, but are not included in this dissertation:
\subsubsection*{Aggregative games and applications}
\begin{refcontext}[sorting=ydnt]
\makenamesbold
\defbibfilter{others_aggregative}{ keyword=others_aggregative }
\printbibliography[filter=others_aggregative,heading=none,title=none]
\end{refcontext}
\subsubsection*{Utility Design}
\begin{refcontext}[sorting=ydnt]
\makenamesbold
\defbibfilter{others_gamedesign}{ keyword=others_gamedesign }
\printbibliography[filter=others_gamedesign,heading=none,title=none]
\end{refcontext}
\subsubsection*{Others}
\begin{refcontext}[sorting=ydnt]
\makenamesbold
\defbibfilter{others}{ keyword=others }
\printbibliography[filter=others,heading=none,title=none]
\end{refcontext}

\newlength\figureheight 
\newlength\figurewidth 

\newcommand{\N}{M}
\part{Strategic agents: aggregative games}
\label{part:1}
\chapter{Introduction}
\label{ch:p1introduction}
In the first part of the thesis we consider large scale systems composed of \emph{mutual influencing} and \emph{strategic} agents. We use the term ``mutual influencing'' to describe the fact that agents' actions have influence on one another, while the term ``strategic'' captures the self-interested nature of the agents. As an example, consider that of traders in the stock exchange market. In a simplistic setup, each trader's goal is to maximize his profit by carefully buying and selling various financial products. At the same time, the value of one such product depends on what action the other traders take, making the final outcome difficult to predict. While this is only one example, similar scenarios arise in a number of real-life applications ranging from road traffic network to opinion dynamics and even missile defense or racing cars.
A setup in which multiple agents behave strategically and influence each others' objectives is typically referred to as a game, and the corresponding field of study termed \emph{game theory}.\footnote{The terminology derives from the fact that chess, poker, go, and many other board games are prototypical examples of strategic decision making.}
 
Game theory originated as a set of tools to model the interaction of selfish decision makers and has been given formal recognition as an independent research area thanks to the pioneering work of Von Neumann \cite{neumann1928theorie} and to the celebrated existence result of Nash \cite{Nash01011950}. With the modern terminology of game theory, a game is fully specified by four elements:
\begin{itemize}
\item[-]{\bf players or agents:} these are the decision makers, e.g., the traders in the stock exchange market. 
In the following we identify each agent with an index $i\in\{1,\dots,\N\}$.
\item[-]{\bf strategy or action sets}: these are the actions available to each agent, e.g., which financial products a trader can buy/sell, in what amount,  and when. 
In the following we denote with $\mc{X}\i$ the set containing all the possible actions available to agent $i\in\{1,\dots,\N\}$.
\item[-]{\bf utilities or cost functions}: a measure that quantifies whether the goal of each agent has been satisfied and to what extent. This is typically captured through a function mapping an element of the joint action space $\mc{X}\defeq\mc{X}^1\times\dots\times \mc{X}^\N$ to a real number. In the following we concentrate on cost minimization games and thus introduce the function $J\i:\mc{X}\rightarrow \R$ representing the cost incurred by player $i\in\{1,\dots,\N\}$, e.g., the negative profit incurred by each trader in the stock market.\footnote{Observe that any cost minimization game can be transformed in a utility maximization game upon reversing the sign of the cost functions.} 
\item[-]{\bf equilibrium concept}: while player $i$ aims at minimizing his cost function $J^i$, this function depends on both $x^i\in\mc{X}^i$ and the choices of all the other agents, typically referred to as $x^{-i}$.
Thus, we need to define what is a descriptive outcome of the game. This concept is captured by the notion of equilibrium, the most celebrated of which is known as Nash equilibrium. Informally, a joint strategy $x\NE\in\mc{X}$ is a (pure) Nash equilibrium, if no agent can lower his cost by unilaterally changing his action. 
\end{itemize}
Building on these foundation, we wish to introduce an additional dimension to the problem and to address games with a large number of players. This is motivated by the observation that a relevant number of applications are indeed large scale, e.g., stock exchange markets, road traffic networks, online advertising, and many more.
The first difficulty that one is faced with, when thinking about these \emph{large system}, is that of complexity or more formally \emph{tractability}. In order to alleviate this issue, in the remainder of \cref{part:1} of this thesis, we will consider \emph{aggregative games}. Aggregative games are games where the cost function of each agent does not directly depend on the choice of all the other players, but instead is a function of the \emph{aggregate} players' behaviour. As a purely conceptual example consider the following.

\begin{example}[Guess 2/3 of the average]
\label{ex:2/3}
\emph{During the first lecture of the course in algorithmic game theory each student is asked to pick an integer number in the interval $\{0,\dots,100\}$. The student(s) that selects the number closest to $2/3$ of the average wins. What number should you pick?}	

This puzzle can be modeled as a game where the players are the students, and the action set of each player is $\{0,\dots,100\}$. Further, each player's cost function is captured by the distance from his selection to the $2/3$ of the average. According to the previous definition, this game is aggregative in that the cost function of every player \emph{does not} depend on which number each player selected, but only on an aggregate measure, i.e., the average in this case.\footnote{The answer to this puzzle is more subtle than what it might appear at first, and is more of an exercise in behavioral psychology than a question related to game theory. Indeed, it immediate to observe that the only pure Nash equilibrium of the game consists in all players selecting the number $0$. Nevertheless, the fundamental question we need to answer is different: is the notion of Nash equilibrium an appropriate equilibrium concept for the given setup? Real world experiments show that this is not the case, as the average of the players' actions is usually much higher than 0.}
\end{example}
Besides \cref{ex:2/3}, the aggregative structure arises in various real world applications: in a stock exchange market, the price of a product depends on the total demand and supply, but not on the specific choice of each trader. Similarly, in a road traffic network, the travel time on each link depends (ideally) only on the total number of vehicles on that link.  
\section{Equilibrium efficiency and algorithms}
The notion of Nash equilibrium describes a strong stability condition, requiring no agent to be capable of improving by means of unilateral deviations.\footnote{In order to make sense of the following discussion, we assume existence of a Nash equilibrium. We will formally tackle the existence question in \cref{ch:p1equilibriaVIandSMON}.} 
On the other hand, the quality of an allocation is often measured at the system level with a single scalar cost function $J\SO:\mc{X}\rightarrow \R$.\footnote{In some cases the function $J\SO$ is simply the sum of each agents' cost function. This need not be the general case.} As an example, consider that of a road traffic network, where agents move from origin to destination with the goal of minimizing their own travel time. In this scenario, each agent's cost function captures the time spent on the the road. Nevertheless, a system-level measure describing how well the infrastructure is used is the sum of all users' travel time. 
Thus, of great interest from a system's perspective is to further understand to what extent equilibrium strategies are efficient. Formally, given a game and a social cost function $J\SO$, the efficiency of a Nash equilibrium $x\NE$ is measured by the ratio between $J\SO(x\NE)$ and the minimum possible social cost, i.e., $\min_{x\in \mc{X}}J\SO(s)$. The worst-case (best-case) efficiency over all possible equilibria is known as price of anarchy (price of stability). While non uniqueness of the equilibrium set means that these quantities can be quite different, the notion of price of anarchy has received greater attention. Indeed, knowledge of the price of anarchy can be used to bound the efficiency of any possible Nash equilibrium. Additionally, the system regulator can exploit knowledge of the price of anarchy to influence or design better-performing systems. For example, in relation to the road traffic network mentioned previously, the system operator could impose tolls on specific streets or dynamically modify the speed limit so as to improve the efficiency of the overall system. Following this research direction, the first objective of \cref{part:1} of this thesis is to study the price of anarchy relative to a class of aggregative games.  

Once the equilibrium efficiency problem has been settled (and measures have been taken in case of non-satisfactory performance), of fundamental importance is the problem of coordinating the agents towards an equilibrium of the underlying game. With this respect, we are particularly interested in the use of decentralized algorithms. The advantage in using this class of algorithms includes privacy-preserving features and  computational tractability. In this spirit, the second objective we pursue in  \cref{part:1} of this thesis is the development of decentralized algorithms for a class of aggregative games.

\vspace*{5mm}
\noindent \fbox{\parbox{0.985\textwidth}{
\label{problem:utilitydesign}
 To summarize, the {\bf goal} of \cref{part:1} of this thesis is twofold.
 \begin{itemize}
 \item[-] First, we wish to provide guarantees on the efficiency of Nash and Wardrop equilibria as formally defined in \cref{ch:p1equilibriaVIandSMON}.
 \item[-] Second, we want to devise decentralized algorithms to coordinate the agents toward one such equilibrium.
 \end{itemize}}}
\section{Related works}
In this section, we limit ourselves to connect our work and the aggregative game framework with other research threads and models.
In particular, we do \emph{not} provide a comparison between the contributions presented in \cref{part:1} of this thesis and the existing literature. On the contrary, we postpone this task after the presentation of the results in each of the \cref{ch:p1distanceWENEandPOA,,ch:p1distribtuedalgorithms,,ch:p1applications}. \mbox{This allows us to provide a sharper literature comparison.}

\subsection*{Aggregative games}
While well-known models studied in game theory belong to the class of \emph{aggregative games} (e.g., the classical Cournot model of competition \cite{cournot1838recherches}, or the traffic user equilibrium of Wardrop \cite{wardrop1952some}), a systematic study of this class was initiated only at the turn of the last century, with significant effort coming from the economic literature \cite{corchon1994comparative,dubey2006strategic}. Early studies have been devoted to proving existence of the equilibria, and to the analysis of parametric equilibrium problems. A particular class of which is that of comparative statics, where the goal is to predict how the modification of a parameter in the game would alter the set of equilibria \cite{acemoglu2013aggregate}. Additional results include convergence analysis for best-response like algorithms, but their scope is generally limited to scalar valued aggregate functions \cite{Jensen2010,cornes2012fully}. 
Within the engineering and control community there has been a recent surge of interest in the class of aggregative games, in particular because of their potential applications to road traffic dispatch, wireless network routing, and demand-response schemes \cite{gentile2017nash,scutari2012monotone,ma2013decentralized}. Under technical assumptions, gradient-based algorithms have been proposed to coordinate the agents towards a Nash equilibrium, for example in \cite{koshal2016distributed,chen2014autonomous}.

\subsection*{Mean field games}
Mean field games are a class of continuous-time dynamic games, where the evolution of each agent's trajectory is governed by a stochastic controlled differential equation. In the simplest setup, agents are coupled purely through the cost function, which is assumed to depend \emph{only} on the average state of the agents \cite{huang2007large}. The analysis is carried out in the limiting regime of large populations, since the problem ``simplifies'' to a system comprising a Hamilton-Jacobi equation (backward in time, capturing the optimality condition) and a Fokker-Planck equation (forward in time, capturing the distribution of the agents in the state space) \cite{lasry2007mean}.
While there are some elements of contact between mean field and aggregative games (e.g., the dependence of each agent's cost on the average), some fundamental differences prevent from deeming one class of problems a subset of the other. In particular, the presence of input constraints in aggregative games does not allow for a reformulation in terms of mean field games. The converse is also true, for example due to the fundamental role played by \mbox{stochasticity in the realm of mean field games.} 
\subsection*{Population games}
A population game consists of a game played by a splittable unit mass of players. To facilitate the comprehension, one can think of this model as a game with infinitely many identical agents. By choosing an action from a finite and common set, each agent receives a payoff that depends on the chosen action and on the total mass of agents selecting the same strategy \cite{sandholm2010population}. 
We note that this class of games differs from that of aggregative games for at least two reasons. First, aggregative games are a modeling language capable of describing games with any number of agents, in contrast to population games. Additionally, the result available for aggregative games are not confined to the limiting case of infinite number of players, but the analysis is possible without restoration to the limit. Second, in aggregative games the strategy sets are typically thought of as continuous sets, while this is not the case for population games. Classical results in population games include, amongst others, convergence analysis of evolutionary dynamics including the replicator dynamics and extension thereof \cite{bomze1983lotka,cressman2014replicator}.

\chapter{Mathematical preliminaries}
\label{ch:p1mathpreliminaries}
In this chapter we introduce the mathematical tools required for the development of the first part of this thesis. We begin discussing and connecting useful properties of finite dimensional operators. We then turn our attention to variational inequalities, discuss existence and uniqueness of the solution and present two classical algorithms. While all the material is already available in the literature, we redirect the reader to \cite{facchinei2007finite} for a comprehensive treatment. 
 
\section{Operator properties}
\label{sec:operatorsprop}
In this section we introduce some useful properties of finite dimensional operators. Our interest stems from the key role they play in the study of \mbox{variational inequalities.}
\begin{definition}[Lipschitz, nonexpansive, contractive]
\label{def:nonexpansive}
The operator $F: \mc{X} \subseteq \R^n \rightarrow \R^n$ is Lipschitz with Lipschitz constant $L>0$ if
\be
\label{eq:lipschitz}
||F(x)-F(y)||\le L ||x-y||\qquad \forall{x,y\in\mc{X}}.
\ee
The operator $F$ is non-expansive if \eqref{eq:lipschitz} holds with $L=1$. The operator $F$ is contractive if \eqref{eq:lipschitz} holds with $L<1$.
\end{definition}
\begin{definition}[Monotone and strongly monotone~\cite{facchinei2007finite}]
\label{def:SMON}
The operator $F: \mc{X} \subseteq \R^n \rightarrow \R^n$ is strongly monotone with monotonicity constant $\alpha>0$ if 
\begin{equation}
(F(x)-F(y))^\top(x-y) \ge \alpha \|x-y\|^2\qquad \forall x,y \in \mc{X}.
\label{eq:def_strongly monotone}
\end{equation}
The operator $F$ is monotone if~\eqref{eq:def_strongly monotone} holds for $\alpha=0$.
\hfill{$\square$}
\end{definition}
An example of monotone operator is that of the gradient of a convex function, as detailed in the next proposition.
\begin{proposition}[Convex functions have monotone gradients 
\label{prop:convgrad}
\textup{\cite[Prop. 17.10]{bauschke2011convex}}]
	Let $\mc{X}\subseteq \R^n$ be convex, and consider $f:\mc{X}\rightarrow \R$ a continuously differentiable and (strongly) convex function. The operator $F:\mc{X} \rightarrow \R^n$ defined by $F(x)=\nabla_xf(x)$ is (strongly) monotone.
\end{proposition}
\begin{definition}[Co-coercive]
\label{def:cocercive} 
The operator $F: \mc{X} \subseteq \R^n \rightarrow \R^n$ is co-coercive with constant $\eta>0$ if 
\[
(F(x)-F(y))^\top(x-y)\ge\eta ||F(x)-F(y)||^2 \qquad \forall{x,y\in\mc{X}}.
\]
\end{definition}
The notion of co-coercivity sits in between that of strong monotonicity and monotonicity. In particular, for a given Lipschitz continuous operator it \mbox{is possible to show that}
\be
\text{strong monotonicity} \implies \text{co-coercivity} \implies \text{monotonicity}.
\label{eq:smonimpliedcoc}
\ee
These results follow directly from the corresponding definitions and can be found in \cite[p. 164]{facchinei2007finite}.
The following figure is typically employed to give a visual interpretation of the properties just defined.
\begin{figure}[H]
\begin{center}
\begin{tikzpicture}[scale=1]
 
\draw[thick,->] (-3,0) -- (6.3,0) node[anchor=north east]{};
\draw[thick,->] (0,-3) -- (0,3) node[anchor=north east]{};

\filldraw[fill=black] (0,0) circle (0.08);
\filldraw[fill=black] (-1.5,0) circle (0.08);
\filldraw[fill=black] (1.5,0) circle (0.08);
\filldraw[fill=black] (2,0) circle (0.08);
\filldraw[fill=black] (5,0) circle (0.08);

\filldraw[%
          fill=black,
          opacity=0.2
         ] (0,0) circle (1.5);
\draw (0,0) circle (1.5);

\filldraw[%
          fill=black,
          opacity=0.2
         ] (2.5,0) circle (2.5);
\draw (2.5,0) circle (2.5);

\filldraw[draw=teal,%
          fill=teal,
          opacity=0.3,
         ]
               (2.0,-3)
             --(2.0,3)
             --(6.3,3)  
             --(6.3,-3)
             -- cycle
             ;
\draw (2,-2.5) -- (2,2.5);
\draw[dashed] (2,-2.5) -- (2,-3);         
\draw[dashed] (2,2.5) -- (2,3);

\draw (0.6,-0.3) node[anchor=west]{\footnotesize $(1,0)$};  
\draw (-0.9,-0.3) node[anchor=west]{\footnotesize $(0,0)$};    
\draw (-2.6,-0.3) node[anchor=west]{\footnotesize $(-1,0)$};
\draw (2,-0.3) node[anchor=west]{\footnotesize $(\alpha,0)$};  
\draw (4.9,-0.3) node[anchor=west]{\footnotesize $(1/\eta,0)$}; 
\draw (-1.2,0.5) node[anchor=west]{\footnotesize NE};
\draw (0.55,-1.6) node[anchor=west]{\footnotesize $\eta-$COC};
\draw (2.1,-2.8) node[anchor=west]{\footnotesize $\alpha-$SMON};
\end{tikzpicture}
\end{center}
\caption{Two dimensional representation of nonexpansive operator (NE), co-coercive operator with constant $\eta$ ($\eta-$COC), and strongly monotone operator with constant $\alpha$ ($\alpha-$SMON). For each of these properties, the corresponding colored region represents the locus of points where $F(1,0)$ must lie, under the assumption that $\zeros[2]$ is a fixed point of $F$, i.e., that $F(\zeros[2])=\zeros[2]$. The regions can be easily derived from the corresponding definitions.}
\end{figure}
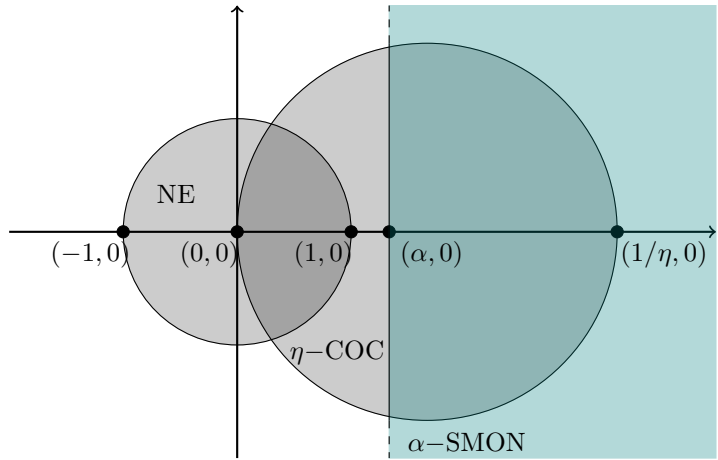
\section{Variational inequalities}
Variational inequalities are fundamental mathematical tools that lend their power from their chameleonic nature. Indeed, surprisingly different problems can be formulated and studied using the language of variational inequality. Examples include systems of equations, optimization problems, Nash equilibrium problems, contact problem in mechanics, options pricing. While the term ``variational inequality'' was coined by Stampacchia in relation to partial differential equation \cite{hartman1966some}, in the following we focus on the treatment of finite dimensional variational inequalities, as defined next. 
\begin{definition}[Variational inequality]
\label{def:vi}
Consider a set $\mc{X}\subseteq \R^n$ and an operator $F:\mc{X} \rightarrow \R^n$. A point $\bar x\in\mc{X}$ is a solution of the variational inequality VI$(\mc{X},F)$ if 
\begin{equation}
F(\bar x)^\top (x-\bar x)\ge 0 \quad \forall x\in\mc{X}. 
\label{eq:videf}
\end{equation}
\end{definition}

\begin{proposition}[Existence and uniqueness \textup{\cite[Cor. 2.2.5, Thm. 2.3.3]{facchinei2007finite}}]
\label{prop:exun}
Consider the variational inequality VI$(\mc{X},F)$, where $\mc{X}$ is compact convex and $F$ continuous.
\begin{enumerate}
	\item The solution set of VI$(\mc{X},F)$ is nonempty and compact.
	\item If the operator $F$ is strongly monotone, the solution of VI$(\mc{X},F)$ is unique.  
\end{enumerate}
\end{proposition}

\subsection*{Connection to convex optimization}
The variational inequality problem is tightly connected with that of mathematical programming. In a mathematical program we are given a set $\mc{X}\subseteq \R^n$ and a real valued function $f:\R^n\rightarrow \R$. Our goal is to select an element of $\mc{X}$ that minimizes $f$ over such set. The next proposition makes this connection clear.
\begin{proposition}[Minimum principle \textup{\cite[Prop 3.1]{bertsekas1989parallel}}]
\label{prop:minimumprinc}
Given $\mc{X}\subseteq\R^n$ closed convex and $f:\mc{X}\rightarrow \R$ continuously differentiable, consider the problem of minimizing \mbox{$f$ over $\mc{X}$.}\footnote{In the following, we say that $f$ is continuously differentiable in a \emph{closed} set $\mc{X}$ if there exists an \emph{open} set $\mc{Y}\supset \mc{X}$ where $f$ is continuously differentiable.}
\begin{enumerate}
\item If $\bar{x}\in\mc{X}$ is a local minimizer of $f$, then $\bar{x}$ solves VI$(\mc{X},\nabla_x f)$
\item If $f$ is convex on $\mc{X}$, then any solution to VI$(\mc{X},\nabla_x f)$ is a global minimizer of $f$.
\end{enumerate}
\end{proposition}
In a nutshell, a convex optimization problem is equivalent to a variational inequality where the operator $F$ represents the gradient of the original function and the set captures the constraint set $\mc{X}$. It is important to observe that the converse does \emph{not} hold. Indeed, there are variational inequalities that do not represent the first order condition for any optimization problem. To convince ourselves of this, it suffices to observe that not all operators $F:\R^n\rightarrow\R^n$ can be written as the gradient of some underlying function.
We also note that the gradient of a strongly convex function is strongly monotone (see \cref{prop:convgrad}), so that existence and uniqueness of the solution is already guaranteed by the corresponding result on variational inequalities presented in \cref{prop:exun}.
A geometric interpretation of condition \eqref{eq:videf} and the corresponding interpretation in terms of mathematical program is illustrated in the following figure.
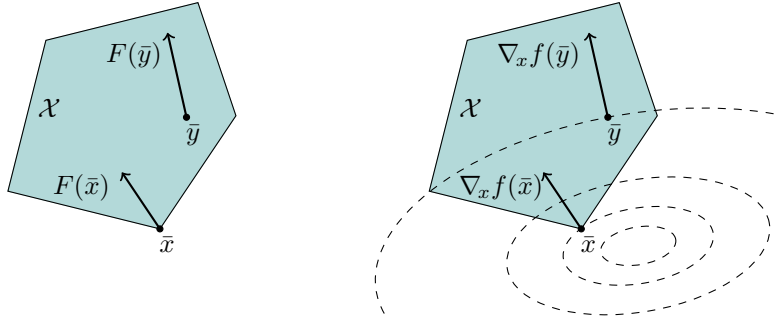
\begin{figure}
\begin{center}
\begin{tikzpicture}[scale=0.5]
\filldraw[draw=teal,%
          fill=teal,
          opacity=0.3,
         ]
               (1,-3)
             --(2.0,1)
             --(6,2)  
             --(7,-1)
             --(5,-4)
             -- cycle
             ;
\draw (1,-3)
             --(2.0,1)
             --(6,2)  
             --(7,-1)
             --(5,-4)
             -- cycle
             ;
\draw[rotate around={10:(6.5,-4.45)},dashed,opacity=0](6.5,-4.45) ellipse (2 and 1);
\draw[rotate around={10:(6.5,-4.45)},dashed,opacity=0](6.5,-4.45) ellipse (1 and 0.5);
\draw[rotate around={10:(6.5,-4.45)},dashed,opacity=0](6.5,-4.45) ellipse (3.5 and 1.75);
\draw[rotate around={10:(6.5,-4.45)},dashed,opacity=0](6.5,-4.45) [partial ellipse=54:190:7 and 3.5];
\draw[thick,->] (5,-4)--(4,-2.5);
\draw[thick,->] (5.7,-1.04)--(5.2,1.2);
\draw (1.5,-0.8) node[anchor=west]{\footnotesize$\mathcal{X}$}; 
\draw (1.9,-2.9) node[anchor=west]{\footnotesize$F(\bar{x})$}; 
\draw (3.3,0.6) node[anchor=west]{\footnotesize$F(\bar{y})$};
\draw (4.7,-4.4) node[anchor=west]{\footnotesize$\bar x$}; 
\draw (5.4,-1.5) node[anchor=west]{\footnotesize$\bar y$};
\filldraw[fill=black] (5,-4) circle (0.08);          
\filldraw[fill=black] (5.7,-1.04) circle (0.08);   
\end{tikzpicture}
\begin{tikzpicture}[scale=0.5]
\filldraw[draw=teal,%
          fill=teal,
          opacity=0.3,
         ]
               (1,-3)
             --(2.0,1)
             --(6,2)  
             --(7,-1)
             --(5,-4)
             -- cycle
             ;
\draw (1,-3)
             --(2.0,1)
             --(6,2)  
             --(7,-1)
             --(5,-4)
             -- cycle
             ;
\draw[rotate around={10:(6.5,-4.45)},dashed](6.5,-4.45) ellipse (2 and 1);
\draw[rotate around={10:(6.5,-4.45)},dashed](6.5,-4.45) ellipse (1 and 0.5);
\draw[rotate around={10:(6.5,-4.45)},dashed](6.5,-4.45) ellipse (3.5 and 1.75);
\draw[rotate around={10:(6.5,-4.45)},dashed](6.5,-4.45) [partial ellipse=54:190:7 and 3.5];
\draw[thick,->] (5,-4)--(4,-2.5);
\draw[thick,->] (5.7,-1.04)--(5.2,1.2);
\draw (1.5,-0.8) node[anchor=west]{\footnotesize$\mathcal{X}$}; 
\draw (1.5,-2.9) node[anchor=west]{\footnotesize$\nabla_{\!x} f(\bar{x})$}; 
\draw (2.5,0.6) node[anchor=west]{\footnotesize$\nabla_{\!x} f(\bar{y})$};
\draw (4.7,-4.4) node[anchor=west]{\footnotesize$\bar x$}; 
\draw (5.4,-1.5) node[anchor=west]{\footnotesize$\bar y$};
\filldraw[fill=black] (5,-4) circle (0.08);          
\filldraw[fill=black] (5.7,-1.04) circle (0.08);  
\end{tikzpicture}
\end{center}	
\caption{On the left: illustration of the condition \eqref{eq:videf} for a general variational inequality. The point $\bar{x}$ is a solution of VI$(\mc{X},F)$ since the scalar product of $F(\bar{x})$ with any other vector attached to $\bar{x}$ and pointing inside the set $\mc{X}$ is non-negative. With a similar reasoning, it is immediate to note that the point $\bar{y}$ is \emph{not} a solution of VI$(\mc{X},F)$. On the right: the special case of variational inequality VI$(\mc{X},\nabla_x f)$ corresponding to the convex optimization program $\min_{x\in\mc{X}}f(x)$. Similarly to the case on the left, $\bar{x}$ is a solution of VI$(\mc{X},\nabla_x f)$ and thus a global minimizer of $f$ (see \cref{prop:minimumprinc}), \mbox{while $\bar{y}$ is not.}}
\end{figure}

\subsection*{Projection based algorithms}
In the following we introduce two classical algorithms for the solution of variational inequalities with a strongly monotone (\cref{alg:gradproj}) or monotone (\cref{alg:extragrad}) operator. Before doing so, we recall the definition of metric projection \mbox{of a point onto a convex set.}

\begin{definition}[Metric projection]
\label{def:proj}
Given $\mc{X}\subseteq \R^n$, we define the metric projection of $x$ onto $\mc{X}$ as the map $\Pi_{\mc{X}} : \R^n\rightarrow \R^n$ with
\be
\label{eq:proj}
\Pi_{\mc{X}}(x)=\argmin_{y\in\mc{X}} ||y-x||.
\ee
\end{definition}
Informally, the projection of $x$ onto the convex set $\mc{X}$ is the closest point in $\mc{X}$ to $x$. From the computational point of view, computing the projection of a point onto a convex set amounts to solving the program in \eqref{eq:proj}. We observe that the program \eqref{eq:proj} reduces to a quadratic program if $\mc{X}$ is a polytope. Since quadratic programs can be solved efficiently, \eqref{eq:proj} can be used as subroutine in the following algorithms.

\begin{algorithm}[h!]
\caption{(Projection algorithm)}
\label{alg:gradproj}
\begin{algorithmic}[1]
\State {\bf Initialise} $k =0$, $\tau>0$, $x_{(0)} \in\R^n$
\While {not converged}
\State $x_{(k+1)} = \Pi_{\mc{X}}(x_{(k)}-\tau F(x_{(k)}))$
\State $k\gets k+1$
\EndWhile
\end{algorithmic}
\end{algorithm}

\begin{proposition}[\textup{\cite[Thm. 12.1.8]{facchinei2007finite}}]
Let $\mc{X}\subseteq \R^n$ be compact convex and $F:\mb{R}^n\rightarrow \R^n$ be co-coercive with constant $\eta$.
Then \cref{alg:gradproj} converges to a solution of VI$(\mc{X},F)$ for any choice of $\tau< 2\eta$ and $x_{(0)}$. 	
\end{proposition}
Since any strongly monotone and Lipschitz operator is also co-coercive as seen in \eqref{eq:smonimpliedcoc}, the previous proposition applies in particular to the special case of strongly monotone operators. Observe that strongly convex optimization problems are equivalent to strongly monotone variational inequalities with the corresponding gradient as operator as discussed in \cref{prop:minimumprinc}. Thus, the previous proposition gives an alternative proof for the convergence of the well-known gradient projection algorithm for strongly convex programs.

If the operator $F$ is not strongly monotone, \cref{alg:gradproj} might not converge in general (it does if we restrict ourselves to variational inequalities representing convex optimization problems). A counterexample is provided in \cite[Ex. 12.1.3]{facchinei2007finite}. It is possible to recover convergence of the algorithm at the price of one extra projection per each iteration, as detailed next.
\begin{proposition}[\textup{\cite[Thm. 12.1.11]{facchinei2007finite}}]
Let $\mc{X}\subseteq \R^n$ be compact convex and $F:\mb{R}^n\rightarrow \R^n$ be monotone and Lipschitz with constant $L$.
Then \cref{alg:extragrad} converges to a solution of VI$(\mc{X},F)$ for any choice of $\tau< 1/L $ and $x_{(0)}$.	
\end{proposition}

\begin{algorithm}[h!]
\caption{(Extragradient algorithm)}
\label{alg:extragrad}
\begin{algorithmic}[1]
\State {\bf Initialize} $k =0$, $\tau>0$, $x_{(0)} \in\R^n$
\While {not converged}
\State $y_{(k+1)} = \Pi_{\mc{X}}(x_{(k)}-\tau F(x_{(k)}))$
\State $x_{(k+1)} = \Pi_{\mc{X}}(x_{(k)}-\tau F(y_{(k+1)}))$
\State $k\gets k+1$
\EndWhile
\end{algorithmic}
\end{algorithm}
\chapter{Nash and Wardrop equilibria \mbox{in aggregative games}}
\label{ch:p1equilibriaVIandSMON}
In the first section of this chapter we introduce the class of average aggregative games as well as the notions of Nash and Wardrop equilibrium. In \cref{sec:connection_VI} we show how these can be reformulated as variational inequalities. We conclude the chapter discussing the monotonicity properties of the operators associated to the Nash and Wardrop problems in \cref{sec:SMON}. All the proofs are reported in the Appendix (\cref{sec:proofsp1-part1}). The formulation presented in this chapter has been published in \cite{gentile2017nash}.
\section{Equilibria with coupling constraints}
\label{sec:game}
We consider a population of $\N$ agents, where each agent $i\in\{1,\dots,\N\}$ can choose a strategy $x\i$ in his individual constraint set $\mc{X}\i\subset\mb{R}^n$. 
In addition to the constraint $x\i\in\mc{X}\i$, each agent's strategy has to satisfy a coupling constraint, which involves the decision variables of other agents. Upon stacking together the strategies of all players as in $x\defeq[x^1;\ldots;x^\N] \in \R^{Mn}$, the coupling constraint takes the form
\begin{equation}
x\in \mc{C} \defeq \{x\in\mb{R}^{\N n}\,\vert\,g(x)\le \zeros[m]\} \subset\mb{R}^{\N n},\qquad g: \R^{\N n} \to \R^m.
\label{eq:coupling_constraints_general}
\end{equation}

We assume that the cost function of agent $i$ depends on his own strategy $x\i$ and on the strategies of the other agents via the average population strategy $\sigma(x)\coloneqq\frac{1}{\N}\sum_{j=1}^{\N}x\j$, as typical of aggregative games~\cite{Jensen2010}. The cost function of agent $i$ is identified with $J\i:\R^n\times\R^n\rightarrow \R$ and takes the form 
\begin{equation}
J\i(x\i,\sigma(x)).
\label{eq:costs_generic}
\end{equation}
The cost and constraints introduced above give rise to the game $\mc{G}$ identified with
\begin{equation}
 \mc{G} \defeq \left\{ 
 \begin{aligned}
& \textup{ agents}  \; && \{1,\dots,\N\}\\
& \textup{ cost of agent } i\quad &&J^i(x\i,\sigma(x)) \\
&\textup{ individual constraint}  &&\mc{X}^i\\
& \textup{ coupling constraint}   &&\mc{C}
\end{aligned}\right. ,
\label{eq:GNEP}
\end{equation}
which is the focus of \cref{part:1} of this thesis. We denote for convenience $\mc{X} \defeq \mc{X}^1\times\ldots\times\mc{X}^\N$ and define
\begin{equation}\label{eq:Q}
\mc{Q}\i(x\mi)\coloneqq\{x\i\in\mc{X}\i\, \vert \, g(x) \le \zeros[m]\}, \quad \quad \mc{Q} \defeq \mc{X}\cap\mc{C}.
\end{equation}
Note that $\mc{Q}\i(x\mi)$ represents the feasible set of player $i$, given that the other players have selected the strategy $x\mi$, while $\mc{Q}$ represents the feasible set for the stacked strategy profile $x$. 
We consider two notions of equilibrium for the game $\mc{G}$ in~\eqref{eq:GNEP}. The first  is a generalization of the celebrated Nash equilibrium concept \cite{Nash01011950} to games with coupling constraints~\cite{arrow1954existence, rosen1965existence}.

\begin{definition}[Nash Equilibrium]\label{def:NE}
A set of strategies $x\NE = [x^1\NE; \dots; x^\N\NE] \in \R^{\N n}$ is an $\varepsilon$-Nash equilibrium of the game $\mathcal{G}$,  if $x\NE\in\mc{Q}$ and for all $ i\in\{1,\dots,\N\}$ and all $ x\i  \in  \mc{Q}\i(x\mi\NE)$ 
\begin{align}
 J\i(x\i\NE,\sigma(x\NE))  \le \textstyle  J\i\left(  x\i,\frac 1\N x\i  +  \frac 1\N \sum_{j \neq i} x\j\NE \right) + \varepsilon\,. 
\label{eq:def_GNE}
\end{align}
If~\eqref{eq:def_GNE} holds with $\varepsilon = 0$ then  $x\NE$ is a Nash equilibrium. 
\hfill $\square$
\end{definition}
\noindent Intuitively, a feasible set of strategies $\left\{ x\i\NE \right\}_{i=1}^\N$ is a Nash equilibrium if no agent can lower his cost by unilaterally deviating from his strategy, assuming that the strategies of the other agents are fixed.
If no coupling constraint is present, i.e., if $\mc{C}=\R^{\N n}$, the previous definition reduces to the well known notion of Nash equilibrium introduced in \cite{Nash01011950}. 
In order to differentiate the two definitions, a Nash equilibrium for a game with coupling constraints is usually referred to in the literature as \emph{generalized Nash equilibrium}~\cite{facchinei2007generalized}. Nevertheless, in \cref{part:1} of this thesis we refer to one such equilibrium simply as a Nash equilibrium. 

Note that on the right-hand side of~\eqref{eq:def_GNE} the decision variable $x^i$ appears in both arguments of $J^i(\cdot,\cdot)$. However, as the number of agents grows the contribution of agent $i$ to $\sigma(x)$ decreases. This motivates the definition of Wardrop equilibrium.
\begin{definition}[Wardrop Equilibrium]\label{def:WE}
A set of strategies $x\WE = [x^1\WE; \dots; x^\N\WE] \in \R^{\N n}$ is a Wardrop equilibrium of the game $\mathcal{G}$  if $x\WE\in\mc{Q}$ and for all $i\in\{1,\dots,\N\}$ and all $x\i  \in  \mc{Q}\i(x\mi\WE)$
\begin{equation}
 \hspace{0.7cm}J\i(x\i\WE,\sigma(x\WE)) \le J\i ( x\i,\sigma(x\WE) ). \tag*{\qed} 
\end{equation}
\end{definition}
\noindent 
Intuitively, a feasible set of strategies $\left\{ x\i\WE \right\}_{i=1}^{\N}$ is a Wardrop equilibrium if no agent can lower his cost by unilaterally deviating from his strategy, assuming that the average strategy is fixed (i.e., he does not influence the average $\sigma(x\WE)$). Similarly to the terminology introduced to indicate a Nash equilibrium, in the following we will refer to a \emph{generalized Wardrop Equilibrium} simply as a Wardrop equilibrium.
The term ``Wardrop equilibrium'' originates from the fact that \cref{def:WE} can be used to model, amongst others, the equilibrium concept introduced in \cite{wardrop1952some} in relation to the study of road traffic networks and often referred to as \emph{traffic user equilibrium} or \emph{Wardrop equilibrium}.
\begin{remark}[On the definition of Wardrop equilibrium]
	Even though the notion of Wardrop equilibrium is thought of as a classical concept, the existing literature defines a Wardrop equilibrium only in terms of the aggregate behaviour $\sigma(x)$ \cite{wardrop1952some, altman2002nash, altman2004equilibrium, marcotte1995convergence, Dafermos87}, while \cref{def:WE} is presented in terms of the agents' strategies $\{x\i\}_{i=1}^\N$. 
	It is important to observe that \cref{def:WE} can be reformulated as a condition on the aggregate $\sigma(x)$ \emph{only} in specific cases (e.g., for applications in transportation networks \cite{wardrop1952some} or competitive markets \cite{Dafermos87}), while there are games for which one such aggregate reformulation is just not possible. Thus, \cref{def:WE} is not a mere revisitation of the classical notion of Wardrop equilibrium, but instead can be used to address a larger class of equilibrium problems.
	Additionally, all the aforementioned works define a Wardrop equilibrium in relation to a specific application, and thus restrict themselves to specific constraint sets or cost functions. On the other hand \cref{def:WE} does not pose any such limitation.

To the best of our knowledge, the first formulation of a Wardrop equilibrium in terms of agents' strategies appears in \cite{ma2013decentralized,grammatico:parise:colombino:lygeros:14}, where however it is not recognized as an equilibrium concept on its own, but rather characterized as an $\varepsilon$-Nash equilibrium for an appropriate value of $\varepsilon$.
\end{remark}

\section{Variational reformulations}
\label{sec:connection_VI} 
In this section we show how Nash and Wardrop equilibria introduced in~\cref{def:NE,,def:WE} can be obtained by solving a corresponding variational inequality. The connection we will draw between these equilibrium notions and the theory of variational inequalities is \emph{fundamental} for the development of \cref{part:1} of this thesis. As a matter of fact, most of the results we will derive in relation to the concepts of Nash and Wardrop equilibria are based on the analysis of their corresponding variational inequalities.

Recall from \cref{def:vi} that a variational inequality is fully specified by its constraint set $\mathcal{X}$ and operator $F$ (see \cref{ch:p1mathpreliminaries} for a brief introduction to the theory of variational inequalities). Towards this goal, we introduce the operators $F\NE,~F\WE : \mc{X} \rightarrow \mathbb{R}^{\N n}$, where
\begin{subequations}
\label{eq:F}
\begin{align}
\label{eq:F_N}
F\NE(x)&\defeq[ \nabla_{x\i} J^i(x\i,\sigma(x)) ]_{i=1}^\N\,,\\
F\WE(x)&\defeq [ \nabla_{x\i} J^i(x\i,z)_{|{z=\sigma(x)}} ]_{i=1}^\N \,. \label{eq:F_W}
\end{align}
\end{subequations}
The operator $F\NE$ is obtained by stacking together the gradients of each agent's cost with respect to his decision variable. $F\WE$ is obtained similarly, but considering $\sigma(x)$ as fixed when differentiating. The following proposition provides a sufficient characterization of the Nash and Wardrop equilibria introduced in~\cref{def:NE,,def:WE} as solutions of two variational inequalities. Both variational inequalities feature the same constraint set $\mathcal{Q}$, defined in~\eqref{eq:Q}, but different operators $F\NE$ and $F\WE$, defined in~\eqref{eq:F_N} and \eqref{eq:F_W}. 
\begin{assumption}
\label{A1}
For all $i\in\{1,\dots,\N\}$, the constraint set $\mathcal{X}\i$ is closed and convex.
The set $\mc{Q}$ in \eqref{eq:Q} is non-empty.
The cost functions $J\i(x\i,\sigma(x))$ are convex in $x\i$ for any fixed $x^j\in\mc{X}^{j}$, ${j\neq i}$.
The cost functions $J^i(x\i,z)$ are convex in $x\i$ for any $z\in \frac{1}{\N} \sum_{j=1}^\N \mc{X}^j$.
The cost functions $J^i(z_1,z_2)$ are continuously differentiable in $[z_1;z_2]$ for any $z_1 \in \mathcal{X}^i$ and $z_2\in \frac{1}{\N} \sum_{j=1}^\N \mc{X}^j$.
The function $g$ in~\eqref{eq:coupling_constraints_general} is convex.
\end{assumption}
\begin{proposition}\label{prop:vi_ref}
Under \cref{A1}, the following hold.
\begin{enumerate}
\item Any solution $\VNE{x}$ of VI$(\mathcal{Q},F\NE)$ is a Nash equilibrium of the game $\mc{G}$ in~\eqref{eq:GNEP}.
\item Any solution  $\VWE{x}$ of VI$(\mathcal{Q},F\WE)$  is a Wardrop equilibrium of the game $\mc{G}$ in~\eqref{eq:GNEP}. 
\end{enumerate}
\end{proposition}
\Cref{prop:vi_ref} states that any solution of the variational inequality VI$(\mathcal{Q},F\NE)$ is a Nash equilibrium and, similarly, any solution of VI$(\mathcal{Q},F\WE)$ is a Wardrop equilibrium. The converse \emph{does not hold} in general, in that there might be strategy profiles that are Nash equilibria but do not satisfy the corresponding variational inequality. This is due to the presence of the coupling constraint $\mc{C}$. Indeed, if $\mathcal{C}=\R^{\N n}$, then $\mathcal Q = \mathcal X$ and one can show that $x\NE$ solves the VI$(\mathcal{X},F\NE)$ \emph{if and only if} it is a Nash equilibrium of $\mc{G}$~\cite[Cor. 1]{facchinei2007generalized}. A similar result holds in the case of Wardrop equilibrium. 
The equilibria that can be obtained as solution of the corresponding variational inequality are called \textit{variational equilibria}~\cite[Def. 3]{facchinei2007generalized} and are here denoted with $\VNE{x},\VWE{x}$ instead of $x\NE,x\WE$ (indicating any equilibrium satisfying ~\cref{def:NE} or \cref{def:WE}). 
We next provide sufficient conditions for the existence and uniqueness of variational equilibria by exploiting two well-known results in the theory of variational inequalities.
\begin{lemma}\textup{\cite[Cor. 2.2.5, Thm. 2.3.3]{facchinei2007finite}}
\label{lem:exun}
Let~\cref{A1} hold.
\begin{enumerate}
\item 
If $\mc{Q}$ is bounded, then there exist a variational Nash equilibrium and a variational Wardrop equilibrium.\footnote{The convexity of the cost functions required by~\cref{A1} is not needed for the first statement of~\cref{lem:exun}, continuity is enough.}
\item If $F\NE$ is strongly monotone on $\mc{Q}$, then the variational Nash equilibrium is unique. If $F\WE$ is strongly monotone on $\mc{Q}$ then the variational Wardrop equilibrium is unique.
\end{enumerate}
\end{lemma}
In light of \cref{prop:vi_ref}, the proof of the first statement in \cref{lem:exun} amounts to showing that \cref{A1} ensures the existence of a solution to VI$(\mathcal{Q},F\NE)$ and VI$(\mathcal{Q},F\WE)$. This is guaranteed if the constraint set $\mathcal{Q}$ is compact and convex, and the operator is continuous \cite[Cor. 2.2.5]{facchinei2007finite}. Such conditions follow immediately form \cref{A1}. 
 Similarly, the proof of the second statement relies on the fact that the solution of a variational inequality is unique if the constraint set $\mathcal{Q}$ is compact and convex, and the operator is continuous and strongly monotone \cite[Thm. 2.3.3]{facchinei2007finite}. The proofs are not reported here, but can be found in the above-mentioned references.

Since any variational Nash equilibrium is a Nash equilibrium, the first claim in \cref{lem:exun} guarantees the existence of a Nash equilibrium. A similar conclusion hold for the existence of a Wardrop equilibrium.

\subsection*{A hierarchy of equilibria: variational  and normalized equilibria}
\label{sec:norm}
The notion of games with coupling constraints  has been introduced in the seminal works \cite{arrow1954existence,rosen1965existence}. In \cite{rosen1965existence} the author defines the concept of \textit{normalized equilibria} to describe the fact that one should expect a manifold of equilibria when the agents are subject to a coupling constraint, even under
strong monotonicity conditions. 
Formally, the strategy profile $x\NE$ is a normalized Nash equilibrium if there exists a vector of weights $r\in\R^M_{>0}$, such that $x\NE$ solves the VI$(\mathcal{Q},F\NE^r )$ where $F\NE^r(x)\defeq[ r_i\nabla_{x\i} J^i(x\i,\sigma(x)) ]_{i=1}^\N$.
It is proven that any normalized Nash equilibrium is a Nash equilibrium in the sense of \cref{def:NE}. Additionally, \cite{rosen1965existence} shows that different choices of $r$ correspond to a different division of the burden of satisfying the coupling constraints $\mc{C}$ among the agents.
In the context of aggregative games, however, each agent contributes equally to the average. Therefore it is typically assumed that the burden of satisfying the coupling constraint should also be split equally among the agents by selecting $r=\ones[M]$, see~\cite{facchinei2007generalized,pan2009games,facchinei2007generalized_2}. It is immediate to see that the subclass of normalized equilibria for which this property holds is the class of \textit{variational equilibria} introduced in the previous section. Nonetheless we note that our results could be easily extended to normalized equilibria by using the operator $F\NE^r$ instead of $F\NE$. We conclude observing that the set of Nash equilibria, normalized Nash equilibria and variational Nash equilibria are all nested as in \cref{fig:nested_equilibria}. A similar result holds for Wardrop equilibria.
\begin{figure}[h!]
\centering
\includegraphics[scale=0.5]{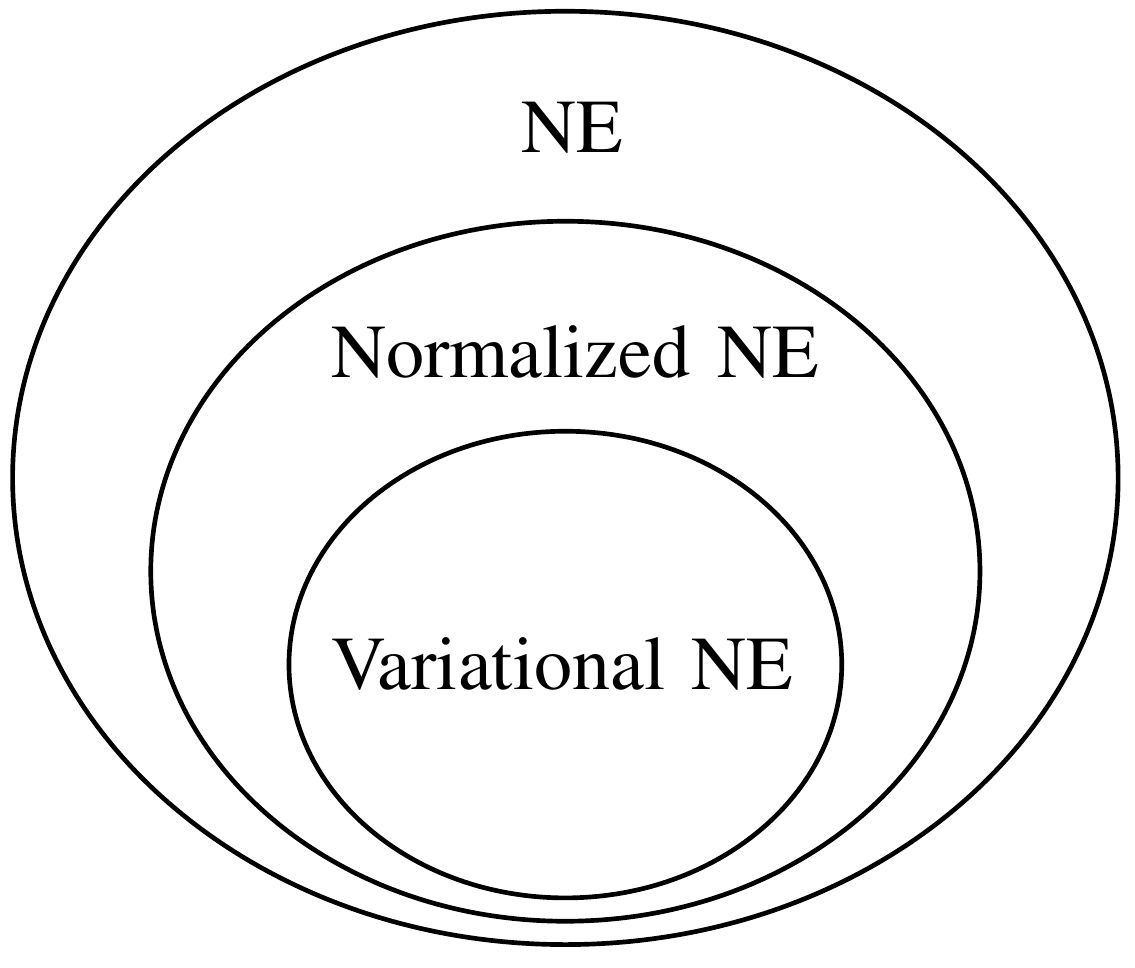}
\caption{The set of Nash equilibria (NE), normalized NE and variational NE are all nested.}
\label{fig:nested_equilibria} 
\end{figure}

In the following we exemplify how the presence of the coupling constraint $\mc{C}$ is typically associated with a manifold of equilibria, regardless of the monotonicity properties of the operators $F\NE$ or $F\WE$.
\begin{example}[Coupling constraints and manifold of equilibria]
Consider the aggregative game $\mc{G}$ defined as in \eqref{eq:GNEP} where there are only two players, and

\be
\label{eq:examplegame}
\begin{split}
\mc{X}^1 &= \{x^1\in\R\,|\, 0\le x^1\le 1\},\\
\mc{X}^2 &= \{x^2\in\R\,|\, 0\le x^2\le 1\},\\
J^1(x^1,\sigma(x))&=\frac{3}{2}(x^1)^2-2\sigma(x)x^1,\\ 
J^2(x^2,\sigma(x))&=2\sigma(x)x^2.
\end{split}
\ee 
We first study the case where there is no coupling constraint, i.e., $\mc{C}=\R^2$, and observe that for such game \cref{A1} is satisfied. 
Thus, any Nash equilibrium is a solution of VI$(\mc{X},F\NE)$ and \emph{vice versa} as discussed immediately after \cref{prop:vi_ref}.
The operator $F\NE$ and the corresponding $\nabla_x F\NE(x)$ are obtained from \eqref{eq:F_N} as 
\[
F\NE(x^1,x^2)=
\begin{bmatrix}
	x^1-x^2\\
	x^1+2x^2\\
\end{bmatrix}
\qquad
\nabla_x F\NE(x^1,x^2) = 
\begin{bmatrix}
	1 & 1 \\ -1 & 2
\end{bmatrix}.
\] 
\cref{lemma:pd} in \cref{sec:SMON} ensures that $F\NE$ is strongly monotone since $\nabla_x F\NE(x^1,x^2)+\nabla_x F\NE(x^1,x^2)^\top \succ 0$. Thus, the solution of the variational inequality VI$(\mc{X},F\NE)$ is unique (thanks to \cref{lem:exun}), and so is the Nash equilibrium. It is immediate to verify that the unique Nash equilibrium is given by $(x^1,x^2)=(0,0)$.

Let us now consider the same game defined in \eqref{eq:examplegame} and introduce the additional coupling constraint
\[
\mc{C}= \{(x^1,x^2)\in\R^2\,|\, x^1+x^2\ge 1\}.
\]  
\cref{A1} is still satisfied so that any solution of the variational inequality VI$(\mc{Q},F\NE)$ is a Nash equilibrium, but the reverse \emph{does not hold} in this case, due to the presence of $\mc{C}$. As a matter of fact, the solution of VI$(\mc{Q},F\NE)$ (i.e., the variational equilibrium) is unique thanks to the strong monotonicity of $F\NE$. On the contrary, it can be verified that any point in the set $\{x\in\R^2\,|\,x^1+x^2=1,~x^1\ge 1/2\}$ is a Nash equilibrium as no player can improve by means of unilateral deviations. 

\end{example}

\section{Sufficient conditions for monotonicity}
\label{sec:SMON}
In this section we derive sufficient conditions that guarantee the monotonicity or strong monotonicity of the operators $F\NE$, $F\WE$ associated with the Nash and Wardrop equilibrium problems.
The importance in assessing whether these operators posses any monotonicity property stems from the following three observations.
\begin{enumerate}
\item[i)] Uniqueness of the variational equilibrium is guaranteed by the strong monotonicity of the corresponding operator, as already discussed in \cref{lem:exun}.
\item[ii)] Strong monotonicity is crucial to control the behaviour of the variational equilibria and their corresponding efficiency in large populations regimes (\cref{ch:p1distanceWENEandPOA}).
\item[iii)] Monotonicity of $F\NE$, $F\WE$ allows to compute the corresponding equilibria using tractable algorithms and to bound their distance (\cref{ch:p1distribtuedalgorithms}).
\end{enumerate}

To verify whether $F\NE$, $F\WE$  are monotone or strongly monotone one can exploit the following equivalent characterizations.
\begin{lemma}
\label{lemma:pd}
\textup{\cite[Prop. 2.3.2]{facchinei2007finite}}
A continuously differentiable operator $F: \mc{K} \subseteq \R^\di \to \R^\di$ is strongly monotone with monotonicity constant $\alpha$ (resp. monotone) if and only if $\nabla_x F(x)\succeq \alpha I$ (resp. $\nabla_x F(x)\succeq 0$) for all $x \in \mc{K}$. Moreover, if $\mc{K}$ is compact, there exists $\alpha>0$ such that $\nabla_x F(x)\succeq \alpha I$ for all $x \in \mc{K}$ if and only if $\nabla_x F(x)\succ 0$ for all $x \in \mc{K}$.
\end{lemma}
In the following we specialize this result to the case when the cost functions~\eqref{eq:costs_generic} reduce to 
\begin{equation} J\i(x\i,\sigma(x)) \defeq v\i(x\i) +p(\sigma(x))^\top x\i.
\label{eq:costs_specific}
\end{equation}

The cost functions in~\eqref{eq:costs_specific} can describe, for example, applications where $x^i$ denotes the usage level of a certain commodity, whose negative utility is modeled by $v^i: \mc{X}\i \to \R$ and whose per-unit cost $p : \frac{1}{\N} \sum_{i=1}^\N \mc{X}^i \to \R^n$ depends on the average usage level of the entire population. Cost functions of the form~\eqref{eq:costs_specific} are widely used in the applications, see \cite{chen2014autonomous,ma2013decentralized}. We refer to $p$ in the following as to the \emph{price function}.
The operators in~\eqref{eq:F}  become
\begin{subequations}
\label{eq:F_decomp_specific}
\begin{align}
F\WE(x) &= [\nabla_{x\i} v\i(x\i)]_{i=1}^\N + [p(\sigma(x))]_{i=1}^\N, \label{eq:F_W_decomp} \\
F\NE(x) &= \textstyle F\WE(x) + \frac1 \N [  \nabla_z p(z)_{|{z=\sigma(x)}} {x\i}  ]_{i=1}^\N. \label{eq:F_N_decomp}
\end{align}
\end{subequations}
\begin{lemma}[Sufficient conditions for strong monotonicity of \eqref{eq:F_decomp_specific}]
\label{lem:FNstrongly monotone}
\leavevmode
\begin{enumerate}
\item Suppose that for each agent $i\in\{1,\dots,\N\}$ the function $v^i$ in~\eqref{eq:costs_specific} is convex and that $p$ is monotone; then $F\WE$ is monotone. Under the further assumption that $p$ is affine and strongly monotone, $F\NE$ is strongly monotone.
\item Suppose that for each agent $i\in\{1,\dots,\N\}$ the function $v^i$ in~\eqref{eq:costs_specific} is strongly convex and that $p$ is monotone. Then $F\WE$ is strongly monotone. \qed
\end{enumerate}
\end{lemma}
\subsection{Linear price function}
In the following we refine the sufficient conditions of \cref{lem:FNstrongly monotone} to the important class of aggregative games with cost functions of the following form
\begin{equation}
J\i(x\i,\sigma(x)) \defeq \frac{1}{2} (x\i)^\top Q x\i +(C\sigma(x)+c\i)^\top x\i\,,
\label{eq:costs_quad}
\end{equation}
where $Q \in \R^{n\times n}$ is symmetric, $C\in\R^{n\times n}$ (not necessarily symmetric), $c\i \in\R^n$. We observe that \eqref{eq:costs_quad} is a special case of \eqref{eq:costs_specific}, obtained setting $v\i(x\i)=(x\i)^\top Q x\i+(c\i)^\top x\i$ and $p(\sigma(x))=C\sigma(x)$. We refer to this case as to the case of \emph{linear price function}.
The cost functions in  \eqref{eq:costs_quad} have been used for example in~\cite{huang2007large,grammatico:parise:colombino:lygeros:14,bauso:pesenti:13}. Since the operators $F\NE,F\WE$ defined in \eqref{eq:F} are obtained by differentiating quadratic functions, their expression is affine, and given by
\begin{subequations}
\label{eq:F_quad}
\begin{align}
\label{eq:F_quad_W}
F\WE(x)&= \textstyle \left(I_\N\otimes Q + \frac 1 \N \mathbbold{1}_\N\mathbbold{1}_\N^\top \otimes C \right) x+  c, \\
F\NE(x)&= \textstyle F\WE(x)+  \frac{1}{\N} (I_\N \otimes C^\top)  x,
\label{eq:F_quad_N}
\end{align}
\end{subequations}
where $c=[c^1;\ldots; c^\N]$.
The following lemma exploits the structure in~\eqref{eq:F_quad} to derive sufficient conditions for strong monotonicity of $F\WE$, $F\NE$.
\begin{lemma}[Sufficient conditions for strong monotonicity of \eqref{eq:F_quad}]
\label{lem:quadratic}
\begin{enumerate}
\item[]
\item If $Q\succ 0$, $C\succeq 0$ then $F\WE$ in~\eqref{eq:F_quad_W} is strongly monotone.
\item If $Q\succ 0$, $Q - C^\top Q^{-1} C \succ 0$ then $F\WE$ in~\eqref{eq:F_quad_W} is strongly monotone.
\item If $Q\succ 0$, $C\succeq 0$ or if $Q\succeq 0$, $C\succ 0$ then $F\NE$ in~\eqref{eq:F_quad_N} is strongly monotone.
\end{enumerate}
\end{lemma}
\subsection{Diagonal price function}
In the following we consider the case when the price function $p(\sigma)$ has diagonal structure, i.e., the $t$-th component of $p$ depends only on the corresponding component of the average. Formally, we assume that $p(\sigma(x))$ can be decomposed as $p(\sigma(x))=[p_t(\sigma_t(x_t))]_{t=1}^n$, with $p_t:\R\rightarrow \R$ for all $t$, $\sigma_t(x_t) = \frac{1}{N}\sum_{i=1}^\N x_t^i$ and $x_t\defeq[x^1_t,\dots,x^n_t]$. This corresponds to cost functions of the following form 
\be
J\i(x\i,\sigma(x))=v\i(x\i)+\sum_{t=1}^np_t(\sigma_t(x_t))x\i_t\,.
\label{eq:diagonalprice}
\ee
Cost functions as in \eqref{eq:diagonalprice} are typically used in the literature to describe congestion costs of road traffic networks (\cite{wardrop1952some, correa2004selfish} and \cref{sec:traffic}) or the charging of electric vehicles (\cite{ma2013decentralized,grammatico:parise:colombino:lygeros:14} and \cref{sec:PEVs}).
We refer to this case as to the case of \emph{diagonal price function}. A sufficient condition ensuring the monotonicity or strong monotonicity of $F\WE$ can be obtained directly exploiting the structure of \eqref{eq:diagonalprice} and the result in \cref{lemma:pd}. The situation is more complicated when we turn our attention to $F\NE$ due to the presence of the additional term $[  \nabla_z p(z)_{|{z=\sigma(x)}} {x\i}  ]_{i=1}^\N$ in \eqref{eq:F_decomp_specific}. The following lemma provides a sufficient condition.
\begin{lemma}
\label{lem:diagSMON}
Let $\mc{X}$ be closed and convex. Assume that $v\i(x\i)$ in \eqref{eq:diagonalprice} is convex for each agent $i\in[\N]$ and that $p_t$	 is continuously differentiable and strictly increasing for all $t\in[n]$. Further, suppose that $\mc{X}\i\subseteq [0,x\zero]^n$ for each $i\in[\N]$. If 
\[\minn{\substack{t \in \{1,\dots,n\}\\z \in [0,x^0]}}{\left(p'_t(z) - \frac{\tilde x^0 p''_t(z)}{8}\right)} > 0,
\label{eq:bound_PEV_smonfirst}
\]
then the operator $F\NE$ is strongly monotone.
\end{lemma}
We note that the positivity requirement on the agent strategies is satisfied in many applications such as those studied in \cref{ch:p1applications}. Nevertheless, the previous lemma can be extended adjusting the condition \eqref{eq:bound_PEV_smonfirst} to the case where $\mc{X}\i\subseteq [-x\zero,x\zero]^n$, see \cite{gentilethesis}.

An immediate consequence of the previous lemma is that, when $p_t$ is continuously differentiable, strictly increasing and \emph{concave} for all $t$, the operator $F\NE$ is strongly monotone.
It is worth noting that~\cite{yin2011nash} considers a similar setup to what studied in this section. In~\cite[Lem. 3]{yin2011nash} the authors exploit the structure in~\eqref{eq:diagonalprice} and give conditions for $\nabla_x F\NE(x)$ to be a $P$-matrix, which in turn guarantees uniqueness of the Nash equilibrium in the absence of coupling constraints. This is, to the best of our knowledge, the only work providing sufficient conditions for equilibrium uniqueness and convergence of the algorithms. It is interesting to note that uniqueness in~\cite{yin2011nash} holds assuming $p'_t > 0, p''_t > 0$, whereas our result holds if the opposite condition is satisfied, namely if $p'_t > 0, p''_t < 0$.
\section{Appendix}
\label{sec:proofsp1-part1}
\subsection{Proofs of the results presented in \cref{sec:connection_VI,,sec:SMON}}
\subsubsection*{Proof of \cref{prop:vi_ref}}
\begin{proof}
\begin{enumerate}
\item[]
\item The proof of the first statement can be also found in~\cite[Thm. 2.1]{facchinei2007generalized_2}.\\
By definition $\VNE{x}$ is a solution of VI$(\mc{Q},F\NE)$, that is
\be
F\NE(\VNE{x})^\top(x-\VNE{x})\ge0,\qquad \forall x \in \mc{Q}\,.
\label{eq:proof_NEintermediate}
\ee
In the following we fix the strategies of all the players but $i$ to $x\mi = \VNE{x}\mi$, so that all the summands in \eqref{eq:proof_NEintermediate} vanish, except for the $i$-th term
\[
\nabla_{x\i}J\i(\VNE{x}\i,\sigma(\VNE{x}))^\top(x\i-\VNE{x}\i)\ge0, \qquad\forall x\i\in\mc{Q}(\VNE{x}\mi)\,.
\]
Consider the function $x\i\mapsto J\i(x\i,\frac{1}{\N}x\i+\frac{1}{\N}\sum_{j\neq i} \VNE{x}\j)$ and observe that $J\i:\mc{Q}(\VNE{x}\mi)\rightarrow\R$ is convex by assumption. Since $\mc{Q}(\VNE{x}\mi)$ is also convex by assumption, it follows from~\cite[Prop. 3.1]{bertsekas1989parallel} that $\VNE{x}\i$ must be a minimizer of $J\i:\mc{Q}(\VNE{x}\mi)\rightarrow\R$, i.e., that 
\[
J\i(\VNE{x}\i,\VNE{x}\mi)\le  J\i\left(x\i,\frac{1}{\N}x\i+\frac{1}{\N}\sum_{j\neq i} \VNE{x}\j\right),\qquad \forall x\i\in\mc{Q}(\VNE{x}\mi)\,.
\]
Since this holds for all $i\in \{1,\dots,\N\}$ and since $\VNE{x}\in\mathcal{Q}$ by definition of variational inequality, it follows that  $\VNE{x}$ is a Nash equilibrium of $\mc{G}$.
\item We rewrite the operator $F\WE(x)$ as $\tilde F\WE(x,\sigma(x))$, where $\tilde F\WE(x,z)\defeq [ \nabla_{x\i} J^i(x\i,z) ]_{i=1}^\N.$ Fix $ \bar{z}=\sigma(\VWE{x})$. By definition, if $\VWE{x}$ solves VI$(\mathcal{Q},F\WE)$ then $F\WE(\VWE{x} )^\top(x-\VWE{x} )\ge 0$ for all $x\in\mc{Q}$, i.e.,
\begin{equation}
\tilde F\WE(\VWE{x}, \bar{z} )^\top(x-\VWE{x} )\ge 0,\qquad \forall x\in\mc{Q}.
\label{eq:proof_intermediate}
\end{equation}
Consider $i \in \{1,\dots,\N\}$, set $x\mi=\VWE{x}\mi$ in~\eqref{eq:proof_intermediate} and consider an arbitrary $x^{i} \in \mc{Q}^i(\VWE{x}^{-i})$; then all the summands in~\eqref{eq:proof_intermediate} vanish except the $i^{\textup{th}}$ one and~\eqref{eq:proof_intermediate} reads
\begin{equation}
\nabla_{x\i} J\i(\VWE{x}^i,\bar{z})^\top (x\i-\VWE{x}\i)\ge 0, \qquad \forall  x\i\in\mc{Q}\i(\VWE{x}\mi).
\label{eq:minimum_principle}
\end{equation}  
Consider the convex function $J\i(\cdot,\bar{z}):\mc{Q}\i(\VWE{x}\mi) \rightarrow \R$. Since $\mc{Q}\i(\VWE{x}\mi)$ is a convex set, by~\eqref{eq:minimum_principle} and~\cite[Prop. 3.1]{bertsekas1989parallel} we have that $\VWE{x}^i\in \arg  \min_{x^i\in\mc{Q}\i(\VWE{x}\mi) } J\i\left(x\i,\bar{z} \right)$.
Substituting $\bar{z}=\sigma(\VWE{x})$, one has $J\i\left(\VWE{x}^i,\sigma(\VWE{x} )\right) \le J\i\left(x\i,\sigma(\VWE{x} )\right)$ for all $x\i\in\mc{Q}\i(\VWE{x}\mi)$. Since this holds for all $i\in \{1,\dots,\N\}$ and since $\VWE{x}\in\mathcal{Q}$, it follows that  $\VWE{x}$ is a Wardrop equilibrium of $\mc{G}$.
\end{enumerate}
\end{proof}
\subsubsection*{Proof of \cref{lem:FNstrongly monotone}}
\begin{proof}
\begin{enumerate} 
\item[]
\item  Let us first show that $F\WE$ is monotone.
Since $v\i$ is convex, then $\nabla_{x\i} v\i(x\i)$ is monotone in $x\i$ by~\cite[Sec. 4.2.2]{scutari2012monotone}.
Hence $[\nabla_{x\i} v\i(x\i)]_{i=1}^\N$ is monotone.
Moreover, for any $x_1,x_2$
\begin{equation}
\begin{aligned}
& ([p(\sigma(x_1))]_{i=1}^\N - [p(\sigma(x_2))]_{i=1}^M)^\top (x_1 - x_2) \\
&= \N ( p(\sigma(x_1)) - p(\sigma(x_2)) )^\top ( \sigma(x_1) - \sigma(x_2) ) \ge 0,
\label{eq:p_mon}
\end{aligned}
\end{equation}
where the last inequality follows from the fact that $p$ is monotone.
By~\eqref{eq:F_W_decomp} and the fact that the sum of two monotone operators is monotone, one can conclude that $F\WE$ is monotone. 

To show that $F\NE$ is strongly monotone, we write the affine expression of $p$ as $p(x) = C x + c$, where there exists $\alpha > 0$ such that $C \succ \alpha I_n$ by~\cref{lemma:pd}.
Then the term $\frac1 \N [  \nabla_z p(z)_{|{z=\sigma(x)}} {x\i}]_{i=1}^\N$ in~\eqref{eq:F_N_decomp} equals
$\frac{1}{\N} (I_\N \otimes C^\top) x$.
Since $\nabla_x (\frac{1}{\N} (I_\N \otimes C^\top) x) \succ \frac{\alpha}{\N} I_{\N n}$,
then $\frac1 \N [  \nabla_z p(z)_{|{z=\sigma(x)}} {x\i}]_{i=1}^\N$ is strongly monotone by~\cref{lemma:pd}.
Having already shown that $F\WE$ is monotone, the proof is concluded upon noting that the sum of a monotone operator and a strongly monotone operator is strongly monotone.

\item Strong convexity of $v^i$ is equivalent to strong monotonicity of $\nabla_{x\i}v\i(x\i)$ in $x\i$~\cite[Sec. 4.2.2]{scutari2012monotone}.
Then $[\nabla_{x\i} v\i(x\i)]_{i=1}^\N$ is strongly monotone.
Monotonicity of $[p(\sigma(x))]_{i=1}^\N$ in~\eqref{eq:F_W_decomp} can be shown as in~\eqref{eq:p_mon}.
\end{enumerate}	
\end{proof}

\subsubsection*{Proof of \cref{lem:quadratic}}
\begin{proof}
\begin{enumerate}
\item[]
\item By \cref{lemma:pd}, strong monotonicity of $F\WE$ in~\eqref{eq:F_quad_W} is equivalent to $\nabla_x F\WE(x) = \left(I_\N\otimes Q + \frac{1}{\N} \mathbbold{1}_\N\mathbbold{1}_\N^\top \otimes C \right)^\top \succ 0$, which is independent from $x$.
If $Q\succ 0$ and $C\succeq0$, it holds $\left(I_\N\otimes Q + \frac{1}{\N} \mathbbold{1}_\N\mathbbold{1}_\N^\top \otimes C \right)^\top \succ 0$, proving the statement.
\item Since $Q$ is symmetric, $Q\succ0$, and $Q-C^\top Q^{-1}C \succ 0$, by Schur's Complement we have
\[
\begin{bmatrix}
	Q & C^\top \\
	C & Q
\end{bmatrix}\succ 0\,.
\]
It follows that 
\[
\begin{bmatrix}
x\\
x	
\end{bmatrix}^\top
\begin{bmatrix}
	Q & C^\top \\
	C & Q
\end{bmatrix}
\begin{bmatrix}
x\\
x	
\end{bmatrix}
=x^\top\left(2Q + C+C^\top \right) x \ge 0\,, 
\]
from which it must be $Q+C\succ 0$. We conclude the proof by showing that $Q\succ0$ symmetric and $Q+C\succ 0$ imply $\nabla_xF\WE(x) = \left(I_\N\otimes Q + \frac{1}{\N} \mathbbold{1}_\N\mathbbold{1}_\N^\top \otimes C \right)^\top \succ 0$. Recall that $\nabla_xF\WE(x)\succ 0$ is equivalent to showing positive definiteness of 
\be
\label{eq:proofschur}
I_\N\otimes 2Q + \frac{1}{\N} \mathbbold{1}_\N\mathbbold{1}_\N^\top \otimes {(C+C^\top)}.
\ee
 To prove the latter inequality, let us consider $\lambda_j$ an eigenvalue of the matrix appearing in \eqref{eq:proofschur} with corresponding eigenvector $v_j\neq \zeros[\N n]$. It must be 
 \be
 \left(
 I_\N\otimes 2Q + \frac{1}{\N} \mathbbold{1}_\N\mathbbold{1}_\N^\top \otimes {(C+C^\top)}
 \right)v_j = \lambda_j v_j \iff 2Qv^i_j+\frac{C+C^\top}{\N}\sum_{i=1}^{\N}v_j^i = \lambda_j v_j^i,
 \label{eq:partialeigen}
 \ee
for all $i\in[\N]$. Summing the previous expressions over $i$ gives
 \[
 \left(
 2Q+C+C^\top
 \right)
 \sum_{i=1}^{\N}{v_j^i}
 =
 \lambda_j
 \sum_{i=1}^{\N}{v_j^i}\,.
 \]
 Thus, if $\sum_{i=1}^{\N}{v_j^i}\neq \zeros[n]$, $\lambda_j$ is also an eigenvalue of $2Q+C+C^\top$ and it must be $\lambda_j>0$ since $Q+C\succ 0$. If, on the contrary, $\sum_{i=1}^{\N}{v_j^i}= \zeros[n]$, it follows from \eqref{eq:partialeigen} that $2Qv_j = \lambda_j v_j$, i.e., $\lambda_j$ is also an eigenvalue of $2Q$ and it must be $\lambda_j> 0$ since $Q\succ 0$ and symmetric. We conclude, as required, that the matrix appearing in \eqref{eq:proofschur} is positive definite, since all its eigenvalues are strictly positive.
\item 
Similarly to the first point, strong monotonicity of $F\NE$ in~\eqref{eq:F_quad_N} is equivalent by \cref{lemma:pd} to $\left(I_\N\otimes Q + \frac{1}{\N} \mathbbold{1}_\N\mathbbold{1}_\N^\top \otimes C \right)^\top +  \frac{1}{\N} (I_\N \otimes C^\top)^\top \succ 0$. 
If $Q\succ 0$ and $C\succeq0$ or if $Q\succeq0$ and $C\succ 0$, it follows that $\left(I_\N\otimes Q + \frac{1}{\N} \mathbbold{1}_\N\mathbbold{1}_\N^\top \otimes C \right)^\top +  \frac{1}{\N} (I_\N \otimes C^\top)^\top \succ 0$, completing the proof. 
\end{enumerate}
\end{proof}

\subsubsection*{Proof of \cref{lem:diagSMON}}
\begin{proof}
First, observe that the operator $p:\R^{n}\rightarrow \R^n$ is monotone. Indeed, since $p_t$ is strictly increasing it holds for all $y$, $z$ that
\[
(p(y)-p(z))^\top (y-z) = \sum_{t=1}^n(p_t(y_t)-p_t(z_t))(y_t-z_t)>0\,.
\]	
Thanks to \cref{lem:FNstrongly monotone}, we conclude that $F\WE$ is also monotone.
According to~\eqref{eq:F_N_decomp}, to show strong monotonicity of $F\NE$ it is sufficient to show that the term $[\nabla_z p(z)_{|z=\sigma(x)} {x\i}  ]_{i=1}^\N$ is strongly monotone for all $x \in \mc{X}$. The latter is equivalent to proving $\nabla_x [\nabla_z p(z)_{|z=\sigma(x)} {x\i}  ]_{i=1}^\N \succ 0$  for all $x \in \mc{X}$ by~\cref{lemma:pd}.
We have
\begin{equation}
\nabla_x [\nabla_z p(z)_{|z=\sigma(x)} x\i]_{i=1}^\N =  I_\N \otimes \nabla_z p(z)_{|z=\sigma(x)} + \frac{1}{\N} \ones[\N] \otimes \left([ \text{diag} \{p''_t(\sigma_t) x^i_t \}_{t=1}^n ]_{i=1}^{\N}\right)^\top,
\label{eq:PEV_proof_intermediatefirst}
\end{equation}
where $\text{diag} \{p''_t(\sigma_t) x^i_t \}_{t=1}^n$ is the diagonal matrix whose entry in position $(t,t)$ is $p''_t(\sigma_t) x^i_t$. The permutation matrix $P = [[{\bf e}_{t+(i-1)n}^\top]_{i=1}^\N]_{t=1}^n$ (${\bf e}_i$ denotes the $i$-th vector of the canonical basis) permutes~\eqref{eq:PEV_proof_intermediatefirst} into block-diagonal form

\begin{align}
&P \nabla_x [\nabla_z p(z)_{|z=\sigma(x)} x\i]_{i=1}^\N P^\top = \label{eq:PEV_gradient_FN} \\[3mm]
&
\begin{bmatrix}
p_1'(\sigma_1) I_{ \N}  & & \\
&  \ddots  & \\
& &  p_n'(\sigma_n) I_{ \N} 
\end{bmatrix}
 +  \frac{1}{\N} 	
\begin{bmatrix}
p_1''(\sigma_1) x_1 \ones[\N]^\top  & & \\
&  \ddots  & \\
& &  p_n''(\sigma_n) x_n \ones[\N]^\top
\end{bmatrix}
\end{align}
\noindent where $x_t  =  [x\i_t]_{i=1}^\N$. To conclude, it suffices to show that $p_t'(\sigma_t) I_\N + \frac{1}{\N} p''_t(\sigma_t) x_t \ones[\N]^\top \succ 0$ for all $t$. \cref{lem:min_eigenval} (reported at the end of this proof) guarantees that  $\lambda_\text{min} \left( x_t \ones[\N]^\top + \ones[\N] x_t^\top \right)/2 \ge -\frac{\tilde x^0 \N}{8}$, which terminates the proof.
\end{proof}

\begin{lemma}
\label{lem:min_eigenval}
For all $\N\in\mb{N}$, it holds  
\begin{equation}
\minn{y \in [0,1]^\N} \lambda_\textup{min} \left( y \ones[\N]^\top + \ones[\N] y^\top \right) \ge -\frac{\N}{4}.
\label{eq:min_eigenval_statement}
\end{equation}

\end{lemma}
\begin{proof}
The statement is trivially true for $\N = 1$.
For $\N > 1$, the left hand side of \eqref{eq:min_eigenval_statement} is equivalent to
\begin{equation}
\minn{\substack{y \in [0,1]^\N \\ \| v \| = 1}}{v^\top  \left( y \ones[\N]^\top + \ones[\N] y^\top \right) v} = 
\minn{\substack{y \in [0,1]^\N \\ \| v \| = 1}}{2\left( v^\top y \right) \left( \ones[\N]^\top v \right)}.
\label{eq:min_eigenval_expanded}
\end{equation}
Let us consider a pair $y^\star, v^\star$ minimizing~\eqref{eq:min_eigenval_expanded}. If $\ones[\N]^\top v^\star = 0$, the bound \eqref{eq:min_eigenval_expanded} is trivially satisfied. We are left with two cases, $\ones[\N]^\top v^\star > 0$ and $\ones[\N]^\top v^\star < 0$.
Let us start from the case of $\ones[\N]^\top v^\star > 0$.
To minimize $2\left( v^\top y \right) \left( \ones[\N]^\top v \right)$, it must be
\begin{equation}
y^\star_i = \begin{cases} 0 \quad &\text{if} \; v^\star_i > 0  \\ 
1 \quad & \text{if} \; v^\star_i < 0,
\end{cases} \;\; \text{for all} \; i \in \{1,\dots,\N\}.
\label{eq:min_eigenval_y_01}
\end{equation}
Without loss of generality, we can assume $y_i^\star \in \{0,1\}$ if $v^\star_i = 0$.
Hence we conclude that $y^\star \in \{0,1\}^\N$ and~\eqref{eq:min_eigenval_statement} reduces to
\begin{equation}
\minn{p \in \{0,\dots,\N\}}{\lambda_\textup{min} \left[
\begin{array}{c|c}
2 (\ones[p] \ones[p]^\top) & \ones[p] \ones[(\N-p)]^\top \\[0.1cm]
\hline  \\[-0.3cm]
\ones[(\N-p)] \ones[p]^\top & \zeros[(\N-p)] \zeros[(\N-p)]^\top
\end{array}
\right]},
\label{eq:min_eigenval_matrix_01}
\end{equation}
where without loss of generality we assumed the first $p$ components of $y^\star$ to be $1$ and the remaining to be $0$.
Note that the matrix in~\eqref{eq:min_eigenval_matrix_01} features $p$ identical rows followed by $\N-p$ other identical rows.
Hence any of its eigenvectors must have $p$ identical components followed by $\N-p$ other identical components.
With this observation and the definition of eigenvalue, algebraic calculations show that
the matrix in~\eqref{eq:min_eigenval_matrix_01} has only two distinct eigenvalues, the minimum of the two being $p-\sqrt{\N p}$.
The function $p-\sqrt{\N p}$ is minimized over the reals for $p=\N/4$ with corresponding minimum $\lambda_\textup{min} = -\N/4$,
as it can be seen by using the change of variables $p = q^2$ and minimizing the quadratic function $q^2 - \sqrt{\N} q$.
Since $p \in \{0,\dots,\N\}$ in~\eqref{eq:min_eigenval_matrix_01},
the value $-\N/4$ is a lower bound for the minimum eigenvalue, and it is attained only if $\N$ is a multiple of $4$.
We conclude by noting that the derivation for the case $\ones[\N]^\top v^\star < 0$ is identical to the derivation for the case $\ones[\N]^\top v^\star > 0$ just shown, upon switching $0$ and $1$ in~\eqref{eq:min_eigenval_y_01}.
\end{proof}
\chapter{Equilibria and efficiency in large populations}
\label{ch:p1distanceWENEandPOA}
Many real world applications where agents behave strategically feature the interaction of a large population of individuals. As an example, consider that of drivers moving on the road network of a city, with the objective of reaching their destination as swiftly as possible. As an alternative example consider that of traders in a stock market. Motivated by this observation, the current chapter is dedicated to the study of aggregative games with a large number of players. The chapter is divided in two parts. In \cref{sec:Wardrop_Nash} we provide bounds on the distance between Wardrop and Nash equilibria, while in \cref{sec:eqilibriumefficiency} we study the efficiency of these equilibria, i.e., we study how much selfish behaviour degrades the performance of a centrally controlled system.
All the proofs are reported in the Appendix (\cref{sec:proofsp1-part2}). The results presented in this chapter have been published in \cite{gentile2017nash,paccagnan18efficiency}.
 
Specifically, we consider a sequence of games $\GNS$.
For fixed $\N$, the game $\GN$ is played among $\N$ agents and is defined as in~\eqref{eq:GNEP} with an arbitrary coupling constraint $\mc{C}$,  arbitrary costs $\{J\i(x\i,\sigma(x))\}_{i=1}^\N$ and arbitrary local constraints $\{\mc X\i\}_{i=1}^\N$. For the sake of readability, we avoid the explicit dependence on $\N$ in denoting these quantities and in denoting $x\NE$, $x\WE$, $F\NE$, $F\WE$.
\section{Distance between Nash and Wardrop equilibria}
\label{sec:Wardrop_Nash}
As we have learnt from the previous chapter, the monotonicity properties of the operators $F\NE$ and $F\WE$ might not coincide. For example, $F\NE$ might be strongly monotone for a given game, while for the same game $F\WE$ might not be. Unfortunately, if the operator associated to a variational inequality is not monotone, determining the corresponding solution is in general an  intractable problem.
Motivated by this shortcoming, in this section we provide bounds on the distance between $\VNE{x}$ and $\VWE{x}$, so that, should one of the two equilibria  be difficult to compute (e.g., due to the lack of monotonicity), we might be able to compute the other and still be able to learn something about the former.

\begin{assumption}
\label{A2}
There exists a convex, compact set $\mathcal{X}\zero\subset\R^n$ such that $\cup_{i=1}^\N \mathcal{X}^i\subseteq{\mathcal{X}\zero}$ for each $\GN$ in the sequence $\GNS$. Let $ \Dx \defeq  \max_{y \in{\mathcal{X}\zero} } \{\|y \|\}$. For each $\N$ and $i \in \{1,\dots,\N\}$, the function $J\i(z_1,z_2)$ is Lipschitz with respect to $z_2$ in $\mc{X}^0$ with Lipschitz constant $L_2$ independent from $\N$, $i$ and $z_1 \in \mc{X}^i$.
\end{assumption}
We note that~\cref{A2} implies that $\sigma(x) \in \mc{X}^0$ for any $\N$ and any $x \in \mc{X}^1 \times \dots \times \mc{X}^\N$. Furthermore, if the cost function~\eqref{eq:costs_generic} takes the specific form~\eqref{eq:costs_specific}, then $p$ being Lipschitz in $\mc{X}^0$ with constant $L_p$ implies $J\i(z_1,z_2)$ being Lipschitz with respect to $z_2$ in $\mc{X}^0$ with constant $L_2 = \Dx L_p$, as by the Cauchy-Schwartz inequality
\begin{equation}
\begin{aligned}
&\| J\i(z_1,z_2) - J\i(z_1,z_2') \| = \| (p(z_2) - p(z_2'))^\top z_1 \| \\
&\le \| p(z_2) - p(z_2') \| \| z_1 \| \le \Dx L_p \| z_2 - z_2' \|.
\end{aligned}
\label{eq:Lipschitz_implies_Lipschitz}
\end{equation}

The next proposition shows that every Wardrop equilibrium is an $\varepsilon$-Nash equilibrium, with $\varepsilon$ vanishing as $\N$ grows.
\begin{proposition}~\label{prop:conv_cost}
Let the sequence of games $\GNS$ satisfy~\cref{A2}. For each $\GN$, every Wardrop equilibrium is an $\varepsilon$-Nash equilibrium, with $\varepsilon=\frac{2\Dx L_2}{\N}$.
\end{proposition}
\Cref{prop:conv_cost} is a strong result as it guarantees that for relatively large $\N$ a Wardrop equilibrium is \emph{almost} stable in the sense of the Nash equilibrium definition. In particular, at any given Wardrop equilibrium no player can improve upon its cost by more than an additive factor $\varepsilon$, considering the strategies of the others fixed. 
Unfortunately, \cref{prop:conv_cost} provides no information on the distance between the set of strategies constituting a Nash and a Wardrop equilibrium. This question is addressed in the following theorem.
\begin{theorem}
\label{thm:conv_strategies}
Let the sequence of games $\GNS$ satisfy~\cref{A2}, and each $\GN$ satisfy~\cref{A1}. Then:
\begin{enumerate}
\item
If the operator $F\NE$ relative to $\GN$ is strongly monotone on $\mc{Q}$ with monotonicity constant $\alpha\emm>0$, then there exists a unique variational Nash equilibrium $\VNE{x}$ of $\GN$. Moreover, for any variational Wardrop equilibrium $\VWE{x}$
\begin{align}
\| \VNE{x} - \VWE{x} \| \le \frac{L_2}{\alpha\emm{\sqrt \N}}.
\label{eq:convergence_strategies_2}
\end{align}
As a consequence, if $\alpha\emm{\sqrt \N} \to \infty$ as $\N \to \infty$, then $\| \VNE{x} - \VWE{x} \| \to 0$ as $\N \to \infty$.
\item If the operator $F\WE$ relative to $\GN$ is strongly monotone on $\mc{Q}$ with monotonicity constant $\alpha\emm>0$, then there exists a unique variational Wardrop equilibrium $\VWE{x}$ of $\GN$. Moreover, for any variational Nash equilibrium $\VNE{x}$
\begin{align}
\| \VNE{x} - \VWE{x} \| \le \frac{L_2}{\alpha\emm{\sqrt \N}}.
\label{eq:convergence_strategies}
\end{align}
As a consequence, if $\alpha\emm{\sqrt \N} \to \infty$ as $\N \to \infty$,
then $\| \VNE{x} - \VWE{x} \| \to 0$ as $\N \to \infty$.
\item If in each game $\GN$ the cost function $J^i(x^i,\sigma(x))$ takes the form~\eqref{eq:costs_specific}, with $v^i = 0$ and $p$ being strongly monotone on $\mc{X}^0$ with monotonicity constant $\alpha$, then there exists a unique $\bar \sigma$ such that $\sigma(\VWE{x})=\bar \sigma$ for any variational Wardrop equilibrium $\VWE{x}$ of $\GN$. Moreover, for any variational Nash equilibrium $\VNE{x}$ of $\GN$ and for any variational Wardrop equilibrium $\VWE{x}$ of $\GN$
\begin{equation}
\| \sigma(\VNE{x}) - \sigma(\VWE{x}) \| \le  \sqrt{\frac{2\Dx L_2}{\alpha \N}}.
\label{eq:convergence_sigma}
\end{equation}
Hence, $\| \sigma(\bar x\NE) - \sigma(\bar x\WE) \| \to 0$ as $\N \to \infty$.\footnote{If $p$ is Lipschitz with constant $L_p$, then in~\eqref{eq:convergence_sigma} $L_2$ can be replaced by $R L_p$, as by~\eqref{eq:Lipschitz_implies_Lipschitz}. This is used in the application in~\cref{sec:PEVs,,sec:traffic}.}
\end{enumerate}
\end{theorem}
We point out that ~\eqref{eq:convergence_strategies_2} and~\eqref{eq:convergence_strategies} can be used to derive a bound on the average strategies similar to~\eqref{eq:convergence_sigma}.
\subsubsection*{Related Works}
\Cref{prop:conv_cost} ensures that, under minimal assumptions, any Wardrop equilibrium is an $\varepsilon$-Nash equilibrium. Such result follows directly from the aggregative structure of the game, and from the Lipschitz continuity of the cost functions. A similar idea is used to prove analogous results in various previous contributions. For example, the case of  potential games is investigated in  \cite{altman2004equilibrium,altman2006survey},  routing games are considered in \cite{altman2011routing},  flow control and routing in communication networks are discussed in \cite{altman2002nash}, while a similar argument is used in \cite{grammatico:parise:colombino:lygeros:14} for the case of average aggregative games with no coupling constraints. \Cref{prop:conv_cost} is a direct extension of these works to  generic aggregative games with coupling constraints.

\cref{thm:conv_strategies} show that it is possible to derive bounds on the Euclidean distance between Nash and Wardrop equilibria at the price of introducing further assumptions.  
More precisely, strong monotonicity of either the Nash or Wardrop operator ensures that the actual strategies $\VNE{x}$ and $\VWE{x}$ converge to each other as $\N$ grows large. A weaker requirement, i.e., the strong monotonicity of $p$ ensures instead convergence in the aggregate. To the best of our knowledge, the only result bounding the Euclidean distance between the two equilibria is obtained in \cite{haurie1985relationship}. Therein a similar bound to~\eqref{eq:convergence_sigma} is derived limitedly to routing/congestion games. However, \cite{haurie1985relationship} requires the population to increase by means of identical replicas of the agents. We here prove that a similar argument can be used to address the case of generic new agents. In addition, the first two results of \cref{thm:conv_strategies} address a more general class of aggregative games (i.e., not necessarily congestion games) by employing a new  type of argument, based on a sensitivity analysis result for variational inequalities with perturbed  strongly  monotone operators \cite[Thm. 1.14]{nagurney2013network}. We note that the works~\cite{Dafermos87,altman2004equilibrium,altman2006survey} guarantee convergence of Nash to Wardrop in terms of Euclidean distance, but do not provide a bound on the convergence rate.

Finally, we observe that our results are derived in relation to variational equilibria. Nevertheless, if there is no coupling constraint as in all the above-mentioned works, then any equilibrium is a variational equilibrium. Hence our results subsume the previous. 
\section{Equilibrium efficiency: the price of anarchy}
\label{sec:eqilibriumefficiency}

In this section we study the efficiency of Nash and Wardrop equilibria by means of the concept of price of anarchy. The notion of equilibrium efficiency was first formalized in \cite{Koutsoupias} and is used to describe the performance degradation incurred when moving from a centralized solution to distributed and strategic decision making. The motivations that lead us to the study of the price of anarchy are essentially two. The first is analytical: given an optimization problem and the corresponding competitive counterpart, we wish to know how inefficient an equilibrium might be. The second stems from the possibility to engineer the behaviour of a large population of strategic thinkers. For example, in the application considered in \cref{sec:PEVs}, the system operator has the freedom to select the price function $p$. In these cases we wish to understand how to modify the game so as to make it as efficient as possible. 

Similarly to previous section, we consider a sequence of games $\GNS$, where each game $\GN$ is defined as in~\eqref{eq:GNEP} with arbitrary constraint sets $\{\mc X\i\}_{i=1}^\N$, and cost functions of the following form
\begin{equation} J\i(x\i,\sigma(x)) \defeq p(\sigma(x)+d)^\top x\i,\qquad d\in\R^n.
\label{eq:costs_specific_POA}
\end{equation}
In order to simplify the exposition, throughout this section we consider the case when no coupling constraint is present, i.e., $\mc{C}=\R^{\N n}$.\footnote{Most of the results hold with minor adaptations in the presence of coupling constraints too.}
We observe that the cost functions in \eqref{eq:costs_specific_POA} have a similar structure to those in \eqref{eq:costs_specific}.
More precisely, it is possible to reduce \eqref{eq:costs_specific_POA} to \eqref{eq:costs_specific}, upon setting $v\i(x\i)=0$ in the latter equation and introducing an additional player whose constraint set is given by $\{x\in\R^n\mid x=d\cdot \N\}$. 
Since we are interested in the case of large population, we do not purse this approach because the unboundedness of this set (as $\N\rightarrow\infty$) will complicate the analysis. 
The costs in~\eqref{eq:costs_specific_POA} can be used to describe applications where $x^i$ denotes the usage level of a certain commodity, whose per-unit cost $p$ depends on the average usage level plus some inflexible normalized usage level $d$~\cite{ma2013decentralized,chen2014autonomous}.
As the notion of equilibrium efficiency relates the behaviour of an equilibrium allocation with that of a socially optimal one, we begin with the following definition.

\begin{definition}[Social optimizer]
\label{def:socopt}
A set of actions $x\SO = [x^1\SO; \dots; x^\N\SO] \in \R^{\N n}$ is a social optimizer of $\GN$  if $x\SO\in\mc{X}$ and it minimizes the cost 
$
J\SO(\sigma(x))\coloneqq p(\sigma(x)+d)^\top(\sigma(x)+d).
$
\end{definition}
\noindent Note that the cost $J\SO$ is the sum of all the players costs, divided by $\N$, and the additional term $p(\sigma(x)+d)^\top d$. The reason why the latter term is included is that we want to compute the  total cost of buying the commodity for both the flexible ($\sigma(x)$) and inflexible ($d$) users.
 This cost was first introduced in \cite{ma2013decentralized} and successively used in \cite{Gonz2015, deori2016nash, deconvergence}.
 For a given a game $\mc{G}_\N$, we quantify the efficiency of equilibrium allocations using the notion of price of anarchy~\cite{Koutsoupias} 
\begin{equation}\label{eq:poa}
\poa_\N \coloneqq \frac{\max_{x_N\in \textup{NE}_\N}J\SO(\sigma(x_N)) }{J\SO(\sigma(x\SO))}\,,
\end{equation}
 where $\textup{NE}_\N\subseteq \mc{X}$ is the set of Nash equilibria of $\mc{G}_\N$ and $x\SO$ is a social optimizer of $\mc{G}_\N$. The price of anarchy captures the ratio between the cost at the worst Nash equilibrium and the optimal cost; by definition $\poa_\N\ge1$. 
In the following we study the behavior of $\poa_\N$, for three different classes of admissible price functions $p$.

\subsection{Linear price function}
Throughout this subsection we consider cost functions of the form \eqref{eq:costs_specific_POA}, where the price functions $p$ is linear as detailed in \cref{ass:lin}. Linear price functions have been used in \cite{Gonz2015,deori2016nash} to model, e.g., the competitive charging of electric vehicles.
\begin{assumption}
\label{ass:lin}
The cost functions are as in \eqref{eq:costs_specific_POA}, where the price function $p$ takes the form $p(z+d)=C(z+d)$, with $C=C^\top\in\R^{n\times n}$, $C\succ0$.
\end{assumption}
Under \cref{ass:lin}, let $L_s$, $L_p$ be the Lipschitz constants of $J\SO$, $p$, and $\alpha$ the monotonicity constant of $p$.\footnote{The function $p(z+d)=C(z+d)$ is strongly monotone since $C\succ0$ with monotonicity constant given by the smallest eigenvalue of $C$.}
The following theorem shows that, under minimal assumptions, any Wardrop equilibrium is also socially optimum irrespective of the population size $\N$. This is no longer the case for Nash equilibria, which nevertheless recover this property when the population size grows.

\begin{theorem}[$\poa_\N$ bound and convergence to 1]
\label{thm:lin}
Let \cref{ass:lin} hold.
\begin{enumerate}
	\item{
	Let each of the constraint set $\{\mathcal{X}\i\}_{i=1}^\N$ be closed, convex, non empty.
	Then, for any fixed game $\mc{G}_\N$ in the sequence $\GNS$, every Wardrop equilibrium $x\WE$ is a social optimizer, i.e., $J\SO(\sigma(x\WE))\le J\SO(\sigma(x)),~\forall x\in \mc{X}$.
	} 
	\item{Assume, in addition, that there exists a convex, compact set $\mathcal{X}\zero\subset\R^n$ such that $\cup_{i=1}^\N \mathcal{X}^i\subseteq{\mathcal{X}\zero}$ for each $\GN$ in $\GNS$.
	Define the constant $c=RL\SO\sqrt{2L_p\alpha^{-1}}$, where $ \Dx = \max_{y \in{\mathcal{X}\zero} } \{\|y \|\}$. Then,
	\begin{equation}
	\textstyle
	J\SO(\sigma(x\SO))\le J\SO(\sigma(x\NE))\le J\SO(\sigma(x\SO))+c/{\sqrt{M}}\,,
	\label{eq:boundjlin}
	\end{equation}
	for any fixed game $\mc{G}_\N$ in the sequence.
	{Thus, if there exists $\hat J\ge 0$ s.t. $J\SO(\sigma(x\SO))>\hat J$ for every game in the sequence $(\mc{G}_\N)_{\N=1}^\infty$, one has}
	{\[
	1\le \poa_\N\le 1+c/\bigl(\hat J\sqrt{\N}\bigl)~~~\text{and}~~
	\lim_{\N\to\infty}\poa_\N=1\,.\]}}
\end{enumerate}
\end{theorem}
\begin{remark}
The previous theorem extends the results of \cite{ma2013decentralized,Gonz2015,deori2016nash,deconvergence}  simultaneously allowing for arbitrary convex constraints, finite populations, and non diagonal price function. Note that the condition $J\SO(\sigma(x\SO))>\hat J\ge0$ is merely technical and required to properly define $\poa_\N$. This condition is trivially satisfied in the most of the applications considered, see, e.g., \cref{sec:PEVs}. Even if the latter condition does not hold, the cost at any Nash equilibrium converges to the minimum cost as $\N\to\infty$, see~\eqref{eq:boundjlin}.
\end{remark}

\subsection{Diagonal price function}
In the following we study the efficiency of Nash and Wardrop equilibria when the cost functions take the form \eqref{eq:costs_specific_POA} and the price function $p(z+d)$ has diagonal structure, i.e., the $t$-th component of $p$ depends only on the corresponding component of the average. We distinguish two cases depending on wether $p_t$ has the same structure for different values of $t$, or not. Towards this goal, we first introduce two useful assumptions.
\begin{assumption}
\label{A1bis}
For $i\in\{1,\dots,\N\}$, the constraint set $\mathcal{X}\i$ is closed, convex, non empty. For $z\in \frac{1}{\N} \sum_{i=1}^\N \mc{X}^i $, the function $z\mapsto p(z+d)$ is continuously differentiable and strongly monotone while $z\mapsto p(z+d)^\top(z+d)$ is strongly convex. Let $L\SO$, $L_p$ be the Lipschitz constant of $J\SO$, $p$, and $\alpha$ be the monotonicity constant of $p$.
\end{assumption}
 
\begin{assumption}
\label{ass:sequence}
There exists a convex, compact set $\mathcal{X}_0\subset\R^n$ s.t. $\cup_{i=1}^\N \mathcal{X}^i\subseteq{\mathcal{X}_0}$ for each game $\mc{G}_\N$ in $(\mc{G}_\N)_{\N=1}^\infty$. Moreover, $J^i(x^i,\sigma(x))$ is convex in $x^i\in\mathcal{X}^i$ for all fixed $x^{-i}\in\mathcal{X}^{-i}$, for all $i\in\{1,\dots,\N\}$. We let $R= \max_{y\in\mc{X}_0}||y||$. 
\end{assumption}
\subsubsection*{Homogeneous price function}
In this section we consider $p(z+d)$ to be a nonlinear function, and assume its $t$-th component to depend only on the $t$-th component $z_t+d_t$, for all $t\in\{1,\dots,n\}$. Additionally, we assume that the functions $p_t$ have the same structure for all the values of $t$. This describes, for example, electricity markets where the unit cost of electricity at every instant of time is captured by a time invariant function depending on the total consumption at that same instant.
\begin{assumption}
\label{ass:nonlin}
The price function $p$ takes the form 
\[
p(z+d)=
\begin{bmatrix}
f(z_1+d_1),\hdots,
f(z_n+d_n)
\end{bmatrix}^\top,
\]
with $f(y):\R_{>0}\rightarrow\R_{>0}$.
Further $\mc{X}\i\subseteq \R^n_{\ge0}$ and $d\in\R^n_{>0}$\,.
\end{assumption}
If $f(y)$ is not linear, a simple check shows that, in general, $\nabla_{x^j}(\nabla_{x^i} J^i(x^i,\sigma(x)))\neq\nabla_{x^i}(\nabla_{x^j} J^j(x^j,\sigma(x)))$ when $i\neq j$. Consequently,  the game is not potential, \cite[Theorem 1.3.1]{facchinei2007finite}. Hence methods to bound the $\poa$ based on the existence of an underlying potential function \cite{Gonz2015, deori2016nash}, can not be used here. The following theorem provides a necessary and sufficient condition on the structure of $f$ that ensures the efficiency of the resulting equilibria.
\begin{theorem}[$\poa_\N$ convergence and counterexample]
\label{thm:polypoa}
Suppose that \cref{A1bis,,ass:sequence,,ass:nonlin} hold. Further assume that $J\SO(\sigma(x\SO))$ $>$$\hat J$ for some $\hat J \ge 0$, for every game in $(\mc{G}_\N)_{\N=1}^\infty$. 
\begin{enumerate}
\item If $f(y)=\alpha y^k$ with $\alpha>0$ and $k>0$, it holds 
\[
	1\le \poa_\N\le 1+c/\bigl(\hat J\sqrt{\N}\bigl)~~~\text{and}~~
	\lim_{\N\to\infty}\poa_\N=1\,,
	\]
	with $c = RL\SO\sqrt{2 L_p \alpha^{-1}}$ constant.
\item For $n\ge 2$, if $f(y)$ satisfies the assumptions, but does not take the form $\alpha y^k$ for some $\alpha>0$ and $k>0$, it is possible to construct a sequence of games $(\mc{G}_\N)_{\N=1}^\infty$ for which $\lim_{\N\to\infty}\poa_\N>1$.
\end{enumerate}
\end{theorem}

The counterexample relative to the second claim is constructed using $\mc{X}\i=\bar{\mc{X}}$. In other words our  impossibility result holds also for the case of 
homogeneous  populations. This is not in contrast with the result in \cite{ma2013decentralized} or \cite{deconvergence}, because therein the sets $\bar{\mc{X}}$ were assumed to be simplexes with upper bounds constraints. Here we claim that there exists a convex set $\bar{\mc{X}}$ (not a simplex with upper bounds) such that $\poa_\N$ does not converge to $1$.
\begin{remark}
	The previous theorem is of fundamental importance in applications where the system operator has the possibility to freely set the price function. In these cases,  \cref{thm:polypoa} suggests the use of monomial price functions to guarantee the highest achievable efficiency (all Nash equilibria become social optimizers for large $\N$). If different price functions are chosen, it is always possible to construct a problem instance such that the worst Nash equilibrium is \emph{not} a social optimizer.
\end{remark}

\subsubsection*{Heterogeneous price function}
In the previous subsection we showed that if the price function is not a monomial, then  $\poa_\N$ may not converge to one. In this section we derive upper bounds for $\poa_\N$ when the price function belongs to a general class of functions, as formalized next.
 \begin{assumption}
\label{ass:nonlin2}
The price function $p$ takes the form 
\[
p(z+d)=
\begin{bmatrix}
l_1(z_1+d_1),
\hdots,
l_n(z_n+d_n)
\end{bmatrix}^\top,
\]
where $l_t(y):\R_{\ge 0}\rightarrow\R_{\ge0}$, $ l_t\in \mathcal{L}$ for all $t$ and $\mathcal{L}$ is a given set of continuous and nondecreasing price functions.  Further let $\mc{X}\i\subseteq \R^n_{\ge0}$ be non empty, closed and convex.
\end{assumption}
Note that \cref{ass:nonlin2} is \emph{less restrictive} than \cref{ass:nonlin} as we let the price $l_t$ depend on the time instant $t$. The key idea in this case is to show that  standard results derived for Wardrop equilibria in routing games \cite{roughgarden2003price}, \cite{correa2004selfish} can be applied to the setup studied here. The resulting bounds on $\poa_\N$ can then be derived using the convergence result in \cref{thm:conv_strategies}.
Formally, given a game $\mc{G}_{\N}$ with cost functions as in \eqref{eq:costs_specific_POA}, we consider an equivalent nonatomic routing game over a parallel network with a number of links equal to $n$, the dimension of decision variables $x\i$. To present our next result, we first introduce the quantity
\[\beta(\mathcal{L}):=\sup_{l\in\mathcal{L}}  \sup_{v\ge 0} \left( \frac{1}{vl(v)}\max_{w\ge 0} [ (l(v)-l(w))w] \right).\]
defined in \cite[Eq 3.8]{correa2004selfish}. Therein, the authors show that $\beta(\mathcal{L})\le 1$ and   $[ 1- \beta(\mathcal{L}) ]^{-1}=\alpha(\mathcal{L})$. The quantity $\alpha(\mathcal{L})$ describes, essentially, the worst-case price of anarchy over all possible cost functions in the set $\mc{L}$, as detailed in the following theorem. The key is to show that the games considered here are $(1,\beta(\mathcal{L}))$-smooth, as defined in \cite[Def. 1.1]{roughgarden2009intrinsic}.

\begin{theorem}[$\poa_\N$ for heterogeneous price function]\label{thm:routing}
~
\begin{enumerate}
	\item 
Suppose that \cref{ass:nonlin2} holds. Then for any fixed game $\mathcal{G}_M$ and any Wardrop equilibrium $x_W$ it holds
\begin{equation}\label{eq:stepThm3}
J_S(\sigma(x_W))\le J_S(\sigma(x_S))\alpha(\mathcal{L})
\end{equation}
\item Further suppose \cref{A1bis,,ass:sequence} hold, and there exists $\hat J\ge 0$ s.t. $J\SO(\sigma(x\SO))>\hat J$ for every game in $(\mc{G}_\N)_{\N=1}^\infty$. 
  Then, for any game $\mathcal{G}_M$ in the sequence
\[
J_S(\sigma(x_S))\le J_S(\sigma(x_N))\le J_S(\sigma(x_S))\alpha(\mathcal{L})+c/\sqrt{M},
\] 
and $1\le\poa_M\le \alpha(\mathcal{L})+c/\bigl(\hat J\sqrt{\N}\bigl),$ thus implying 
$\lim_{M\rightarrow \infty} \poa_M\le \alpha(\mathcal{L}),$
with $c = RL_s\sqrt{2L_p\alpha^{-1}}$.
\end{enumerate}
\end{theorem}
\begin{remark}
In \cite[Table 1]{roughgarden2003price},  $\alpha(\mathcal{L})$ is computed for classes of functions such as affine, quadratic, polynomials. 
If $\mathcal{L}$ contains constant functions, then  \eqref{eq:stepThm3} is tight (see \cite{roughgarden2003price} and the application discussed in \cref{sec:PEVs}).
This is not a contradiction of \cref{thm:lin,,thm:polypoa} because therein either constant functions are not allowed or the price function $p_t$ is assumed to be independent of $t$. \Cref{thm:lin,,thm:polypoa} can be seen as refinements of \cref{thm:routing} and guarantee that $\lim_{M\rightarrow \infty} \poa_\N=1$ by restricting the admissible class of price functions.
\end{remark}

\section{Appendix}
\label{sec:proofsp1-part2}
\subsection{Proofs of the results presented in \cref{sec:Wardrop_Nash}}
\subsubsection*{Proof of \cref{prop:conv_cost}}
\begin{proof}
Consider any Wardrop equilibrium $x\WE$ of $\GN$ (not necessarily a variational one). By~\cref{def:WE}, $x\WE\in\mc{Q}$ and for each agent $i$ 
\[J\i(x\WE^i, \sigma(x\WE)) \le J\i(x^i, \sigma(x\WE)), \qquad \forall x^i\in \mc{Q}\i(x\WE\mi).\] 
It follows that for each agent $i$ and for all  $x^i\in \mc{Q}\i(x\WE\mi)$
\begin{align}
& J\i(x\WE^i, \sigma(x\WE)) - J\i\left(x^i, \frac{1}{\N}\left(x\i + \sum_{j\neq i }x\WE^j\right)\right) \\
&= \underbrace{J\i(x\WE^i, \sigma(x\WE)) - J\i(x^i, \sigma(x\WE))}_{\le 0}+ J\i(x^i, \sigma(x\WE)) - J\i\left(x^i, \frac{1}{\N}\left(x\i + \sum_{j\neq i }x\WE^j\right)\right) \\
& \le  L_2 \left\| \sigma(x\WE) - \left(\frac{1}{\N}\left(x\i + \sum_{j\neq i }x\WE^j\right)\right) \right\| = \frac{L_2}{\N} \left\|\left(x\i\WE + \sum_{j\neq i }x\WE^j\right) - \left(x\i + \sum_{j\neq i }x\WE^j\right) \right\| \\
&=\frac{L_2}{\N} \| x\i\WE - x\i  \| \le \frac{2\Dx L_2}{\N}.
\end{align}
Hence $x\WE$ is an $\varepsilon$-Nash equilibrium of $\GN$ with $\varepsilon = \frac{2\Dx L_2}{\N}$. 
\end{proof}
\subsubsection*{Proof of \cref{thm:conv_strategies}}
\begin{proof}
\begin{enumerate}
\item[]
\item 
We first bound the distance between the operators $F\NE$ and $F\WE$ in terms of $\N$. By~\eqref{eq:F} it holds
\[\begin{split}
\|F\NE(x)-F\WE(x)\|^2 &=\left\| [ \nabla_{x\i} J^i(x\i,\sigma(x)) ]_{i=1}^\N - [\nabla_{x\i} J^i(x\i,z)_{|{z=\sigma(x)}} ]_{i=1}^\N \right\|^2 \\
&= \sum_{i=1}^\N \left\| \frac{1}{\N} \nabla_{z} J\i(x\i,z)_{|{z=\sigma(x)}} \right\|^2 \le  \frac{1}{\N^2} \sum_{i=1}^\N L_2^2 = \frac{L_2^2}{\N},
\end{split}\]
where the inequality follows from the fact that $J^i(z_1,z_2)$ is Lipschitz in $z_2$ on ${\mc{X}\zero}$ with constant $L_2$ by~\cref{A2} and hence the term $\| \nabla_{z} J\i(x\i,z)_{|{z=\sigma(x)}} \|$ is bounded by $L_2$ by definition of derivative.
Taking the square root, it follows that
\begin{equation}
\|F\NE(x)-F\WE(x)\|\le\frac{L_2}{\sqrt{M}}.
\label{eq:distop}
\end{equation}
for all $x \in \mc{X}^0$.
We exploit~\eqref{eq:distop} to bound the distance between Nash and Wardrop strategies. Since $F\NE$ is strongly monotone on $\mc{Q}$ by assumption, $\textup{VI}(\mathcal{Q},F\NE)$ has a unique solution $\VNE{x}$ by~\cref{lem:exun}. Moreover, the distance between the solutions of two variational inequalities differing in the operator used can be bounded using~\cite{nagurney2013network}. Formally, for all solutions $\VWE{x} $ of $\textup{VI}(\mc{Q},F\WE)$ \cite[Thm. 1.14]{nagurney2013network} shows that
\begin{equation}
 \|\VNE{x}-\VWE{x}\| \le\frac{1}{\alpha\emm}\|F\NE(\VWE{x})-F\WE(\VWE{x})\|.
\end{equation}
Combining this with equation \eqref{eq:distop} yields the result.
\item As in the above, with Nash in place of Wardrop and vice versa.
\item Any solution $\VWE{x}$ to the $\textup{VI}(Q,F\WE)$ satisfies
\begin{equation}
\begin{aligned}
&  F\WE(\VWE{x})^\top(x-\VWE{x}) \ge 0, \; \forall x \in Q \Leftrightarrow \\
&\sum_{i=1}^\N p(\sigma(\VWE{x}))^\top(x\i-\VWE{x}^i) \ge 0, \; \forall x \in Q \Leftrightarrow \\
&p(\sigma(\VWE{x}))^\top(\sigma(x) - \sigma(\VWE{x})) \ge 0, \; \forall x \in Q.
\label{eq:VI_WE_sigma}
\end{aligned}
\end{equation}
Any solution $\VNE{x}$ to the $\textup{VI}(Q,F\NE)$ satisfies
\begin{equation}
\begin{aligned}
& F\NE(\VNE{x})^\top(x-\VNE{x}) \ge 0, \; \forall x \in Q \Leftrightarrow \\
& p(\sigma(\VNE{x}))^\top (\sigma(x) - \sigma(\VNE{x}))  +
\frac{1}{\N^2} \sum_{i=1}^\N (\nabla_z p(z)_{|z=\sigma(\VNE{x})} \VNE{x}\i)^\top (x\i - \bar x\i\NE) \ge 0, \; \forall x \in Q .
\label{eq:VI_VNE_sigma}
\end{aligned}
\end{equation}
Exploiting the strong monotonicity of $p$ on $\mc{X}^0$, one has \\
\[\begin{aligned}
 \alpha \| \sigma(\VWE{x}) - \sigma(\VNE{x}) \|^2  &\le ( p(\sigma(\VWE{x}))-p(\sigma(\VNE{x})) )^\top (\sigma(\VWE{x})-\sigma(\VNE{x}) ) \\[2mm]
&={ p(\sigma(\VWE{x}))^\top  (\sigma(\VWE{x})  -  \sigma(\VNE{x}) )}
 -  p(\sigma(\VNE{x}))^\top  (\sigma(\VWE{x})  -  \sigma(\VNE{x})) \\[2mm]
&\hspace*{-3.5mm}\underset{\text{by}~\eqref{eq:VI_WE_sigma}}{\le}
 -  p(\sigma(\VNE{x}))^\top (\sigma(\VWE{x})-\sigma(\VNE{x})) \\
 &\hspace*{-3.5mm}\underset{\text{by~\eqref{eq:VI_VNE_sigma}}}{\le}  \frac{1}{\N^2} \sum_{i=1}^\N (\VNE{x}\i)^\top (\nabla_z p(z)_{|z=\sigma(\VNE{x})})^\top(\VWE{x}^i - \bar x\i\NE) \\
 & = \frac{1}{\N^2} \sum_{i=1}^\N (\VNE{x}\i)^\top (\nabla_z J\i(\VWE{x}\i,z)_{|z=\sigma(\VNE{x})} - \nabla_z J\i(\VNE{x}\i,z)_{|z=\sigma(\VNE{x})})
\\
&  \le  \frac{1}{\N^2} \sum_{i=1}^\N \|\VNE{x}\i\| (\| \nabla_z J\i(\VWE{x}\i,z)_{|z=\sigma(\VNE{x})} \| + \| \nabla_z J\i(\VNE{x}\i,z)_{|z=\sigma(\VNE{x})} \|) \\
&{\le   \frac{2 L_2}{\N^2} \sum_{i=1}^\N \|\VNE{x}\i\| \le \frac{2L_2}{\N^2} \sum_{i=1}^\N \Dx \le \frac 1\N 2 \Dx L_2,}
\end{aligned}\]
where we have used the Chauchy-Schwartz inequality, the triangular inequality, and the Lipschitzianity of $J\i(z_1,z_2)$ in addition to \eqref{eq:VI_WE_sigma} and \eqref{eq:VI_VNE_sigma}.
\noindent We  conclude that $ \| \sigma(\VWE{x}) - \sigma(\VNE{x}) \| \le \sqrt{\frac{2\Dx L_2}{\alpha \N}}.$
\end{enumerate}
\end{proof}

\subsection{Proofs of the results presented in \cref{sec:eqilibriumefficiency}}
Before proving any of the claims in \cref{sec:eqilibriumefficiency}, we provide a lemma that will be useful in the forthcoming analysis. Throughout the following proofs, we denote with 
$\s\coloneqq \frac{1}{\N} \sum_{i=1}^\N \mc{X}^i$.
\begin{lemma}[Equivalent characterizations of $x\WE$, $x\SO$]
 \label{lemma:averageVI}
 Let the cost functions be given as in \eqref{eq:costs_specific_POA}, and each of the constraint set $\{\mathcal{X}\i\}_{i=1}^\N$ be closed, convex, non empty. Additionally, assume that the function $z\mapsto p(z+d)$ is continuously differentiable and strongly monotone while $z\mapsto p(z+d)^\top(z+d)$ is strongly convex, for all $z\in \s$. The following holds.
 \begin{enumerate} 
 \item
 Given $x\WE$ a Wardrop equilibrium, its average $\sigma(x\WE)$ solves $\textup{VI}(\s,F\WE)$, with $F\WE:\mathbb{R}^n\rightarrow\mathbb{R}^n$, 
$F\WE(z) \coloneqq p(z+d)$.
The  $\textup{VI}(\s,F\WE)$ admits a unique solution $\sigma\WE$. Let us define $\mc{X}\WE\coloneqq\{x\in\mc{X}~\text{s.t.}~\frac{1}{\N}\sum_{j=1}^\N x\j=\sigma\WE\}$. Then any vector of strategies $x\WE\in\mc{X}\WE$ is a Wardrop equilibrium.
\item
Given $x\SO$ a social optimizer, its average $\sigma(x\SO)$ solves $\textup{VI}(\s,F\SO)$, with $F\SO:\mathbb{R}^n\rightarrow\mathbb{R}^n$, 
$F\SO(z)\coloneqq p(z+d)+[\nabla_z p(z+d)](z+d).$ The $\textup{VI}(\s,F\SO)$ admits a unique solution $\sigma\SO$. Define $\mc{X}\SO\coloneqq\{x\in\mc{X}~\text{s.t.}~\frac{1}{\N}\sum_{j=1}^\N x\j=\sigma\SO\}$. Then any vector of strategies $x\SO\in\mc{X}\SO$ is a social optimizer.
\end{enumerate}
\end{lemma}
\begin{proof}
~
\begin{enumerate}
	\item 
 The sets $\mc{X}\i$ are convex and closed by assumption; further, for fixed $z\in \s$, the functions $J\i(x\i,z)$ are linear and thus convex in $x\i\in\mc{X}\i$ for all $i\in\{1,\dots,\N\}$. It follows by \cref{prop:vi_ref} that a Wardrop equilibrium $x\WE$ satisfies\footnote{\Cref{prop:vi_ref} goes in both directions here as there is no coupling constraint, see the discussion in  \cref{sec:connection_VI}.}
  \begin{equation}
  [\ones[\N]\otimes \,p(\sigma(x\WE)+d)]^\top (x-x\WE)\ge0,~~\forall x\in\mc{X}.
  	\label{eq:bigvi}
  \end{equation}
Rearranging and dividing by $\N$  we get
$
p(\sigma(x\WE)+d)^\top(\frac{1}{\N}\sum_{j=1}^\N x\i-\frac{1}{\N}\sum_{j=1}^\N x\WE\i)\ge 0,	
$
for all $x\in\mathcal{X}$, or equivalently
$p(\sigma(x\WE)+d)^\top(z-\sigma(x\WE)  )\ge 0,~\forall z\in\s,$
  that is, $\sigma(x\WE)$ solves $\textup{VI}(\s,F\WE)$.
  
  \indent  By assumption $F\WE(z)=p(z+d)$ is strongly monotone and $\s$ is closed, convex (since the sets $\mathcal{X}^i$ are closed, convex), hence by \cite[Thm. 2.3.3]{facchinei2007finite}
   $\textup{VI}(\s,F\WE)$ has a unique solution $\sigma\WE$. By definition of variational inequality, for any $z\in \s$ it holds $p(\sigma\WE+d)^\top(z-\sigma\WE)\ge0$. By definition of $x\WE\in\mc{X}\WE$, we have $\sigma(x\WE)=\sigma\WE$. It follows that $p(\sigma(x\WE)+d)^\top(z-\sigma(x\WE))\ge0$ for any $z\in \s$.  By definition of $\s$, we conclude that \eqref{eq:bigvi} holds for all $x\in\mc{X}$. By \cref{prop:vi_ref}, we conclude that $x\WE$ is a Wardrop equilibrium.
\item
By assumption the set $\mc{X}$ is convex and closed and $J\SO(\sigma(x))$ is  convex. Hence, any social optimizer $x\SO$ satisfies the first order condition in \cref{prop:minimumprinc}
 \begin{equation}
  \label{eq:bigvi2}
\nabla_x[p(\sigma(x)+d)(\sigma(x)+d)]_{|x={x\SO}}^\top (x-x\SO)\ge0
  ~~\forall x\in\mc{X}\,.
  \end{equation}
 Note that $M \nabla_{x^i} (p(\sigma(x)+d)^\top(\sigma(x)+d)) =p(\sigma(x\SO)+d)+[\nabla_z p(\sigma(x\SO)+d)](\sigma(x\SO)+d)$ for all $i\in\{1,\ldots,M\}$. Consequently, \eqref{eq:bigvi2} is equivalent to
$
  [ p(\sigma(x\SO)+d)+\nabla_z p(\sigma(x\SO)+d)(\sigma(x\SO)+d)]^\top 
  (\sigma(x)-\sigma(x\SO))\ge0\,
$. Thus $\sigma(x\SO)$ solves $\textup{VI}(\s,F\SO)$. The remaining claims are shown similarly to those for $x\WE$.
\end{enumerate}
 \end{proof}

\subsubsection*{Proof of \cref{thm:lin}}
\begin{proof}
~
\begin{enumerate}
	\item 
Note that \cref{ass:lin} implies  strong monotonicity of $z\mapsto p(z+d)$, and strong convexity of $z\mapsto p(z+d)^\top(z+d)$. Thus the assumptions of \cref{lemma:averageVI} are satisfied. 
 Let $x\WE$ be a Wardrop equilibrium. By \cref{lemma:averageVI} part 1, $\sigma(x\WE)$ solves $\textup{VI}(\s,F\WE)$. Thanks to \cref{ass:lin}, $F\SO(z)=C(z+d)+C^\top(z+d)=2C(z+d)=2F\WE(z)$. Since the two operators $F\WE(z)$ and $F\SO(z)$ are parallel for each $z\in\s$, it follows from the definition of variational inequality that  $\sigma(x\WE)$ must solve $\textup{VI}(\s,F\SO)$ too. Using \cref{lemma:averageVI} part 2 we conclude that $x\WE$ must be a social optimizer.
\item
By definition $J\SO(\sigma(x\SO))\le J\SO(\sigma(x\NE))$ and so $1\le \poa_\N$.
Observe that the assumptions on the sets $\{\mathcal{X}\i\}_{i=1}^\N$ together with \cref{ass:lin} imply \cref{A1} and ensures that $J\i(z_1,z_2)$ is Lipschitz with respect to $z_2$ in $\mc{X}\zero$. Thus, the assumptions of \cref{thm:conv_strategies} part 3 are satisfied. It follows that for any Nash equilibrium $x\NE$ and Wardrop equilibrium $x\WE$ of the game $\mc{G}_\N$, it holds  
$
||\sigma(x\WE)-\sigma(x\NE)||\le \sqrt{2 R^2L_p \alpha^{-1}{\N}^{-1}}.$ Thus, using the Lipschitz property of $J\SO$ one has that $|J\SO(\sigma(x\NE))- J\SO(\sigma(x\WE))|\le L\SO R\sqrt{2 L_p \alpha^{-1}{\N}^{-1}} = c\sqrt{M^{-1}}.$ Since every Wardrop equilibrium is socially optimum (previous point of this proof), one has $|J\SO(\sigma(x\NE))- J\SO(\sigma(x\SO))|\le c\sqrt{M^{-1}}$ and thus
$J\SO(\sigma(x\NE))\le J\SO(\sigma(x\SO))+c\sqrt{M^{-1}}$. The final result regarding the price of anarchy follows from the latter inequality upon dividing both sides by $J\SO(\sigma(x\SO))>\hat J\ge 0$.
\end{enumerate}
\end{proof}

\subsubsection*{Proof of \cref{thm:polypoa}}
\begin{proof}
~
\begin{enumerate}
\item
We first show that any Wardrop equilibrium is a social optimizer.
To do so, observe that the function $f(y)=\alpha y^k$ satisfies all the assumptions required by \cref{lemma:averageVI} (see \cref{lem:ass} in the Appendix).
Let $x\WE$ be a Wardrop equilibrium of $\mc{G}_\N$. By \cref{lemma:averageVI}, $\sigma(x\WE)$ solves $\textup{VI}(\s,F\WE)$. Thanks to \cref{ass:nonlin} and the choice of $f(y)$,
\[F\SO(z)=(k+1)
[
\alpha(z_1+d_1)^k,
\hdots,
\alpha(z_n+d_n)^k
]^\top
=(k+1)F\WE(z)\,.
\]
Hence $\sigma(x\WE)$ solves $\textup{VI}(\s,F\SO)$ too. Using \cref{lemma:averageVI} we conclude that $x\WE$ must be a social optimizer. 
The proof is now identical to the proof of the second part of \cref{thm:lin}.

\item
  If $f(y)$ does not take the form $\alpha y^k$ for some $\alpha>0$ and $k>0$, by \cref{lemma:notaligned} there exists a point $\bar z\in\R^n_{>0}$ for which $F\WE(\bar z)$ and $F\SO(\bar z)$ are not aligned, i.e., for which $F\SO(\bar z)\neq h F\WE(\bar z)$ for all $h\in\R$. We intend to construct a sequence of games $\mc{G}_\N$ so that for every $\mc{G}_\N$ in the sequence the unique average at the Wardrop equilibrium is exactly $\bar z$, that is  $\bar z$ solves $\textup{VI}(\s,F\WE)$, but $\bar z$ does not solve $\textup{VI}(\s,F\SO)$. This fact indeed proves, by \cref{lemma:averageVI}, that for any game $\mc{G}_\N$  the Wardrop equilibria of $\mc{G}_\N$ are not social minimizers. By \cref{thm:conv_strategies}, $\sigma(x\NE)\to\sigma(x\WE)$ as $\N\to\infty$. Thus, $\poa$ cannot converge to~$1$.

In the following we construct a sequence of games with the above mentioned properties. To this end let us define $\mc{X}\i\coloneqq\bar{\mc{X}}\subseteq \R^n$, so that $\s=\bar{\mc{X}}$ with $\bar{\mc{X}}\coloneqq\{\bar z+\alpha v_1  +\beta v_2~~ \alpha,\beta\in[0~1]\}\cap\R^n_{\ge0},$
  where $v_1\coloneqq\bar F\WE$, $v_2\coloneqq(\bar F\WE^\top\bar F\SO )\bar F\WE- (\bar F\WE^\top \bar F\WE)\bar F\SO $ and $\bar F\WE\coloneqq F\WE(\bar z)$, $\bar F\SO\coloneqq F\SO(\bar z)$; see \cref{fig:setX}. The intuition is that $-v_2$ is the component of $\bar F_S$ that lives in the same plane as $\bar F_S$ and $\bar F_W$ and is orthogonal to $\bar F_W$, so that $\bar F_W^\top v_2=0$. Observe that $\s=\bar{\mc{X}}$ is the intersection of a bounded and convex set with the positive orthant and thus satisfies Assumptions \ref{A1}, \ref{ass:sequence} and \ref{ass:nonlin}.
  It is easy to verify that $\bar z \in \bar{\mc{X}}$ and that $F\WE(\bar z)^\top(z-\bar z)=\alpha ||F\WE(\bar z)||^2\ge0$ for all $z\in \s=\bar{\mc{X}}$, so that $\bar z$ solves $\textup{VI}(\s,F\WE)$. Let us pick $\hat z=\bar z+ \beta v_2$. Note that since $\bar z>0$, for $\beta$ small enough  $\hat z$ belongs to $\R^n_{>0}$ as well and thus to $\bar{\mc{X}}$. Then $F\SO(\bar z)^\top(\hat z-\bar z)=\beta (\bar F\SO^\top\bar F\WE)^2-\beta ||\bar F\SO||^2||\bar F\WE||^2< 0$. The inequality is strict because $\bar F\WE$, $\bar F\SO$ are neither parallel nor zero (\cref{lemma:notaligned}). Thus, $\bar z$ does not solve $\textup{VI}(\s,F\SO)$.
    \begin{figure}[h!]
  		\vspace*{-3mm}
        \centering
        \includegraphics[scale=1]{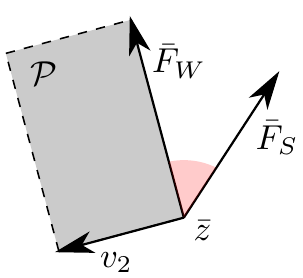}
        \vspace*{-2mm}
        \caption{Construction of the set $\bar{\mc{X}}$.}
        \label{fig:setX}
        \vspace*{-5mm}
   \end{figure}
\end{enumerate}
\end{proof}
\begin{lemma}
\label{lemma:notaligned}
For $n\ge 2$, if $f(y)$ satisfies \cref{A1,,ass:sequence,,ass:nonlin}, but does not take the form $\alpha y^k$ for some $\alpha>0$ and $k>0$, then there exists $\bar z \in\R^n_{>0}$ such that $F\SO(\bar z)\neq h F\WE(\bar z)$, $\forall h\in\R$. Moreover, $F\SO(\bar z)\neq 0$, $F\WE(\bar z)\neq 0$. 
\end{lemma}

\begin{proof}
Let us consider the first statement.
By contradiction, assume there exists $\beta(z):\R^n_{>0}\rightarrow\R$ such that $F\SO(z)= \beta(z) F\WE(z)$ for all $z \in\R^n_{>0}$. This implies
\begin{equation}
\label{eq:alphas}
f'(z_t+d_t)(z_t+d_t)=(\beta(z_1,\dots,z_n)-1)f(z_t+d_t)\,,
\end{equation}

for all $t\in\{1,\dots,n\}$ and for all $z \in\R^n_{>0}$, $d\in\R^n_{>0}$. By \cref{ass:nonlin}, $f(z_t+d_t)>0$. Hence one can divide \eqref{eq:alphas} for $f(z_t+d_t)$ without loss of generality, and conclude that $\beta(z_1,\dots,z_n)=\beta_1(z_1)=\dots=\beta_n(z_n)$ with $\beta_i:\R\rightarrow\R$ for all $z\in\R^n_{>0}$. For $n\ge2$ the last condition implies $\beta(z_1,\dots,z_n)=b$ constant. Equation \eqref{eq:alphas} reads as $f'(y)y=(b-1)f(y)~~\forall y>0$, whose continuously differentiable solutions are all and only $f(y)=a y^{b-1}$. Note that if $a\le0$ or $b\le 1$, \cref{A1} is not satisfied, while if $a>0$ and $b>1$ we contradicted the assumption that $f(y)$ did not take the form $\alpha y^k $ for some $\alpha>0$ and $k>0$. Setting $h=0$ in the previous claim gives $F\SO(\bar z)\ne0$. Since $f:\R_{>0}\rightarrow\R_{>0}$, one has $F\WE(\bar z):=[ f(\bar z_t+d_t)]_{t=1}^n\neq0$.
\end{proof}

\begin{lemma}\label{lem:ass}
Suppose that the price function $p$ is as in \cref{ass:nonlin} with $f(y)=\alpha y^k$, $\alpha>0,k>0$. Then $p$ satisfies \cref{A1bis,,ass:sequence}.
\end{lemma}
\begin{proof}
Note that $\nabla_z p(z+d)$ is a diagonal matrix with entry $f'(z_t+d_t)$ in position $(t,t)$.  Since $f'(y)=\alpha k y^{k-1}>0$ for all $y>0$ and since $z_t+d_t$ is positive by assumption for all $t$, we get that $p(z+d)$ is continuously differentiable and that $\nabla_z p(z+d)\succ 0$, i.e., that $z\mapsto p(z+d)$ is strongly monotone. Similarly, one can show that the Hessian of $p(z+d)^\top (z+d)$ and the Hessian of $J^i(x^i,\sigma(x))$ with respect to $x\i$ are positive definite. Thus, $z\mapsto p(z+d)^\top (z+d)$ and $x^i \mapsto J^i(x^i,\sigma(x))$ are strongly convex.
\end{proof}
\subsubsection*{Proof of \cref{thm:routing}}
\begin{proof} 
We prove only the first claim as the second can be shown as in \cref{thm:lin}. To do so, we define
$C^{\sigma_1}(\sigma_2):=p(\sigma_1+d)^\top(\sigma_2+d)$ so that $J_S(\sigma)=C^{\sigma}(\sigma)$. Let $x_W$ be any Wardrop equilibrium. Then, the average $\bar \sigma:=\sigma_W$ solves VI$(\Sigma,F_W)$, i.e., $F_W(\bar \sigma)^\top(\sigma-\bar \sigma)\ge 0,~ \forall \sigma\in\Sigma.$ This can be seen following the proof of \cref{lemma:averageVI}, and observing that only convexity and closedness of $\mc{X}^i$ are required. Equivalently, $J_S(\bar \sigma) \le C^{\bar \sigma}(\sigma),~\forall \sigma\in\Sigma.$ However,
\[
\begin{split}
C^{\bar \sigma}(\sigma)&=\sum_{t} l_t(\bar \sigma_t+d_t) (\sigma_t+d_t)\\
&=J_S(\sigma)+ \sum_{t} [l_t(\bar \sigma_t+d_t) - l_t( \sigma_t+d_t)] (\sigma_t+d_t) \\
&=J_S(\sigma)+ \sum_{t} \frac{[l_t(v_t) - l_t(w_t)]w_t  }{l_t(v_t)v_t}   l_t(v_t) v_t \\
&\le J_S(\sigma)+ \sum_{t} \beta(\mathcal{L})   l_t(v_t) v_t\\ 
&=  J_S(\sigma)+  \beta(\mathcal{L}) J_S(\bar\sigma)
\end{split}
\] where we used $v_t:=\bar \sigma_t+d_t\ge d_t$, $w_t:= \sigma_t+d_t\ge d_t$ and $d_t\ge 0$. The previous relation holds for all $\sigma\in\Sigma$. Selecting $\sigma=\sigma_S$ (the optimum average), we get $ J_S(\bar \sigma) \le J_S(\sigma_S)+  \beta(\mathcal{L}) J_S(\bar\sigma).$ Rearranging we obtain \eqref{eq:stepThm3}.
\end{proof}

\chapter{Decentralized algorithms}
\label{ch:p1distribtuedalgorithms}
\label{sec:algorithms}

In this chapter we are interested in the design of algorithms that converge to a Nash or a Wardrop equilibrium of a given game $\GN$, formally defined in \eqref{eq:GNEP}. 
All the proofs are reported in the Appendix (\cref{sec:proofsp1-part3}). The results presented in this chapter have been published in \cite{paccagnan2017coupl,gentile2017nash}. 
Throughout the following sections we assume that no agent $i$ wishes to disclose information about his cost function $J\i$ or individual constraint set $\mathcal{X}\i$, to other agents, or to a central operator. Thus, we turn our attention to \emph{decentralized algorithms}. The advantage in using such algorithms is not limited to privacy-preserving issues, but decentralized algorithms are generally preferred when dealing with large scale systems for various reasons, including that of computational tractability. For a comprehensive list of advantages and shortcomings in the use of distributed computing, we redirect the reader to the monograph \cite{bertsekas1989parallel}.
In the following we assume the presence of a central operator able to measure only aggregate quantities, such as the population average $\sigma(x)$, and to broadcast aggregate signals to the agents. \cref{fig:gatherbroadcast} describes the setup more clearly, in relation to \cref{alg:asp}. Based on this information structure, we focus on the design of decentralized algorithms to obtain a solution of either VI$(\mathcal Q, F\NE)$ or VI$(\mathcal Q, F\WE)$. As the techniques are the same for Nash and Wardrop equilibrium, we consider the general problem VI$(\mathcal Q, F)$, where $F$ can be replaced with $F\NE$ or $F\WE$.

Throughout this chapter we assume linearity of the coupling constraints as by \cref{ass:lin}, and observe that this property arises in a range of applications, as detailed, e.g., in~\cite[p. 188]{facchinei2007generalized} and \cite{yi2017distributed}.
\begin{assumption}
\label{ass:lin}
The coupling constraint in~\eqref{eq:coupling_constraints_general} is of the form
\begin{equation}
x\in \mc{C} \defeq \{x\in\mb{R}^{\N n}\,\vert\,Ax\le b\} \subset\mb{R}^{\N n},
\label{eq:coupling_constraints_affine}
\end{equation}
with $A \defeq [A_{(:,1)},\ldots,A_{(:,\N)}] \in\R^{m\times{\N n}}$, $A_{(:,i)} \in\R^{m\times{n}}$ for all $i\in \{1,\dots,\N\}$, $b\in\R^m$. Moreover, for all $i \in \{1,\dots,\N\}$, the set $\mc{X}\i$ can be expressed as $\mc{X}\i = \{x\i \in \R^n \vert g\i(x\i) \le 0 \}$, where $g\i: \R^n \to \R^{p_i}$ is continuously differentiable. The set $Q$, which can thus be expressed as $\mc{Q} = \{x \in \R^{Mn} \vert g\i(x\i) \le 0, \: \forall i,\: Ax \le b \}$, satisfies Slater's constraint qualification~\cite[Eq. (5.27)]{boyd2004convex}. Each agent $i$ has information on the sub-matrix $A_{(:,i)}$ in~\eqref{eq:coupling_constraints_affine}, i.e., he is aware of his influence on the coupling constraint.
\label{A3}
\end{assumption}
If the operator $F$ associated with the variational inequality VI$(\mc{Q},F)$ is integrable\footnote{A necessary and sufficient condition for the integrability of the operator $F$ is that $\nabla_x F(x)=\nabla_x F(x)^\top$ for all $x \in \mc{Q}$~\cite[Thm. 1.3.1]{facchinei2007finite}.} and monotone on $\mc{Q}$, that is, if there exists a convex function $E(x): \R^{\N n}\rightarrow \R$ such that $F(x)=\nabla_x E(x)$  for all $x\in\mc{Q}$, then VI$(\mathcal Q, F)$ is equivalent to the convex optimization problem~\cite[Sec. 1.3.1]{facchinei2007finite}
\begin{equation}
\argmin_{x\in\mathcal{Q}} E(x).
\label{eq:opt}
\end{equation}
Therefore a solution of VI$(\mathcal Q, F)$ and thus a variational equilibrium can be found by applying any of the decentralized optimization algorithms available in the literature of convex optimization \cite{bertsekas1989parallel}; the decentralized structure arises because each agent can evaluate $\nabla_{x^i} E(x)$ by knowing only his strategy $x\i$ and $\sigma(x)$. Since the integrability assumption guarantees that $\mc{G}$ is a \textit{potential game} with potential function $E(x)$ \cite{monderer1996potential}, decentralized convergence tools for potential games \mbox{such as~\cite{dubey2006strategic,MardenCoopControl} can also be employed.}

In light of this observation, our objective is to determine a solution of VI$(\mathcal Q, F)$ when $F$ is not necessarily integrable, so that the previous methods do not apply. In oder to construct a decentralized scheme, we begin by reformulating VI$(\mathcal Q, F)$ in an extended space $[x;\lambda]$ following the spirit of primal-dual methods used in optimization.  The variable $\lambda$ represents the Lagrange multipliers associated to the coupling constraint $\mc C$. 
The following two reformulations will be used to propose two corresponding decentralized algorithms. Formally, for any given $\lambda\in\R^m_{\ge0}$, we define the $\lambda$-dependent game as
\begin{equation}
\mc{G}(\lambda)  \defeq  \left\{ 
\begin{aligned}
&\textup{agents}& &\{1,\dots,M\}\\
&\textup{cost of agent } i& &J^i(x^i,\sigma(x)) + \lambda^\top A(:,i)x^i \\
&\textup{individual constraint}& &\mc{X}^i\\
&\textup{coupling constraint}& &\R^{\N n}
\end{aligned}\right.
\label{eq:GNEP_ext}
,
\end{equation}
and introduce the extended VI$(\mc{Y},T)$, where
\begin{equation}
\label{eq:T}
\begin{split}
\mathcal{Y} \defeq \mc{X}\times \R^m_{\ge0}\,,\quad 
T(x,\lambda) \defeq
\begin{bmatrix}
F(x)+ A^\top\lambda\\
 -(Ax-b)
\end{bmatrix}\,.
\end{split}
\end{equation}
The following proposition draws a connection between VI$(\mathcal Q, F)$, the game $\mc{G}(\lambda)$ and VI$(\mc{Y},T)$.
\begin{proposition}\textup{\cite[Sec. 4.3.2]{scutari2012monotone}}
\label{prop:ext_vi}
Let~\cref{A1,,A3} hold.
The following statements are equivalent.
\begin{enumerate}
\item The vector $\bar x$ is a solution of VI$(\mc{Q},F)$.
\item There exists $\bar \lambda \in \R^m_{\ge 0}$ such that $\bar x$ is a variational equilibrium of $\mc{G}(\bar \lambda)$ and $0\le \bar \lambda \perp b-A\bar x\ge 0$.
\item There exists $\bar \lambda \in \R^m_{\ge 0}$ such that the vector $[\bar x; \bar \lambda]$ is a solution of VI$(\mathcal{Y},T)$. \hfill{$\square$}
\end{enumerate}
\end{proposition}
While the proof is an adaptation of \cite[Sec. 4.3.2]{scutari2012monotone}, we provide a sketch of it for completeness at the end of this chapter. In the following~\cref{subsec:fully,subsec:boundedly} we exploit the equivalence between the statements in \cref{prop:ext_vi} to propose two algorithms that converge to a Wardrop or Nash equilibrium. A numerical comparison of their performance can be found in \cref{ch:p2applications}.
We summarize in \cref{tb:summary} the main conditions that guarantee their convergence.
\begin{table}[h]
\begin{center}
\begin{tabular}{ccc}\hline
& Nash & Wardrop\\ \hline
Best-response& \multirow{2}{*}{ -} &  $F\WE$ strongly monotone \\
(\cref{alg:two})&  &  and \cref{A4} \\ \hline
Gradient-based & \multirow{2}{*}{$F\NE$ strongly monotone} & \multirow{2}{*}{$F\WE$ strongly monotone} \\
(\cref{alg:asp})& & \\ \hline
\end{tabular}
\end{center}
\caption{Range of applicability of the presented algorithms, under~\cref{A1} and~\cref{A2}.}
\label{tb:summary}
\end{table}%

\section{Best-response algorithm for Wardrop equilibrium}
\label{subsec:fully}
Based on the equivalence between the first two statements of~\cref{prop:ext_vi}, we introduce \cref{alg:two}. The algorithm features i) an outer loop, in which the central operator updates and broadcasts to the agents the dual variables $\lambda_{(k)}$ based on the current constraint violation, and ii) an inner loop, in which the agents update their strategies to reach a Wardrop equilibrium of the game $\mc{G}(\lambda_{(k)})$.
Since $\mc{G}(\lambda_{(k)})$ is a game \textit{without} coupling constraints, the Wardrop equilibrium can be found, e.g., via the iterative algorithm proposed in~\cite[Alg. 1]{grammatico:parise:colombino:lygeros:14}.
In order to ease the forthcoming notation, we define for each agent $i\in\{1,\dots,\N\}$ the best-response map to $z\in \frac{1}{\N} \sum_{i=1}^\N \mc{X}\i$ and dual variables $\lambda\in\R^m_{\ge0}$ as
\begin{equation}
x^i_{\textup{br}}(z,\lambda) \defeq \argmin_{x^i \in\mc{X}^i} \; J^i(x^i, z )+\lambda^\top A(:,i)x^i.
\label{eq:or}
\end{equation}
\begin{assumption} \label{A4}
For all $i\in\{1,\dots,\N\}$ and $\lambda\in\R^m_{\ge0}$, the mapping $z\mapsto x^i_{\textup{br}}(z,\lambda)$ is single valued and Lipschitz with constant $L$. Moreover, one of the following holds.
\begin{enumerate}
\item For each $i\in\{1,\dots,\N\}$ and $\lambda \in \R_{\ge 0}^m$, the mapping $z\mapsto x^i_{\textup{br}}(z,\lambda)$ is non-expansive (see \cref{ch:p1mathpreliminaries}).
\item For each $i\in\{1,\dots,\N\}$ and $\lambda \in \R_{\ge 0}^m$, the mapping $z\mapsto z-x^i_{\textup{br}}(z,\lambda)$ is strongly monotone. \qed
\end{enumerate}
\end{assumption}

\begin{algorithm}[h!]
\caption{(Best-response algorithm for Wardrop equilibrium)}
\label{alg:two}
\begin{algorithmic}[1]
\State {\bf Initialise} $k =0$, $\tau>0$, $x_{(0)} \in\R^{n\N}$, $\lambda_{(0)}\in\R^m_{\ge0}$
\While {not converged}
\State $h =0$, $ \tilde x^i_{(0)}=x^i_{(k)}$, $z_{(0)}\in\R^n$.
\vspace*{-1mm}
\Statex \hspace*{6mm}\rule{0.5\textwidth}{.4pt}
\vspace*{-1.8mm}
\While {not converged}
\State $\tilde  x^i_{(h+1)}= x^i_{\textup{br}}(z_{(h)},\lambda_{(k)}) \quad\forall i\in\{1,\dots,\N\}$ \label{alg:eqinf_inner_a} \vspace*{1mm}
\State $\tilde \sigma_{(h+1)} = \frac{1}{\N}\sum_{j=1}^\N \tilde x^j_{(h+1)} $\label{alg:eqinf_inner_b} 
\vspace*{1mm}
\State $z_{(h+1)} = (1-\frac 1 h) z_{(h)}+\frac 1 h \tilde \sigma_{(h+1)} $\label{alg:eqinf_inner_c} 
\vspace*{1mm}
\State $h \leftarrow h+1$
\EndWhile
\vspace*{-4mm}
\Statex \hspace*{6mm}\rule{0.5\textwidth}{.4pt}
\State $x_{(k+1)}=\tilde x_{(h)}$
\State $\lambda_{(k+1)} = \Pi_{\mathbb{R}^{m}_{\ge0}}\left(\lambda_{(k)}-\tau (b-Ax_{(k+1)})\right)$ \label{alg:eqouter}
\State $k\gets k+1$
\EndWhile
\end{algorithmic}
\end{algorithm}

Convergence of the inner loop to a Wardrop equilibrium of the game $\mc{G}(\lambda_{(k)})$ is guaranteed by~\cref{A4} in \cite[Thm. 3 and Cor. 1]{grammatico:parise:colombino:lygeros:14}. 
Additionally, \cite{grammatico:parise:colombino:lygeros:14} provides sufficient conditions for~\cref{A4} to hold, relative to cost functions of the form \eqref{eq:costs_quad}. More precisely, it is shown that $Q\succ 0$ and $C=C^\top\succ0$  or $Q\succ 0$ and $Q-C^\top Q^{-1}C\succ 0$ imply \cref{A4} \cite[Thm.2 ]{grammatico:parise:colombino:lygeros:14}.

\begin{theorem}[Convergence of \cref{alg:two}]
\label{thm:conv_two}
Suppose that the operator $F\WE$ in~\eqref{eq:F_W} is strongly monotone on $\mc{X}$ with constant $\alpha$, that~\cref{A1,,A3,,A4} hold, and that $\mc{X}\i$ is bounded for all $i \in \{1,\dots,\N\}$. If $\tau<\frac{2\alpha}{\|A\|^2}$, then $x_{(k)}$ in \cref{alg:two} converges to a variational Wardrop equilibrium of $\mc{G}$.
\end{theorem}
Two observations on \cref{thm:conv_two} follow. First, we note that the convergence result of \cref{thm:conv_two} holds in the ideal case when, for every fixed $\lambda_{(k)}$, the inner loop converges to the \emph{exact} Wardrop equilibrium. Since this assumption is hardly satisfied due to the finite precision offered by traditional computers, one would like to obtain a guarantee on the convergence of the overall algorithm even if the internal loop provides only an approximate solution. We do not further pursue this direction and instead leave this as a future work.
Second, we observe that the convergence speed of \cref{alg:two} is, to the best of our knowledge, an open question. Nevertheless, it is possible to characterize the convergence rate in each of the two levels separately. In particular, if in \cref{A4} it holds that $z\mapsto z-x^i_{\textup{br}}(z,\lambda)$ is strongly monotone, then it is possible to modify \cref{alg:eqinf_inner_c} with $z_{(h+1)} \leftarrow (1-\frac1\mu)z_{(h)}+\frac{1}{\mu}\tilde \sigma_{(h+1)}$ and guarantee geometric convergence for $\mu\in[0,1]$ small enough, see \cite[Thm. 3.6 (iii)]{berinde}. The outer loop on the other hand has geometric convergence under the additional assumption that the mapping $\Phi$ as defined in the proof of~\cref{thm:conv_two} is not only co-coercive but also strongly monotone.

To the best of our knowledge \cref{alg:two} is the first algorithm that guarantees convergence to a Wardrop equilibrium in games with coupling constraints using a best-response algorithm. We note that, for the case of specific costs \eqref{eq:costs_quad}, \cite{grammatico2017dynamic} proposes a  best-response algorithm that converges to a pair $(\bar x,\bar \lambda)$ such that $\bar x$ is a Wardrop equilibrium of the game $\mc{G}(\bar \lambda)$ satisfying the coupling constraint $\mc{C}$. However such point is not a Wardrop equilibrium because the complementarity condition $0\le \bar \lambda \perp b-A\bar x\ge 0$ is not guaranteed. A \textit{gradient-step} algorithm based on two nested loops for Nash equilibrium with coupling constraints has been proposed in~\cite[Alg. 2]{pang2010design} and in~\cite[Sec. 4]{pavel2007extension}.
\section[Gradient-based algorithm for Nash and Wardrop equilibria]{Gradient-based algorithm for Nash and\\ Wardrop equilibria}
\label{subsec:boundedly}
In this section we devise a decentralized algorithm to achieve a Nash or a Wardrop equilibrium using the reformulation of VI$(\mathcal Q, F)$ as a variational inequality in the the extended space $\mathcal{Y}$, see ~\cref{prop:ext_vi}. 

\cref{alg:asp} proceeds as in the following. After an initialization phase, the agents communicate their current decision variables to the central operator, which in turn broadcasts the initial average and dual variable $\sigma_{(0)}$, $\lambda_{(0)}$ to all agents. At every subsequent iteration the agents update their decision variable and communicate their updated strategy to the central operator, which in turn updates the dual variable to $\lambda_{(k+1)}$ and broadcasts $\sigma_{(k+1)}$, $\lambda_{(k+1)}$ to the agents. \cref{fig:gatherbroadcast} describes the flow of information for \cref{alg:asp}.
\\
\begin{algorithm}[h!]
\caption{(Gradient-based algorithm for Nash equilibrium)}
\label{alg:asp}
\begin{algorithmic}[1]
\State {\bf Initialise} $k =0$, $\tau>0$, $x_{(0)} \in\R^{n\N}$, $\lambda_{(0)}\in\R^m_{\ge0}$
\While {not converged}
\vspace*{2mm}
\State \hspace*{3.8mm}$\sigma_{(k)} =\frac{1}{\N}\sum_{i=1}^\N x^{i }_{(k)}$ \label{alg:eqapa_inner_a}
\vspace*{2mm}
\State $x^{i}_{(k+1)} =\Pi_{\mathcal{X}^i}\left(x^{i }_{(k)} -  \tau \left(\nabla_{ x\i} J\i(x\i_{(k)},\sigma(x_{(k)})) +  {A}_{(:,i)}^\top\lambda_{(k)} \right)\right) \quad  \forall i \in\{1,\dots,\N\}$  \label{alg:eqapa_inner_b}
\vspace*{2mm}
\State $\lambda_{(k+1)}  = \Pi_{\mathbb{R}^{m}_{\ge0}}\left(\lambda_{(k)}-\tau (b-2Ax_{(k+1)}+Ax_{(k)})\right)$alg:eqouter
\label{alg:eqapa_inner_c}
\vspace*{2mm}
\State $k\gets k+1$
\EndWhile
\end{algorithmic}
\end{algorithm}
\\
\begin{remark}
While \Cref{alg:asp} is presented here for the computation of a Nash equilibrium, the same algorithm can be used to compute a Wardrop equilibrium upon replacing $\nabla_{x\i} J\i(x\i_{(k)},\sigma(x_{(k)}))$ with $\nabla_{x\i} J\i(x\i_{(k)},z)\eval$ in \cref{alg:eqapa_inner_b}.
\end{remark}
\begin{figure}[ht]
\centering
\includegraphics[scale=1]{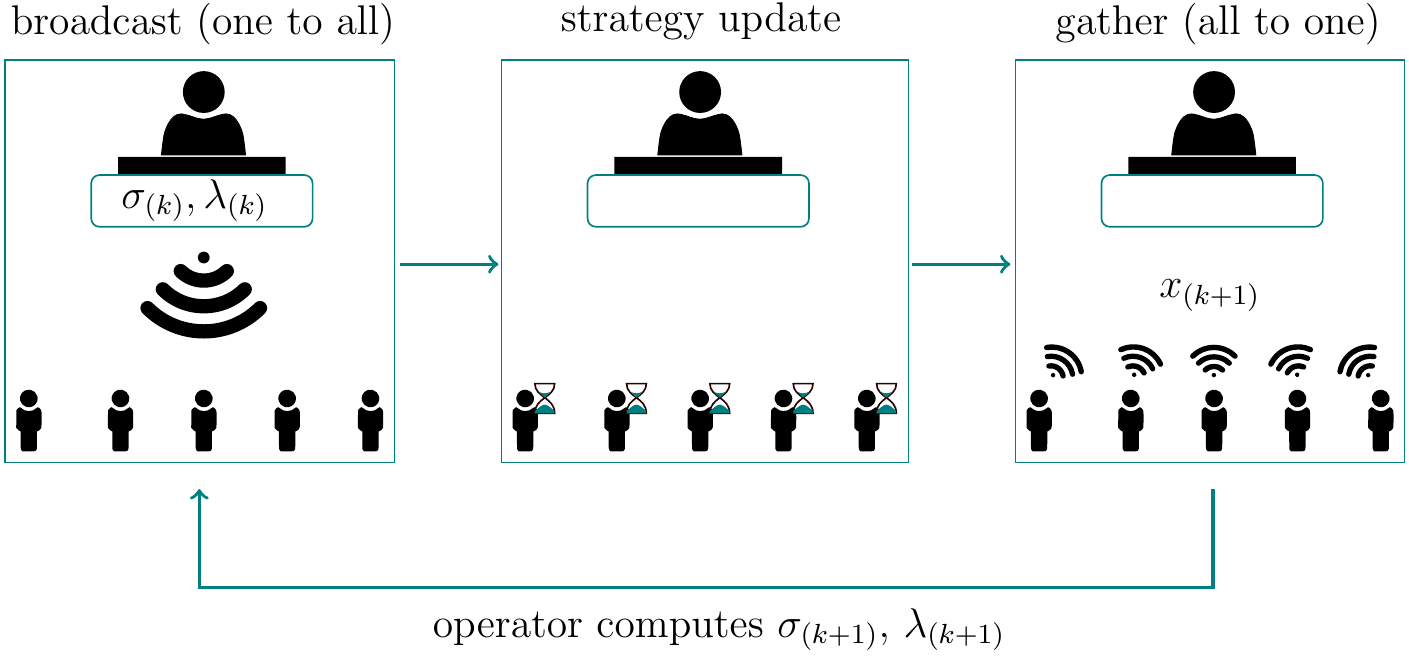}
\caption{Information flow for \cref{alg:asp}}
\label{fig:gatherbroadcast}	
\end{figure}
The fundamental ingredient that guarantees the convergence of \cref{alg:two} is the strong monotonicity of the operator associated to the corresponding variational inequality, as formalized next.
\begin{theorem}
\label{thm:convergence_asp}
Let~\cref{A1} and~\cref{A3} hold. Then
\begin{enumerate}
\item Let $F\NE$ in~\eqref{eq:F_N} be strongly monotone on $\mc{X}$ with constant $\alpha$ and Lipschitz on $\mc{X}$ with constant $L_F$. Let $\tau>0$ s.t.
\begin{equation}
\textstyle \tau <\frac{-L_F^2+\sqrt{L_F^4+4\alpha^2\|A\|^2}}{2\alpha \|A\|^2 }\ . \label{eq:condition_tau_APA}
\end{equation}
Then $x_{(k)}$ in~\cref{alg:asp} converges to a variational Nash equilibrium of $\mc{G}$ in~\eqref{eq:GNEP}.
\item Let $F\WE$ in~\eqref{eq:F_W} be strongly monotone and  Lipschitz on $\mc{X}$ with constants $\alpha$, $L_F$, respectively. Let $\tau$ satisfies~\eqref{eq:condition_tau_APA}. Then~\cref{alg:asp} with $\nabla_{x\i} J\i(x\i_{(k)},z)\eval$ in place of $\nabla_{x\i} J\i(x\i_{(k)},\sigma(x_{(k)}))$ in \cref{alg:eqapa_inner_b} converges to a variational \mbox{Wardrop equilibrium.}
\end{enumerate}
\end{theorem}
\begin{remark}[Convergence rate]
If the operator $F$ is not only monotone but also affine, and the set $\mathcal{X}$ is a polyhedron, then~\cref{alg:asp} converges $R$-linearly for $\tau$ sufficiently small, i.e., $\lim\sup_{k \rightarrow \infty} (\|y_{(k)}-\bar y\|)^{\frac1k}<1$, \cite[Prop. 1]{paccagnan2017coupl}.
\end{remark}

We conclude this section observing that, while there are other gradient-based algorithms that allow to solve VI$(\mathcal{Y},T)$ in a decentralized fashion, they typically require a higher number of gradient steps in each iteration. For example, the extragradient algorithm \cite[Alg. 12.1.9]{facchinei2007finite} requires two updates for both $x$ and  $\lambda$ at each iteration.
  \newenvironment{sma}
  {\left[\begin{smallmatrix}}
  {\end{smallmatrix}\right]}
\section{Appendix}
\label{sec:proofsp1-part3}
\subsection{Proofs of the results presented in \cref{subsec:fully,,subsec:boundedly}}
\subsubsection*{Proof of \cref{prop:ext_vi}}
\begin{proof}
	Under \cref{A1,,A3} the set $\mc{Q}$, and consequently the sets $\{\mc{X}^i\}_{i=1}^\N$, $\mc{X}$ and $\mc{Y}$, are convex and satisfy Slater's constraint qualification.
The VI$(\mathcal{Q},F)$ is therefore equivalent to its KKT system~\cite[Prop. 1.3.4]{facchinei2007finite}.
Moreover, since $\mc{X}^i$ satisfies Slater's constraint qualification, the optimization problem of agent $i$ in the game~\eqref{eq:GNEP_ext} is equivalent to its KKT system, for each $i$.
Finally, by~\cite[Prop. 1.3.4]{facchinei2007finite}, the VI$(\mathcal{Y},T)$ is equivalent to its KKT system.
We do not report the three KKT systems here,
but it can be seen by direct inspection that they are identical~\cite[Section 4.3.2]{scutari2012monotone}.
\end{proof}

\subsubsection*{Proof of \cref{thm:conv_two}}
\begin{proof}
We split the proof of the theorem into two parts. First we show convergence of the inner loop and then of the outer loop.

\textit{Inner loop}.
Using the same approach of \cite[Thm. 3 and Cor. 1]{grammatico:parise:colombino:lygeros:14}, it is possible to show that under~\cref{A4} for any $\lambda_{(k)}\in\R^m_{\ge0}$ the sequences of $z_{(h)}$ and of $\tilde x(h)$ converge respectively to $\bar z$ and to $\bar{{x}}$ such that $\bar z=\frac{1}{\N}\sum_{i=1}^\N x^i_{\textup{or}} (\bar z,\lambda_{(k)})\eqdef\frac{1}{\N}\sum_{i=1}^\N \bar{{x}}^i=\sigma(\bar{{x}})$. In~\cite[Thm. 1]{grammatico:parise:colombino:lygeros:14} it is shown that the set $\{\bar{{x}}^i\}_{i=1}^\N$ is an $\varepsilon$-Nash equilibrium for the game $\mc{G}(\lambda_{(k)})$, with $\varepsilon=\mc{O}(\frac{1}{\N})$. In the following, we show that $\{\bar{{x}}^i\}_{i=1}^\N$ is actually a Wardrop equilibrium of $\mc{G}(\lambda_{(k)})$.
Indeed, for each agent $i$, by the definition of optimal response in~\eqref{eq:or}, one has
\[J^i(\bar x^i,\bar z) +\lambda_{(k)}^\top A_{(:,i)} \bar  x^i \le J^i(x^i,\bar z) +\lambda_{(k)}^\top A_{(:,i)}  x^i, \forall x^i\in\mc{X}^i\,.
\]
Using the fact that $\bar z=\sigma(\bar x)$, we get
\[
J^i(\bar x^i,\sigma(\bar x)) +\lambda_{(k)}^\top A_{(:,i)} \bar x^i \le J^i(x^i,\sigma(\bar x)) +\lambda_{(k)}^\top A_{(:,i)}  x^i, 
\]
for all $x^i\in\mc{X}^i$ and for all $i\in\{1,\dots,\N\}$. Thus $\{\bar x^i\}_{i=1}^\N$ is a Wardrop equilibrium of $\mc{G}(\lambda_{(k)})$ by~\cref{def:WE}.

\textit{Outer loop.}
We follow the steps of the proof of~\cite[Proposition 8]{pang2010design}. For each $\lambda\in\R^m_{\ge0}$ define $F\WE(x;\lambda) \defeq F\WE(x)+A^\top\lambda$. Such operator is strongly monotone in $x$ on $\mc{Q}$ with the same constant $\alpha$ as $F\WE(x)$. It follows by \cref{lem:exun},  that $\mc{G(\lambda)}$ has a unique variational Wardrop equilibrium which we denote by $\bar x\WE(\lambda)$. Note that the outer loop update can be written as 
\[\lambda_{(k+1)}=\Pi_{\R^m_{\ge0}}[\lambda_{(k)}-\tau(b-A\bar x\WE(\lambda_{(k)}))],\]
which is a step of the projection algorithm~\cite[Alg. 12.1.4]{facchinei2007finite} applied to VI$(\R^m_{\ge0}, \Phi)$, with $\Phi(\lambda) \defeq b-A\bar x\WE(\lambda)$.  To conclude, it suffices to show that $\lambda_{(k)}$ converges to a solution $\bar \lambda$ of such VI, because by~\cite[Prop. 1.1.3]{facchinei2007finite}, $\bar \lambda$ solves VI$(\R^m_{\ge0}, \Phi)$ if and only if $0\le \bar \lambda \perp (b-A \bar x\WE(\bar \lambda)) \ge 0$. Having already proved convergence of the inner loop, the conclusion then follows from the second statement of \cref{prop:ext_vi}.

To show that the sequence $\lambda_{(k)}$ converges to a solution of the $\textup{VI}(\R^m_{\ge0},\Phi)$, we prove that the mapping $\Phi$ is co-coercive (see \cref{ch:p1mathpreliminaries}) with co-coercitivity constant $c_\Phi = \alpha/\|A\|^2$ and apply \cite[Thm. 12.1.8]{facchinei2007finite} to conclude the proof. Note that~\cite[Thm. 12.1.8]{facchinei2007finite} requires $\textup{VI}(\R^m_{\ge0},\Phi)$ to have at least a solution; this is guaranteed by the equivalence between the first two statements in~\cref{prop:ext_vi} upon noting that a solution of VI($Q,F$) exists by~\cref{lem:exun}.

To show co-coercitivity of $\Phi$, consider $\lambda_1 , \lambda_2 \in \R^m_{\ge0}$ and the corresponding unique solutions $x_1 \defeq \bar x\WE(\lambda_1)$ of VI($\mc{X}$,$F\WE+A^\top\lambda_1)$ and $x_2 \defeq \bar x\WE(\lambda_2)$ of VI($\mc{X}$,$F\WE+A^\top\lambda_2)$. By definition
\begin{subequations}
\begin{align}
&(x_2-x_1)^\top(F\WE(x_1)+A^\top\lambda_1) \ge 0 \label{eq:VI_1}\,, \\
&(x_1-x_2)^\top(F\WE(x_2)+A^\top\lambda_2) \ge 0 \label{eq:VI_2} \,.
\end{align}
\end{subequations}
Adding~\eqref{eq:VI_1} and~\eqref{eq:VI_2} we obtain $(x_2-x_1)^\top(F\WE(x_1)-F\WE(x_2) + A^\top (\lambda_1-\lambda_2)) \ge 0 $, i.e., $(x_2-x_1)^\top A^\top(\lambda_1-\lambda_2) \ge (x_2-x_1)^\top(F\WE(x_2)-F\WE(x_1))$.
Since $F\WE$ is strongly monotone, it follows from the last inequality that
\begin{equation}
(Ax_2-Ax_1)^\top (\lambda_1-\lambda_2) \ge \alpha \| x_2 - x_1 \|^2 \,.
\label{eq:proof_2_intermed_2}
\end{equation}
Since by definition $\|A(x_2-x_1)\| \le \| A \| \|x_2-x_1 \|$, then
\begin{equation}
\| x_2 - x_1 \|^2 \ge \frac{\| A(x_2-x_1)\|^2}{\| A \|^2} \,.
\label{eq:norm_def_ineq}
\end{equation}
Combining~\eqref{eq:proof_2_intermed_2},~\eqref{eq:norm_def_ineq}, and adding and subtracting $b$, we obtain
\begin{equation}
(b-Ax_2 - (b-Ax_1))^\top   (\lambda_2 - \lambda_1) \ge \frac{\alpha}{\| A \|^2} \| b-Ax_2 - (b-Ax_1) \|^2   ,
\end{equation}
hence $\Phi$ is co-coercive in $\lambda$ with constant $c_\Phi =\alpha/\|A\|^2$.
\end{proof}

\subsubsection*{Proof of \cref{thm:convergence_asp}}
\begin{proof}
We give the proof for a strongly monotone operator $F$, which is to be interpreted as $F\NE$ in the first statement and $F\WE$ in the second statement. We divide the proof into two parts: i) we prove that~\cref{alg:asp} is a particular case of a class of algorithms known as asymmetric projection algorithms (APA) \cite[Alg. 12.5.1]{facchinei2007finite} applied to VI$(\mathcal{Y},T)$; ii) we prove that our algorithm satisfies a convergence condition for APA. It can be shown that if $\tau$ satisfies~\eqref{eq:condition_tau_APA} then also $\tau<1/\|A\|$ holds. 
\\
i) The APA are parametrized by the choice of a matrix $D\succ0$. For a fixed $D$ a step of the APA  for VI$(\mathcal{Y},T)$ is
\begin{equation}
\label{eq:APA}
y_{(k+1)}=\textup{solution of VI}(\mathcal{Y}, T^k_D),
\end{equation}
where $y_{(k)}$ is the state at iteration $k$ and $T^k_{D}(y)\defeq T(y_{(k)})+D(y-y_{(k)})$. Every step of the APA requires the solution of a different variational inequality that depends on the operator $T$, on a fixed matrix $D$ and on the previous strategies' vector $y_{(k)}$. We choose
\begin{equation}
D\coloneqq \left[\begin{array}{cc}\frac{1}{\tau} I_{\N n} & 0 \\-2A & \frac{1}{\tau} I_m\end{array}\right],
\label{eq:choice_D_apa}
\end{equation}
which by using the Schur complement condition can be shown to positive definite because $\tau<1/\|A\|$. It is shown in~\cite[Sec. 12.5.1]{facchinei2007finite} that with the choice~\eqref{eq:choice_D_apa} the update~\eqref{eq:APA} coincides with the steps of \cref{alg:asp}.
\\
\noindent ii) As illustrated in the previous point,~\cref{alg:asp} is the specific APA associated with the choice of $D$ given in \eqref{eq:choice_D_apa}. According to \cite[Prop. 12.5.2]{facchinei2007finite}, this algorithm converges if the mapping $G(y)=D_s^{-1/2}  T(D_s^{-1/2} y) -D_s^{-1/2}  (D-D_s) D_s^{-1/2} y$ is co-coercive with constant $1$, where $D_s = (D+D^\top)/2$ and $D_s^{-1/2}$ denotes the principal square root of the symmetric positive definite matrix $D_s^{-1}$ and is therefore symmetric positive definite. Let us rename $L \defeq D_s^{-1/2}$ and $Ly=\begin{bmatrix} v\\w\end{bmatrix}$ and simplify the expression of $G(y)$\\
\be
\begin{split}
\label{eq:express_G}
G(y)&=L  T(L y) -L  (D-D_s) L y\\
&= L \left( 
\begin{bmatrix} F(v) \\ 0 
\end{bmatrix} +
\begin{bmatrix} 0 & A^\top \\ -A & 0 
\end{bmatrix}Ly +
\begin{bmatrix} 0 \\ b
\end{bmatrix} \right)
-L  
\begin{bmatrix} 0 & A^\top \\ -A & 0 
\end{bmatrix} L y\\
&= L \left( \begin{bmatrix} F(v) \\ 0 \end{bmatrix} +\begin{bmatrix} 0 \\ b\end{bmatrix} \right).
\end{split}
\ee

We now prove that $G(y)$ is co-coercive with constant $1$, i.e., that
\begin{equation}
(y_1-y_2)^\top(G(y_1)-G(y_2)) -  \|G(y_1)-G(y_2)\|^2\ge0.
\label{eq:cocerc_repeated}
\end{equation}
Let us substitute~\eqref{eq:express_G} in the left-hand side of~\eqref{eq:cocerc_repeated}
\be
\begin{split}
&(y_1   -y_2)^\top   (G(y_1)-G(y_2)) -  \|G(y_1)-G(y_2)\|^2 \\
&=(y_1   -y_2)^\top   (L  \begin{bmatrix} F(v_1) \\ 0 \end{bmatrix}   -L  \begin{bmatrix} F(v_2) \\ 0 \end{bmatrix} )    -    \|L  \begin{bmatrix} F(v_1) \\ 0 \end{bmatrix}    -L  \begin{bmatrix} F(v_2) \\ 0 \end{bmatrix} \|^2 \\
&=(Ly_1-Ly_2)^\top(  \begin{bmatrix} F(v_1) -F(v_2)\\ 0 \end{bmatrix}  ) -  \| L  \begin{bmatrix} F(v_1) -F(v_2)\\ 0 \end{bmatrix}  \|^2 \\
&=(\begin{bmatrix} v_1-v_2 \\ w_1-w_2\end{bmatrix})^\top   (  \begin{bmatrix} F(v_1) -F(v_2)\\ 0 \end{bmatrix}  )   -   \begin{bmatrix} F(v_1) -F(v_2)\\ 0 \end{bmatrix} ^\top     L^2 \begin{bmatrix} F(v_1) -F(v_2)\\ 0 \end{bmatrix}   \\
&=  ( F(v_1)  - F(v_2) )^\top [(v_1  -  v_2 )- [L^2]_{11} (F(v_1)  -  F(v_2))] \\
& \ge \alpha \|v_1-v_2\|^2 -  \|[L^2]_{11}\|\| F(v_1) -F(v_2)) \|^2 \\
& \ge \left(\alpha -  \|[L^2]_{11}\|L_F^2 \right)  \|v_1-v_2\|^2 \eqdef K \|v_1-v_2\|^2, 
\end{split}
\ee 
The proof is concluded if $K\ge0.$ Let us compute $[L^2]_{11} = [D_s^{-1}]_{11}$. By inverting the block matrix $D_s$ we get
\begin{equation}\label{eq:l2}
[L^2]_{11} =\tau(I-\tau^2 A^\top A)^{-1} \succ 0. 
\end{equation}
Since $\tau^2 A^\top A$ is symmetric positive semidefinite, $\lambda_\text{max}(\tau^2 A^\top A) = \tau^2 \|A\|^2 <1$ because $\tau<1/\|A\|$ and  $\rho(\tau^2 A^\top A)<1$, i.e., the matrix is convergent. Hence, the Neumann series $\sum_{k=0}^\infty (\tau^2 A^\top A)^k$ converges to $(I-\tau^2 A^\top A)^{-1}$. Substituting in \eqref{eq:l2} yields 
\[[L^2]_{11} = \tau \sum_{k=0}^\infty (\tau^2 A^\top A)^k\succeq 0 \quad\text{and}\quad \|[L^2]_{11}\|\le \tau \sum_{k=0}^\infty (\tau^2\|A\|^2)^k= \frac{\tau}{1-\tau^2\|A\|^2},\]
where we used the fact that the geometric series converges since $\tau^2 \|A\|^2 <1$. Therefore $K\ge \alpha -  \frac{\tau}{1-\tau^2\|A\|^2} L_F^2$. By condition \eqref{eq:condition_tau_APA} we get 
$\alpha \tau^2 \|A\|^2+\tau L_F^2 <\alpha$ and thus
\[
K\ge \frac{\alpha - \alpha \tau^2 \|A\|^2 - \tau L_F^2}{1-\tau^2 \|A\|^2 } >0.
\]
\end{proof}

\chapter{Applications}
\label{ch:p1applications}
In this chapter we verify the theoretical results derived in the previous two chapters. In particular, we consider a coordination problem arising in the charging of electric vehicles, and a selfish routing model used in road traffic network.
All the proofs are reported in the Appendix (\cref{sec:proofsp1-part4}). The results presented in this chapter have been published in \cite{paccagnan2016aggregative,paccagnan2017coupl,gentile2017nash}.
\section{Charging of electric vehicles}
\label{sec:PEVs}
Electric-vehicles (EV) are foreseen to significantly penetrate the market in the coming years~\cite{nemry2010plug}, therefore coordinating their charging schedules can provide useful services for the operation of the grid, e.g., peak shaving, ancillary services~\cite{gan2013optimal}. In the following we model this  problem as a game, where vehicles owners wish to minimize their total electricity bill, while requiring a sufficient final state of charge. By assuming that the electricity price depends on the aggregate consumption,~\cite{ma2013decentralized,grammatico:parise:colombino:lygeros:14} formulate the EV charging problem as an aggregative game and propose decentralized schemes, in the absence of coupling constraints. In this section, we show how the results derived in the previous chapters can be used to study this problem. In particular, our formulation extends the existing literature by introducing coupling constraints and by relaxing the assumptions required for the convergence of the corresponding algorithms.\footnote{Coupling constraints model limits on the aggregate peak consumption or on the local consumption of EVs connected to the same transformer.} In addition, we study the performance degradation of an equilibrium configuration, when compared to the centralized optimal solution. Finally, we establish uniqueness of the dual variables associated to the violation \mbox{of the coupling constraints.}

In the remainder of this section, we consider a population of $\N$ electric vehicles and identify with agent $i$ the corresponding vehicle $i \in \{1,\dots,\N\}$. Additionally, we identify with $s\i_t$ the state of charge of vehicle $i$ at time $t$.  The time evolution of $s^i_t$ is specified by the discrete-time system $s\i_{t+1} = s\i_t + b\i x\i_t \,,  t = 1, \dots, n$, where $x\i_t$ is the charging input and the parameter $b\i > 0$ captures the charging efficiency.
\subsection*{Constraints}
 We assume that the charging input cannot take negative values and that at time $t$ it cannot exceed $\tilde x^i_t \ge 0$. The final state of charge is constrained to $s_{n+1}^i\ge\eta\i$, where $\eta\i \ge 0$ is the desired state of charge of agent $i$. Denoting with $x\i =[x\i_1, \dots, x^i_n]^\top \in \R^n$, the individual constraint of agent $i$ can be expressed as
\begin{equation}
\label{eq:vehicle_constraint}
x\i \in \mc{X}\i \vcentcolon= \left\{ x\i \in \mathbb{R}^n  \left|
\begin{array}{l}
0 \le x\i_t \le \tilde x\i_t, \;\; \forall \, t=1,\dots,n \\ 
\sum_{t=1}^{n} x\i_t \ge \theta\i
\end{array}
\right.
\right\},
\end{equation}
where $\theta\i \coloneqq {(b\i)}^{-1} (\eta\i - s\i_1)$, with $s\i_1 \ge 0$ the state of charge at the beginning of the time horizon. Besides the individual constraints $x\i \in \mc{X}\i$, we introduce \mbox{the coupling constraint}
\begin{equation}
x\in\mc{C}\defeq\left\{x\in\R^{\N n}\left|\, \frac{1}{\N}\sum_{i=1}^\N x\i_t \le K_t,  \, \forall \, t=1,\dots,n\right.\right\},
\label{eq:coupling_global_PEVs}
\end{equation}
indicating that  at time $t$ the grid cannot deliver more than $\N K_t$ units of power to the vehicles. In compact form~\eqref{eq:coupling_global_PEVs} reads as
$(\ones[\N]^\top \otimes I_n) x \le\N K \,,$
where $K \defeq [K_1, \dots, K_n]^\top$.
\subsection*{Cost function}
The cost function of each vehicle represents its electricity bill, which we model as
\begin{equation}
 J\i(x\i,\sigma(x))=\sum_{t=1}^n p_t \left( \frac{d_t + \sigma_t(x)}{\kappa_t}  \right) x\i_t \eqdef p(\sigma(x))^\top x^i,
\label{eq:PEV_energy_bill}
\end{equation}
where we have assumed that the energy price for each time interval $p_t:\R_{\ge0}\rightarrow \R_{>0}$ depends on the ratio between total consumption and total capacity $(d_t + \sigma_t(x))/ \kappa_t$, where $d_t$ and $\sigma_t(x)\defeq\frac{1}{\N}\sum_{i=1}^\N x^i_t$ are the non-EV and EV demand at time $t$ divided by $\N$ and $\kappa_t$ is the total production capacity divided by $\N$ as in~\cite[Eq. (6)]{ma2013decentralized}. The quantity $\kappa_t$ is in general not related to $K_t$.
\subsection{Theoretical guarantees}
We  define the  game $\mc{G}^\text{EV}_\N$ as in~\eqref{eq:GNEP}, with $\mathcal{X}\i$, $\mathcal{C}$ and $J\i(x\i,\sigma(x))$ as in \eqref{eq:vehicle_constraint}, \eqref{eq:coupling_global_PEVs} and \eqref{eq:PEV_energy_bill} respectively.
In the following corollary we refine the main results of ~\cref{ch:p1equilibriaVIandSMON}, \ref{ch:p1distanceWENEandPOA} and \cref{ch:p1distribtuedalgorithms} for the EV application.
\begin{corollary}
\label{cor:pev} 
Consider a sequence of games $(\mc{G}^\textup{EV}_M)_{M=1}^\infty$. Assume that there exists $\tilde x^0$ such that $\tilde x^i_t \le \tilde x^0$ for all $t \in \{1,\dots,n\},i \in \{1,\dots,\N\}$ and for each game $\mc{G}^\textup{EV}_M$. Moreover, assume that for each game $\mc{G}^\textup{EV}_M$ the set $\mc{Q}=\mc{C}\cap\mc{X}$ is non-empty and that for each $t$ the price function $p_t$ in \eqref{eq:PEV_energy_bill} is twice continuously differentiable, strictly increasing and Lipschitz in $[0,\tilde x^0]$ with constant $L_p$.
Then:
\begin{enumerate}
\item
A Wardrop and a Nash equilibrium exist for each game $\mathcal{G}^\textup{EV}_M$ of the sequence. Furthermore, every Wardrop equilibrium is an $\varepsilon$-Nash equilibrium with $\varepsilon = \frac{2n (\tilde{x}^0)^2 L_p}{\N}$.
\item The function $p$ is strongly monotone, hence for each game $\mathcal{G}^\textup{EV}_M$ there exists a unique $\bar \sigma$ such that $\sigma(\VWE{x})=\bar \sigma$ for any variational Wardrop equilibrium $\VWE{x}$ of $\mathcal{G}^\textup{EV}_M$. Moreover for any variational Nash equilibrium $\VNE{x}$ of $\mathcal{G}^\textup{EV}_M$, $\| \sigma(\VNE{x}) - \sigma(\VWE{x}) \| \le \tilde x^0 \sqrt{\frac{2 n L_p}{\alpha M}}$, where $\alpha$ is the monotonicity constant of $p$. 
\item Assume that there is no coupling constraint, i.e., $\mc{C}=\R^{\N n}$, that $d_t>0$ for all $t$, and that $\sum_{i=1}^{\N} \theta\i >0$. If $p_t\left(\frac{d_t+\sigma_t(x)}{\kappa_t}\right)=\alpha\left(\frac{d_t+\sigma_t(x)}{\kappa_t}\right)^k$ with $\alpha>0$, $k>0$, then 
\[
	1\le \poa_\N\le 1+\mathcal{O}\left(1/\sqrt{\N}\right)~~~\text{and}~~
	\lim_{\N\to\infty}\poa_\N=1\,.
\]
\item Assume that
\begin{equation}
\minn{\substack{t \in \{1,\dots,n\}\\z \in [0,\tilde x^0]}}{\left(p'_t(z) - \frac{\tilde x^0 p''_t(z)}{8}\right)} > 0.
\label{eq:bound_PEV_smon2}
\end{equation}
For each game $\mathcal{G}^\textup{EV}_\N$ the operator $F\NE$ is strongly monotone. Hence, if \cref{ass:lin} holds, \cref{alg:asp} converges to a variational Nash equilibrium of $\mathcal{G}^\textup{EV}_M$. 
\end{enumerate}
\end{corollary}
We note that the previous corollary provides guarantees on the equilibrium efficiency for the case of polynomial price functions. Nevertheless, different results can be obtained in the case of affine or diagonal price function by applying the bounds derived in \cref{thm:lin,,thm:routing}. In this respect, the third statement of \cref{cor:pev} is purely exemplificative.

\subsection*{Uniqueness of dual variables}
\cref{cor:pev} shows that under condition~\eqref{eq:bound_PEV_smon2} the operator $F\NE$ of $\mathcal{G}^\textup{EV}_\N$ is strongly monotone, hence the game $\mathcal{G}^\textup{EV}_\N$ admits a unique variational Nash equilibrium (\cref{lem:exun}). We study here the uniqueness of the associated dual variables $\bar \lambda\NE$ introduced in~\cref{prop:ext_vi}. Guaranteeing unique dual variables is important to convince the vehicles owners to participate in the proposed scheme, as $\bar \lambda\NE$ represent the penalty price associated to the coupling constraint. 
Define $R^\text{tight} \subseteq \{1,\dots,n\}$ as the set of instants in which $\mc{C}$ is active. We provide a sufficient condition for uniqueness of the dual variables which relies on a modification of the linear-independence constraint qualification~\cite{Wachsmuth2013}.
\\
\begin{lemma}
\label{lem:uniqueness_for_PEVs}
Assume that condition~\eqref{eq:bound_PEV_smon2} holds and consider the unique variational Nash equilibrium $\VNE{x}$ of $\mc{G}^\textup{EV}_M$. If there exists a vehicle $i$ such that 
$\bar x_{\textup{N},t}^i \notin \{0,\tilde x\i_t \}$ for all $t \in R^\textup{tight}$ and
$\bar x_{\textup{N},t'}^i \notin \{0,\tilde x\i_{t'} \}$ for some $t'\notin R^\textup{tight}$,
then the dual variables $\VNE{\lambda}$ associated to the coupling constraint~\eqref{eq:coupling_global_PEVs} are unique. 
\end{lemma}

We note that the sufficient condition of~\cref{lem:uniqueness_for_PEVs} has to be verified a-posteriori as it depends on the primal solution $\VNE{x}$. In the numerical analysis presented in the following such sufficient condition always holds. Uniqueness of the dual variables associated to the coupling constraint of an aggregative game has been studied also in~\cite[Thm. 4]{yin2011nash}, where the conditions in the bullets of~\cref{lem:uniqueness_for_PEVs} are not required, but $p$ is restricted to be affine.
\subsection{Numerical analysis}\label{sec:pev_num}
The numerical study is conducted on a heterogeneous population of agents. We set the price function to  $p_t(z_t)=0.15 \sqrt{z_t}$ and $n=24$. The agents differ in $\theta\i$, randomly chosen according to $\mathcal{U}[0.5,1.5]$; they also differ in $\tilde x\i_t$, which is chosen such that the charge is allowed in a connected interval, with left and right endpoints uniformly randomly chosen. Within this interval, $\tilde x\i_t$ is constant and randomly chosen for each agent according to $\mathcal{U}[1,5]$, while outside this interval $\tilde x\i_t = 0$.
The demand $d_t$ is taken as the typical (non-EV) base demand over a summer day in the United States~\cite[Fig. 1]{ma2013decentralized}; $\kappa_t=12$ kW for all $t$, and the upper bound $K_t=0.55$ kW is chosen such that the coupling constraint~\eqref{eq:coupling_global_PEVs} is active in the middle of the night.
Note that with these choices all the assumptions of~\cref{cor:pev} are met. In particular, for the given choice of $p$ condition~\eqref{eq:bound_PEV_smon2} holds because $p''_t(z) < 0$ for all $z$ and all $t$. \Cref{fig:primal_variables} presents the aggregate consumption at the Nash equilibrium found by~\cref{alg:asp}, with stopping criterion $\|(x_{(k+1)},\lambda_{(k+1)})-(x_{(k)},\lambda_{(k)})\|_{\infty} \le 10^{-4}$. Note that without the coupling constraint the quantity $\bar \sigma + d$ would be constant overnight, as shown in~\cite{ma2013decentralized}.
\vspace*{5mm} 
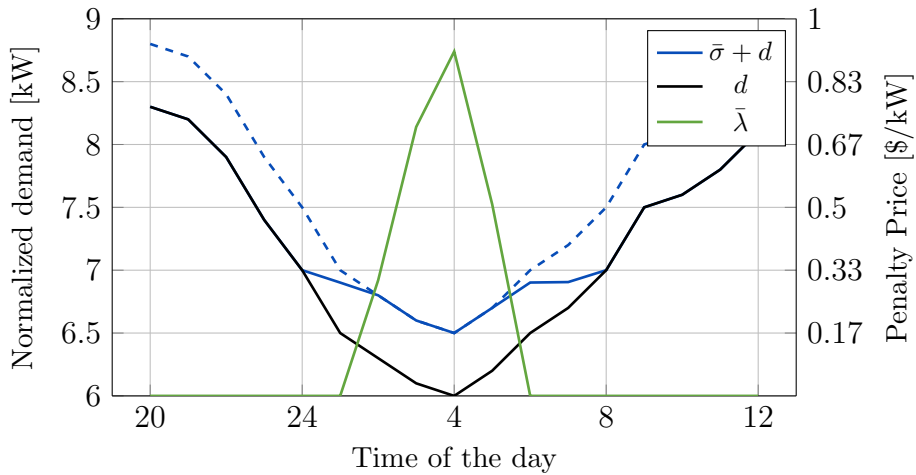
\begin{figure}[H] 
\begin{center}
\setlength\figureheight{5cm} 
\setlength\figurewidth{9cm} 
%
%
\definecolor{mycolor1}{rgb}{0.03,0.3,0.72}%
\definecolor{PortlandGreen}{RGB}{99,166,63}%
\begin{tikzpicture}

\begin{axis}[%
width=\figurewidth,
height=\figureheight,
axis y line* = left,
scale only axis,
xmin=7, xmax=25, xtick={8,12,16,20,24}, xticklabels={{20},{24},{4},{8},{12}}, xlabel={\small Time of the day}, xmajorgrids,
ymin=6, ymax=9, ytick={6,6.5,7,7.5,8,8.5,9}, ylabel={\small Normalized demand [kW]}, ylabel near ticks, ymajorgrids,
]
\addplot [color=mycolor1,solid,line width=1.0pt]
  table[row sep=crcr]{%
8	8.30000000390642\\
9	8.20000000478831\\
10	7.9000000074229\\
11	7.40000002443321\\
12	7.00003584722159\\
13	6.90146739363718\\
14	6.80006156578355\\
15	6.60004473212242\\
16	6.50001520367497\\
17	6.70001392123971\\
18	6.90188130551423\\
19	6.90551484838831\\
20	7.0000246029291\\
21	7.50000000931836\\
22	7.60000000381587\\
23	7.80000000115141\\
24	8.1\\
};\label{plot_one}

\addplot [color=black,solid,line width=1.0pt]
  table[row sep=crcr]{%
8	8.3\\
9	8.2\\
10	7.9\\
11	7.4\\
12	7\\
13	6.5\\
14	6.3\\
15	6.1\\
16	6\\
17	6.2\\
18	6.5\\
19	6.7\\
20	7\\
21	7.5\\
22	7.6\\
23	7.8\\
24	8.1\\
};\label{plot_two}

\addplot [color=mycolor1,dashed,line width=1.0pt]
  table[row sep=crcr]{%
8	8.8\\
9	8.7\\
10	8.4\\
11	7.9\\
12	7.5\\
13	7\\
14	6.8\\
15	6.6\\
16	6.5\\
17	6.7\\
18	7\\
19	7.2\\
20	7.5\\
21	8\\
22	8.1\\
23	8.3\\
24	8.6\\
};
\end{axis}

\begin{axis}[
width=\figurewidth,
height=\figureheight,
axis y line*=right,
ymin = 0, ymax = 1, ytick={0.16666666,0.3333333,0.5,0.666666,0.83333333,1}, yticklabels={0.17,0.33,0.5,0.67,0.83,1}, ylabel={\small Penalty Price [${\$}$/{kW}]}, ylabel near ticks,
xmin=7, xmax=25, axis x line=none,
scale only axis,
legend style={at={(0.985,0.97)},font=\footnotesize}
     ]
\addplot [color=PortlandGreen,solid,line width=1.0pt,forget plot]
  table[row sep=crcr]{%
8	0\\
9	0\\
10	0\\
11	0\\
12	0\\
13	0\\
14	0.306979585957893\\
15	0.713221133089155\\
16	0.913221133089153\\
17	0.508096714131238\\
18	0\\
19	0\\
20	0\\
21	0\\
22	0\\
23	0\\
24	0\\
};\label{plot_three}
\addlegendimage{/pgfplots/refstyle=plot_one}\addlegendentry{$\bar{\sigma}+d$}
\addlegendimage{/pgfplots/refstyle=plot_two}\addlegendentry{$d$}
\addlegendimage{/pgfplots/refstyle=plot_three}\addlegendentry{$\bar{\lambda}$}
\end{axis}

\end{tikzpicture}%
\caption{ Aggregate EV demand $\sigma(\VNE{x})$ and dual variables $\bar \lambda\NE$ for $\N=100$, subject to $\sigma(x) \le 0.55$ kW. The region below the dashed line satisfies $\sigma(x)+d \le  0.55$ kW$+d$.}
\label{fig:primal_variables}
\end{center}
\end{figure}
\Cref{fig:distance_aggregates} illustrates the bound $\| \sigma(\VNE{x}) - \sigma(\VWE{x}) \| \le \tilde x^0 \sqrt{\frac{2 n L_p}{\alpha M}}$ of the second statement of~\cref{cor:pev}. The Wardrop equilibrium is computed with the extragradient algorithm with stopping criterion $\|(x_{(k+1)},\lambda_{(k+1)})-(x_{(k)},\lambda_{(k)})\|_{\infty} \le 10^{-4}$.
The framework introduced above can also be used to enforce local coupling constraints, i.e., constraints on a subset of all the vehicles. These can for instance be used to model capacity limits for local substations as we discuss in \cite[Fig. 4]{paccagnan2017coupl}.
\begin{figure} [H]
\begin{center}
\setlength\figureheight{5cm} 
\setlength\figurewidth{9.cm} 
%
%
\definecolor{mycolor1}{rgb}{0.00000,0.54118,0.90196}%
\begin{tikzpicture}

\begin{axis}[%
width=\figurewidth,
height=\figureheight,
scale only axis,
xmin=0,
xmax=800,
ymin=0,
ymax=0.15,
xtick={0,100,200,300, 400, 500, 600, 700, 800},
tick label style={/pgf/number format/fixed},
ytick={0,0.03,0.06,0.09,0.12,0.15},
xmajorgrids,
ymajorgrids,
xlabel ={Population size $M$},
legend style={at={(0.97,0.95)},legend cell align=left,align=left,draw=white!15!black,font=\footnotesize}
]
\definecolor{mycolor1}{rgb}{0.00000,0.54118,0.90196}

\addplot [color=black,solid,line width=1.0pt]
  table[row sep=crcr]{%
50	0.0961423497667815\\
100	0.0549965612736588\\
150	0.0368506226735783\\
200	0.0281709955720355\\
250	0.0233737077229369\\
300	0.0218642951846739\\
350	0.0180020163292226\\
400	0.0162783969936725\\
450	0.01454678132398\\
500	0.0119194105921108\\
550	0.0115542608715654\\
600	0.0101954546561331\\
650	0.0101114713636592\\
700	0.00906486857544332\\
750	0.00816693178529436\\
800	0.00792344623874586\\
};
\addlegendentry{\footnotesize $\|\sigma(\bar x_{_{N}})-\sigma(\bar x_{_{W}})\|$}

\addplot [dashed,line width=1.0pt]
  table[row sep=crcr]{%
50	0.14142135623731\\
100	0.1\\
150	0.0816496580927726\\
200	0.0707106781186548\\
250	0.0632455532033676\\
300	0.0577350269189626\\
350	0.0534522483824849\\
400	0.05\\
450	0.0471404520791032\\
500	0.0447213595499958\\
550	0.0426401432711221\\
600	0.0408248290463863\\
650	0.0392232270276368\\
700	0.0377964473009227\\
750	0.0365148371670111\\
800	0.0353553390593274\\
};
\addlegendentry{\footnotesize $1/\sqrt{M}$}

\end{axis}
\end{tikzpicture}%
\caption{Distance between the aggregates at the Nash and Wardrop equilibrium (solid line). \Cref{cor:pev} ensures that such distance is upper bounded by  $\tilde x^0 \sqrt{2 n L_p\alpha^{-1}/\N}$. The dotted line shows $1/\sqrt{M}$  proving that our bound captures the correct trend.}
\label{fig:distance_aggregates}
\end{center}
\end{figure}
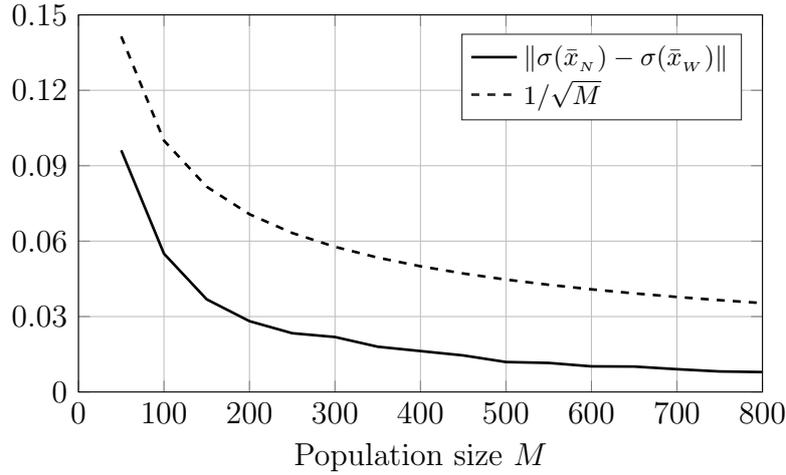
\vspace*{-5mm}
\subsubsection*{The case of linear price function}
Different works in the EV literature~\cite{grammatico:parise:colombino:lygeros:14,kristoffersen2011optimal} use the cost~\eqref{eq:costs_quad}, with $Q \succ 0$ and $C\succ 0$, diagonal. Existence of a Nash and of a Wardrop equilibrium is guaranteed by~\cref{lem:exun}, while~\cref{prop:conv_cost} gives the $\varepsilon$-Nash property.  Further,~\cref{lem:quadratic} shows that the resulting operators $F\NE$ and $F\WE$ are strongly monotone with monotonicity constant independent from $\N$. \Cref{thm:conv_strategies} ensures then that $\| \VNE{x} - \VWE{x} \| \le L_2/(\alpha \sqrt{\N})$, with $L_2 = R \cdot \lambda_\textup{M}$, where $\lambda_\textup{M}$ represents the largest eigenvalue of $C$.
A Nash equilibrium can be found using \cref{alg:asp}, while a Wardrop equilibrium can be achieved using both \cref{alg:two,,alg:asp}. \Cref{fig:iterations_x} presents a comparison between the two algorithms in terms of iteration count, where $Q=0.1 I_n$, $C=I_n$, $c^i=d\; \text{for all} \; i$. \Cref{fig:iterations_x} (top) represents the number of strategy updates required to converge, i.e., the number of times \cref{alg:eqinf_inner_a} in \cref{alg:two} or \cref{alg:eqapa_inner_b} in \cref{alg:asp} is used.  
\Cref{fig:iterations_x} (bottom) depicts the number of dual variables updates, i.e., the number of times \cref{alg:eqouter} in \cref{alg:two} or \cref{alg:eqapa_inner_c} in \cref{alg:asp}  
is used. For both algorithms the number of iterations does not seem to increase  with the population size.  \Cref{alg:asp} requires fewer primal iterations, while~\cref{alg:two} needs much fewer dual iterations.
\begin{figure}[ht!]
\begin{center}
\setlength\figureheight{5cm} 
\setlength\figurewidth{9cm}  
%
%
\begin{tikzpicture}
\definecolor{mycolor1}{rgb}{0.03,0.3,0.72}%
\definecolor{mycolor2}{rgb}{0.93333,0.69804,0.00000}%
\begin{axis}[%
width=\figurewidth,
height=\figureheight,
scale only axis,
xmin=45,
xmax=205,
xmajorgrids,
ymin=0,
ymax=450,
ymajorgrids,
xticklabels={,,},
ylabel = {\small Primal updates},
legend pos = outer north east,
legend image post style={scale=0.3},
]
\addplot[color=mycolor1,line width=1.0pt]
 plot [error bars/.cd, y dir = both, y explicit, error bar style={line width=1pt},   error mark options={
      rotate=90,
      line width=1pt}]
 table[row sep=crcr, y error plus index=2, y error minus index=3]{%
50	366.9	33.8214429023955	33.8214429023955\\
60	345.6	19.4020617461135	19.4020617461135\\
70	363.7	35.8637700193385	35.8637700193385\\
80	343.9	31.1077160845987	31.1077160845987\\
90	349.7	22.1858062733812	22.1858062733812\\
100	347.6	22.4686448189471	22.4686448189471\\
110	338.5	22.9314194937863	22.9314194937863\\
120	333.6	6.78527818147495	6.78527818147495\\
130	343.6	23.5592869162036	23.5592869162036\\
140	350.7	40.2369233416274	40.2369233416274\\
150	334.5	6.72681202353685	6.72681202353685\\
160	336.9	10.8761206319165	10.8761206319165\\
170	345.1	24.2670558576849	24.2670558576849\\
180	334.3	5.93380147965872	5.93380147965872\\
190	335.1	4.92848861214064	4.92848861214064\\
200	346.7	20.6060670677352	20.6060670677352\\
};
\addlegendentry{\cref{alg:two}};

\addplot [solid,line width=1.0pt]
 plot [error bars/.cd, y dir = both, y explicit, error bar style={line width=1pt},   error mark options={
      rotate=90,
      line width=1pt}]
 table[row sep=crcr, y error plus index=2, y error minus index=3]{%
50	38.8	11.7456374880208	11.7456374880208\\
60	60.3	42.2470117286418	42.2470117286418\\
70	59.4	29.8167738026769	29.8167738026769\\
80	75	47.6969600708473	47.6969600708473\\
90	69.6	32.3425416440947	32.3425416440947\\
100	50.9	19.423954283307	19.423954283307\\
110	63.7	21.8771570365073	21.8771570365073\\
120	77.4	43.1490440218552	43.1490440218552\\
130	63.5	22.477766792989	22.477766792989\\
140	94.4	50.5711380136931	50.5711380136931\\
150	74.7	20.0800896412342	20.0800896412342\\
160	77.5	24.2002066106883	24.2002066106883\\
170	90.3	25.6009765438743	25.6009765438743\\
180	71.2	18.4325798519903	18.4325798519903\\
190	93	38.5201246103903	38.5201246103903\\
200	79.6	24.654411369976	24.654411369976\\
};
\addlegendentry{\cref{alg:asp}};
\end{axis}
\end{tikzpicture}%
\\[0.5cm]
%
%
\begin{tikzpicture}
\definecolor{mycolor1}{rgb}{0.03,0.3,0.72}%
\definecolor{mycolor2}{rgb}{0.93333,0.69804,0.00000}%
\begin{axis}[%
width=\figurewidth,
height=\figureheight,
scale only axis,
xmin=45,
xmax=205,
xmajorgrids,
ymin=0,
ymax=150,
ymajorgrids,
ylabel = {\small Dual updates},
xlabel = {\small Population size $M$},
legend pos = outer north east,
legend image post style={scale=0.3},
]
\addplot [color=mycolor1,solid,line width=1.0pt]
 plot [error bars/.cd, y dir = both, y explicit, error bar style={line width=1pt},   error mark options={
      rotate=90,
      line width=1pt}]
 table[row sep=crcr, y error plus index=2, y error minus index=3]{%
50	29	3.03315017762062	3.03315017762062\\
60	26.5	0.670820393249937	0.670820393249937\\
70	29.1	2.77308492477241	2.77308492477241\\
80	26.5	1.62788205960997	1.62788205960997\\
90	28.1	1.92093727122985	1.92093727122985\\
100	27.3	2.00249843945008	2.00249843945008\\
110	27.2	1.72046505340853	1.72046505340853\\
120	26.8	0.4	0.4\\
130	27.4	1.85472369909914	1.85472369909914\\
140	27.7	1.552417469626	1.552417469626\\
150	27.2	0.6	0.6\\
160	26.8	1.16619037896906	1.16619037896906\\
170	26.9	1.04403065089105	1.04403065089105\\
180	26.8	0.4	0.4\\
190	26.9	0.3	0.3\\
200	27.3	1.73493515728975	1.73493515728975\\
};
\addlegendentry{\cref{alg:two}};

\addplot[solid,line width=1.0pt]
 plot [error bars/.cd, y dir = both, y explicit, error bar style={line width=1pt},   error mark options={
      rotate=90,
      line width=1pt}]
 table[row sep=crcr, y error plus index=2, y error minus index=3]{%
50	38.8	11.7456374880208	11.7456374880208\\
60	60.3	42.2470117286418	42.2470117286418\\
70	59.4	29.8167738026769	29.8167738026769\\
80	75	47.6969600708473	47.6969600708473\\
90	69.6	32.3425416440947	32.3425416440947\\
100	50.9	19.423954283307	19.423954283307\\
110	63.7	21.8771570365073	21.8771570365073\\
120	77.4	43.1490440218552	43.1490440218552\\
130	63.5	22.477766792989	22.477766792989\\
140	94.4	50.5711380136931	50.5711380136931\\
150	74.7	20.0800896412342	20.0800896412342\\
160	77.5	24.2002066106883	24.2002066106883\\
170	90.3	25.6009765438743	25.6009765438743\\
180	71.2	18.4325798519903	18.4325798519903\\
190	93	38.5201246103903	38.5201246103903\\
200	79.6	24.654411369976	24.654411369976\\
};
\addlegendentry{\cref{alg:asp}};
\end{axis}
\end{tikzpicture}
\caption{ Primal (top) and dual (bottom) updates required to converge; mean and standard deviation for $10$ repetitions. As \cref{alg:asp} performs one primal and one dual update in each iteration, the black lines appearing in the two figures coincide.}
\label{fig:iterations_x}
\end{center}
\end{figure}
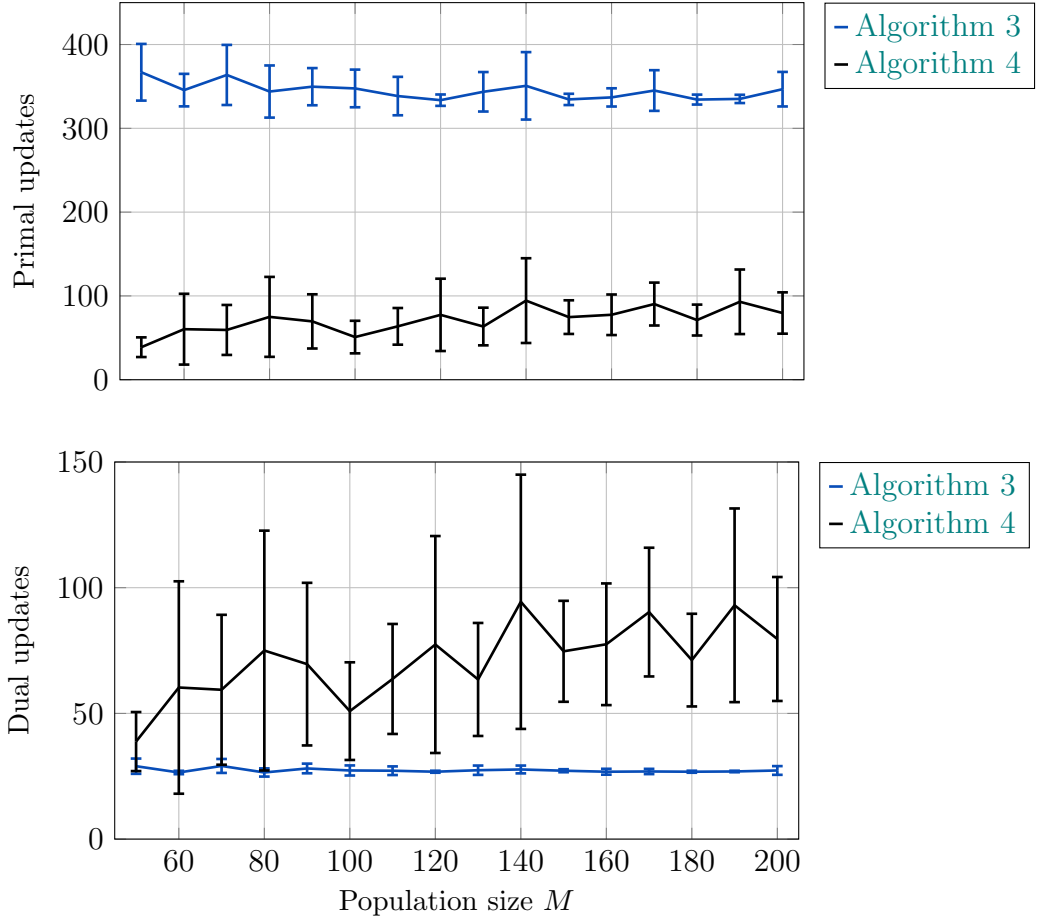

\subsubsection*{Equilibrium efficiency}
In this section we verify the theoretical results on the efficiency of equilibria obtained in \cref{cor:pev}, by means of numerical simulations. We consider four cases as follows.
\begin{itemize}[leftmargin=23mm]
	\item[{\bf Case 1.}] 
	We set $p_t(y)=0.15y^3$ and choose $\tilde x^i_t$ to allow charging in $[t^i_{\textup{min}},t^i_{\textup{max}}]$, with $t^i_{\textup{min}},t^i_{\textup{max}}$ uniformly randomly
distributed between 5pm and 10am; 
$\theta^i\sim\mathcal{U}[5, 15]$ and $d_t$ as in \cite[Fig. 1]{ma2013decentralized}.
\item[{\bf Cases $2$-$4$.}] 
We set $p_t(y)=0.15$ from 5pm to 1am and $p_t(y)=0.15y$ from 2am to 10am. For all vehicles, we choose $\tilde x^i_t$ to allow charging from 5pm to 10am. Cases 2-4 differ in $\theta^i$, $d_t$ as in the following table.
\begin{table}[H]
\centering
\begin{tabular}{|c|c|c|}
\hline
 Case& $\theta^i$ & $d_t$ \\ \hline
2 & $9$ & $\zeros[n]$\\
3 &$9$ & as in \cite[Fig. 1]{ma2013decentralized} \\
4 &$\mathcal{U}[5,13]$ & $\zeros[n]$ \\ \hline
\end{tabular}
\end{table}
\end{itemize}

\noindent For each case, we report the (numerical) price of anarchy as a function of $\N$ in \cref{fig:poa_and_diff} (top).
Observe that case $1$ and $4$ feature heterogenous charging needs. For these cases, we have randomly extracted $100$ games $\mc{G}^{\text{EV}}_\N$ (for any fixed $\N$) and report the worst $\poa$ amongst the $100$ realization.
In order to plot the price of anarchy, we computed the ratio between \emph{one} (instead of the \emph{worst}) Nash equilibrium of $\mc{G}^\text{EV}_\N$ and the social optimum. This choice is imposed by the fact that computing all Nash equilibria of $\mc{G}^\text{EV}_\N$ is in general a hard problem.\footnote{This is due to the fact that the operator associated with the variational inequality of the Nash problem is not guaranteed to be strongly monotone since condition \eqref{eq:bound_PEV_smon2} does not hold due to the choices of $p(t)$ in Cases 1-4.
To compute a Nash equilibrium we applied the extragradient  algorithm \cite{facchinei2007finite}, which is though not guaranteed to converge. We thus verified a posteriori that the point where the algorithm stopped was a Nash equilibrium.}
In \cref{fig:poa_and_diff} (bottom) we plot the difference between the cost at the Nash and at the social optimizer, relative to case~1.
\setlength\figureheight{4cm} 
\setlength\figurewidth{8cm} 
\begin{figure}[H]
\begin{center}
\begin{tikzpicture}

\begin{axis}[%
width=\figurewidth,
height=\figureheight,
at={(1.011111in,0.641667in)},
scale only axis,
xmin=0,
xmax=153,
xmajorgrids,
xtick={3, 10, 20, 30, 40 , 50, 60, 70, 80, 90, 100, 120, 150},
ymin=0.99,
ymax=1.35,
ytick={1, 1.05, 1.10, 1.15, 1.20, 1.25, 1.30, 1.35},
ymajorgrids,
tick label style={font=\small},
ylabel={$\poa_\N$},
legend style={at={(0.98,0.87)},anchor=north east, row sep=-3pt}
]

\addplot [color=black,solid,mark=o,mark options={solid},line width = 0.8pt]
  table[row sep=crcr]{%
3	1.146440320202752\\
5	1.084144206845967\\
7	1.060048849309918\\
10	1.024232150863788\\
15	1.016372686491602\\
20	1.014257536279855\\
30	1.004045564995405\\
40	1.002948190905122\\
50	1.003581851865993\\
60  1.002083234902703\\
70	1.003119089620615\\
80  1.001584796583014\\
90	1.001117731557795\\
100	1.001180347871256\\
120 1.000782030475284\\
150 1.000777440384723\\ 
};
\addlegendentry{\footnotesize Case 1}

\addplot [color=red,solid,mark=triangle,mark options={solid},line width = 0.8pt]
  table[row sep=crcr]{%
3	1.083333271041991\\
5	1.148147949310946\\
7	1.187499715409593\\
10	1.223140306615917\\
15	1.255208121274737\\
20	1.272864593882444\\
30	1.291709937733426\\
40	1.301606129192305\\
50	1.307702065977234\\
60 	1.311833603676443\\
70	1.314818263630529\\
80  1.317075457595903\\
90  1.318841734023247\\
100	1.320261976515319\\
120 1.322404157939918\\
150 1.324560203528079\\ 
};
\addlegendentry{\footnotesize Case 2}

\addplot [color=blue,solid,mark=diamond,mark options={solid},line width = 0.8pt]
  table[row sep=crcr]{
3	1.051328159113605\\
5	1.059896867679024\\
7	1.064396255666115\\
10	1.068183912171345\\
15	1.071412416249527\\
20	1.073131774140623\\
30	1.074924501049531\\
40	1.075850149130869\\
50	1.076415196445974\\
60  1.076795962142782\\
70	1.077069926099951\\
80  1.077268560416244\\
90  1.077437745707750\\
100	1.077567147972555\\
120 1.077749962830918\\
150 1.077942306748157\\ 
};
\addlegendentry{\footnotesize Case 3}

\addplot [color=green,solid,mark=square,mark options={solid},line width = 0.8pt]
  table[row sep=crcr]{
3	1.116195356881952\\
5	1.176091225026618\\
7   1.207709846007565 \\
10	1.235579439042558\\
15	1.260437393215178\\
20	1.277438692108011\\
30	1.295324853591825\\
40	1.303891968957424\\
50	1.309541572633434\\
60  1.312817040389996\\
70	1.315639950431473\\
80  1.317885457093774\\
90  1.319246690488960\\
100	1.320741186315272\\
120 1.322749824875607\\
150 1.324731164796894\\
};
\addlegendentry{\footnotesize Case 4}

\end{axis}
\end{tikzpicture}%
\\[0.5cm]
 \input{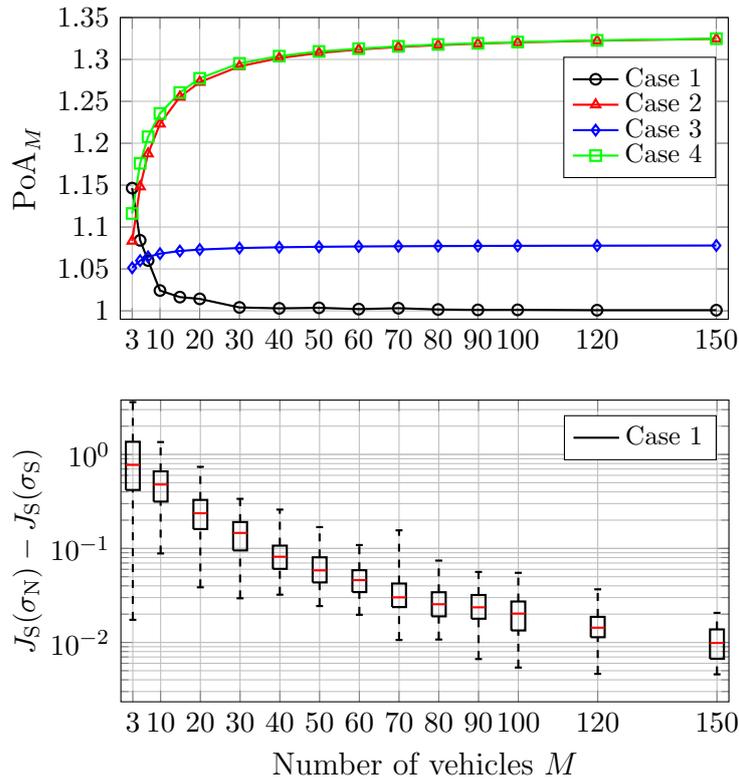}
\caption{Price of anarchy (top), and cost difference between Nash and social optimum (bottom) as a function of $\N$.}
\label{fig:poa_and_diff}
\end{center}
\end{figure}

Thanks to the choice of parameters and price function, the third statement in \cref{cor:pev} guarantees that $\lim_{M\rightarrow \infty} \poa_M=1$. The numerical results reported in~\cref{fig:poa_and_diff} (top, black line) are consistent with it: the ratio between the cost at the Nash and the cost at the social optimum converges to one. In addition to this, \cref{fig:poa_and_diff} (bottom) shows that also the difference between these costs converges to zero, as guaranteed in the proof of \cref{thm:polypoa} by the boundedness of $\mc{X}_0$.
Case 2 has been constructed so that the corresponding Wardrop equilibrium features the worst possible asymptotic price of anarchy within the class of affine cost functions (for which $\alpha(\mathcal{L})=4/3$, see \cite{roughgarden2003price}). The numerics of \cref{fig:poa_and_diff} (top, red line) show that $\poa_\N$ (i.e., the efficiency of Nash equilibria) converges to $1.33\approx 4/3=\alpha(\mathcal{L})$. Cases 3 and 4 are a modification of case 2. While the presence of base demand (case 3) helps in lowering the price of anarchy, the impact of heterogeneity (case 4) on the asymptotic price of anarchy is minor (blue and green plots in \cref{fig:poa_and_diff}).
\vspace*{-3mm}
\section{Route choice in a road network}
\label{sec:traffic}
As second application we consider that of traffic routing in a road network.
Traffic congestion is a well-recognized issue in densely populated cities, and the corresponding economic costs are significant~\cite{arnott1994economics}. Since every driver seeks his own interest (e.g., minimizing the travel time) and is affected by the others' choices via congestion, a classic approach is to model the traffic problem as a game~\cite{dafermos1980traffic}. In the following we focus on a stationary model that aims at capturing the basic interactions among the vehicles flow during rush hours. Building upon our theoretical findings, we derive results specific for the route choice game. Moreover, we perform a realistic numerical analysis based on the data set of the city of Oldenburg in Germany~\cite{OldenburgDataSet}. Specifically, we investigate via simulation the effect of road access limitations, expressed as coupling constraints~\cite{sandmo1975pigovian}.

We consider a strongly-connected directed graph $(\mc{V},\mc{E})$ with vertex set $\mc{V} = \{1,\dots,V\}$, representing geographical locations, and directed edge set $\mc{E} = \{1,\dots,E\} \subseteq \mc{V} \times \mc{V}$, representing roads connecting the locations. Each agent $i \in \{1,\dots,\N\}$ represents a driver who wants to drive from his origin $o^i \in \mc{V}$ to his destination $d^i \in \mc{V}$.
\vspace*{-2mm}
\subsection*{Constraints}
Let us introduce the vector $x^i \in [0,1]^E$ to describe the strategy (route choice) of agent $i$, with $[x^i]_e$ representing the probability that agent $i$ transits on edge $e$~\cite{de2005route}. To guarantee that agent $i$ leaves his origin and reaches his destination with probability 1, the strategy $x\i$ has to satisfy
\begin{equation}
\sum_{e \in \text{in}(v)} [x\i]_e - \sum_{e \in \text{out}(v)} [x\i]_e =
\begin{cases}
-1 &\text{if} \quad v = o^i \\
1 &\text{if} \quad v = d^i \\
0 &\text{otherwise},
\end{cases} \qquad \forall \; v \in \mc{V},
\label{eq:traffic_constraints}
\end{equation}
where $\text{in}(v)$ and $\text{out}(v)$ represent the set of in-edges and the set of out-edges of node $v$. We denote  the graph incidence matrix by $B \in \mb{R}^{V \times E}$,  so that $[B]_{ve} = 1$ if edge $e$ points to vertex $v$, $[B]_{ve} = -1$ if edge $e$ exits vertex $v$ and $[B]_{ve} = 0$ otherwise. The individual constraint set of agent $i$ is then
\begin{equation}
\mc{X}^i \defeq \{x \in [0,1]^E \,|\, Bx = b\i \},
\label{eq:local_traffic}
\end{equation}
where $b^i \in \mb{R}^V$ is such that $[b^i]_v = -1$ if $v = o^i$,  $[b^i]_v = 1$ if $v = d^i$ and $[b^i]_v = 0$ otherwise. We introduce the constraint
\begin{equation}
x\in\mc{C}\defeq\{x\in\R^{\N E}\mid \textstyle \frac{1}{\N}\sum_{i=1}^\N x\i_e \le K_e,  \, \forall \, e=1,\dots,E\},
\label{eq:coupling_global_traffic}
\end{equation}
expressing the fact that the number of vehicles on edge $e$ cannot exceed $\N K_e$. Such coupling constraint can be imposed by authorities to decrease the congestion in a specific road or neighborhood, with the goal of reducing noise or pollution.
\subsection*{Cost function}
We assume that each driver $i\in\{1,\dots,\N\}$ wants to minimize his travel time and, at the same time, does not want to deviate too much from a preferred route $\tilde x^i \in \mc{X}\i$. We model this objective with the following cost function 
\begin{equation}
J^i(x\i,\sigma(x)) = \frac{\gamma^i}{2} \|x^i-\tilde x^i\|^2 + \sum_{e=1}^E t_e(\sigma_e(x_e)) x\i_e,
\label{eq:cost_traffic}
\end{equation}
with $\gamma^i\ge0$ a weighting factor, $x_e \defeq [x^1_e,\ldots,x^\N_e]^\top$, $\sigma_e(x_e) = \frac{1}{\N} \sum_{i=1}^{\N} x\i_e$ and $t_e(\sigma_e(x_e))$ the travel time on edge $e$.
\subsection*{Travel time}
This subsection is devoted to the derivation of the analytical expression of the travel time $t_e(\sigma_e(x_e))$. The reader not interested in the technical details of the derivation can proceed to the expression of $t_e(\sigma_e(x_e))$ in~\eqref{eq:travel_time_smoothed}, which is illustrated in~\cref{fig:t_cont_diff}.

In the following, we introduce the quantity $D_e(x_e) = \sum_{i=1}^\N x\i_e$ to describe the total demand on edge $e$ and consider a rush-hour interval $[0,h]$. We assume that the instantaneous demand equals ${D_e(x_e)}/{h}$ at any time $t\in[0,h]$ and zero for $t > h$. Additionally, we assume that edge $e$ can support a maximum flow $F_e$ (vehicles per unit of time) and features a free-flow travel time $t_{e,\text{free}}$. As we are interested in comparing populations of different sizes, we further assume that the peak hour duration $h$ is independent from the population size $\N$ and that  the road maximum capacity flow $F_e$ scales linearly with the population size, i.e., $F_e(\N) = f_e \cdot \N$, with $f_e$ constant in $\N$. The consideration underpinning this last assumption is that the road infrastructure scales with the number of vehicles to accommodate the increasing demand, similarly as what assumed in~\cite{ma2013decentralized} for the energy infrastructure.

If $D_e(x_e)/h\le F_e$ then every car has instantaneous access to edge $e$ and no queue accumulates, hence the travel time equals $t_{e,\text{free}}$. We focus in the rest of this paragraph on the case $D_e(x_e) / h> F_e$. An increasing queue forms in the interval $[0,h]$ and decreases at rate $F_e$ for $t>h$. The number of vehicles $q_e(t)$ queuing on edge $e$ at time $t$ obeys then the dynamics
\begin{equation}
\dot q_e(t) =  \begin{cases} \frac{D_e(x_e)}{h}\cdot \boldsymbol{1}_{[0,h]}(t) - F_e & \text{if} \; q_e(t) \ge 0 \\
\; 0  & \text{otherwise}, \end{cases} \quad q_e(0) = 0,
\label{eq:queue_dynamics}
\end{equation}
where $\boldsymbol{1}_{[0,h]}$ is the indicator function of $[0,h]$. The solution $q_e(t)$ to~\eqref{eq:queue_dynamics} is hence
\begin{equation}
q_e(t) =  \begin{cases} \left(\frac{D_e(x_e)-F_e h}{h}\right) t \quad &\text{if} \; 0 \le t \le h \\
D_e(x_e) - F_e \, t &\text{if} \; h \le  t \le D_e(x_e) / F_e \\
\; 0  & \text{if} \; t \ge D_e(x_e) / F_e. \end{cases}
\label{eq:queue_expression}
\end{equation}
As a consequence, the total queuing time at edge $e$ (i.e, the queuing times summed over all vehicles) is the integral of $q_e(t)$, which equals $D_e(x_e) (D_e(x_e)-F_e h) / (2F_e)$; the queuing time is then $(D_e(x_e)-F_e h)/(2F_e)$.

Since $\sigma_e(x_e)   =   \frac 1\N \sum_{i=1}^\N x\i_e   =   \frac 1\N D_e(x_e)$, the travel time is
\begin{equation}
t_e^\textup{PWA}(\sigma_e(x_e)) =  \begin{cases} t_{e,\text{free}} & \text{if} \; \sigma_e(x_e) \le f_e h \\
t_{e,\text{free}} + \frac{\sigma_e(x_e)-f_e h}{2f_e}  & \text{otherwise,} \end{cases}
\label{eq:travel_time}
\end{equation}
and is reported in \cref{fig:t_cont_diff}. Note that $t_e^\textup{PWA}$ is a continuous and piece-wise affine function of $\sigma_e(x_e)$, but it is not continuously differentiable, hence~\cref{A1} would not hold. Therefore, we define $t_e$ appearing in~\eqref{eq:cost_traffic} as the smoothed version of $t_e^\textup{PWA}$
\begin{equation}
t_e(\sigma_e(x_e))   =   \begin{cases} t_{e,\text{free}} & \text{if} \; \sigma_e(x_e) \le f_e h - \Delta_e \\
t_{e,\text{free}} + \frac{\sigma_e(x_e)-f_e h}{2f_e} & \text{if} \; \sigma_e(x_e) \ge f_e h + \Delta_e \\
a \sigma_e(x_e)^2   +   b \sigma_e(x_e)   +   c    & \text{otherwise,} \end{cases}
\label{eq:travel_time_smoothed}
\end{equation}
where the values of $\Delta_e$, $a$, $b$, $c$ are such that $t_e$ is continuously differentiable\footnote{The values are $\Delta_e = 0.5(\sqrt{(f_e h)^2+4f_e h} - f_e h)$, $a = 1/(8f_e\Delta_e)$, $b=1/(4f_e)-h/(4 \Delta_e)$, $c = t_\textup{e,free} + (f_e h)^2/(8f_e \Delta_e) - h/4 - (\Delta_e)/(8f_e)$.}, as illustrated in~\cref{fig:t_cont_diff}.
\begin{figure}[h]
\setlength\figureheight{4cm} 
\setlength\figurewidth{8cm} 
\begin{center}
%
%
\begin{tikzpicture}

\begin{axis}[%
width=\figurewidth,
height=\figureheight,
at={(1.011in,0.642in)},
scale only axis,
xmin=0.56,
xmax=3.46,
xtick={1.2679,2,2.7321},
xticklabels={{$f_e h-\Delta_e$},{$f_e h$},{$f_e h+\Delta_e$}},
xmajorgrids,
ymajorgrids,
xminorgrids,
yminorgrids,
ymin=1.75,
ymax=4.25,
ytick={2,2.5,3,3.5,4},
yticklabels={{$t_{\text{free}}$}},
axis background/.style={fill=white},
legend style={at={(0.43,0.92)}, font=\footnotesize, legend cell align=left, align=left, draw=white!15!black}
]

\addplot [color=black, dashed, line width=1.5pt]
  table[row sep=crcr]{%
0.535898384862246	2\\
0.565476195269069	2\\
0.595054005675892	2\\
0.624631816082716	2\\
0.654209626489539	2\\
0.683787436896362	2\\
0.713365247303186	2\\
0.742943057710009	2\\
0.772520868116832	2\\
0.802098678523655	2\\
0.831676488930479	2\\
0.861254299337302	2\\
0.890832109744125	2\\
0.920409920150949	2\\
0.949987730557772	2\\
0.979565540964595	2\\
1.00914335137142	2\\
1.03872116177824	2\\
1.06829897218507	2\\
1.09787678259189	2\\
1.12745459299871	2\\
1.15703240340554	2\\
1.18661021381236	2\\
1.21618802421918	2\\
1.24576583462601	2\\
1.27534364503283	2\\
1.30492145543965	2\\
1.33449926584648	2\\
1.3640770762533	2\\
1.39365488666012	2\\
1.42323269706695	2\\
1.45281050747377	2\\
1.48238831788059	2\\
1.51196612828742	2\\
1.54154393869424	2\\
1.57112174910106	2\\
1.60069955950789	2\\
1.63027736991471	2\\
1.65985518032153	2\\
1.68943299072836	2\\
1.71901080113518	2\\
1.748588611542	2\\
1.77816642194883	2\\
1.80774423235565	2\\
1.83732204276247	2\\
1.8668998531693	2\\
1.89647766357612	2\\
1.92605547398294	2\\
1.95563328438977	2\\
1.98521109479659	2\\
2.01478890520341	2.02957781040682\\
2.04436671561024	2.08873343122047\\
2.07394452601706	2.14788905203412\\
2.10352233642388	2.20704467284776\\
2.1331001468307	2.26620029366141\\
2.16267795723753	2.32535591447506\\
2.19225576764435	2.3845115352887\\
2.22183357805118	2.44366715610235\\
2.251411388458	2.502822776916\\
2.28098919886482	2.56197839772964\\
2.31056700927164	2.62113401854329\\
2.34014481967847	2.68028963935694\\
2.36972263008529	2.73944526017058\\
2.39930044049211	2.79860088098423\\
2.42887825089894	2.85775650179788\\
2.45845606130576	2.91691212261152\\
2.48803387171258	2.97606774342517\\
2.51761168211941	3.03522336423882\\
2.54718949252623	3.09437898505246\\
2.57676730293305	3.15353460586611\\
2.60634511333988	3.21269022667976\\
2.6359229237467	3.2718458474934\\
2.66550073415352	3.33100146830705\\
2.69507854456035	3.3901570891207\\
2.72465635496717	3.44931270993434\\
2.75423416537399	3.50846833074799\\
2.78381197578082	3.56762395156164\\
2.81338978618764	3.62677957237528\\
2.84296759659446	3.68593519318893\\
2.87254540700129	3.74509081400258\\
2.90212321740811	3.80424643481622\\
2.93170102781493	3.86340205562987\\
2.96127883822176	3.92255767644352\\
2.99085664862858	3.98171329725716\\
3.0204344590354	4.04086891807081\\
3.05001226944223	4.10002453888446\\
3.07959007984905	4.1591801596981\\
3.10916789025587	4.21833578051175\\
3.1387457006627	4.2774914013254\\
3.16832351106952	4.33664702213904\\
3.19790132147634	4.39580264295269\\
3.22747913188317	4.45495826376634\\
3.25705694228999	4.51411388457998\\
3.28663475269681	4.57326950539363\\
3.31621256310364	4.63242512620728\\
3.34579037351046	4.69158074702092\\
3.37536818391728	4.75073636783457\\
3.40494599432411	4.80989198864822\\
3.43452380473093	4.86904760946186\\
3.46410161513775	4.92820323027551\\
};
\addlegendentry{$t_e^\textup{PWA} (\sigma_e(x_e))$}

\addplot [color=black, line width=1.5pt]
  table[row sep=crcr]{%
0.535898384862246	2\\
0.565476195269069	2\\
0.595054005675892	2\\
0.624631816082716	2\\
0.654209626489539	2\\
0.683787436896362	2\\
0.713365247303186	2\\
0.742943057710009	2\\
0.772520868116832	2\\
0.802098678523655	2\\
0.831676488930479	2\\
0.861254299337302	2\\
0.890832109744125	2\\
0.920409920150949	2\\
0.949987730557772	2\\
0.979565540964595	2\\
1.00914335137142	2\\
1.03872116177824	2\\
1.06829897218507	2\\
1.09787678259189	2\\
1.12745459299871	2\\
1.15703240340554	2\\
1.18661021381236	2\\
1.21618802421918	2\\
1.24576583462601	2\\
1.27534364503283	2.00003734572021\\
1.30492145543965	2.00093364300527\\
1.33449926584648	2.00302500333706\\
1.3640770762533	2.0063114267156\\
1.39365488666012	2.01079291314087\\
1.42323269706695	2.01646946261289\\
1.45281050747377	2.02334107513165\\
1.48238831788059	2.03140775069714\\
1.51196612828742	2.04066948930938\\
1.54154393869424	2.05112629096836\\
1.57112174910106	2.06277815567408\\
1.60069955950789	2.07562508342654\\
1.63027736991471	2.08966707422574\\
1.65985518032153	2.10490412807167\\
1.68943299072836	2.12133624496435\\
1.71901080113518	2.13896342490377\\
1.748588611542	2.15778566788993\\
1.77816642194883	2.17780297392284\\
1.80774423235565	2.19901534300248\\
1.83732204276247	2.22142277512886\\
1.8668998531693	2.24502527030198\\
1.89647766357612	2.26982282852184\\
1.92605547398294	2.29581544978844\\
1.95563328438977	2.32300313410179\\
1.98521109479659	2.35138588146187\\
2.01478890520341	2.38096369186869\\
2.04436671561024	2.41173656532226\\
2.07394452601706	2.44370450182256\\
2.10352233642388	2.4768675013696\\
2.1331001468307	2.51122556396339\\
2.16267795723753	2.54677868960391\\
2.19225576764435	2.58352687829118\\
2.22183357805118	2.62147013002519\\
2.251411388458	2.66060844480593\\
2.28098919886482	2.70094182263342\\
2.31056700927164	2.74247026350764\\
2.34014481967847	2.78519376742861\\
2.36972263008529	2.82911233439632\\
2.39930044049211	2.87422596441077\\
2.42887825089894	2.92053465747195\\
2.45845606130576	2.96803841357988\\
2.48803387171258	3.01673723273455\\
2.51761168211941	3.06663111493596\\
2.54718949252623	3.11772006018411\\
2.57676730293305	3.170004068479\\
2.60634511333988	3.22348313982063\\
2.6359229237467	3.278157274209\\
2.66550073415352	3.33402647164411\\
2.69507854456035	3.39109073212596\\
2.72465635496717	3.44935005565455\\
2.75423416537399	3.50846833074799\\
2.78381197578082	3.56762395156164\\
2.81338978618764	3.62677957237528\\
2.84296759659446	3.68593519318893\\
2.87254540700129	3.74509081400258\\
2.90212321740811	3.80424643481622\\
2.93170102781493	3.86340205562987\\
2.96127883822176	3.92255767644352\\
2.99085664862858	3.98171329725716\\
3.0204344590354	4.04086891807081\\
3.05001226944223	4.10002453888446\\
3.07959007984905	4.1591801596981\\
3.10916789025587	4.21833578051175\\
3.1387457006627	4.2774914013254\\
3.16832351106952	4.33664702213904\\
3.19790132147634	4.39580264295269\\
3.22747913188317	4.45495826376634\\
3.25705694228999	4.51411388457998\\
3.28663475269681	4.57326950539363\\
3.31621256310364	4.63242512620728\\
3.34579037351046	4.69158074702092\\
3.37536818391728	4.75073636783457\\
3.40494599432411	4.80989198864822\\
3.43452380473093	4.86904760946186\\
3.46410161513775	4.92820323027551\\
};
\addlegendentry{$t_e (\sigma_e(x_e))$}

\end{axis}
\end{tikzpicture}%
\end{center}
\caption{Piece-wise affine travel time $t_e^\textup{PWA}(\sigma_e(x_e))$ and its smooth approximation $t_e(\sigma_e(x_e))$ as functions of $\sigma_e(x_e)$.}
\label{fig:t_cont_diff}
\end{figure}
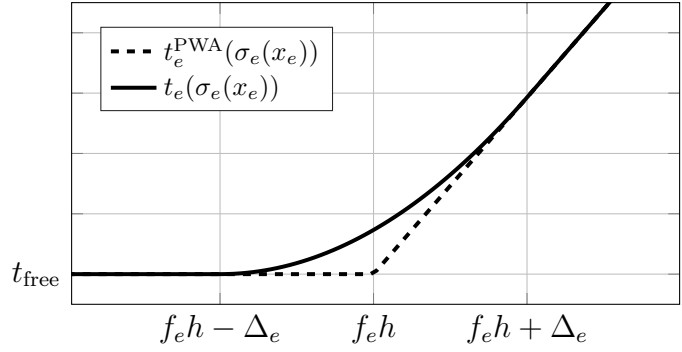
We note that the function $t_e(\sigma_e(x_e))$ is used within a stationary traffic model but includes the average queuing time which is based on the dynamic function~\eqref{eq:queue_expression}. A thorough analysis of a dynamic traffic model is subject of future work.

Finally, we remark that a travel time with similar monotonicity properties can be derived from the piecewise affine fundamental diagram of traffic~\cite[Fig. 7]{FDLi}, but $t_e(\sigma_e(x_e))$ would present a vertical asymptote which is absent here.
\subsection{Theoretical guarantees}
We  define the route-choice game $\mc{G}^\text{RC}_\N$ as in~\eqref{eq:GNEP}, with $\mathcal{X}\i$~as in \eqref{eq:local_traffic}, $\mathcal{C}$ as in~\eqref{eq:coupling_global_traffic} and $J\i(x\i,\sigma(x))$ as in~\eqref{eq:cost_traffic},~\eqref{eq:travel_time_smoothed}. In the following we summarize the main results from the previous chapters.
\begin{corollary}
\label{cor:traffic1}
Consider the sequence of games $(\mc{G}^\textup{RC}_M)_{M=1}^\infty$.
Assume that for each game $\mc{G}^\textup{RC}_M$ the set $\mc{Q}=\mc{C}\cap\mc{X}$ is non-empty, that $h > 0$ and $t_\textup{e,free}, f_e > 0$ for each $e \in \mc{E}$. Moreover, assume that there exists $\hat \gamma > 0$ such that $\gamma\i \ge \hat \gamma$ for all $i \in \{1,\dots,\N\}$, for all $\N$.
Then:
\begin{enumerate}
\item The operator $F\WE$ is strongly monotone, hence each game $\mathcal{G}^\textup{RC}_M$ admits a unique variational Wardrop equilibrium. For every $\N$ satisfying
\begin{equation}
M > \maxx{e \in \mc{E}} \frac{1}{32 f_e \Delta_e \hat \gamma}
\label{eq:bound_M_traffic}
\end{equation}
the operator $F\NE$ is strongly monotone, hence each game $\mathcal{G}^\textup{RC}_M$ admits a unique variational Nash equilibrium. Every Wardrop equilibrium is an $\varepsilon$-Nash equilibrium with $\varepsilon=\frac{E}{\N f_\text{min}}$, where $f_\textup{min} = \min_{e \in \mc{E}} f_e$.
\item  For any variational Nash equilibrium $\VNE{x}$ of $\mathcal{G}^\textup{RC}_M$, the unique variational Wardrop equilibrium $\VWE{x}$ of $\mathcal{G}^\textup{RC}_M$ satisfies
\begin{align}
\| \VNE{x} - \VWE{x} \| \le \frac{\sqrt{E}}{2 f_\textup{min} \hat \gamma {\sqrt \N}}.
\label{eq:convergence_strategies_traffic}
\end{align}
\item For any $\N$,~\cref{alg:asp} with operator $F\WE$ converges to a variational Wardrop equilibrium of $\mathcal{G}^\textup{RC}_\N$. For $\N$ satisfying~\eqref{eq:bound_M_traffic},~\cref{alg:asp} with operator $F\NE$ converges to a variational Nash equilibrium of $\mathcal{G}^\textup{RC}_\N$.
\qed
\end{enumerate}
\end{corollary}
\subsection{Numerical analysis}
For the numerical analysis we use the data set of the city of Oldenburg~\cite{OldenburgDataSet}, whose graph features 175 nodes, 213 undirected edges, and is reported in~\cref{fig:queuing_time_portion}.\footnote{The graph in the original data set features 6105 vertexes and 7035 undirected edges. We reduce it by excluding all the nodes that are outside the rectangle $[3619,4081] \times [3542,4158]$ and all the edges that do not connect two nodes in the rectangle. The resulting graph is strongly connected.}
For each agent $i$ the origin $o\i$ and the destination $d\i$ are chosen uniformly at random. Regarding the cost~\eqref{eq:cost_traffic}, $t_{e,\text{free}}$ is computed as the ratio between the road length, which is provided in the data set, and the free-flow speed.
Based on the road topology, we divide the roads into main roads, where the free-flow speed is $50$ km/h, and secondary roads, where the free-flow speed is $30$ km/h. Moreover, we assume a peak hour duration $h$ of $2$ hours, and for all $e \in \mc{E}$, we set $f_e = 4\cdot10^{-3}$ vehicles per second, which corresponds to 1 vehicle every 4 seconds for a population of $\N = 60$ vehicles. Finally, the parameter $\gamma\i$ is picked uniformly at random in $[0.5,3.5]$ and $\tilde x\i$ is such that $\tilde x\i_e = 1$ if $e$ belongs to the shortest path from $o\i$ to $d\i$, while $\tilde x\i_e = 0$ otherwise. The shortest path is computed based on $\{t_{e,\text{free}}\}_{e=1}^E$. Note that with the above values the bound~\eqref{eq:bound_M_traffic} becomes $\N > 16.14$, which is satisfied for relatively small-size populations.

We compute the Wardrop equilibrium with~\cref{alg:asp} relatively to a population of $\N = 60$ drivers without coupling constraint, i.e., with $K_e = 1$ for all $e \in \mc{E}$. We report in~\cref{fig:queuing_time_portion} the corresponding queuing time $t_e(\sigma_e(x_e)) - t_{e,\text{free}}$ as by~\eqref{eq:travel_time_smoothed}.
\begin{figure}[H]
\begin{center}
\includegraphics[width = 0.75\columnwidth]{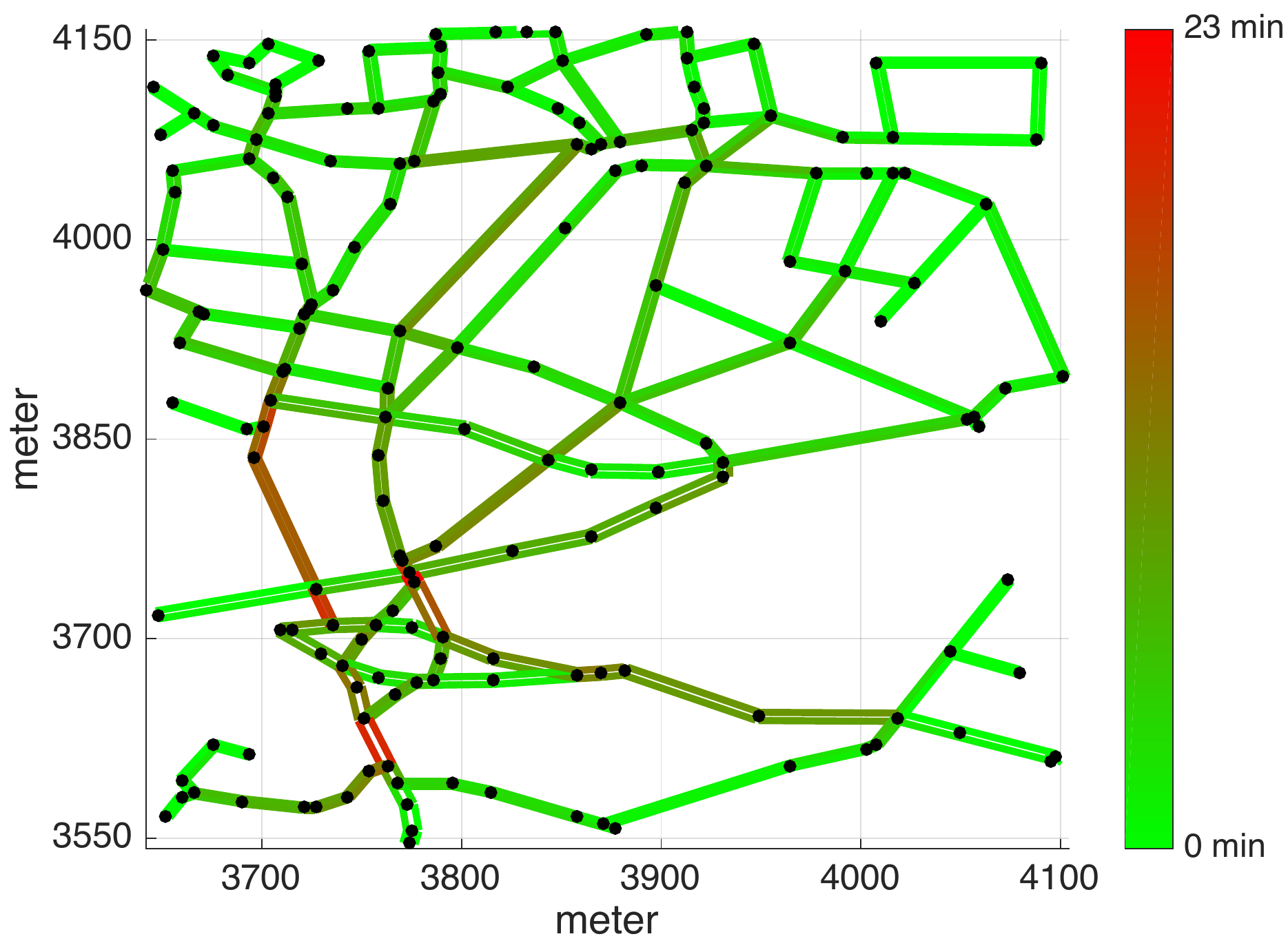}
\end{center}
\caption{ The queuing time reported in green-red color scale.
Note that this pattern changes if one modifies the pairs origin-destination.}
\label{fig:queuing_time_portion}
\end{figure}
We illustrate in~\cref{fig:comparison_coupling} the change in the queuing time of an entire neighborhood when introducing a coupling constraint that upper bounds the total number of cars on a single edge, relatively to a Wardrop equilibrium with $\N = 60$. Finally, we illustrate the second statement of~\cref{cor:traffic1} by reporting in~\cref{fig:traffic_conv_x} the distance between the unique variational Wardrop equilibrium and the variational Nash equilibrium found by~\cref{alg:asp}. 
\begin{figure}[H]
\begin{center}
\includegraphics[width = 1\columnwidth]{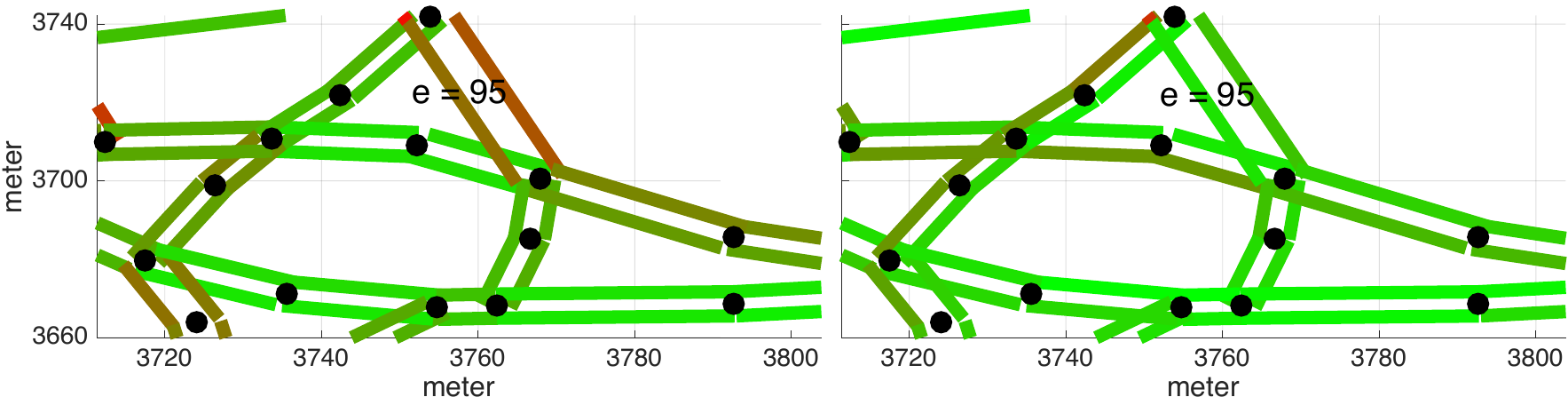}
\end{center}
\caption{On the left, the queuing time in a neighborhood without any coupling constraints; 10\% of the population transits on edge 95, and the queuing time is 7.28 minutes. On the right, the queuing time in presence of a coupling constraint allowing at most 3\% of the entire population on edge 95; the queuing time is reduced to 1.42 minutes, but it visibly increases on the edges of the alternative route.}
\label{fig:comparison_coupling}
\end{figure}

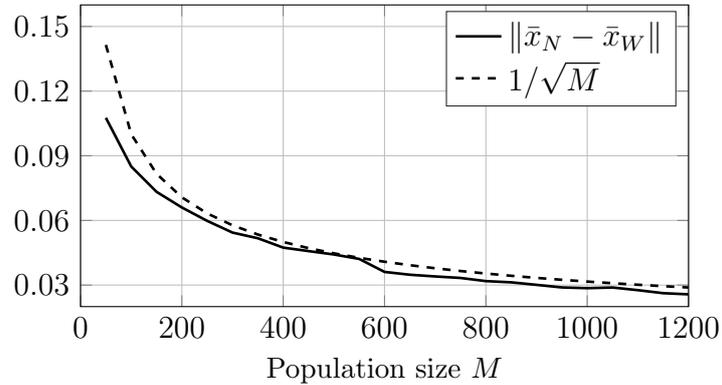
\begin{figure} [h!]
\begin{center}
\setlength\figureheight{4cm} 
\setlength\figurewidth{8cm} 
%
%
\begin{tikzpicture}

\begin{axis}[%
width=\figurewidth,
height=\figureheight,
at={(1.011in,0.642in)},
scale only axis,
xmin=0,
xmax=1200,
tick label style={/pgf/number format/.cd,1000 sep={}},
xtick={0, 200, 400, 600, 800, 1000, 1200},
ymin=0.02,
ymax=0.16,
tick label style={/pgf/number format/fixed},
ytick={0,0.03,0.06,0.09,0.12,0.15},
axis background/.style={fill=white},
xmajorgrids,
ymajorgrids,
xlabel ={\small Population size $M$},
legend style={legend cell align=left, align=left, draw=white!15!black}
]
\addplot [color=black, line width=1.0pt]
  table[row sep=crcr]{%
50	0.107597186907606\\
100	0.0850078973215346\\
150	0.0732622497727821\\
200	0.0660039393423034\\
250	0.0597843249632177\\
300	0.0543514624193683\\
350	0.0517201549776713\\
400	0.047375072456806\\
450	0.0457663821274737\\
500	0.0441920556441806\\
550	0.0421236118876949\\
600	0.0361\\
650	0.0348\\
700	0.034\\
750	0.0333\\
800	0.0318333333333333\\
850	0.0312666666666667\\
900	0.0301\\
950	0.0289\\
1000	0.0285666666666667\\
1050	0.0288666666666667\\
1100	0.0276333333333333\\
1150	0.0262333333333333\\
1200	0.0257\\
};
\addlegendentry{$\| \bar x_N - \bar x_W \|$}

\addplot [color=black, dashed, line width=1.0pt]
  table[row sep=crcr]{%
50	0.14142135623731\\
100	0.1\\
150	0.0816496580927726\\
200	0.0707106781186548\\
250	0.0632455532033676\\
300	0.0577350269189626\\
350	0.0534522483824849\\
400	0.05\\
450	0.0471404520791032\\
500	0.0447213595499958\\
550	0.0426401432711221\\
600	0.0408248290463863\\
650	0.0392232270276368\\
700	0.0377964473009227\\
750	0.0365148371670111\\
800	0.0353553390593274\\
850	0.0342997170285018\\
900	0.0333333333333333\\
950	0.0324442842261525\\
1000	0.0316227766016838\\
1050	0.0308606699924184\\
1100	0.0301511344577764\\
1150	0.0294883912309794\\
1200	0.0288675134594813\\
};
\addlegendentry{$1/\sqrt{M}$}

\end{axis}
\end{tikzpicture}%
\caption{ Distance between Nash and Wardrop variational equilibria. The quantity $1/\sqrt{M}$ illustrates the trend of the bound in~\cref{cor:traffic1} and not the \mbox{specific constant.}}
\label{fig:traffic_conv_x}
\end{center}
\end{figure}

\section{Appendix}
\label{sec:proofsp1-part4}
\subsection{Proofs of the results presented in \cref{sec:PEVs,,sec:traffic}}
\subsubsection*{Proof of \cref{cor:pev}}
\begin{proof}
~
\begin{enumerate}
\item 
First, we show that ~\cref{A1} holds. Indeed the sets $\mc{X}^i$ in \eqref{eq:vehicle_constraint} are convex and compact, the function $g$ in~\eqref{eq:coupling_constraints_general} is affine and hence convex, and $\mc{Q}$ is non-empty by assumption.
For each $z$ fixed, the function $J^i(x^i,z)$ is linear hence convex in $x^i$. We prove in the last statement that $F\NE$ is strongly monotone. This is equivalent to $\nabla_x F\NE(x) \succ 0$ by~\cref{lemma:pd}, which by definition of $F\NE(x)$ implies $\nabla_{x\i} (\nabla_{x\i} J\i(x\i,\sigma(x))) \succ 0$, which implies convexity of $J\i(x\i,\sigma(x))$. Finally, $J^i(z_1,z_2)$ is continuously differentiable in $[z_1;z_2]$ because $p_t$ is twice continuously differentiable.
Having verified ~\cref{A1}, \cref{lem:exun} guarantees the existence of a Nash and of a Wardrop equilibrium. The $\varepsilon$-Nash property is guaranteed by~\cref{prop:conv_cost} upon verifying~\cref{A2}. This holds because: i) $\cup_{i=1}^\N \mathcal{X}^i\subseteq{ [0,\tilde x^0]^n}$, ii)  $J^i(z_1,z_2)$ is Lipschitz in $z_2$ on $[0,\tilde x^0]^n$ with Lipschitz constant $L_2 = R L_p$, iii) \eqref{eq:Lipschitz_implies_Lipschitz} holds and iv) $p_t$ is assumed Lipschitz in $[0,\tilde x^0]$ with Lipschitz constant $L_p$ for all $t$. We conclude by noting that $R = \tilde x^0 \sqrt{n}$.
\item
 The fact that each $p_t$ is strictly increasing in $[0,\tilde x^0]$ implies that $\nabla_z p(z) \succ 0$ in $[0,\tilde x^0]^n$, where $p(z)\defeq\left[p_1(\frac{d_1 + z_1}{\kappa}),\ldots,p_n(\frac{d_n + z_n}{\kappa})\right]^\top$. In turn $\nabla_z p(z) \succ 0$ guarantees strong monotonicity of $p$ in $[0,\tilde x^0]^n$ by~\cref{lemma:pd}. This, together with \cref{A1} and \cref{A2} verified above, allows us to use the third result in \cref{thm:conv_strategies}.
\item
Given the special form of the sets $\{\mc{X}\i\}_{i=1}^\N$ and the price $p_t\left(\frac{d_t+\sigma_t(x)}{\kappa_t} \right)$, \cref{A1bis,,ass:sequence,,ass:nonlin} are satisfied. In addition since $\sum_{i=1}^\N \theta\i >0$, it must be that $J\SO(\sigma(x\SO))>\hat J$ for some $\hat J \ge 0$. Thus, the assumptions of \cref{thm:polypoa}.
\item
The strong monotonicity of $F\NE$ follows immediately thanks to \cref{lem:diagSMON}.
Additionally, \cref{A1} holds as previously shown. Since \cref{ass:lin} holds, we can directly employ~\cref{thm:convergence_asp} and conclude the proof. 
\end{enumerate}
\end{proof}

\subsubsection*{Proof of \cref{lem:uniqueness_for_PEVs}}
\begin{proof}
 The constraints in \eqref{eq:vehicle_constraint}, \eqref{eq:coupling_global_PEVs} can be expressed as $\Gamma x \le \gamma$ with
\begin{equation}
\Gamma = \begin{bmatrix}
 I_{\N  n} \\
 -I_{\N  n}  \\
  -I_\N \otimes \ones[n]^\top \\
  \ones[\N]^\top \otimes I_n
\end{bmatrix}
,\quad \gamma=\begin{bmatrix}
\tilde x  \\
 0\\
 -\theta \\
 \N K
\end{bmatrix}
\,,
\end{equation}
where $\theta = [\theta^1,\dots,\theta^\N]^\top$, and $\tilde x = [[\tilde x^i_t]_{t=1}^n]_{i=1}^\N$.
Let us partition the constraint matrix $\Gamma$ into its individual part $\Gamma_1$ and coupling part $\Gamma_2$
\begin{equation}\label{eq:H_PEVs_partitioned}
\Gamma = \begin{bmatrix} \Gamma_1 \\ \Gamma_2 \end{bmatrix},\;
 \Gamma_1 = \begin{bmatrix}
 I_{\N  n} \\
 -I_{\N  n}  \\
  -I_\N \otimes \ones[n]^\top \\
\end{bmatrix},\;
\Gamma_2 = \begin{bmatrix}
  \ones[\N]^\top \otimes I_n
\end{bmatrix}
\end{equation}
and $\gamma = [\gamma_1^\top, \gamma_2^\top]^\top$ accordingly.
The KKT conditions for VI$(\mc{Q},F\NE)$ at the primal solution $\VNE{x}$ are~\cite[Prop. 1.3.4]{facchinei2007finite}
\begin{subequations}
\label{eq:KKT_PEVs}
\begin{align}
&F\NE(\VNE{x}) + \Gamma_1^\top \mu + \Gamma_2^\top \lambda = 0, \label{eq:KKT_PEVs_stat} \\
&0 \le \mu \perp \gamma_1 - \Gamma_1 \VNE{x} \ge 0, \label{eq:KKT_PEVs_compl_individ} \\
&0 \le \lambda \perp \gamma_2 - \Gamma_2 \VNE{x} \ge 0 .\label{eq:KKT_PEVs_stat_compl_coupl}
\end{align}
\end{subequations}
Define $\tilde \mu$ and $\tilde \lambda$ as the dual variables corresponding to the active constraints (the other dual variables must be zero due to~\eqref{eq:KKT_PEVs_compl_individ} and~\eqref{eq:KKT_PEVs_stat_compl_coupl}).
The KKT system~\eqref{eq:KKT_PEVs}  in $\tilde \mu, \tilde \lambda$ only reads
\begin{equation}
\begin{aligned}
&\tilde \Gamma_1^\top \tilde \mu + \tilde \Gamma_2^\top \tilde \lambda = -F\NE(\VNE{x}) ,\\
&\tilde \mu, \tilde \lambda \ge 0 \,,
\end{aligned}
\label{eq:KKT_PEVs_reduced}
\end{equation}
where $\tilde \Gamma_1, \tilde \Gamma_2$ contain the subset of rows of $\Gamma_1, \Gamma_2$ corresponding to active constraints.
To conclude the proof we need to show that \eqref{eq:KKT_PEVs_reduced} has a unique solution $\tilde \lambda$.
To this end we apply the subsequent~\cref{lem:sufficient_for_uniqueness_subvector}.
To verify its assumption, we note that  its negation is equivalent, given the expressions of $\tilde \Gamma_1, \tilde \Gamma_2$ in~\eqref{eq:H_PEVs_partitioned}, to
the existence of $R' \subseteq R^\text{tight}$ such that for each vehicle $i$ it holds $\bar x_{\textup{N},t}^i  \in \{0,\tilde x^i_r\}$ for all $t \in R'$ or $\bar x_{\textup{N},t}^i  \in \{0,\tilde x^i_t\}$ for  $t\in\{1,\dots,n \}\setminus R'$ and  such $R'$ cannot exist by assumption.	
\end{proof}
\begin{lemma}
\label{lem:sufficient_for_uniqueness_subvector}
Consider $A_1 \in \R^{m \times n_1}$, $A_2 \in \R^{m \times n_2}$, $b \in \R^m$.
If the implication $A_1 x_1 + A_2 x_2 = 0 \; \Rightarrow \; x_1 = 0$ holds, then
the linear system of equations $A_1 x_1 + A_2 x_2 = b$
has at most one solution in $x_1$.
\end{lemma}
\begin{proof}
Assume $A \tilde x = b$ and $A \hat x = b$, then $A_1 \tilde{x}_1 + A_2 \tilde{x}_2 = b$ and $A_1 \hat{x}_1 + A_2 \hat{x}_2 = b$ imply 
$A_1(\hat{x}_1 - \tilde{x}_1) + A_2 (\hat{x}_2 - \tilde{x}_2) = 0$, which by assumption implies $\hat{x}_1 = \tilde{x}_1$.
\end{proof}

\subsubsection*{Proof of \cref{cor:traffic1}}
\begin{proof}
~
\begin{enumerate}
	\item Satisfaction of \Cref{A1} and the consequent existence of a variational Nash and of a variational Wardrop equilibrium for any $\N$ can be shown as in~\cref{cor:pev}. The operator $F\WE$ for the cost~\eqref{eq:cost_traffic} reads
\begin{equation}
F\WE(x)=[\gamma^i(x^i-\hat x^i) + t(\sigma(x))]_{i=1}^\N.
\end{equation}
where $t(\sigma(x))\defeq[t_e(\sigma_e(x_e))]_{e=1}^E$. Since $t_e(\sigma_e(x_e))$ in~\eqref{eq:travel_time_smoothed} is a monotone function of $\sigma_e(x_e)$, the operator $t(\sigma(x))$ is monotone. Then $F\WE$ is strongly monotone with constant $\hat \gamma$ because it is the sum of a monotone and a strongly monotone operator with constant $\hat \gamma$. As a consequence, each $\mathcal{G}^\textup{RC}_\N$ admits a unique variational Wardrop equilibrium.

To prove strong monotonicity of $F\NE$ we use the result of~\cref{lemma:pd}.\footnote{\Cref{lemma:pd} requires $F\NE$ to be continuously differentiable, which is not the case here. The more general result~\cite[Prop. 2.1]{schaible1996generalized} extends the statement of~\cref{lemma:pd} to operators which are not continuously differentiable. It then suffices to show $\nabla_x F\NE(x) \succ 0$ for $\sigma(x)$ in each of the three intervals defined by~\eqref{eq:travel_time_smoothed}, because in each of them $F\NE$ is continuously differentiable.} We first note that each $t_e$ only depends on the corresponding $\sigma_e$, hence $\nabla_x F\NE(x)$ can be permuted into diagonal form similarly to what done in~\eqref{eq:PEV_gradient_FN}. It then suffices to show $\hat \gamma I_\N + \frac{1}{\N} t_e'(\sigma_e) I_\N + \frac{1}{\N^2} t''_e(\sigma_e) x_e \ones[\N]^\top \succ 0$ for all $\sigma_e$ and for all $e$. This matrix is indeed positive definite if $\sigma_e(x_e) \notin [f_e h - \Delta_e, f_e h + \Delta_e]$, because then $t'_e(\sigma_e) \ge 0$ and $t''_e(\sigma_e) = 0$ by~\eqref{eq:travel_time_smoothed}. For $\sigma_e(x_e) \in [f_e h - \Delta_e, f_e h + \Delta_e]$ it suffices to show $\hat \gamma I_\N + \frac{1}{\N^2 4f_e \Delta_e} x_e \ones[\N]^\top \succ 0$, because $t_e'(\sigma_e) \ge 0$ and $t''_e(\sigma_e) = \frac{1}{4f_e \Delta_e}$. By~\cref{lem:min_eigenval}, $\lambda_\text{min} \left( x_e \ones[\N]^\top + \ones[\N] x_e^\top \right)/2 \ge -\frac{\N}{8}$, which proves strong monotonicity of $F\NE$ under~\eqref{eq:bound_M_traffic}.
Consequently, if $M$ satisfies~\eqref{eq:bound_M_traffic} then $\mathcal{G}^\textup{RC}_\N$ admits a unique variational Nash equilibrium.
Finally, we verify~\cref{A2} in order to use~\cref{prop:conv_cost}. We have $\mc{X}^0 = [0,1]^E$ and $t$ is continuously differentiable and hence Lipschitz in $\mc{X}^0$, with constant $L_p = 1/(2f_\textup{min})$. Moreover, $  \Dx   \defeq   \max_{y \in{\mathcal{X}\zero} } \{\|y \|\} = \sqrt{E}$. Using~\eqref{eq:Lipschitz_implies_Lipschitz} concludes the proof.

\item Since all the assumptions of~\cref{thm:conv_strategies} have just been verified, it is a direct consequence of its second statement.

\item As~\cref{A3} holds trivially (the others have already been verified), we apply~\cref{thm:convergence_asp} and conclude the proof.
\end{enumerate}
\end{proof}


\renewcommand{\N}{N}
\part[Programmable machines: game design]{Programmable machines:\\ game design}
\label{part:2}
\chapter{Introduction}
\label{ch:p2introduction}
In this part of the thesis we focus on large scale cooperative systems composed of programmable machines, which we refer as multiagent systems. As discussed in the overview of \cref{ch:overview}, one of the main challenges in the control of large scale cooperative systems rests in the design of control algorithms that achieve a given global objective by relying solely on local information. 
The problem of designing local control algorithms is typically posed as an optimization problem (finite or infinite dimensional), where the system-level objective is captured by an \emph{objective function} (functional), while physical laws and informational availability are incorporated as \emph{constraints} on the decision variables. The design is complete once a distributed decision making algorithm has been found, satisfying the constraints and maximizing the objective function \cite{Raffo, CortesBullo}.
A well-established approach to tackle this problem consists in the design of a centralized maximization algorithm, that is later distributed by leveraging the structure of the problem considered. Examples in continuous optimization include algorithms such as distributed gradient ascent, primal-dual and Newton's method \cite{bertsekas1989parallel,nedic2009distributed, wei2013distributed}. While the existing approaches,  including the above-mentioned one, have produced a variety of algorithms for the control of distributed systems, the design question has not been entirely solved. A perspective article recently appearing in Science Robotics summarizes the difficulties:``{\it There are currently no systematic approaches for designing such multidimensional feedback loops}'' \cite[p.\,6]{yang2018grand}.
\section{The game-design framework}
A promising approach, termed \emph{game design}, has recently emerged as a tool to complement the partial understanding offered by more traditional techniques \cite{shamma2007cooperative}. The game design approach is tightly connected with the notion of equilibrium in game theory, and its origin stems from a novel engineering perspective on this field. While game theory has originated as a set of tools to model the interaction of multiple decision makers (players) \cite{von2007theory}, its relevance to distributed control stems from the observation that players of a game are required to take local decisions based on partial information of the entire system. Motivated by this consideration, the seminal works of \cite{Mannor, arslan2007autonomous} proposed the use of game theoretic tools to tackle distributed optimization problems arising in the area of multi-agent systems. Rather than using game theory to describe existing interactions, \cite{arslan2007autonomous} suggested a \emph{paradigm shift} and proposed the use of {game theory to \emph{design}} control architectures with the aim of meeting a given system level objective.

In lieu of directly specifying a decision-making process, the game design approach consists in assigning local utility functions to the agents, so that their selfish maximization translates in the achievement of the system level objective. The potential of this technique stems from the possibility to inherit a pool of algorithms from the literature of learning in games \cite{blume1993statistical, fudenberg1998theory, marden2012revisiting, yi2017distributed} that are \emph{distributed} by nature, \emph{asynchronous}, and \emph{resilient} to external disturbance \cite{arslan2007autonomous}.

Given an optimization problem we wish to solve distributedly, the {\it game design} procedure proposed in \cite{shamma2007cooperative,marden2013distributed} is summarized in \cref{fig:gamedesign} and consists in the following steps:\footnote{In the following, we identify the agents of the optimization problem and their local constraint sets with the players of the game and their action sets.}

\begin{itemize}
	\item[1)] {\bf Utility design}: assign utility functions to each agent and an equilibrium concept for the corresponding game.
	\item[2)] {\bf Algorithm design}: devise a distributed algorithm to guide agents to the chosen equilibrium concept. 
\end{itemize}
\noindent

\begin{figure}[h!]
\centering
\includegraphics[scale=.9]{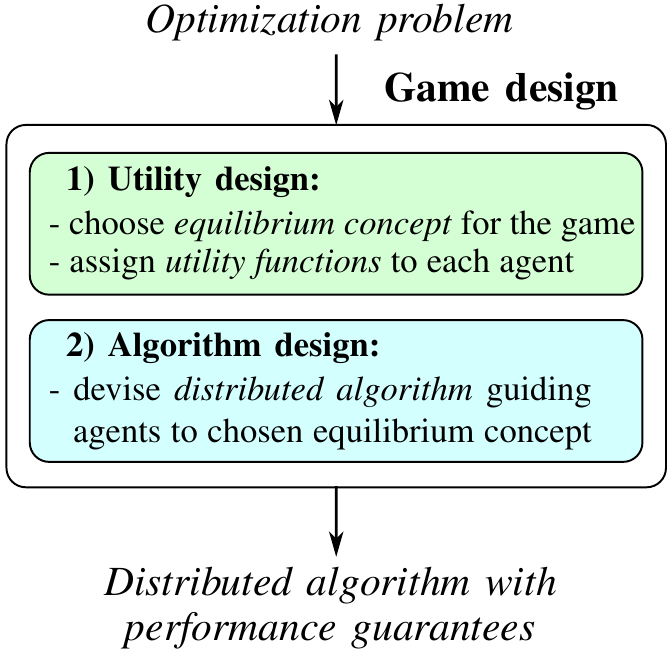}
\caption{Game theoretic approach for the design of distributed control systems. }
\label{fig:gamedesign} 
\end{figure}

The objective of the game design procedure is to obtain an efficient and distributed algorithm for the solution of the original optimization problem. While the introduction of an auxiliary equilibrium problem might seem artificial at first, this approach has recently produced a host of new results \cite{gairing2009covering, MardenCoopControl, paccagnan2017arxiv, gebraad2016wind}. Observe that, in order for the game design procedure to be relevant to the original optimization problem, the utility functions need to be carefully designed so that the chosen equilibrium (equilibria) coincide with the global optimizer(s) of the original problem, or is provably close to. 
  
 Within the boundaries of the game design procedure discussed above, it is important to highlight that agents are \emph{not} modeled as competing units, but the system operator is rather designing their utilities to distribute the global objective. For this purpose, agents are considered as purely programmable machines endowed with computational and communication capabilities. Game theory represents, in this context, a mere set of tools that can be exploited to derive distributed algorithms with provable performance certificates, and \emph{not} a modeling language describing the behaviour of egoistic agents.

  While the field of learning in games offers \emph{readily available algorithms} to coordinate agents towards an equilibrium in a distributed fashion (i.e., it addresses the second step of \cref{fig:gamedesign}), the \emph{utility design} problem is much less tracked.

\vspace*{5mm}
\noindent \fbox{\parbox{0.985\textwidth}{
\label{problem:utilitydesign}
 The {\bf goal} of \cref{part:2} of this thesis is to provide a framework to compute the equilibrium efficiency as a function of the given utility functions, and to optimally select utilities so as to maximize such efficiency. 
}}

\section{The general multiagent maximum coverage}
\label{sec:problemstatement}
In this section we introduce the problem considered in \cref{part:2} of this thesis.

Consider $\mc{R}=\{r_1,\dots,r_m\}$ a finite set of resources, where each resource $r\in\mc{R}$ is associated with a value $v_r\ge 0$ describing its importance. Further let $N=\{1,\dots,n\}$ be a finite set of agents. Every agent $i\in N$ selects $a_i$, a subset of the resources, from the given collection $\mc{A}_i\subseteq 2^{\mc{R}}$, i.e., $a_i\in\mc{A}_i$. The welfare of an allocation $a = (a_1,\dots ,a_n)\in \mc{A} \defeq \mc{A}_1 \times \dots \times \mc{A}_n$ is given by 
\be
W(a)\coloneqq\sum_{r\in\cup_{i\in N} a_i}v_r w(|a|_r),
\label{eq:welfaredef}
\ee
where $W:2^{\mc{R}}\times\dots\times2^{\mc{R}}\rightarrow\mb{R}$, $|a|_r=|\{i\in\N\,\text{s.t.}\,r\in a_i\}|$ captures the number of agents selecting resource $r$ in allocation $a$, and $w:[n]\rightarrow \mb{R}_{\ge 0}$ is called the welfare basis function. Informally, $w$ scales the value of each resource depending on how many agents selected it. The goal is to find a feasible allocation maximizing the welfare, i.e.,
\be
\aopt\in\argmax_{a\in\mc{A}}W(a)\,.
\label{eq:problem_findaopt}
\ee

We refer to the above problem as to the \ac{gmmc} problem, due to its connections with coverage problems as discussed in the forthcoming \cref{subsec:related_coverage}.

Observe that we have not posed any constraint on the structure of the sets $\{\mc{A}_i\}_{i\in\N}$ in the sense that they are not required to represent matroid constraints, knapsack constraints, etc. At this stage, they are a mere collection of subsets of $\mc{R}$.

Since the problem in \eqref{eq:problem_findaopt} is $\npclass$-hard (see the discussion  in \cref{subsec:related_coverage}), we seek an \emph{efficient} algorithm (or a class of algorithms) to determine an approximate solution, ideally with the best possible approximation ratio. Additionally, we request the algorithm to be distributed as detailed in \cref{sec:poaperformancemetric}. We pursue this goal by means of the game theoretic approach previously introduced.

\subsection{Applications}
In the following we discuss two classes of problems that can be solved by using the techniques discussed in the second part of this thesis. \\

\subsubsection*{Multiagent task assignment problems}

In a multiagent task assignment problem we are given a list of tasks to be performed, as well as a list of agents. The goal is to match agents and tasks so as to maximize a given welfare function representing the quality of the matching. Such function is typically additive over the tasks and some of the tasks may require a minimum number of agents to be completed. It is typically assumed that the more agents participate in the execution of a task, the higher the welfare generated from that task, and that the problem exhibits diminishing returns. Practical examples of problems belonging to this class include vehicle-target assignment \cite{murphey2000target, arslan2007autonomous}, sensor deployment \cite{cassandras2005sensor, MardenCoopControl}, satellite assignment \cite{qu2015distributed} problems. 

\subsubsection*{Distributed maximum coverage}
In a distributed maximum coverage problem we are given a list of resources  with their respective value and a list of agents. Each agent has access to a collection of subsets of the resources, while different agents typically have access to different collections (due to, e.g., geographical or other limitations). The goal is to allocate the agents so as to maximize the total value of covered resources. A large number of problems can be cast into this framework. Examples include staff scheduling \cite{ernst2004staff}, facility location \cite{farahani2012covering} and wireless scheduling \cite{cohen2008generalized} (see \cite{hochbaum96} for an overview of the applications).
More recent applications include, among others, distributed caching in wireless networks \cite{goemans2006market, de2017competitive}, multi-topic searches \cite{saha2009maximum}, influence maximization \cite{karimi2017}, vehicle scheduling in mobility-on-demand platforms \cite{sayarshad2017non, agatz2012optimization}.
\section{Related work}
\label{sec:p2relatedwork}
The work presented in \cref{part:2} of this thesis is multidisciplinary in that it sits at the interface between approximation theory, distributed optimization and game theory. In the followings we present the most relevant connections to each of these areas and corresponding related works. 

\subsection*{Maximum coverage and approximation guarantees}
\label{subsec:related_coverage}
The general multiagent maximum coverage (\ac{gmmc}) problem defined in \cref{sec:problemstatement} and studied in \cref{part:2} of this thesis is tightly connected with the maximum coverage problem defined in \cite{feige1998threshold}.

In a maximum coverage problem we are given a ground set of elements, and a collection of subsets of the ground set. The objective is to select $n$ subsets from the collection, so as to maximize the total number of covered elements. The greedy algorithm achieves a $1-1/e$ approximation in polynomial time, and no polynomial algorithm can approximate the solution within any ratio better than $1-1/e+\epsilon$ (for all $\epsilon>0$) unless $\pclass=\npclass$, \cite{feige1998threshold}. This inapproximability result applies to all extensions discussed next (including the \ac{gmmc} problem), since they hold the maximum coverage problem as a special case.

 A generalization of the maximum coverage problem is the weighted maximum coverage problem. In a weighted maximum coverage we are given a ground set of elements, and a collection of subsets of the ground set. Every element in the ground set is given a weight. The goal is to select $n$ subsets from the collection in order to maximize the total weight of covered elements. The greedy algorithm gives the best possible polynomial approximation ratio of $1-1/e$. The proof is no longer based on the result of \cite{feige1998threshold}, but on the more general result in submodular maximization subject to cardinality constraints~\cite{nemhauser1978analysis}.

Algorithms based on a continuous relaxation of the previous problems are also available. In particular the result in \cite{calinescu2011maximizing} applies to the problem of monotone submodular maximization subject to matroid constraints, and thus provides a solution for the weighted maximum coverage. The algorithm of \cite{calinescu2011maximizing} computes a non integer solution, which is then rounded using the pipage algorithm producing a $1-1/e$ approximation.
Relative to the problem of monotone submodular maximization over a matroid constraint, a more refined result is available when the objective function has known (total) curvature $c$. The notion of curvature has been introduced in \cite{conforti1984submodular} and describes how far a given function is from being modular.
In this case,
\cite{sviridenko2017optimal} has recently provided a $1-c/e$ approximation and has showed that no polynomial time algorithm can give a better approximation. 
The latter work improves upon the $(1-e^{-c})/c$ of \cite{conforti1984submodular}.
Observe that the maximum coverage problem has $c=1$, so that \cite{sviridenko2017optimal} matches \cite{feige1998threshold}. 

Multiagent versions of the weighted maximum covering problem have been introduced independently in \cite{chekuri2004maximum} with the name of \emph{maximum coverage problem with group budget constraints} and in \cite{gairing2009covering} with the name of \emph{general covering problems}. This class of problems subsumes the previous ones; we refer to it as to the class of \ac{mmc} problems. In a \ac{mmc} problem we are given not one, but $n$ collections of subsets. The objective is to select one set from each collection so as to maximize the total weight of covered elements. Relative to \ac{mmc} problems, the greedy algorithm provides a  $1/2$ approximation \cite{chekuri2004maximum}, and the local search algorithm proposed in \cite{gairing2009covering} achieves the optimal $1-1/e$, under technical assumptions. 

The \ac{gmmc} problem studied in \cref{part:2} of this thesis is a generalization of the \ac{mmc} problem in that we allow for a function $w$ to rescale the weight of each element depending on how many agents cover such element. 
Any \ac{mmc} problem can be recovered by the corresponding \ac{gmmc} problem by setting $w(j)=1$ in \eqref{eq:welfaredef}. Any weighted maximum coverage problem can be recovered from a \ac{gmmc} problem, upon setting $w(j)=1$ in \eqref{eq:welfaredef} and $\mc{A}_i=\mc{A}_j$ for all $i,j$. Further classes of problems such as the multiple-choice knapsack problem or the standard knapsack problem \cite{pisinger1995} can be obtained from the \ac{mmc} problem (and thus from the \ac{gmmc} problem). The former problem can be recovered assuming $\mc{A}_i$ to represent knapsack constraints. The latter problem is obtained by additionally imposing $\mc{A}_i=\mc{A}_j$ for all $i,j$. Observe that when the welfare basis $w$ is increasing and concave (in the discrete sense), the welfare function $W$ defined in \eqref{eq:welfaredef} is monotone submodular. Submodular functions are subject of intense study due to their ability to model engineering problems that feature diminishing returns. Similarly, if $w$ is increasing and convex, $W$ is monotone supermodular. \cref{fig:problem_hierarchy} summarizes the main classes of problems discussed.
\begin{figure}[h!]
\centering
\includegraphics[scale=0.5]{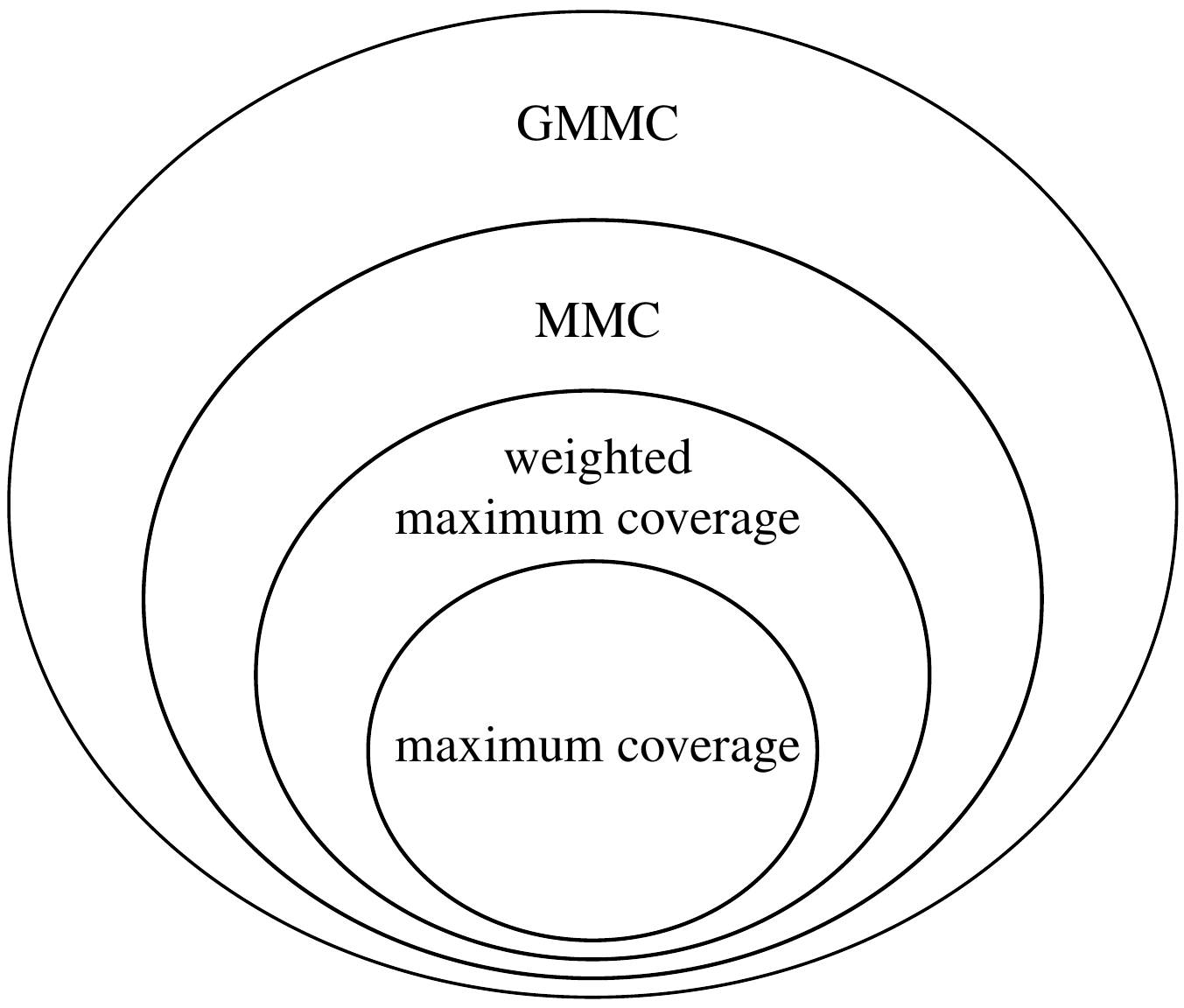}
\caption{Classes of problems discussed in \cref{sec:p2relatedwork}.}
\label{fig:problem_hierarchy} 
\end{figure}

\subsection*{Distributed combinatorial optimization}

While distributed algorithms have been studied since the early nineties in the context of continuous (and convex) optimization  \cite{bertsekas1989parallel}, the interest in their combinatorial counterpart is more recent.

Particular attention has been devoted to the problem of maximizing a submodular function subject to various form of constraints such as cardinality, matroid or knapsack constraints. This is due to the potential applications of submodular maximization in different fields featuring ``large-scale'' systems. A non-exhaustive list includes sensor allocation \cite{summers2016submodularity,krause2008near}, data summarization \cite{mirzasoleiman2016distributed}, task-assignment problems \cite{qu2015distributed}.
While centralized algorithms are available to produce good approximations (e.g., the greedy algorithm and its variations \cite{feige1998threshold}), their sequential implementation makes them unsuited for parallel and distributed execution.
In this respect, there has been recent effort in distributing such algorithms using the so called \emph{MapReduce} programming approach \cite{dean2008mapreduce}.
In \cite{mirzasoleiman2016distributed, barbosa2015power} and references therein, the authors propose to divide the original optimization problem into smaller parts and to solve each of them on a different machine. The solution is determined by patching together the partial results and is certified to achieve a competitive approximation ratio. Nevertheless, the approach still requires a central coordinator. 

Other classes of combinatorial problems for which distributed algorithms have been recently proposed include graph coloring, maximum coverage, and multiple-choice knapsack \cite{barenboim2013distributed, gairing2009covering, murthy2017distributed}. Finally, \cite{mokhtari2018decentralized} provides distributed algorithms for submodular maximization problems, but admissible objective functions are required to be the sum of agents' individual contributions (unlike here).

\subsection*{Game design and utility design approach} 
The problem of designing local utility functions so as to maximize the efficiency of the emerging equilibria find its roots in the economic literature relative to the design of optimal taxations \cite{ramsey1927contribution}. The approach has been applied to the design of engineering systems only recently. More in details, the use of game theoretic learning algorithms for the distributed solution of optimization problems has been proposed in \cite{arslan2007autonomous}, and since then a number of works have followed a similar approach \cite{soto2009distributed, gairing2009covering, chapman2011unifying, song2011optimal}. We redirect the reader to \cite{Marden2018} for a general overview on equilibrium learning algorithms in distributed control.
What has been less understood so far, is how to provide performance certificates for a given set of utility functions, and more fundamental how to select utility functions so as to maximize such performance certificates.

The performance degradation of an equilibrium allocation compared to the optimal solution has been subject of intense research in the field of algorithmic game theory (through the notions of price of anarchy, price of stability \cite{Koutsoupias,schulz2003performance}). Nevertheless, the results available therein are not helpful for the design problem studied here. The widely used smoothness framework proposed in \cite{roughgarden2009intrinsic} has brought a number of different results under a common language and has produced tightness guarantees for different problems \cite{roughgarden2009intrinsic,roughgarden2017price}. Unfortunately the latter framework requires the sum of the utility functions to be equal (or less equal) to the welfare function (budget-balance condition). While this assumption is well justified for a number of problems modeled through game theory (e.g., cost sharing games \cite{moulin2001strategyproof}), it has little bearing on the design of local utility functions studied~here.

 The utility design problem considered here has been addressed limitedly to specific applications, e.g., concave cost sharing, reverse carpooling problems \cite{Philips17, marden2012price} or confined to particular design methodologies such as the Shapley value or marginal contribution \cite{marden2014generalized, philips2016importance}.
\chapter{Mathematical preliminaries}
In this chapter we introduce the mathematical tools required to move forward and present the results of \cref{ch:p2utilitydesign} and \cref{ch:p2subsupercov}.
\label{ch:p2mathpreliminaries}
\section{Strategic-form games and equilibrium concepts}
\label{sec:equilibria}
\begin{definition}[Strategic-form game]
\label{def:strategicformgame}
A strategic form game $G=(\N,\{\mc{A}_i\}_{i=1}^\N,\{u_i\}_{i=1}^\N)$ is a tuple where $\N=\{1,\dots,n\}$ is a finite set of players, $\mc{A}_i$ is the action set of player $i\in\N$, and $u_i:\mc{A}\rightarrow \mb{R}$ is the utility function of player $i\in \N$, where $\mc{A}\coloneqq \mc{A}_1\times\dots\times \mc{A}_n$. A strategic-form game is called finite if the set $\mc{A}$ is finite.
\end{definition}
Informally, a game is fully specified in its strategic form if every player is given an action set and a utility function depending on the choice of all the players. We refer to $a\coloneqq (a_1,\dots, a_n)\in\mc{A}$ as to an allocation. We will often represent an allocation as $a=(a_i,\ami)$, where $\ami\coloneqq(a_1,\dots,a_{i-1},a_{i+1},\dots,a_n)$ denotes the allocations of all players but $i\in\N$.

In the following we consider strategic-form games only. We do not repeat this in the forthcoming statements for ease of presentation. \begin{definition}[\ac{ne}, \cite{Nash01011950}]
\label{def:nashequilibrium}
A feasible allocation $\ae\in\mc{A}$ is a pure Nash equilibrium for the game $G$, if no player can increase his utility function by unilaterally deviating from his equilibrium allocation, i.e., if 
\be
u_i(\ae)\ge u_i(a_i,\ae_{-i})\qquad\forall a_i \in \mc{A}_i,\quad \forall i\in \N.
\ee
We denote with $\nashe{G}$ the set of pure Nash equilibria of $G$.
\end{definition}

In the remaining of this thesis we will refer to a pure Nash equilibrium as just a Nash equilibrium, if no confusion arises. 
It is not difficult to show that Nash equilibria may not exist. 
This and many other reasons motivate the definition of mixed Nash equilibria. Towards this goal, we first introduce the concept of \emph{mixed strategy}.  A mixed strategy $\sigma_i$ is a probability distribution over the action space of player $i$, i.e., $\sigma_i \in \Delta(\mc{A}_i) $. A mixed strategy profile $\sigma\coloneqq(\sigma_1,\dots,\sigma_n)$ is a distribution $\sigma \in\Sigma\coloneqq \times_{i\in\N}\Delta(\mc{A}_i)$.
 
\begin{definition}[\Ac{mne}, \cite{Nash01011950}]
A mixed strategy profile $\sigmane\in\Sigma$ is a mixed Nash equilibrium for the game $G$ if no player can increase his expected utility by deviating to a pure strategy, i.e., if 
\be
\expected{u_i(a)}{a\sim \sigmane} \ge \expected{u_i(a_i',\ami)}{a\sim \sigmane} \qquad\forall a_i' \in \mc{A}_i,\quad \forall i\in \N.
\ee
We denote with $\mnashe{G}$ the set of mixed Nash equilibria of $G$.
\end{definition}
In the previous definition player $i$ compares $\expected{u_i(a)}{a\sim \sigmane}$ with the expected value of his utility when he deviates and selects the \emph{pure} strategy $a'_i$.
It is possible to show that this is equivalent to requiring player $i$ not to improve even if selecting a \emph{mixed strategy} $\sigma_i\in\Delta(\mc{A}_i)$. Thus, an equivalent definition could be given with respect to deviations in mixed strategies. Additionally, observe that the set of mixed Nash equilibria contains the set of pure Nash equilibria. 

Mixed Nash equilibria are guaranteed to exist in any game where the actions sets are finite, as shown in the celebrated paper by John Nash, \cite{Nash01011950}.
\begin{proposition}[Existence of \ac{mne}, \cite{Nash01011950}]
	Any finite game admits a \ac{mne}. 
\end{proposition}
Despite the fact that existence of mixed Nash equilibria is guaranteed, the problem of computing a \ac{mne} is, in general, intractable \cite{Daskalakis_conference}. For this reason, we consider the following enlarged class of equilibria.
\begin{definition}[\Ac{cce}, \cite{moulin1978strategically}]
\label{def:cce}
A probability distribution $\sigmacce\in\Delta(\mc{A})$ is a coarse correlated equilibrium for the game $G$ if no player can increase his expected utility by deviating to a pure strategy, i.e., if 
\be
\expected{u_i(a)}{a\sim \sigmacce} \ge \expected{u_i(a_i',\ami)}{a\sim \sigmacce} \qquad\forall a_i' \in \mc{A}_i,\quad \forall i\in \N.
\label{eq:ccedef}
\ee
We denote with $\cce{G}$ the set of \ac{cce} of $G$.
\end{definition}
The only difference in the definitions of mixed Nash equilibrium and corse correlated equilibrium is in that $\sigmane$ is required to be a product distribution $\sigmane\in \times_{i\in\N}\Delta(\mc{A}_i)$, while $\sigmacce\in\Delta(\mc{A})$ is not. It follows that the set of coarse correlated equilibria is a superset of the set of mixed Nash equilibria. The interest in \ac{cce} stems from the fact that, unlike \ac{mne} and \ac{ne}, they are computationally tractable \cite{littlestone1994weighted, nisan2007algorithmic}. We will return to this in \cref{sec:complexity_equil}.

We conclude introducing the last equilibrium concept. To do so, we first consider a welfare function $W:\mc{A}\rightarrow \mb{R}_{\ge0}$ and define the allocation $\aopt \in\mc{A}$ as an allocation such that $W(\aopt)\ge W(a)$ for all $a\in\mc{A}$.
\begin{definition}[\Ac{acce},\cite{nadav2010limits}]
Given a game $G$ and a function $W:\mc{A}\rightarrow \mb{R}_{\ge0}$, a probability distribution $\sigmaacce\in\Delta(\mc{A})$ is an average coarse correlated equilibrium with respect to the allocation $\aopt\in\mc{A}$ if 
\be
\expected{\sum_i u_i(a)}{a\sim \sigmaacce}
\ge
\expected{\sum_i u_i(\aopt _i,a_{-i})}{a\sim \sigmaacce}.
\ee	
We denote with $\acce{G}$ the set of \ac{acce} of $G$.
\end{definition}

Average coarse correlated equilibria are a superset of coarse correlated equilibria. This is because the previous condition can be obtained from \cref{def:cce} by summing the condition \eqref{eq:ccedef} over all players, and selecting $a'=\aopt$. 

It follows that the equilibrium sets previously defined are all nested 
\be
\nashe{G}\subseteq \mnashe{G} \subseteq \cce{G} \subseteq \acce{G}.
\ee

While \ac{ne}, \ac{mne}, \ac{cce} are well studied and regularly used equilibrium concepts, the notion of \ac{acce} is rather novel. The latter equilibrium concept will be used here as a purely conceptual tool in connection with the study of  equilibrium efficiency, see \cref{subsec:tightness}.

\section{Potential games and congestion games}
\label{sec:potentialgames}
In the previous section we have introduced three fundamental equilibrium concepts. Additionally, we have commented on their existence and on their tractability (or lack thereof). In this section we refine the analysis to potential games and congestion games.
\subsection*{Potential games}
\begin{definition}[Potential game, \cite{monderer1996potential}]
	A strategic-form game is a potential game if there exists a function $\varphi:\mc{A}\rightarrow \mb{R}$ such that
	\be
	u_i(a_i,\ami)-u_i(a'_i,\ami) = \varphi(a_i,\ami)-\varphi(a'_i,\ami)\qquad\forall a\in\mc{A},\quad \forall a'_i\in\mc{A}_i,\quad\forall i\in\N .
	\ee
	The function $\varphi$ is called potential function.
\end{definition}
Informally, a game is potential if the variation in each player's utility experienced when deviating from $a_i$ to $a_i'$ can be captured by the function $\varphi$, and such function is the \emph{same} for all the players $i\in\N$. The condition is reminiscent of the notion of conservative force and corresponding potential field taken form Physics. Indeed, for games with continuous action space, the two notions coincide.

An immediate consequence of the previous definition is the existence of a pure \ac{ne}.
\begin{proposition}[Existence of pure \ac{ne} in potential games, \cite{monderer1996potential}]
\label{prop:existence_potential}
Any finite potential game admits a pure Nash equilibrium.	
\end{proposition}
\begin{proof}
Consider $a^\star\in\mc{A}$ a global maximizer of the potential. Since the actions sets are finite, $a^\star$ is guaranteed to exist. By definition of maximizer and of potential game, it is
\be
u_i(a_i^\star,\ami^\star)-u_i(a'_i,\ami^\star) =
\varphi(a_i^\star,\ami^\star)-\varphi(a'_i,\ami^\star)\ge 0\qquad\forall a_i'\in\mc{A}_i,\quad\forall i\in N\,.
\ee
Thus, $a^\star$ is a pure Nash equilibrium.
\end{proof}
It can be similarly shown that any local maximizer of the potential function $\varphi$ is a pure Nash equilibrium.
\begin{definition}(Local maximizer)
\label{def:localmax}
Given a function $\varphi:\mc{A}\rightarrow\mb{R}$ with $\mc{A}=\mc{A}_1\times\dots\times\mc{A}_n$, an allocation $a^\star\in\mc{A}$ is a local maximizer of $\varphi$ if $\varphi(a^\star)\ge \varphi(a_i',a^\star_{-i})$ for all $a_i'\in\mc{A}_i$ and for all $i\in\N$.
\end{definition}

 The previous observation builds a \emph{fundamental bridge} between optimization problems and equilibrium problems as it suggests a seemingly simple technique to compute a \ac{ne} for the class of potential games: determine a local maximizer of the potential function. Additionally, it suggests a natural dynamics to compute one such equilibrium. 
\begin{definition}[\Ac{br}]
	Let $t\in\mb{N}_0$ indicate the time step of the algorithm and $a^t\in\mc{A}$ the corresponding allocation. The best-response dynamics is presented in \cref{alg:BR}. Ties are broken according to a pre-specified rule (any rule).
\end{definition}
In the best-response dynamics, players take ordered turns and update their choice by selecting their best action, given the current actions of the others. While the \ac{br} dynamics is not guaranteed to converge for a general game, this is the case if we restrict to the class of potential games.
\newlength\algbrake
\setlength\algbrake{2mm}
\begin{algorithm}[h!]
\caption{Best-response dynamics (round-robin)}\label{alg:BR}
\begin{algorithmic}[1]
\State Initialise $a^0\in\mc{A}$;\quad $t\gets 0$
\vspace*{\algbrake}
\While {not converged}
\LeftComment{Best response}
\State $i\gets (t \mod n)+1$
\State $a^{t+1}_i\gets \argmax_{a_i\in\mc{A}_i} u_i(a_i,\ami^t)$
\State  $a^{t+1}\gets (a^{t+1}_i,\ami^t)$
\State $t\gets t+1$
\EndWhile
\end{algorithmic}
\end{algorithm}

\begin{proposition}[Convergence of the \ac{br} dynamics, \cite{monderer1996potential}]	
\label{prop:brconverges}
The best-response dynamics converges, for any initial condition $a^0\in\mc{A}$, to a \ac{ne} in a finite number of steps, for any potential and finite game.
\end{proposition}
\begin{proof}
The proof is based on the use of the potential function as a Lyapunov function. After every round of the \ac{br} dynamics either no player improved his utility, in which case we are at a Nash equilibrium, or at least the utility of one player has increased. In the latter case, $\varphi$ has increased too. Since $\varphi$ is upper bounded by its maximum value, the \ac{br} dynamics must converge. Additionally, since $\varphi$ strictly increases in every round, the best-response dynamics can not return to an allocation visited in the past. Thus, convergence in a finite number of steps follows by the finiteness of $\mc{A}$. 
\end{proof}
Three important comments follow. First, we considered here a round-robin best response algorithm, i.e., an algorithm where the players revise their decision in a given order. Similar statements to those in \cref{prop:brconverges} can be made \emph{almost surely} if the players updating their decision are uniformly randomly selected. This will produce a totally asynchronous algorithm.
Second, note that the claim in \cref{prop:brconverges} holds even if the players were to update their actions using a better-response dynamics, instead of a best-response dynamics. In the better-response dynamics, players update their previous choice by selecting an action that improves their utility, but need not be the best.
Third, observe that the best-response dynamics (better-response dynamics) might be slow to converge, in that it could visit \emph{all} the allocations in $\mc{A}$ before settling to a \ac{ne}. Additionally, the task of finding a best-response (line $4$ in \cref{alg:BR}) might also be intractable. We return to this in \cref{sec:complexity_equil}.

\subsection*{Congestion games}
Congestion games are defined as follows.
\begin{definition}[Congestion game, \cite{rosenthal1973class}]
\label{def:congestiongame}
Consider $\mc{R}$ a finite set of resources and for every resource $r\in\mc{R}$ a function $w_r:\mb{N}\rightarrow\mb{R}_{\ge0}$. A congestion game is a normal-form finite game where $\N=\{1,\dots,n\}$ is the set of players, $\mc{A}_i\subseteq 2^{\mc{R}}$ and $u_i(a)=\sum_{r\in a_i} w_r(|a|_r)$ are the action set and utility function of player $i$, respectively. The quantity $|a|_r$ represents the number of players selecting resource $r$ in allocation $a$,
\be
|a|_r\coloneqq\{i\in\N\,\text{s.t.}\,r\in a_i\}
\ee
\end{definition}
The next proposition shows that congestion games are a subclass of potential games. Thus, existence of a pure Nash equilibrium is guaranteed as well as convergence of the best-response dynamics, see \cref{prop:existence_potential} and \vref{prop:brconverges}. 
\begin{proposition}[Congestion games are potential, \cite{rosenthal1973class}]
\label{prop:congestionsarepotential}
	Congestion games are potential games with potential function $\varphi$ given by 
	\be
	\varphi(a)=\sum_{
	\substack{r\in\mc{R}\\[0.5mm]|a|_r\ge1}}\sum_{j=1}^{|a|_r} w_r(|a|_r).
	\label{eq:rosenthalpotential}
	\ee
	Thus, a pure Nash equilibrium is guaranteed to exist.
\end{proposition}
The potential function in \eqref{eq:rosenthalpotential} is often referred to as \emph{Rosenthal's potential}.

\section{Price of anarchy and smoothness}
\label{sec:poa_smooth}
The notions of price of anarchy and price of stability have been introduced to quantify the efficiency of the equilibrium allocations with respect to centralized optimal allocations \cite{Koutsoupias, schulz2003performance}. Let us consider a strategic-form game $G=(\N,\{\mc{A}_i\}_{i=1}^\N,\{u_i\}_{i=1}^\N)$ and a corresponding
 welfare function $W:\mc{A}\rightarrow \mb{R}_{\ge0}$. The function $W$ measures the quality of a given allocation, and can be used to model the achievement of a global objective. The price of anarchy represents the ratio between the welfare at the worst performing equilibrium and the optimal welfare. Consequently, it provides a bound on the efficiency for \emph{all} the equilibria. In the following we assume that $W(\aopt)>0$ so that the notion of price of anarchy is well posed.
\begin{definition}[\Ac{poa}, \cite{Koutsoupias}]
\label{def:poa}
Consider the strategic-form game $G$ and the welfare function $W:\mc{A}\rightarrow \mb{R}_{\ge0}$.	
\begin{enumerate}
	\item The price of anarchy for the class of \ac{ne} is defined as
\be
\poane \coloneqq \min_{a\in\nashe{G}}\frac{W(a)}{W(\aopt)}\,.
\ee
\item The price of anarchy for the class of \ac{mne} is defined as
\be
\poamne \coloneqq \min_{\sigma \in\mnashe{G}}\frac{\expected{W(a)}{a\sim\sigma}}{W(\aopt)}\,.
\ee
Replacing the set $\mnashe{G}$ with $\cce{G}$ or $\acce{G}$, one obtains the corresponding definitions for $\poacce$ and $\poaacce$.
\end{enumerate}

\end{definition}
Observe that the expression $W(\aopt)$ also depends on the game instance $G$ considered 
but we do not indicate it explicitly, for ease of presentation.
By definition, the price of anarchy is bounded between zero and one. The higher the price of anarchy, the more efficient the worst performing equilibrium. Since $\nashe{G}\subseteq \mnashe{G} \subseteq \cce{G} \subseteq \acce{G}$ as seen in \cref{sec:equilibria}, it follows that
\[
\poane\ge \poamne \ge \poacce \ge \poaacce\,,
\]
i.e., the efficiency degrades as we move to a richer class of equilibria.

As a prototypical example to clarify the concept of price of anarchy, consider that of a road traffic network where a large number of drivers traveling from a certain origin to their corresponding destination. If each driver was to minimize his own travel time, this will result in an equilibrium configuration such as the \ac{ne}. Instead, if the system operator was to instruct the drivers on which route to take, he will try to minimize the total travel time, i.e., the sum over all the drivers' individual travel time. The price of anarchy precisely capture the ratio between these two quantities.

While we present results relative to welfare maximization problems, analogous definitions and claims are available in case of cost minimization.
\begin{definition}[Smooth game, \cite{roughgarden2009intrinsic}]
\label{def:smoothgame}
	Consider the strategic-form game $G$ and the welfare function $W:\mc{A}\rightarrow \mb{R}_{\ge0}$. The pair $(G,W)$ is $(\lambda,\mu)$-smooth if for some $\lambda,\mu\ge0$ it holds
	\be
	\label{eq:smoothdef}
	\sum_{i\in\N} u_i(a_i',a_{-i})\ge \lambda W(a')-\mu W(a)\,, \quad\forall a,a'\in\mc{A}.
	\ee
\end{definition}
 
 	The following proposition provides a lower bound on the ratio between the expected welfare at any \ac{cce} and the optimum, i.e., it gives a bound on the price of anarchy relative to the specific instance $G$ considered.
\begin{proposition}[$\poa$ bound,\cite{roughgarden2009intrinsic}]
\label{prop:poasmooth}
	Consider a $(\lambda,\mu)$-smooth game with $ \sum_{i\in\N} u_i(a)\le W(a)$ for all $a\in\mc{A}$. Then, for any coarse correlated equilibrium $\sigmacce$ of $G$ it holds
	\be
		\frac{\expected{W(a)}{a\sim\sigmacce}}{W(\aopt)}\ge \frac{\lambda}{1+\mu}\,. 	
		\label{eq:lambdamubound}
	\ee
\end{proposition} 

\begin{proof}
Consider $\sigmacce$ any \ac{cce} of $G$. Setting $a'_i=\aopt_i$ in
\cref{def:smoothgame} it is 
\[
0\le \expected{u_i(a)}{a\sim \sigmacce} - \expected{u_i(\aopt,\ami)}{a\sim \sigmacce} 
\qquad \forall i\in \N.
\] 
Summing over the agents one obtains

\[
\begin{split}
0&\le \expected{\sum_i u_i(a)}{a\sim \sigmacce}
-
\expected{\sum_i u_i(\aopt_i,a_{-i})}{a\sim \sigmacce}
\\
&\le 
\expected{\sum_i u_i(a)}{a\sim \sigmacce}
-\lambda W(\aopt) + \mu \expected{W(a)}{a\sim \sigmacce}\\
&\le -\lambda W(\aopt) + (1+\mu)\expected{W(a)}{a\sim \sigmacce}\,,
\end{split}
\]	
where we used the linearity of the expectation, the definition of $(\lambda,\mu)$-smooth game, and the assumption for which $\sum_{i\in\N}u_i(a)\le W(a)$. 
\mbox{The claim follows from $W(a)\ge0$}
\[
\frac{\mb{E}_{a\sim\sigmacce}[W(a)]}{W(\aopt)}\ge \frac{\lambda}{1+\mu}\,.
\]
\end{proof}
 The smoothness framework has proved useful in bringing a number of different results under a common language and has produced tight bounds on the price of anarchy for different problems \cite{roughgarden2009intrinsic, roughgarden2017price}. 
Its strength amongst others, lies in the recipe it provides to obtain performance bounds for a large class of equilibria.
Indeed, as seen in the previous section, pure \ac{ne} (and \ac{mne}) are a subclass of \ac{cce}. Thus, for any pure Nash equilibrium $\ae$ of $G$ it also holds 
\be
\frac{W(\ae)}{W(\aopt)}\ge \frac{\lambda}{1+\mu}\,.
\ee

The proof presented in \cref{prop:poasmooth} shows that once a game has been shown to be $(\lambda,\mu)$-smooth, the corresponding bound on the price of anarchy follows easily (only linearity of the expectation is additionally used). Thus, the main difficulty in proving bounds on the price of anarchy using a smoothness argument resides in proving the smoothness property itself, i.e., in selecting $\lambda$ and $\mu$ so that \eqref{eq:smoothdef} holds for all $a,a^\star\in\mc{A}$. These parameters have been determined for certain classes of games. A non-comprehensive list include scheduling games \cite{cohen2012smooth}, location and valid utility games \cite{vetta2002nash}, affine congestion games \cite{roughgarden2009intrinsic}, first-price auctions \cite{kaplan2012asymmetric, syrgkanis2013composable}, second-price auctions \cite{christodoulou2008bayesian}.

\subsection{The question of tightness}
\label{subsec:tightness}
An important question we discuss in this section is the capability of the smoothness framework to give good (ideally tight) bounds on the price of anarchy. Given a game $G$, we define the best bound on the price of anarchy attainable via a smoothness argument~as
\[
\begin{split}
\poas\coloneqq &\sup_{\lambda,\mu \ge0}\frac{\lambda}{1+\mu}\\
&\text{s.t.}~
(\lambda,\mu)~\text{satisfy \eqref{eq:smoothdef}}\,.
\end{split} 
\]
 Observe that $\poas\le \poa$ as \cref{prop:poasmooth} provides only a bound on the equilibrium efficiency.
The next proposition shows that $\poas$ is tight for the class of \ac{acce}. 

\begin{proposition}[Smoothness is tight for \ac{acce}, \cite{nadav2010limits}]
For any given game $G$~it~is
\be
\poas = \poaacce = \min_{\sigma\in\acce{G}}\frac{\expected{W(a)}{a\sim\sigma}}{W(\aopt)}\,.
\ee
\end{proposition}
The previous proposition provides a positive result, in that it shows that $\poas$ matches the ``true'' price of anarchy for the class of \ac{acce}. Nevertheless, this result is rather weak. Indeed, it has been shown by means of counterexamples that the best bound on the price of anarchy achievable using a smoothness argument is \emph{not tight} in the class of \ac{cce}, \cite{nadav2010limits, part1Paccagnan2018}. That is, there are instances $G$ where $\poas < \poacce$ and $\poas$ provides a rather weak bound on $\poacce$, \cite{part1Paccagnan2018}. It follows that the best smoothness bound can not be tight for any of the subclasses of \ac{cce} including~\ac{mne}~and~\ac{ne}.

\section{Complexity of computing equilibria}
\label{sec:complexity_equil}
The goal of this section is to present an overview on the complexity issues related to the equilibrium computation problem. We do not delve in the details of different complexity classes, but simply try to highlight which equilibrium concepts are ``hard'' to compute and which are ``easy''.

We divide the presentation in four parts. First, we present an intractability result for pure and mixed Nash equilibria. Second, we restrict our attention to congestion games and show that the best-response dynamics converges in polynomial time under structural assumptions on the actions sets $\{\mc{A}_i\}_{i\in\N}$. Third, we show that coarse correlated equilibria are tractable in general. We conclude discussing the tradeoff between equilibrium efficiency and computational tractability.

\subsubsection*{Pure and mixed Nash equilibria are intractable}
We begin with a negative result showing that the problem of computing a pure \ac{ne} is intractable, even if we restrict to the class of congestion games. In the following $\pls$ represents the complexity class known as \emph{polynomial local search}. Loosely speaking the $\pls$ class models the difficulty of finding a local optimum solution in the sense of \cref{def:localmax}. 
The $\pls$ class lives in between the classes $\pclass$ and $\npclass$ and there is strong evidence suggesting that $\pls\not\subseteq \pclass$, where $\pclass$ is the class of problems that can be solved polynomially. As a matter of fact, many concrete problems including the local Max-Cut problem are in the $\pls$ class and no efficient algorithm is available. We redirect the reader to \cite{roughgarden2016twenty} for an introduction to the polynomial local search class.

\begin{proposition}[Computing a pure \ac{ne} is $\pls$-complete, \cite{fabrikant2004complexity}]
\label{prop:neispls}
The problem of computing a pure \ac{ne} in a congestion game is $\pls$-complete.
\end{proposition}
It follows that for a general strategic form and finite game (not necessarily a congestion game), computing a pure Nash equilibrium is as hard as the hardest problem in the $\pls$ class.
Any modification of the original problem (e.g., determining if a game $G$ has a pure \ac{ne}, determining the \ac{ne} that maximizes a given welfare function) makes it $\npclass$-complete \cite{gottlob2005pure,conitzer2002complexity}.

Computing a mixed Nash equilibrium is also an intractable problem. Its complexity has been settled in \cite{Daskalakis_conference, chen2006settling} with the introduction of the $\ppad$ complexity class.
Similarly to the $\pls$ class, the $\ppad$ class lives in between the classes $\pclass$ and $\npclass$ and despite the great interest in the topic, there are currently no known efficient algorithms to tackle these problems \cite{Daskalakis_conference}. 

\begin{proposition}[Computing a \ac{mne} is $\ppad$-complete, \cite{Daskalakis_conference, chen2006settling}]
The problem of computing a \ac{mne} in a strategic-form finite game is $\ppad$-complete.	
\end{proposition}
\label{subsec:puremixed_complex}
\subsubsection*{Nash equilibria are tractable in matroid congestion games}
\label{subsec:pure_matroid}
While computing a (pure) Nash equilibrium is intractable
even if restricting to the class  of congestion games, it is possible to obtain a more positive result by imposing structural constraints on the actions sets.
\begin{definition}[Matroid, \cite{welsh2010matroid}]
\label{def:matroid}
A tuple $\mc{M}=(\mc{R},\mc{I})$ is a matroid if $\mc{R}$ is a finite set, $\mc{I}\subseteq 2^\mc{R}$ is a collection of subsets of $\mc{R}$, and the following two properties hold:
\begin{itemize}
	\item[-] If $B\in\mc{I}$ and $A\subseteq B$, then $A\in\mc{I}$;
	\item[-] If $A\in\mc{I}$, $B\in\mc{I}$ and $|B|>|A|$, then there exists an element $r\in B\setminus A$ s.t. $A\cup \{r\}\in\mc{I}$.
\end{itemize}	
\end{definition}
\begin{definition}[Basis of a matroid, \cite{welsh2010matroid}]
	A set $S\in\mc{I}$ such that for all $r\in\mc{R}\setminus S$, $(S\cup{r})\notin \mc{I}$ is called a basis of the matroid. 
\end{definition}
It can be shown that all basis have the same number of elements, which is known as the rank of the matroid and indicated with $\rank(\mc{M})$, \cite{welsh2010matroid}. An example of matroid is that of uniform matroid defined as follows. 
\begin{definition}[Uniform matroid, \cite{welsh2010matroid}]
	Given a finite set $\mc{R}$ with $|\mc{R}|=m$, let $\mc{I}\subseteq 2^\mc{R}$ be the collection of all subsets with a number of elements $k\le m$. $\mc{M}=(\mc{R},\mc{I})$ is a matroid, $\rank(\mc{M})=k$ and $\mc{M}$ is called the uniform matroid of rank $k$.
\end{definition}
The following proposition provides sufficient conditions under which the best-response dynamics of \vref{alg:BR} has polynomial running time for the class of congestion games. The main assumption amounts to requiring each of the player's allocation set to coincide with the set of bases of some matroid.
\begin{proposition}
\label{prop:poly}\cite[Thm. 2.5]{ackermann2008impact}
Consider a congestion game $G$ and assume the action sets $\mc{A}_i$ are the set of bases for a matroid $\mc{M}_i=(\mc{R},\mc{I}_i)$ over the set $\mc{R}$. Then, players reach a (pure) Nash equilibrium\mbox{ after at most $n^2 m\,\max_{i\in N}\text{rank}(\mc{M}_i)$ best responses.}	
\end{proposition}

\begin{example}\label{example:matroids}
The case when $\mc{A}_i$ contains only sets with a single element (singletons) does satisfy the assumptions of the previous proposition, even if a player does not have access to all the possible resources. One such example is the following: $\mc{R}=\{r_1,\dots,r_m\}$, $m>2$, $\mc{A}_i=\{\{r_1\},\{r_2\}\}$. Define $\mc{I}_i=\{\emptyset,\{r_1\},\{r_2\}\}$. We have that $\mc{M}_i\coloneqq(\mc{R},\mc{I}_i)$ is a matroid of rank $1$ and that $\mc{A}_i$ is a set of bases for $\mc{M}_i$.

On the negative side, a few examples that \emph{do not} satisfy the requirements are presented next. Consider $\mc{R}=\{r_1,\dots,r_m\}$, $m\ge3$ and $\mc{A}_i=\{\{r_1\},\{r_2,r_3\}\}$. The set $\mc{A}_i$ can not form the set of bases for any matroid $\mc{M}_i$, as all bases must have the same number of elements  while $\{r_1\}$ and $\{r_2,r_3\}$ do not have this property. A more involved example that does not satisfy the requirements is the following: $\mc{R}=\{r_1,\dots,r_m\}$, $m\ge4$ and $\mc{A}_i=\{\{r_1,r_2\},\{r_3,r_4\}\}$. For the given $\mc{A}_i$ to be the set of bases of a matroid $\mc{M}_i=(\mc{R},\mc{I}_i)$, it must be that $\{r_1,r_2\}\in\mc{I}_i$ and $\{r_3,r_4\}\in\mc{I}_i$. But due to definition of matroid, it must also be $\{r_1\}\in\mc{I}_i$ (Definition \cref{def:matroid}, first property), so that also $\{r_1,r_3,r_4\}\in\mc{I}_i$ (\cref{def:matroid}, second property). Thus $\mc{A}_i$ can not be the set of bases for a matroid $\mc{M}_i=(\mc{R},\mc{I}_i)$, as any possible choice of $\mc{I}_i$ will contain at least one set with more elements than $\{r_1,r_2\}\in\mc{A}_i$.	
\end{example}

\begin{remark}
 The previous theorem gives conditions under which the maximum number of best responses required to converge to a Nash equilibrium is polynomially bounded in the number of players and resources. If it is possible to compute a single best response polynomially in the number of resources, then it is possible to compute a \ac{ne} in polynomial time using the best-response algorithm.
\end{remark}

\subsubsection*{Coarse correlated equilibria}
\label{subsec:cce_complex}
Contrary to \ac{ne} and \ac{mne}, (approximate) coarse correlated equilibria can be computed in polynomial time.  We limit ourself to report this result in the following proposition.

Formally, an $\varepsilon$-\ac{cce} is defined as a distribution $\sigma\in\Delta(\mc{A})$ such that the equilibrium condition in \cref{def:cce} holds up to an additive $\varepsilon\ge 0$ term, i.e., 
\be
\expected{u_i(a)}{a\sim \sigma} +\varepsilon \ge \expected{u_i(a_i',\ami)}{a\sim \sigma} \qquad\forall a_i' \in \mc{A}_i,\quad \forall i\in \N.
\ee

\begin{proposition}[$\varepsilon$-\ac{cce} can be computed efficiently, \cite{littlestone1994weighted, roughgarden2016twenty}]
	For every $\varepsilon>0$, an $\varepsilon$-\ac{cce} can be computed polynomially using the multiplicative-weight algorithm.
\end{proposition}

\subsubsection*{The tradeoff between tractability and efficiency}
\label{subsec:tradeoff_complexeffic}
This section connects the efficiency result presented in \cref{sec:poa_smooth} with the complexity results presented in \cref{sec:complexity_equil}.
In the former section we have seen that $\poane\ge \poamne \ge \poacce$ and thus the equilibrium efficiency (\ac{poa}) degrades by moving from \ac{ne} to \ac{mne} and from \ac{mne} to \ac{cce}. In the latter section we have seen that \ac{ne} and \ac{mne} are tractable in limited cases, while \ac{cce} are tractable in general. This shows a fundamental tradeoff between equilibrium efficiency and corresponding tractability: the larger the class of equilibria we consider, the easier to compute one, but the lower the corresponding efficiency. This is depicted in \cref{fig:hierarchy}.

\begin{figure}[h!]
\centering
\includegraphics[scale=0.5]{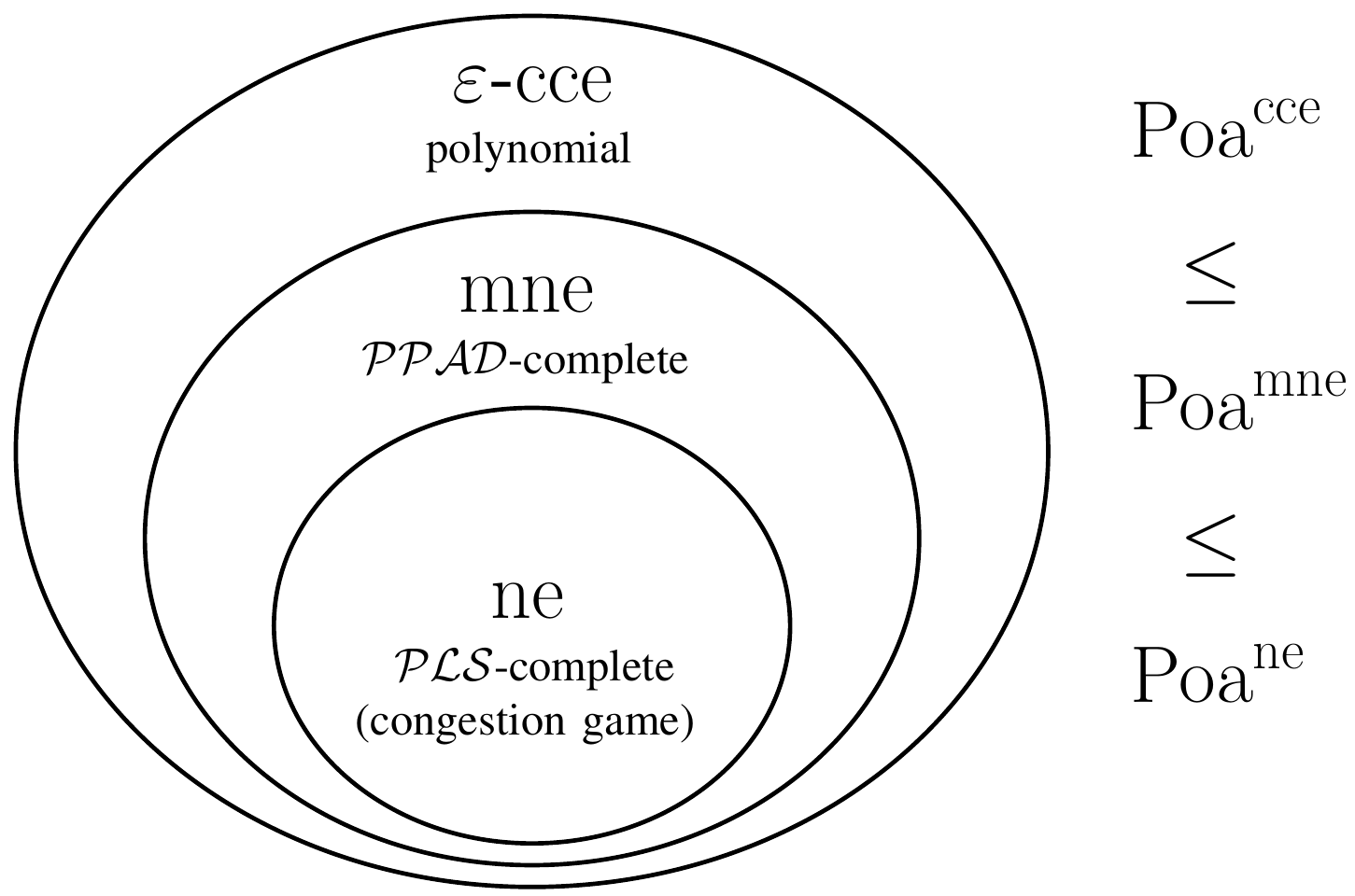}
\caption{Hierarchy of equilibria, corresponding complexity and efficiency.}
\label{fig:hierarchy} 
\end{figure}
\chapter[Tight price of anarchy and utility design: a linear program approach]{Tight price of anarchy and utility design:\\ a linear program approach}
\label{ch:p2utilitydesign}
We seek approximation algorithms for the solution of \ac{gmmc} problems defined in \cref{sec:problemstatement}. Towards this goal, we adopt the game design approach discussed in \cref{ch:p2introduction} and consisting of two separate steps: utility design and algorithm design. In this chapter we formulate and solve the utility design problem. More precisely, in \cref{sec:poaperformancemetric} we pose the utility design problem and introduce the game-theoretic notion of \emph{price of anarchy}. We observe that any algorithm capable of computing an equilibrium, will inherit an approximation ratio matching the price of anarchy. Thus, in a quest to construct good approximating algorithms, we turn our attention to \emph{quantifying} the price of anarchy. In \cref{sec:poa_smooth} we show that standard approaches used to characterize the price of anarchy are rather conservative and not suited for the design problems we are interested in (\cref{thm:smoothnottight}). Motivated by this observation, in \cref{sec:tightpoa} we provide a novel technique based on a linear programming reformulation to characterize the price of anarchy (\cref{thm:primalpoa,,thm:dualpoa}) as a function of the utility functions assigned to the agents. This result is provably tight. We conclude the chapter by addressing the utility design question in \cref{sec:optimalutilitydesign}. In particular, we show how the problem of designing utility functions so as to \emph{optimize} the price of anarchy can be posed as a tractable linear program in $n+1$ variables (\cref{thm:optimizepoa}).

All the proofs are reported in the Appendix (\cref{sec:proofsp2-part1}). The results presented in this chapter have been published in \cite{part1Paccagnan2018,,part2Paccagnan2018}.
\section{The price of anarchy as performance metric}
\label{sec:poaperformancemetric}
Within the combinatorial framework considered, finding a solution to the \ac{gmmc} problem, i.e., determining a feasible allocation that maximizes the welfare function 
\be
W(a)=\sum_{r\in\cup_{i\in N} a_i}v_r w(|a|_r),
\ee
defined in \eqref{eq:welfaredef} is an $\npclass$-hard problem.
Based on such observation, we focus on deriving \emph{efficient} and \emph{distributed algorithms} for attaining approximate solutions to the maximization of $W$, ideally with the best possible ratio. In the following, each agent is assumed to have information only regarding the resources that he can select, i.e., regarding the resources $r\in\mc{A}_i\subseteq \mc{R}$. Agents are requested to make independent choices in response to this local piece of information. 
Rather than directly specifying a decision-making process, we adopt the \emph{game design} approach discussed in \cref{ch:p2introduction}  and depicted in \vref{fig:gamedesign}. The idea is to carefully define an auxiliary problem, namely an equilibrium problem, which will guide the search and serve as a proxy for the original maximization of $W$. The motivations and advantages of this approach have been discussed in \cref{ch:p2introduction}.

In the following we focus on the first component of the game design approach: the \emph{utility design} problem. 

\vspace*{3mm}
\noindent \fbox{\parbox{0.985\textwidth}{
\label{problem:utilitydesign}
The {\bf utility design problem} amounts to the choice of local utility functions that adhere to the above mentioned informational constraints, and whose corresponding equilibria offer the highest achievable performance.
}}
\vspace*{1mm}

We naturally identify the agents of the original optimization problem and their local constraint sets $\{\mc{A}_i\}_{i\in\N}$ with the players of the game and their action sets. In the following we will use the terms agents and players interchangeably.
 
In order to tackle the utility design problem, each agent is assigned a local utility function $u_i:\mc{A}\rightarrow \mb{R}_{\ge0}$ of the form
\be
u_i(a) \coloneqq \sum_{r\in a_i}v_r w(|a|_r)f(|a|_r)\,,
\label{eq:utilities}
\ee
where $f:[n]\rightarrow \mb{R}_{\ge 0}$ describes the fractional benefit that each agent receives by selecting resource $r$ in allocation $a$. The function $f$ constitutes our \emph{design choice}; we refer to it as to the \emph{distribution rule} or simply the \emph{distribution}.
Observe that each utility function in \eqref{eq:utilities} satisfies the required informational constraints in that it only depends on the value of the resources that the agent selected, the distribution rule $f$ and the number of agents that selected the very same resource. 
\begin{remark}[On the choice of the utility functions]
In principle one needs not to restrict himself to utility functions of the form \eqref{eq:utilities}. The reasons for choosing utilities as in \eqref{eq:utilities} are as follows.
First, the utility functions \eqref{eq:utilities} satisfy the required informational constraints, as just discussed. 
Second, restricting ourselves to the above mentioned utilities reduces the design problem to a hopefully tractable problem. Indeed, the utilities \eqref{eq:utilities} are fully determined if the distribution rule $f$ is so. 
While designing the distribution rule $f$ amounts to choosing $n$ real numbers, the problem in its full generality consists in choosing the value of $u_i(a)$ for all $a\in\mc{A}$ and for all $i\in N$, clearly a large number of decision variables (exponential in the worst case in both the number of agents and in the number of resources since $\mc{A}=\mc{A}_1\times\dots\times\mc{A}_n$ and $\mc{A}_i\subseteq 2^\mc{R}$). 
Third, utilities of the form \eqref{eq:utilities} will ensure equilibrium existence and convergence of the best-response dynamics, as explained after this remark.
Fourth, even when restricting to this special class of utilities, we will  obtain performance certificates that are competitive with the state of the art approximation algorithms. We will return to this in the next chapter in \cref{subsec:approxcomparison} and \cref{rmk:1-1/ecomparison}.
Finally, we observe that a different and apparently less restrictive choice of utilities might entail assigning different distribution rules $f_i$ to different players $i\in\N$. However, it is possible to show that working in this larger set of admissible utility functions will not improve the best achievable performance.\footnote{While we do not provide a proof of this statement, a similar conclusions was found in \cite{gairing2009covering}.} 
For all these reasons, in the followings we focus on utility functions of the form \eqref{eq:utilities}.
\end{remark}

The game introduced above and identified with the agents set $\N$, the actions sets $\{\mc{A}_i\}_{i\in N}$ and the utilities $\{u_i\}_{i\in N}$ in \eqref{eq:utilities} is a normal-form finite game, according to \cref{def:strategicformgame}. Additionally, such game belongs to the class of congestion games due to the special structure of the actions sets and utilities, see \cref{def:congestiongame}. Thus, a pure Nash equilibrium is \emph{guaranteed to exist} for any choice of $f$ thanks to \cref{prop:congestionsarepotential}. 

In the forthcoming analysis we focus on the solution concept of pure Nash equilibrium, which we will refer to as just an equilibrium.
	Recall that an allocation $\ae \in \mc{A}$ is a pure Nash equilibrium if $u_i(\ae)\ge u_i(a_i,\ae_{-i})$ for all alternative allocations $a_i\in \mc{A}_i$ and for all agents $i\in N$ (see \vref{def:nashequilibrium}).
We identify one instance of the game introduced above with the tuple 
\begin{equation}	
G=(\mc{R}, \{v_r\}_{r\in \mc{R}},  N, \{\mc{A}_i\}_{i\in N},f)\,,
\label{eq:gameG}
\end{equation}
and for ease of notation remove the subscripts of the above sets, e.g., use $\{\mc{A}_i\}$ instead of $\{\mc{A}_i\}_{i\in N}$. 

In the following we require a system operator to \emph{robustly} design a distribution rule, that is to design $f$ \emph{without} any prior information regarding the resource set $\mc{R}$, the value of the resources $\{v_r\}$ or the action sets of the agents $\{\mc{A}_i\}$.  The \emph{only} datum available to the system designer is an upper bound on the number of players in the game, i.e., $|N|\le n$. This request stems from the observation that the previous pieces of information may be unreliable, or unavailable to the system designer due to, e.g., communication restrictions or privacy concerns. Formally, given a distribution rule $f$, we introduce the following family of games
\[
\mc{G}_f\coloneqq\{(\mc{R},\{v_r\},N,\{\mc{A}_i\},f)~\text{s.t.}~|N|\le n\}\,,
\]
containing all possible games $G$ where the number of agents is bounded by $n$.
In the forthcoming analysis, we restrict our attention to the class of games where the number of players is \emph{exactly} $n$. This is without loss of generality. Indeed the latter class of games and the class of games where the number of players is upper bounded by $n$ have the same price of anarchy. To see this, note that the price of anarchy of any game with $l$ players $1<l<n$ can be obtained as the price of anarchy of a game with $n$ players where we add a resource valued $v_0=0$ and set $\mc{A}_i=\{v_0\}$ for the additional $n-l$ players.

We measure the performance of a distribution rule $f$ adapting the concept of \emph{price of anarchy} introduced in \cite{Koutsoupias} and reported in  \cref{def:poa} as
\be
\poa(f) \coloneqq \inf_{G\in \mc{G}_f}\biggl(\frac{\min_{a\in \nashe{G}} W(a)}{\max_{a\in\mc{A}} W(a)}\biggr)\,,
\label{eq:poadef}
\ee
where $\nashe{G}$ denotes the set of Nash equilibria of $G$.
While the optimal value at the denominator of \eqref{eq:poadef} also depends on the instance $G$ considered, we do not indicate it explicitly, for ease of presentation.
The quantity $\poa(f)$ characterizes the efficiency of the worst-performing Nash equilibrium relative to the corresponding optimal allocation over all instances in the class $\mc{G}_f$. According to the previous definition, $0\le\poa(f)\le 1$ and the higher the price of anarchy, the better performance certificates we can offer. 

It is important to highlight that whenever an algorithm is available to compute one such equilibrium, the price of anarchy also represents the \emph{approximation ratio} of the corresponding algorithm over all instances $G\in\mc{G}_f$.
For this reason, the price of anarchy defined in \eqref{eq:poadef} will serve as the \emph{performance metric} in all the forthcoming analysis.

\begin{remark}[On the choice of pure \ac{ne} as equilibrium concept]
The choice of pure \ac{ne} as equilibrium concept has the benefit of providing us with potentially better performance guarantees compared to that offered by, e.g., mixed Nash equilibria or coarse correlated equilibria, as $\poane\ge \poamne \ge \poacce$, see \cref{sec:poa_smooth}.
The drawback of this choice is the general intractability of pure Nash equilibria. 
Indeed, computing a pure Nash equilibrium is hard ($\pls$-complete, as discussed in \vref{prop:neispls}) even when limited to the class of games considered here (congestion games). 
Nevertheless we have seen that under structural assumptions on the sets $\{\mc{A}_i\}$ similar to those used in combinatorial optimization, computing a pure \ac{ne} is a polynomial task (\vref{prop:poly}).
Finally, the approximation guarantees offered by $\poane$ are deterministic, while the bounds provided by $\poamne$ and $\poacce$ are in expected value.
An antipodal choice might entail using the notion of \ac{cce} instead of \ac{ne} as computing one such equilibrium is known to be tractable in general. The price to pay for this is a potentially worsened performance certificate since $ \poacce \le \poane$.	
\end{remark}

The utility design problem can be decomposed in two tasks: 
\begin{enumerate}
\item[i)] providing a bound (or ideally an exact characterization) of the price of anarchy as a function of $f$;
\item[ii)] optimizing this expression over the admissible distribution rules.
\end{enumerate}
In \cref{sec:tightpoa} we address i), while in \cref{sec:optimalutilitydesign} we tackle ii).

\section{The limitations of the smoothness framework}
\label{sec:smooth}
In this section we recall the definition of smooth games introduced in \cref{sec:poa_smooth}, and show that the corresponding best achievable bounds on the price of anarchy are not tight, but rather conservative when applied to utility design problems. 

Before delving in the details of the smoothness framework, we introduce the notion of budget-balanced and sub budget-balanced utility functions.
\begin{definition}[Budget-balanced utility functions]
	Consider a strategic-form game with actions sets $\{\mc{A}_i\}$, utilities $\{u_i\}$, and a welfare function $W:\mc{A}\rightarrow \mb{R}_{\ge0}$. The utility functions are budget-balanced if for all $a\in\mc{A}$
	\[\sum_{i\in\N}u_i(a)=W(a).\]
	The utility functions are sub budget-balanced if 
	$
	\sum_{i\in\N}u_i(a)\le W(a)$~for~all~$a\in\mc{A}$.
\end{definition}

The notion of smooth game has been introduced in \cite{roughgarden2009intrinsic} and has been successively employed to obtain tight bounds on the price of anarchy for different classes of games. Recall from \vref{def:smoothgame} that the game \eqref{eq:gameG} together with the welfare function \eqref{eq:welfaredef} are $(\lambda,\mu)$-smooth if for some $\lambda,\mu\ge0$ it holds
	\be
	\sum_{i\in\N} u_i(a_i',a_{-i})\ge \lambda W(a')-\mu W(a),\quad\forall a',a\in\mc{A}.
	\label{eq:smoothcondition2}
	\ee
	\vref{prop:poasmooth} showed that 
	the price of anarchy of a $(\lambda,\mu)$-smooth game $G$ is bounded. More precisely, given $G$ a $(\lambda,\mu)$-smooth game with $\sum_{i\in\N} u_i(a)\le W(a)$ $\forall a\in\mc{A}$, the ratio between the total welfare at any coarse correlated equilibrium and the optimum is lower bounded by

	\begin{equation}
		\frac{\expected{W(a)}{a\sim \sigmacce}}{W(\aopt)}\ge \frac{\lambda}{1+\mu}\,,\qquad\forall \sigmacce\in\cce{G}. 	
	\end{equation}
 Since $\nashe{G}\subseteq \cce{G}$, it follows immediately  \[
 \frac{W(\ae)}{W(\aopt)}\ge \frac{\lambda}{1+\mu}\,,\qquad\forall\ae \in\nashe{G}.
 \]
 Note that the smoothness framework forces us to restrict the attention to utilities satisfying $\sum_{i\in\N} u_i(a)\le W(a)$, else no guarantee is provided by \cref{prop:poasmooth}. This corresponds to requesting $f(j)\le 1/j$.
 Thus, in the remaining of this section \emph{only}, we consider utilities satisfying this constraint.

The next lemma shows that when we are allowed to freely choose the players' utilities (i.e., if we are interested in design problems), the best achievable smoothness guarantee is obtained when the assigned utilities are budget-balanced.
\begin{lemma}
\label{lem:neveradvantageous}
Suppose $\sum_{i\in\N} u_i(a)=W(a)$.  Consider a different set of utilities $\tilde{u}_i(a)$ such that $\sum_{i\in \N}\tilde{u}_i(a) \leq \sum_{i\in \N}u_i(a)$ for all $a \in \mc{A}$.  If the game with utilities $\tilde{u}_i(a)$ is $(\lambda, \mu)$-smooth, then the game with utilities $u_i(a)$ is also $(\lambda, \mu)$-smooth.  
\end{lemma}
\begin{proof}
By assumption the game with utilities $\tilde{u}_i(a)$	is $(\lambda,\mu)$-smooth and $\sum_{i} u_i(a)\ge \sum_{i}\tilde{u}_i(a)$, so that for all $a,a'\in\mc{A}$
\[
\sum_{i\in\N} u_i(a_i',a_{-i})\ge \sum_{i\in\N} \tilde u_i(a_i',a_{-i})\ge \lambda W(a')-\mu W(a).
\]
Thus, the game with utilities ${u}_i(a)$ is $(\lambda,\mu)$-smooth too.
\end{proof}
Observe that the statement of \cref{lem:neveradvantageous} holds true in general and does not depend on the specific form of the utility functions or of the welfare considered here.

\cref{lem:neveradvantageous} suggests to design utilities that are budget-balanced, as sub budget-balanced utilities can never be advantageous with regards to the performance guarantees associated with smoothness.
 This observation turns out to be misleading, in that there are utility functions that are sub budget-balanced, but give a better performance certificate compared to what the smoothness argument can offer, as shown next.

Consider $f$ a distribution rule satisfying $f(j)\le 1/j$ for all $j\in[n]$, the best bound on the price of anarchy \eqref{eq:poadef} that can be obtained via smoothness, is given by the solution to the following  program
\[
\begin{split}
\poas(f)\coloneqq &\sup_{\lambda,\mu \ge0}\frac{\lambda}{1+\mu}\\
&\text{s.t.}~
(\lambda,\mu)~\text{satisfy \eqref{eq:smoothcondition2} for all } G\in\mc{G}_f.
\end{split} 
\]
Observe that $\poas(f)\le \poa(f)$ as \cref{prop:poasmooth} provides only a bound on the equilibrium efficiency. In the following we show that the best smoothness bound captured by $\poas(f)$ is not representative of the ``true'' price of anarchy $\poa(f)$ defined in \eqref{eq:poadef}. To do so, we illustrate the \emph{gap} between these two quantities in the special case of multiagent weighted maximum coverage (\ac{mmc}) problems (see \cref{sec:problemstatement}). \ac{mmc} problems are a special class of the resource allocation problems considered here. They are obtained  setting $w(j)=1$ for all $j\in[n]$. Before stating the result, we introduce the distribution rule 
\begin{equation}
\fgar(j)=(j-1)!\frac
{\frac{1}{(n-1)(n-1)!} +\sum_{i=j}^{n-1}\frac{1}{i!}}
{\frac{1}{(n-1)(n-1)!} +\sum_{i=1}^{n-1}\frac{1}{i!}},\quad j\in[n]\,.
\label{eq:fgar}	
\end{equation}
as originally defined in \cite[Eq. (5)]{gairing2009covering}.
\begin{theorem}[\bf{Limitations of the smoothness framework}]
\label{thm:smoothnottight}
Consider the class of \ac{mmc} problems, i.e., fix $w(j)=1$ for all $j\in[n]$.
\begin{enumerate}
\item For any choice of $f$, the best bound on the price of anarchy that can be achieved using a smoothness argument~is 
\[
\poas(f)\le \frac{1}{2-1/n} \coloneqq b(n)~~\xrightarrow[]{n \to \infty}~  \frac{1}{2}\,.
\] 
\item The distribution \eqref{eq:fgar} satisfies $\fgar(j)\le 1/j$ and achieves
\be
\poa(\fgar)=1-\frac{1}{\frac{1}{(n-1)(n-1)!}+\sum_{i=0}^n \frac{1}{i!}}~~\xrightarrow[]{n \to \infty}~ 1-\frac{1}{e}\,,
\label{eq:poagairing}
\ee
where $e$ is Euler's number.
\item For all $n>2$, $\poas(\fgar)<\poa(\fgar)$\,.
\end{enumerate}
\end{theorem}

\begin{remark}[The limitations of smoothness are structural]
While the previous theorem compares the performance guarantees offered by $\poas(f)$ and $\poa(f)$ we recall that $\poas(f)$ bounds the equilibrium efficiency for any coarse correlated equilibrium, while 
$\poa(f)$ provides a certificate limitedly to pure \ac{ne}. Thus, one might think that the result of the previous theorem is simply an artifact due to this observation and to the fact that $\nashe{G} \subseteq \cce{G}$. This is not the case and the limitations of the smoothness framework are structural. Indeed, it can be shown that $\fgar$ has the same price of anarchy of \eqref{eq:poagairing} even in the larger set of \ac{cce}.\footnote{While \cite[Thm. 3]{gairing2009covering} provides a proof limitedly to mixed Nash equilibria, it is not difficult to extend such proof to \ac{cce}.}
\end{remark}

The quantity $b(n)$ bounding the best possible performance certificate offered by the smoothness framework, and the guarantee offered by the ``true'' price of anarchy for $\fgar$ are presented in \cref{fig:smoothness} (left). Additionally, the distribution rules $\fgar(j)$ and $1/j$ are depicted in \cref{fig:smoothness} (right).
\begin{figure}[!h]
    \centering
    \setlength\figureheight{6.2cm} 
	\setlength\figurewidth{6.2cm} 
%
%
\definecolor{mycolor1}{rgb}{0.85000,0.32500,0.09800}%
\begin{tikzpicture}
\begin{axis}[%
width=\figurewidth,
height=\figureheight,
scale only axis,
xmin=1,
xmax=20,
xtick={1,5,10,15,20},
xlabel={$n$\textcolor{white}{j}},
ymin=0.48,
ymax=1,
ytick={.5,.6,.7,.8,.9,1},
ylabel style = {rotate=-90},
]
\addplot [dashed, line width=0.7pt]
  table[row sep=crcr]{%
1   1.000000000000000\\
2   0.666666666666667\\
3   0.600000000000000\\
4   0.571428571428571\\
5   0.555555555555556\\
6   0.545454545454546\\
7   0.538461538461538\\
8   0.533333333333333\\
9   0.529411764705882\\
10   0.526315789473684\\
11   0.523809523809524\\
12   0.521739130434783\\
13   0.520000000000000\\
14   0.518518518518518\\
15   0.517241379310345\\
16   0.516129032258065\\
17   0.515151515151515\\
18   0.514285714285714\\
19   0.513513513513513\\
20   0.512820512820513\\
};
\addlegendentry{\footnotesize $b(n)$};
\addplot [color = blue, line width=0.7pt]
  table[row sep=crcr]{
1   1.000000000000000\\
2   0.666666666666667\\
3   0.636363636363636\\
4   0.632653061224490\\
5   0.632183908045977\\
6   0.632127529123237\\
7   0.632121263731585\\
8   0.632120624393906\\
9   0.632120564455885\\
10   0.632120559276055\\
11   0.632120558861669\\
12   0.632120558830847\\
13   0.632120558828706\\
14   0.632120558828567\\
15   0.632120558828558\\
16   0.632120558828558\\
17   0.632120558828558\\
18   0.632120558828558\\
19   0.632120558828558\\
20   0.632120558828558\\
   };
\addlegendentry{\footnotesize $\poa(\fgar)$};
%
%
\end{axis}
\end{tikzpicture}%
%
%
\definecolor{mycolor1}{rgb}{0.85000,0.32500,0.09800}%
\begin{tikzpicture}
\begin{axis}[%
width=\figurewidth,
height=\figureheight,
scale only axis,
xmin=1,
xmax=10,
xtick={1,2,3,4,5,6,7,8,9,10},
xlabel={$j$\textcolor{white}{n}},
ymin=0,
ymax=1,
ylabel style = {rotate=-90},
]
\addplot [only marks,color=black,line width=1.0pt,mark=x,mark options={solid}, mark size=3pt]
  table[row sep=crcr]{
1   1.000000000000000\\
2   0.418023294250603\\
3   0.254069882751809\\
4   0.180232942506029\\
5   0.138955064274718\\
6   0.112798615624194\\
7   0.094814987995768\\
8   0.081728210220980\\
9   0.071848976018444\\
10  0.064664078416600\\
   };
\addlegendentry{\footnotesize $\fgar(j)$};
\addplot [only marks, dashed,line width=0.5pt,mark=o,mark options={solid}]
  table[row sep=crcr]{%
1  1.000000000000000\\
2   0.500000000000000\\
3   0.333333333333333\\
4   0.250000000000000\\
5   0.200000000000000\\
6   0.166666666666667\\
7   0.142857142857143\\
8   0.125000000000000\\
9   0.111111111111111\\
10   0.100000000000000\\
};
\addlegendentry{\footnotesize $1/j$};
%
%
\end{axis}
\end{tikzpicture}%
    \caption{Left: best achievable bound $b(n)$ on the price of anarchy using a smoothness argument, and actual price of anarchy $\poa(\fgar)$ for the distribution $\fgar$ in \eqref{eq:fgar}. Right: distribution rule $\fgar$ and $1/j$ for $n=10$.}
    \label{fig:smoothness}
  \end{figure}
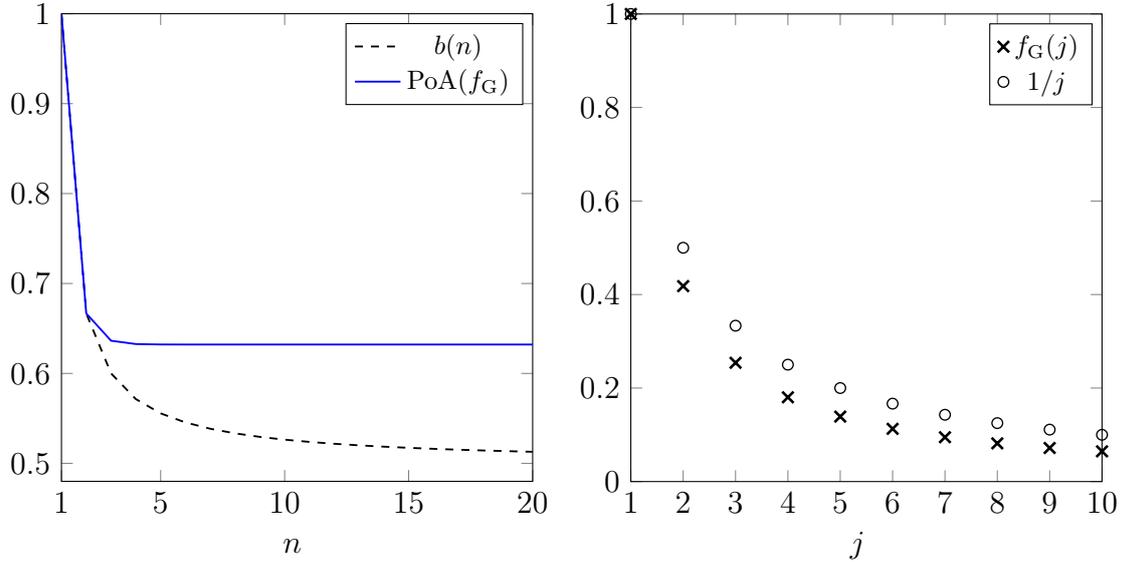
The gap between $b(n)$ and $\poa(\fgar)$ is significant: for a system with, e.g., $n=20$ agents, $\poa(\fgar)$ produces a performance certificate that is at least $25\%$ higher than what $\poas(\fgar)$ can offer.
Thus, the smoothness framework is not the right tool to study the utility design problems considered here. First, it restricts the set of admissible distribution rule to $f(j)\le 1/j$. Second, even for distribution rules satisfying this assumption, it provides performance certificates that are too conservative.
Finally, we observe that the notion of local smoothness (a refinement of the original notion introduced in \cite{roughgarden2015local}) will not be useful here in improving $\poas$.

\section{A tight price of anarchy}
\label{sec:tightpoa}
In the previous section we have highlighted the limitations of the smoothness framework when applied to utility design problems.
In this section we propose a novel approach for the exact characterization of $\poa(f)$ as defined in \eqref{eq:poadef}. More precisely, we reformulate the problem of \emph{computing} the price of anarchy as a \emph{tractable} \ac{lp} involving the components of $w$ and of $f$ (\cref{thm:primalpoa,,thm:dualpoa}). This section is dedicated to the problem of characterizing the price of anarchy in its full generality, while in \cref{ch:p2subsupercov} we specialize the results to a class of submodular and supermodular problems.

In all the forthcoming analysis we make the following regularity assumptions on admissible welfare basis functions and distribution rules. 
\begin{namedtheorem}[\hypertarget{sas}{Standing Assumptions}]
{The sets $\mc{A}_i\subseteq 2^\mc{R}$ are nonempty and $\mc{A}_i\setminus\emptyset \neq \emptyset$ for all $i\in\N$.
Further, $\exists r\in\mc{R}$ s.t. $v_r>0$ and $r\in a_i\in\mc{A}_i$ for some $i\in\N$.}
The welfare basis function $w:[n]\rightarrow \mb{R}_{> 0}$ satisfies $w(j)>0$ for all $j\in[n]$.
A distribution rule $f:[n]\rightarrow \mb{R}_{\ge0}$ satisfies $f(1)\ge1$, $f(j)\ge 0$ for all $j\in[n]$. The latter is equivalent to $f\in\mc{F}$,  with
\[
\mc{F}\coloneqq\left\{f:[n]\rightarrow \mb{R}_{\ge0}~\text{s.t.}~f(1)\ge1,~f(j)\ge 0~\forall j\in[n]\right\}.
\] 
\end{namedtheorem}

{The non-emptiness of $\mc{A}_i$ ensures feasibility of the welfare maximization introduced in \cref{sec:problemstatement}.
The assumptions $\mc{A}_i\setminus\emptyset \neq \emptyset$ for all $i\in\N$ and $\exists r \in\mc{R}$ s.t. $v_r>0$ ensure that the problem is non degenerate, in that every agent has the possibility to select at least one resource, and not all the resources have a value of zero.
Finally, observe that the assumption $f(1)\ge1$ is without loss of generality for all distributions with $f(1)>0$. Indeed, If $f(1) \neq 1$, but $f(1) > 0$, it is possible to scale the value of the resources and reduce to the case $f(1)=1$.}

\subsection{Primal formulation}
\subsubsection*{An informal introduction}
\label{subsec:informal}
 While \eqref{eq:poadef} corresponds to the {\it definition} of price of anarchy, it also describes a (seemingly difficult) \emph{optimization problem}. 
 The aim of this section, is to transform the definition of price of anarchy into a finite dimensional \ac{lp} that can be efficiently solved. Towards this goal, we provide here an informal introduction (based on four steps), as we believe the reader will benefit from it. The formal derivation and justification of each of these steps is \emph{postponed to \cref{thm:primalpoa}} and its proof.\\[\vspacesteps]
  \noindent 
{\bf \emph{Step 1:}} we observe that the price of anarchy computed over the family of games $G\in\mc{G}_f$ is the same of the price of anarchy over the reduced family of games $\hat{\mc{G}}_f$, where the feasible set of every player only contains two allocations: (worst) equilibrium and optimal allocation, that is $\hat{\mc{A}}_i=\{\ae_i,\aopt_i\}$, $\hat{\mc{G}}_f\coloneqq\{(\mc{R},\{v_r\},N,\{\hat{\mc{A}}_i\},f)\}$. Thus, definition \eqref{eq:poadef} reduces to
 \[
 \begin{split}
\poa(f)=&\inf_{G\in \hat{\mc{G}}_f}\biggl(\frac{W(\ae)}
{W(\aopt)}
\biggr)\,,\\
&\quad \text{s.t.}\quad  u_i(\ae)\ge u_i(\aopt_i,\ae_{-i})\quad \forall i\in\N\,,
\end{split}
 \]
 where we have constrained $\ae$ to be an equilibrium. We do not include the additional constraints requiring $\ae$ to be the \emph{worst} equilibrium and $\aopt$ to provide the \emph{highest} welfare. This is because the infimum over $\hat{\mc{G}}_f$ and the parametrization we will introduce to describe an instance $G$ (in step 4) will implicitly
 ensure this. 
\vspace*{\vspacesteps}

\noindent  {\bf \emph{Step 2:}} we assume without loss of generality that $W(\ae)=1$ and get 
\be
\label{eq:poastep2informal}
\begin{split}
\poa(f)=&\inf_{G\in \hat{\mc{G}}_f}\frac{1}{W(\aopt)}\,,\\
&\quad \text{s.t.}\quad  u_i(\ae)\ge u_i(\aopt_i,\ae_{-i})\quad \forall i\in\N\,,\\
&\qquad\quad W(\ae)=1\,.
\end{split}
\ee

\noindent{\bf \emph{Step 3:}} we relax the previous program as in the following 
 \be
 \label{eq:informalpoa}
\begin{split}
\poa(f)=&\inf_{G\in \hat{\mc{G}}_f}\frac{1}{W(\aopt)}\,,\\
&\quad \text{s.t.}\quad  \sum_{i\in\N}u_i(\ae)- u_i(\aopt_i,\ae_{-i})\ge0\,,\\
&\qquad~\quad W(\ae)=1\,,
\end{split}
\ee 
where the $n$ equilibrium constraints (one per each player) have been substituted by their sum. We show that the relaxation gives the same price of anarchy of \eqref{eq:poastep2informal}.
\vspace*{\vspacesteps}\\
{\bf \emph{Step 4:}} for a given instance  in the reduced family $\hat{\mc{G}}_f$, computing the efficiency amounts to identifying an optimal allocation and the corresponding worst Nash equilibrium. The additional difficulty appearing in \eqref{eq:informalpoa} is in how to describe a generic instance $G\in\hat{\mc{G}}_f$ and on how to compute the infimum over all such (infinite) instances.
To do so, we introduce an efficient parametrization that fully describes the objective function and the decision variables of the previous problem. This allows to reduce \eqref{eq:informalpoa} and obtain the result in the following \cref{thm:primalpoa}.
\subsubsection*{The linear program}
\label{subsec:primallp}
The following theorem makes the reasoning presented in \cref{subsec:informal} formal and constitutes the second result of this manuscript.

In order to capture all the instances in $\hat{\mc{G}}_f$, we use a parametrization inspired by \cite{ward2012oblivious} and introduce the variables $\theta(a,x,b)\in\mb{R}$ defined for any tuple of integers $(a,x,b)\in\mc{I}$, where 
\[
	\begin{split}
		\mc{I}&\coloneqq \{(a,x,b)\in\mb{N}_{\ge0}^3~\text{s.t.}~1\le a+x+b\le n\}\,,\\
		\Ir &\coloneqq\{(a,x,b)\in\mc{I}~\text{s.t.}~a\cdotshort x\cdotshort b=0~\text{or}~a+x+b=n\}\,.
	\end{split}
\]
Note that $\Ir$ contains all the integer points on the planes $a=0$, $b=0$, $x=0$, $a+x+b=n$ bounding $\mc{I}$. The set $\mc{I}$ is depicted in \cref{fig:pyramid_setI_cut} for the case of $n=3$.

\begin{figure}[h!]
\centering
\includegraphics[scale=1.5]{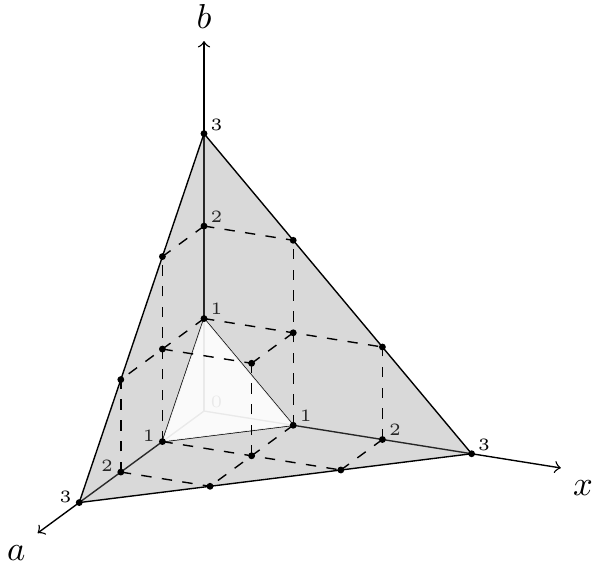}
\caption{The black circles represent all the points belonging to $\mc{I}$, $n=3$.}
\label{fig:pyramid_setI_cut} 
\end{figure}

In the remainder we write $\sum_{a,x,b}$ instead of $\sum_{(a,x,b)\in\mc{I}}$, for readability. 
Additionally, given a distribution rule $f:[n]\rightarrow \mb{R}_{\ge 0}$, and a welfare basis function $w:[n]\rightarrow \mb{R}_{>0}$, we extend their definition, with slight abuse of notation, to $f:\intwithzero{n+1}\rightarrow \mb{R}_{\ge 0}$ and $w:\intwithzero{n+1}\rightarrow \mb{R}_{\ge 0}$, where we set the first and last components to be identically zero, i.e., $f(0)=w(0)=0$, $f(n+1)=w(n+1)=0$.\footnote{This adjustment does not play any role, but is required to avoid the use of cumbersome notation in the forthcoming expressions. Else, e.g., $f(a+x+1)$ and $w(a+x+1)$ in \eqref{eq:primalvalue} will not be defined for $a+x=n$.}

\begin{theorem}[$\poa$ as a linear program]
\label{thm:primalpoa}
Given $f\in\mc{F}$, the price of anarchy \eqref{eq:poadef} is 
\[
\poa(f) = \frac{1}{W^\star}\,,
\]	
where $W^\star$ is the value of the following (primal) linear program in the unknowns $\theta(a,x,b)\in\mathbb{R}_{\ge0}$, $(a,x,b)\in\mc{I}$
\be
\begin{split}
\label{eq:primalvalue}
W^\star &= \max_{\theta(a,x,b)}\sum_{a,x,b}\ones[\{b+x\ge 1\}]w(b+x)\theta(a,x,b)\\
\text{s.t.}&\sum_{a,x,b}[af(a+x)w(a+x)-bf(a+x+1)w(a+x+1)]\theta(a,x,b)\ge0 \\
		   & \sum_{a,x,b}\ones[\{a+x\ge 1\}]w(a+x)\theta(a,x,b)=1\\
		   & \theta(a,x,b)\ge 0\quad\forall (a,x,b)\in\mc{I}\,.
\end{split}
\ee
\end{theorem}
The proof is based on the four steps previously discussed.

Given a distribution rule $f$, the solution to the previous program returns both the price of anarchy, and the corresponding worst case instance (encoded in $\theta(a,x,b)$, see the proof of the Step 4 in \cref{sec:proofsp2-part1}). 
Observe that the number of decision variables in \eqref{eq:primalvalue} is $|\mc{I}|={\frac{1}{2}\sum_{j=0}^n(j+2)(j+1)}-1=n(n+1)(n+2)/6-1\sim\mathcal{O}(n^3)$, while only two scalar constraints are present (neglecting the positivity constraint). The previous program can thus already be solved efficiently.  Nevertheless, we are only interested in the expression of $\poa(f)$ (i.e., ultimately in the \emph{value} of the program), and therefore consider the dual counterpart of \eqref{eq:primalvalue}  in the following.
\subsection{Dual formulation}
\label{subsec:duallp}
Thanks to strong duality, it suffices to solve the dual program of \eqref{eq:primalvalue} to compute the price of anarchy \eqref{eq:poadef}. While the dual program should feature two scalar decision variables and $\mathcal{O}(n^3)$ constraints, the following theorem shows how to reduce the number of constraints to only $|\Ir|=
{2(n^2+1)-1}\sim\mathcal{O}(n^2)$. The overarching goal is to progress towards an \emph{explicit expression} for $\poa(f)$.

\begin{theorem}[Dual reformulation of $\poa$]
\label{thm:dualpoa}
Given $f\in\mc{F}$, the price of anarchy \eqref{eq:poadef} is 
$
\poa(f) = 1/{W^\star}\,,
$	
where $W^\star$ is the value of the following (dual) program
\be
\begin{split}
W^\star &= \min_{\lambda\in\mb{R}_{\ge0},\,\mu\in\mb{R}}~ \mu \\[0.1cm]
&\,\text{s.t.} ~~\ones[\{b+x\ge1\}]w(b+x)
- \mu \ones[\{a+x\ge 1\}]w(a+x)+\\
&\quad~~+ \lambda[af(a+x)w(a+x)-bf(a+x+1)w(a+x+1)]
\le 0\\[0.1cm]
& \hspace*{60mm}\forall (a,x,b)\in\Ir
\label{eq:generalbound}
\end{split}
\ee
\end{theorem}
The proof of the previous theorem (see \cref{sec:proofsp2-part1}) suggests that a further simplification can be made when $f(j)w(j)$ is non-increasing for all $j$. In this case the number of constraints reduces to exactly $n^2$, as detailed in the following corollary.
\begin{corollary}
\label{cor:fwnonincreasingpoadual}
	\noindent Consider a given $f\in\mc{F}$.
	\begin{enumerate}
		\item Assume $f(j)w(j)$ non increasing for $j\in[n]$. Then $\poa=1/W^\star$, where
	\be
	\label{eq:formulacor1}
	\begin{split}
	W^\star &= \min_{\lambda\in\mb{R}_{\ge0},\,\mu\in\mb{R}}~ \mu  \\[0.1cm]
	&\,\text{s.t.}~ \,\mu w(j)\ge w(l)+\lambda [j f(j) w(j)-l f(j+1) w(j+1)]\\
	&~\hspace*{42mm} \forall j,l\in[0, n], \quad 1\le j+l\le n,\\[0.05cm]
	&\qquad \mu w(j)\ge w(l)+\lambda [(n-l) f(j) w(j)-(n-j) f(j+1) w(j+1)]\\
	&~\hspace*{42mm} \forall j,l\in[0, n], \quad~~~~~ j+l> n.\\
	\end{split}
	\ee
	\item If additionally $f(j)\ge \frac{1}{j}f(1)w(1) \min_{l\in[n]} \frac{l}{w(l)}$, then 
	\[
	\lambda^\star=\max_{l\in[n]}\frac{w(l)}{l}\frac{1}{f(1)w(1)}\,.\]
	\end{enumerate}	
\end{corollary}
Mimicking the proof of the previous corollary, it is possible to obtain a similar result when $f(j)w(j)$ is instead non-decreasing. While the requirements on $f(j)w(j)$ being non increasing might seem restrictive at first, similar assumptions were made relative to a simpler class of problems in \cite{marden2014generalized,gairing2009covering}. We remark that this requirement is added to obtain an explicit expression for the price of anarchy. If this is not the goal, one can compute $\poa(f)$ using \cref{thm:dualpoa} without imposing any additional assumption.

\begin{remark}[Explicit expression of $\poa(f)$]
Observe that, if the optimal value $\lambda^\star$ is known a priori, as in the second statement from the previous corollary, the quantity $W^\star$ (and consequently the price of anarchy) can be computed \emph{explicitly} from \eqref{eq:formulacor1} as the maximum between $n^2$ real numbers depending on all the entries of $f$ and $w$. To see this, divide both sides of the constraints in \eqref{eq:formulacor1} by $w(j)$ for $1\le j\le n$, and observe that the solution $\mu^\star$ is then found as the maximum of the resulting right hand side. The corresponding value of $W^\star$ is given by the following expression. 
\begin{equation}
\label{eq:mustar} 	
W^\star =\max
{
\begin{cases}
	\medmath{\max_{\substack{j\neq 0 \\[0.5mm]1\le j+l\le n \\[0.5mm] j,l\in[0,n]}}} 
	{\frac{w(l)}{w(j)}+\lambda^\star[jf(j)-lf(j+1)\frac{w(j+1)}{w(j)}]}\\\\
	\medmath{\max_{\substack{j\neq 0 \\[0.5mm] j+l> n \\[0.5mm] j,l\in[0,n]}}}
	{\frac{w(l)}{w(j)}+\lambda^\star[(n-l)f(j)-(n-j)f(j+1)\frac{w(j+1)}{w(j)}]}
\end{cases}}
\end{equation}
Equation \eqref{eq:mustar} is reminiscent of the result obtained using a very different approach in \cite[Thm. 6]{marden2014generalized} (limited to Shapley value) and \cite[Thm. 3]{gairing2009covering} (limited to set covering problems and sub budget-balanced utilities). 

Finally, observe that for the case of \ac{mmc} problems discussed in \cref{sec:smooth} (it is $w(j)=1$ for all $j\in[n]$) the assumption required in the first statement of the previous corollary reduces to $f(j)$ non increasing. That is, the previous corollary gives us an expression for the $\poa(f)$ also for utilities that do not satisfy $\sum_{i\in\N}u_i(a)\le W(a)$, as instead required in \cite{gairing2009covering}. We discuss further connections with these works and others in \cref{ch:p2subsupercov}.
\end{remark}

\subsection{Related works}
The idea of using an auxiliary linear program to study the equilibrium efficiency has appeared in few works in the literature \cite{nadav2010limits,bilo2012unifying,kulkarni2014robust,thang2017game}. 
Note that, \emph{all} the aforementioned works assume the budget-balance condition to hold true.
In \cite{nadav2010limits}, the authors pose the problem in an abstract form and the corresponding linear program is used as a conceptual tool, rather than as a machinery to explicitly compute the price of anarchy. While \cite{bilo2012unifying} provides result for polynomial latency functions in weighted congestion games, the techniques proposed in \cite{bilo2012unifying,kulkarni2014robust,thang2017game} require an ad-hoc bound on the dual objective to obtain a bound on the price of anarchy. This is not the case with our approach.
Additionally, we note that the linear programming reformulations of \cite{nadav2010limits} capture the price of anarchy for a \emph{given} problem instance, while in this work we consider the \emph{worst case} instance over an admissible class of problems. This additional requirement complicates the analysis, but will produce algorithms that are provably robust to the presence of uncertainty, and are thus better suited for engineering implementation.
Finally, we observe that a direct transposition of the approach in, e.g., \cite{nadav2010limits} to our setting would produce a linear program whose size grows exponentially in the number of resources, making it impossible to solve for real world applications.

\section{Optimal utility design via linear programming}
\label{sec:optimalutilitydesign}
Given $w$ and a distribution rule $f$, \cref{thm:dualpoa} and \cref{cor:fwnonincreasingpoadual} have reduced the computation of the price of anarchy to the solution of a tractable linear program. Nevertheless, determining the distribution rule maximizing $\poa(f)$, i.e., giving the best performance guarantees, is \emph{also a tractable linear program}. The following theorem makes this clear.
\begin{theorem}[Optimizing $\poa(f)$ is a linear program]
	\label{thm:optimizepoa}
	For a given welfare basis $w$, the design problem
	\[
	\argmax_{f\in\mc{F}} \poa(f)
	\]
	is equivalent to the following \ac{lp} in $n+1$ scalar unknowns 
	\be
\label{eq:optf}
	\begin{split}
	f^\star &\in \argmin_{{f\in\mc{F},\,\mu\in\mb{R}}}
	~ \mu \\[0.1cm]
	&\,\text{s.t.} ~~\ones[\{b+x\ge1\}]w(b+x)
	- \mu \ones[\{a+x\ge 1\}]w(a+x)+\\
	&\quad~~+af(a+x)w(a+x)-bf(a+x+1)w(a+x+1)
	\le 0\\[0.1cm]
& \hspace*{62mm}\forall (a,x,b)\in\Ir
	\end{split}
\ee
The corresponding optimal price of anarchy is
\[
\poa(f^\star)=\frac{1}{\mu^\star}\,,
\]
where $\mu^\star$ is the value of the program \eqref{eq:optf}.
\end{theorem}
\begin{remark}
The importance of this results stems from its applicability for the \emph{game design procedure} outlined in \cref{ch:p2introduction}. More precisely, the previous theorem provides a solution to the utility design problem introduced in \cref{problem:utilitydesign}.  As a matter of fact, \cref{thm:optimizepoa} allows to compute the optimal distribution rule, for any given welfare basis function (satisfying the \hyperlink{sas}{Standing Assumptions}), and thus to solve the utility design problem. Applications of these results are presented in \cref{ch:p2subsupercov,,ch:p2applications}.
\end{remark}
\section{Appendix}
\label{sec:proofsp2-part1}
\subsection{Proofs of the results presented in \cref{sec:smooth}}
\subsubsection*{Proof of \cref{thm:smoothnottight}}
\begin{proof}
We prove the first and third claims only, as the second {statement is shown in \cite[Thm. 3]{gairing2009covering}.}
\begin{enumerate}
\item The claim in \cref{prop:poasmooth} requires $f(j)\le 1/j$, so that we need to restrict to this class of admissible utility functions to apply any smoothness argument. 
We proceed dividing the proof in two parts. First, we consider the valid distribution rule $\fsv$ defined for all $j\in [n]$ as $\fsv(j)\coloneqq 1/j$, and show that the best smoothness parameters are $(1,1-1/n)$ so that 
\be
\poas(\fsv)=\frac{1}{2-1/n}= b(n).
\ee
Second, we show that for any distribution with $f(j)\le \fsv(j)$ for all $j\in[n]$ it holds $\poas(f)\le \poas(\fsv)$.
From this, we conclude $\poas(f)\le b(n) = \frac{1}{2-1/n}$ for all admissible distribution rules.

Part 1: with the special choice of $\fsv$, the proof of \cite[Thm. 2]{gairing2009covering} shows that for any pair of feasible $a, a'$ and any $G\in\mc{G}_f$, it holds
\[
\sum_{i\in\N} u_i(a_i',a_{-i})\ge W(a')-\rchi_{\rm SV} W(a)\,,\]
\[\text{where}~~~\rchi_{\rm SV}=\max_{j\in[n-1]}\{j\fsv(j)-\fsv(j+1), (n-1)\fsv(n)\}\,,\]
from which $\rchi_{\rm SV}= 1 - 1/n$. Thus the game is $(1,1-1/n)$-smooth and it follows that $\poas(\fsv)\ge \frac{1}{2-1/n}$.
To show that there is no better pair $(\lambda,\mu)$ we show that the price of anarchy is exactly $\frac{1}{2-1/n}$. To do so, we consider the instance $G$ proposed in \cite[Fig. S2]{paccagnan2017arxiv} and observe that $W(\aopt) = 2-1/n$ while $W(\ae)=1$. Thus, $\poas(\fsv)\le\poa(\fsv)\le \frac{1}{2-1/n}$. 
Since the lower and upper bounds obtained for $\poas(\fsv)$ match, we conclude that $\poas(\fsv)=\frac{1}{2-1/n}$.

Part 2: Consider any distribution rule such that $f(j)\le \fsv(j)$ for all $j\in[n]$. Let us define the set
\[\begin{split}
A(f)\coloneqq\biggl\{&(\lambda,\mu)~\text{s.t. for all $a,a'\in \mc{A}$, for all $G\in\mc{G}_f$}\\
&\sum_{\substack{i\in N\\[0.5mm] r\in a_i}} v_r f(|(a'_i,a_{-i})|_r)w(|(a'_i,a_{-i})|_r) \ge  \lambda W(a')-\mu W(a)
\biggr\},
\end{split}
\] 
and analogously for $A(\fsv)$. With this notation, the claim we intend to prove reduces to
\begin{equation}
\sup_{(\lambda,\mu)\in A(f)} \frac{\lambda}{1+\mu}\le \sup_{(\lambda,\mu)\in A(\fsv)} \frac{\lambda}{1+\mu}\,.
\label{eq:prooflambdamu}
\end{equation}
To show the latter, we prove that $A(f)\subseteq A(\fsv)$. Consider a feasible tuple $(\lambda,\mu)\in A(f)$; by definition of $A(f)$ it is
\[
\sum_{\substack{i\in N\\[0.5mm] r\in a_i}} v_r f(|(a'_i,a_{-i})|_r)w(|(a'_i,a_{-i})|_r) \ge  \lambda W(a')-\mu W(a),
\]
$\forall a,a'\in \mc{A}$, $G\in\mc{G}_f$.
Since $\fsv(j)\ge f(j)$, it follows that
\[
\sum_{\substack{i\in N\\[0.5mm] r\in a_i}} v_r \fsv(|(a'_i,a_{-i})|_r)w(|(a'_i,a_{-i})|_r) \ge \lambda W(a')-\mu W(a),
\]
$\forall a,a'\in \mc{A}$, $G\in\mc{G}_f$.
Thus $(\lambda,\mu)\in A(\fsv)$ too, from which we conclude that $A(f)\subseteq A(\fsv)$ and \eqref{eq:prooflambdamu} must hold.

\setcounter{enumi}{2}
\item Follows from the previous claims upon noticing that $b(n)<\poa(\fgar)$ for $n>2$ (while $b(n)=\poa(\fgar)$ for $n=2$).
\end{enumerate}
\end{proof}
\subsection{Proofs of the results presented in \cref{sec:tightpoa}}
\subsubsection*{Proof of \cref{thm:primalpoa}}
\begin{proof}
The proof formalizes the steps introduced in \cref{subsec:informal}.\\
\begin{itemize}[leftmargin=23mm]
\item[\bf Step 1:] We intend to show that the price of anarchy computed over $G\in\mc{G}_f$ is the same of the price of anarchy computed over a reduced set of games.
Consider a game $G\in\mc{G}_f$ and denote with $\ae$ the corresponding  worst equilibrium (as measured by $W$) and with $\aopt$ an optimal allocation of $G$. 
For every such game $G$, we construct a new game $\hat G$, where $\hat G\coloneqq (\mc{R},\{v_r\},N,\{\hat{\mc{A}}_i\},f)$ and $\hat{\mc{A}}_i=\{\ae_i,\aopt_i\}$ for all $i\in\N$. That is, the feasible set of every player in $\hat G$ contains only two allocations: an optimal allocation, and the (worst) equilibrium of $G$. With slight abuse of notation we write $\hat G(G)$ to describe the game $\hat G$ constructed from $G$ as just discussed.
Observe that $G$ and $\hat G$ have the same price of anarchy, i.e.,
\[\frac{\min_{a\in \nashe{G}} W(a)}{\max_{a\in\mc{A}} W(a)} = 
\frac{\min_{a\in \nashe{\hat G}} W(a)}{\max_{a\in\hat {\mc{A}}} W(a)}\,.
\] 
Denote with $\hat{\mc{G}}_f$ the class of games $
\hat{\mc{G}}_f\coloneqq\{\hat G (G)~\forall G\in\mc{G}_f\}\,.$ Observe that $\hat{\mc{G}}_f\subseteq \mc{G}_f$ (by definition) and since for every game $G\in \mc{G}_f$, it is possible to construct a game $\hat G \in \hat{\mc{G}}_f$ with the same price of anarchy, it follows that \eqref{eq:poadef} can be computed as
\[
\poa(f)  =\inf_{\hat G\in \hat{\mc{G}}_f}\biggl(\frac{\min_{a\in \nashe{\hat G}} W(a)}{\max_{a\in\mc{A}} W(a)}\biggr)\,.
\]

\item[\bf Step 2:] 
 \Cref{lem:lemmapositivewelfare_ateq} ensures for any game $G$, every equilibrium configuration has strictly positive welfare. Thus, we assume without loss of generality that $W(\ae)=1$, where $\ae$ represents the worst equilibrium of $G$.\footnote{If, for a given game $G$, this is not the case, it is possible to construct a new game (by simply rescaling the value of the resources) such that $W(\ae)=1$. Note that the new game has the same game price of anarchy of $G$.} The price of anarchy reduces to
\[
\begin{split}
\poa(f)=&\inf_{\hat G\in \hat{\mc{G}}_f}\frac{1}{W(\aopt)}\,,\\
&\quad \text{s.t.}\quad  u_i(\ae)\ge u_i(\aopt_i,\ae_{-i})\quad \forall i\in\N\,,\\
&\qquad\quad~ W(\ae)=1\,.
\end{split}
\]
\item[\bf Steps 3, 4:] While in \cref{subsec:informal} these steps have been introduced separately for ease of exposition, their proof is presented jointly here. 
First observe, from the last equation, that $\poa(f)=1/W^\star$, where
\be
\label{eq:originalproofprimal}
\begin{split}
W^\star\coloneqq&\sup_{\hat G\in \hat{\mc{G}}_f}{W(\aopt)}\,,\\
&\quad \text{s.t.}\quad  u_i(\ae)\ge u_i(\aopt_i,\ae_{-i})\quad \forall i\in\N\,,\\
&\qquad\quad~ W(\ae)=1\,.
\end{split}
\ee
We relax the previous program as in the following
\be
\label{eq:relaxedproofprimal}
\begin{split}
V^\star\coloneqq &\sup_{\hat G\in \hat{\mc{G}}_f}{W(\aopt)}\,,\\
&\quad \text{s.t.}\quad  \sum_{i\in\N}u_i(\ae)- u_i(\aopt_i,\ae_{-i})\ge0\,,\\
&\qquad\quad~ W(\ae)=1\,,
\end{split}
\ee
where the $n$ equilibrium constraints (one per each player) have been substituted by their sum. Thus, $V^\star \ge W^\star$, but it also holds $V^\star\le W^\star$
as \cref{lemma:relaxedmatchesoriginal} proves, so that $V^\star = W^\star$.

In the following we show how to transform \eqref{eq:relaxedproofprimal} in \eqref{eq:primalvalue} by introducing the variables $\theta(a,x,b)$, $(a,x,b)\in\mc{I}$. This parametrization has been introduced to study covering problems in \cite{ward2012oblivious}, and will be used here to efficiently represent the quantities appearing in \eqref{eq:relaxedproofprimal}.
To begin with, recall that each feasible set is composed of only two allocations, that is $\hat{\mc{A}}_i=\{\ae_i,\aopt_i\}$.
For any given triplet $(a,x,b)$ in $\mc{I}$, we thus define $\theta(a,x,b)\in\mb{R}_{\ge0}$ as the total value of resources that belong to precisely $a+x$ of the sets $\ae_i$, $b+x$ of the sets $\aopt_j$, for which exactly $x$ sets have the same index (i.e., $i=j$). These $\mc{O}(n^3)$ variables suffice to fully describe the terms appearing in \eqref{eq:relaxedproofprimal}. Indeed, extending the formulation of \cite{ward2012oblivious} to the welfare defined in \eqref{eq:welfaredef} and the utilities defined in \eqref{eq:utilities}, we can write
\[
\begin{split}
W(\aopt) &=\sum_{(a,x,b)\in\mc{I}}\ones[\{b+x\ge 1\}]w(b+x)\theta(a,x,b)\,, \\
W(\ae) &= \sum_{(a,x,b)\in\mc{I}}\ones[\{a+x\ge 1\}]w(a+x)\theta(a,x,b)\,.\\
\end{split}
\]
The relaxed equilibrium constraint \[\sum_{i\in\N}u_i(\ae)- u_i(\aopt_i,\ae_{-i})\ge0\] reduces to
\[
\begin{split}
&~~\sum_{i\in\N}u_i(\ae)- u_i(\aopt_i,\ae_{-i})\\
=&\sum_{(a,x,b)\in\mc{I}}        [(a + x)f(a + x)w(a + x)  -  bf(a + x + 1)w(a + x + 1)  \\
&\hspace*{15mm}-  xf(a + x)w(a + x)]\theta(a,x,b) \\
 =&       \sum_{(a,x,b)\in\mc{I}}        [af(a + x)w(a + x) - bf(a + x + 1)w(a + x + 1)]\theta(a,x,b)\ge0\,.
\end{split}
\]
Substituting the latter expressions in \eqref{eq:relaxedproofprimal}, one gets 
\[
\begin{split}
W^\star &= \sup_{\theta(a,x,b)}\sum_{a,x,b}\ones[\{b+x\ge 1\}]w(b+x)\theta(a,x,b)\\
\text{s.t.}& \sum_{a,x,b}[af(a + x)w(a + x) - bf(a + x + 1)w(a + x + 1)]\theta(a,x,b) \ge 0 \\
		   &  \sum_{a,x,b}\ones[\{a+x\ge 1\}]w(a+x)\theta(a,x,b)=1\\
		   & \theta(a,x,b)\ge 0\quad\forall (a,x,b)\in\mc{I}\,.
\end{split}
\]
To transform the latter expression in \eqref{eq:primalvalue} (i.e., the desired result) it suffices to show that the supremum is attained. To see this observe that the decision variables $\theta(a,x,b)$ live in a compact space. Indeed $\theta(a,x,b)$ are constrained to the positive orthant for all $(a,x,b)\in\mc{I}$.
Additionally, the decision variables with $a+x\neq0$ must be bounded due to the constraint $W(\ae)=1$
\[
\sum_{\substack{(a,x,b)\in\mc{I}\\a+x\ge1}}w(a+x)\theta(a,x,b)=1\,,
\]
where $w(j)\neq0$ by assumption. Finally, the decision variables left, i.e., those of the form $\theta(0,0,b)$, $b\in[n]$ are bounded due to the equilibrium constraint, which can be rewritten as
\[\begin{split}
&~~\sum_{b\in[n]}bf(1)w(1)\theta(0,0,b)\le \\
&
\sum_{\substack{(a,x,b)\in\mc{I}\\a+x\ge1}}[af(a+x)w(a+x)-bf(a+x+1)w(a+x+1)]\theta(a,x,b),
\end{split}
\]
where $f(1)w(1)\neq0$ by assumption.
\end{itemize}
\end{proof}

\begin{lemma}
\label{lem:lemmapositivewelfare_ateq}
For any game $G\in\mc{G}_f$, it holds
\[
W(\ae)>0\text{~~for all~~} \ae\in\nashe{G}\,.
\]
\end{lemma}
\begin{proof}
Let us consider a fixed game $G\in \mc{G}_f$.
By contradiction, let us assume that $W(\ae)=0$ for some $\ae\in\nashe{G}$.
It follows that all the players must have distributed themselves on resources that are either valued zero, or have selected the empty set allocation (since $w(j)>0$). Thus, their utility function must also evaluate to zero.
However, by \hyperlink{sas}{Standing Assumptions}, there exists a player $p$ and a resource $r\in a_p\in\mc{A}_{p}$ with $v_r>0$. Observe that no other player is currently selecting this resource, else $W(\ae)>0$. If player $p$ was to deviate and selected instead $a_p$, his utility would be strictly positive (since $f(1)>0$). Thus $\ae$ is not an equilibrium: a contradiction.
Repeating the same reasoning for all games $G\in \mc{G}_f$ yields the claim.
\end{proof}

\begin{lemma}
\label{lemma:relaxedmatchesoriginal}
Consider $W^\star$ and $V^\star$ defined respectively in \eqref{eq:originalproofprimal} and \eqref{eq:relaxedproofprimal}. It holds that $V^\star \le W^\star$.
\end{lemma}
\begin{proof}
	Since \eqref{eq:relaxedproofprimal} is equivalent to \eqref{eq:primalvalue} as shown in the proof of \cref{thm:primalpoa}, we will work with \eqref{eq:primalvalue} to prove $V^\star \le W^\star$. To do so, for any $\theta(a,x,b)$, $(a,x,b)\in \mc{I}$ feasible solution of \eqref{eq:primalvalue} with value $v$, we will construct an instance of game $\hat G$ satisfying the constraints of the original problem \eqref{eq:originalproofprimal} too. This allows to conclude that $V^\star\le W^\star$. To ease the notation we will use $\sum_{a,x,b}$ in place of $\sum_{(a,x,b)\in\mc{I}}$.
	
	Consider $\theta(a,x,b)$, $(a,x,b)\in \mc{I}$ a feasible point for \eqref{eq:primalvalue} with value $v$. For every $(a,x,b)\in \mc{I}$ and for each $i\in\N$ we create a resource $r(a,x,b,i)$ and assign to it the value of $\theta(a,x,b)/n$, i.e., $v_{r(a,x,b,i)}=\theta(a,x,b)/n$ $\forall i\in\N$. We then construct the game $\hat G$ by defining $\forall i \in\N$, $\hat{\mc{A}}_i=\{\ae_i,\aopt_i\}$ and assigning the resources as follows 
	\[
	\begin{split}
	\ae_i=&\cup_{j=1}^n \{r(a,x,b,j)~\text{s.t.}~a+x\ge 1+g(i,j)\}\,,\\
	\aopt_i=&\cup_{j=1}^n \{r(a,x,b,j)~\text{s.t.}~b+x\ge 1 + h(i,j)\}\,,
	\end{split}
	\]
	where 
	\[\begin{split}
	g(i,j)\coloneqq & (j-1+(n-1)(i-1)) \mod n\,,\\
	=&(j-i)\mod n\\
	h(i,j)\coloneqq & (j+(n-1)(i-1)) \mod n\\
	=&(j-i+1)\mod n
	\,.
	\end{split}
	\]
	We begin by showing $W(\ae)=1$ and $W(\aopt)=v$.
	Aside from the cumbersome definition of $g$ and $h$, it is not difficult to verify that for any fixed resource (i.e., for every fixed tuple $(a,x,b,j)$), there are exactly $a+x$ (resp. $b+x$) players selecting it while at the equilibrium (resp. optimum) allocation. It follows that
	\[\begin{split}
	W(\ae)&=\sum_{j\in[n]}\sum_{a+x>0} v_{r(a,x,b,j)}w(a+x)\\
	      &=\sum_{j\in[n]}\sum_{a+x>0} \frac{\theta(a,x,b)}{n}w(a+x)\\
	      &=\sum_{a,x,b}\ones[\{a+x\ge 1\}]w(a+x)\theta(a,x,b)=1\,,
	\end{split}
	\]
	With an identical reasoning, one shows that 
	\[
	\begin{split}
	W(\aopt)&=\sum_{j\in[n]}\sum_{b+x>0} v_{r(a,x,b,j)}w(b+x)\\
	      &=\sum_{a,x,b}\ones[\{b+x\ge 1\}]w(b+x)\theta(a,x,b)=v\,.
	\end{split}
	\]
Finally, we prove that $\ae$ is indeed an equilibrium, i.e., it satisfies 
$u_i(\ae)-u_i(\aopt_i,\ae_{-i})\ge 0$ for all $i\in\N$. Towards this goal, we recall that the game under consideration is a congestion game with potential $\varphi:\mc{A}\rightarrow\mb{R}_{\ge0}$
\[
\varphi(a)=\sum_{r\in\mc{R}}\sum_{j=1}^{|a|_r} v_r w(j)f(j)
\]
	It follows that $u_i(\ae)-u_i(\aopt_i,\ae_{-i})=\varphi(\ae)-\varphi(\aopt_i,\ae_{-i})$ and so we equivalently prove that 
	\[
	\varphi(\ae)-\varphi(\aopt_i,\ae_{-i})\ge 0\quad\forall i\in\N\,.
	\]
Thanks to the previous observation, according to which every resource $(a,x,b,j)$ is covered by exactly $a+x$ players at the equilibrium, we have
\[
\begin{split}
\varphi(\ae)=&\sum_{j\in [n]}\sum_{a,x,b}\frac{\theta(a,x,b)}{n}\sum_{j=1}^{a+x}w(j)f(j)\\
		 =&\frac{1}{n}\sum_{a,x,b}n\,{\theta(a,x,b)}\sum_{j=1}^{a+x}w(j)f(j)\,.
\end{split}
\]
Additionally, observe that there are $b$ resources selected by one extra agent and $a$ resources selected by one less agent when moving from $\ae$ to $(\aopt_i,\ae_{-i})$. The remaining resources are chosen by the same number of agents. 
It follows that
\[\begin{split}
&\varphi(\ae)-\varphi(\aopt_i,\ae_{-i})=\frac{1}{n}\sum_{a,x,b}n\,{\theta(a,x,b)}\sum_{j=1}^{a+x}w(j)f(j)\\
&~~~~~~-\frac{1}{n}\sum_{a,x,b}\,{\theta(a,x,b)}\left(
b\sum_{j=1}^{a+x+1}w(j)f(j)+a\sum_{j=1}^{a+x-1}w(j)f(j)+
(n-a-b)\sum_{j=1}^{a+x}w(j)f(j)
\right)\\
&=\frac{1}{n}\sum_{a,x,b}\theta(a,x,b)\left(a\,w(a+x)f(a+x)-b\,w(a+x+1)f(a+x+1)\right)\ge0\,,\\
\end{split}
\] 
where the inequality holds because $\theta(a,x,b)$ is assumed feasible for \eqref{eq:primalvalue}. This concludes the proof. 
\end{proof}
\subsubsection*{Proof of \cref{thm:dualpoa}}
\begin{proof}
	We divide the proof in two steps. In the first step we write the dual of the original program in \eqref{eq:primalvalue}. With the second step we show that only the constraints obtained for $(a,x,b)\in\Ir$ are binding.\\
	\noindent{\bf Step 1.} Upon stacking the decision variables $\theta(a,x,b)$ in the vector $y\in\mb{R}^{\ell}$, ${\ell=|\mc{I}|}$, and after properly defining the coefficients $c$, $d$, $e\in\mb{R}^{\ell}$, the program \eqref{eq:primalvalue} can be compactly written as 
\[
\begin{split}
W^\star = &\max_{y} \, c^\top y\\
&~\text{s.t.}\quad -e^\top y\le 0\,, \quad (\lambda)\\
&\qquad\, d^\top y -1 = 0\,, \quad (\mu)\\
&\qquad~~\,\quad -y\le 0\,. \quad \,(\nu)
\end{split}
\]
The Lagrangian function is defined for $\lambda\ge0$, $\nu\ge0$ as  $\mathcal{L}(y,\lambda,\mu,\nu)=c^\top y-\lambda(-e^\top y)-\mu(d^\top y -1)-\nu^\top(-y)=(c^\top+\lambda e^\top+\nu-\mu d^\top)y+\mu$, while the dual function reads as
\[
g(\lambda,\mu,\nu) = \mu \quad \text{if}\quad c^\top+\lambda e^\top+\nu^\top-\mu d^\top=0\,,
\]
and it is unbounded elsewhere.
Hence the dual program takes the form 
\[\begin{split}
&\min_{\lambda\in\mb{R}_\ge0,\,\mu\in\mb{R}}~ \mu \\
&~~~~~\text{s.t.}\quad c+\lambda e-\mu d\le0\,,
\end{split}
\]
which corresponds, in the original variables, to 
\be
\label{eq:proofdual}
\begin{split}
&\min_{\lambda\in\mb{R}_{\ge0},\,\mu\in\mb{R}}~ \mu \\[0.1cm]
&\,\text{s.t.} ~~\ones[\{b+x\ge1\}]w(b+x)
- \mu \ones[\{a+x\ge 1\}]w(a+x)+\\
&\quad~~~+\lambda[af(a+x)w(a+x)-bf(a+x+1)w(a+x+1)]
\le 0 \qquad \forall (a,x,b)\in\mc{I}\,.
\end{split}
\ee
By strong duality\footnote{The primal LP \eqref{eq:primalvalue} is always feasible, since $\theta(0,1,0)=1/w(1)$, $\theta(a,x,b)=0$ $\forall\,(a,x,b)\in\mc{I}\setminus(0,1,0)$ satisfies all the constraints in \eqref{eq:primalvalue}.}, the value of \eqref{eq:primalvalue} matches \eqref{eq:proofdual}.

\noindent{\bf Step 2.} In this step we show that only the constraints with $(a,x,b)\in\Ir$ are necessary in \eqref{eq:proofdual}, thus obtaining \eqref{eq:generalbound}. 

Observe that when $(a,x,b)\in\mc{I}$ and $a+x=0$, $b$ can take any value $1\le b\le n$, and these indices are already included in $\Ir$. 
Similarly for the indices $(a,x,b)\in\mc{I}$ with $b+x=0$. Thus, we focus on the remaining constraints, i.e., those with $a+x\neq 0$ and $b+x\neq 0$.
We change the coordinates from the original indices $(a,x,b)$ to $(j,x,l)$, $j\coloneqq a+x$, $l\coloneqq b+x$. 
The constraints in \eqref{eq:proofdual} now read as 
\be
\label{eq:reducetobound}
\begin{split}
\mu w(j)&\ge w(l)+\lambda[(j-x)f(j)w(j)-(l-x)f(j+1)w(j+1)]\,,\\
& =  w(l)+\lambda[jf(j)w(j)-lf(j+1)w(j+1)+{ x(f(j+1)w(j+1)-f(j)w(j))}]
\end{split}
\ee	
where $(j,x,l)\in \hat{\mc{I}}
$ and 
$
\hat{\mc{I}}=\{(j,x,l)\in \mathbb{N}_{\ge0}^3~\text{s.t.}~ 1\le j-x+l\le n,~ j\ge x,~ l\ge x,~ j,l\neq 0\}.
$
In the remaining of this proof we consider $j$ fixed, while $l,~x$ are free to move within $\hat{\mc{I}}$. This corresponds to moving the indices in the rectangular region defined by the blue and green patches in \cref{fig:pyramid_decreasing,,fig:pyramid_increasing}.
Observe that for $j=n$ it must be $l=x$ (since $-x+l\le 0$ and $l-x\ge 0$), i.e., in the original coordinates $b=0$, which represents the segment on the plane $b=0$ with $a+x=n$. These indices already belong to $\Ir$. Thus, we consider the case $j\ne n$ and divide the reasoning in two parts. 
\begin{itemize}
\item[a)]{Case of $f(j+1)w(j+1)\le f(j)w(j)$.}\\[0.1cm]
 The term $f(j+1)w(j+1)-f(j)w(j)$ is non-positive and so the most binding constraint in \eqref{eq:reducetobound} is obtained picking $x$ as small as possible. In the following we fix $l$ as well (recall that we have previously fixed $j$). This corresponds to considering points on a black line on the plane $j=$const in \cref{fig:pyramid_decreasing}). Since it must be $x\ge 0$ and $x\ge j+l-n$, for fixed $j$ and $l$ we set $x = \max\{0,j+l-n\}$. In the following we show that these constraints are already included in \eqref{eq:generalbound}.
\begin{itemize}
\item[-] If $j+l\le n$, i.e., if $a+b+2x\le n$, we set $x=0$. These indices correspond to points on the plane $x=0$, ($1\le a+b\le n$) bounding the pyramid and so they are already included in $\Ir$.
\item[-] If $j+l> n$, i.e., if $a+b+2x> n$, we set $x=j+l-n$, i.e., $a+b+x=n$. These indices correspond to points on the plane $a+b+x=n$, and so they are included in $\Ir$ too.
\end{itemize}
\begin{figure}[h!]
\centering
\includegraphics[width=0.5\textwidth]{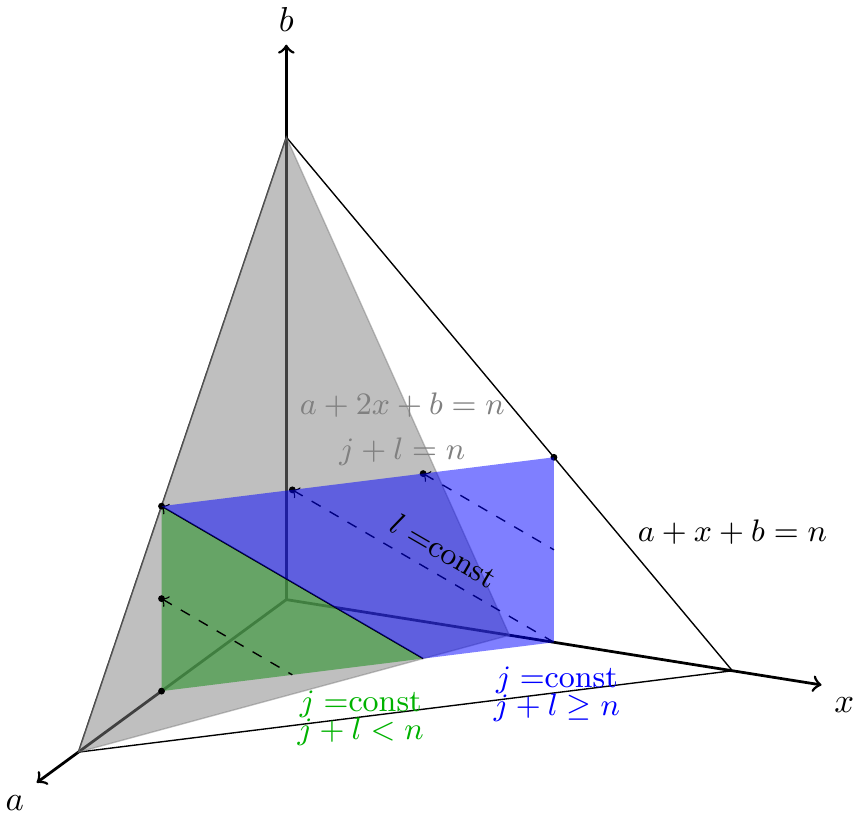}
\caption{Indices representation for case a).}
\label{fig:pyramid_decreasing}
\end{figure}

\item[ b)] {Case of $f(j+1)w(j+1)>f(j)w(j)$}. \\[0.1cm]
The term $f(j+1)w(j+1)-f(j)w(j)$ is positive and so the most binding constraint in \eqref{eq:reducetobound} is obtained picking $x$ as large as possible. In the following (after having fixed $j$) we fix $l$ as well (this means we are moving on a black line on the plane $j=$const in \cref{fig:pyramid_increasing}). Since it must be $x\le l$, $x\le j$ and $x\le j+l-1$, we set $x = \min\{j,l\}$. In the following we show that these constraints are already included in~\eqref{eq:generalbound}.
 
\begin{itemize}
\item[-] If $j\le l$, i.e., if $a\le b$, we set $x=j$, i.e., $a=0$. These indices correspond to points on the plane $a=0$, ($1\le x+b\le k$) and so they are included in $\Ir$.
\item[-] If $j>l$, i.e., if $a>b$, then we set $x=l$, i.e., $b=0$. These indices correspond to points on the plane $b=0$, ($1\le a+b\le k$) and so they are included in $\Ir$.
\end{itemize}
\begin{figure}[ht!]
\centering
\includegraphics[width=0.5\textwidth]{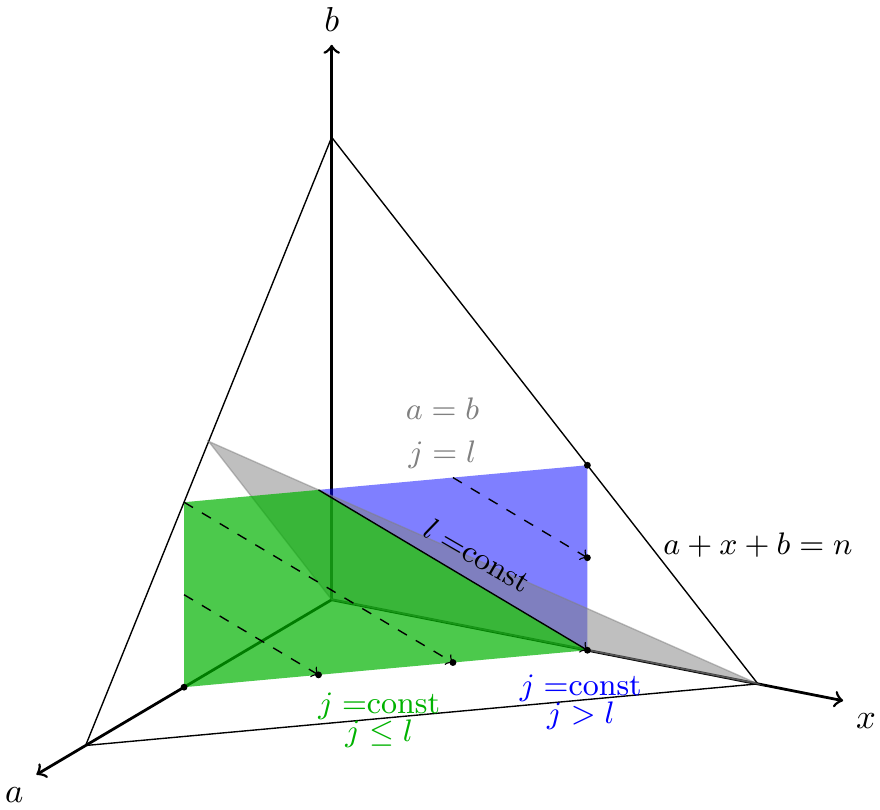}
\caption{Indices representation for case b).}
\label{fig:pyramid_increasing}
\end{figure}
\end{itemize}
\end{proof}

\subsubsection*{Proof of \cref{cor:fwnonincreasingpoadual}}
\begin{proof}
	\begin{enumerate}
	\item[]
		\item  Following the proof of \cref{thm:dualpoa} (second step, case a)),  we note that if $f(j)w(j)$ is non increasing for $j\in\N$, the only binding indices are those lying on the the two surfaces $x=0$, $1\le a+b\le n$ and $a+x+b\le n$. 
The surface $x=0$, $1\le a+b\le n$ gives 
\be
\mu w(j)\ge w(l)+\lambda [j f(j) w(j)-l f(j+1) w(j+1)]
\label{eq:part1reduced}
\ee
for $1\le j+l\le n$ and $j,l\in\intwithzero{n}$, where we used $j,l$ instead of $a,b$. The surface $a+x+b=n$ gives 
\[
\begin{split}
\mu w(n-b)=&\lambda[af(n-b)w(n-b)-bf(n-b+1)w(n-b+1)]+w(n-a)	
\end{split}
\]
which can be written as 
\be
\mu w(j)\ge w(l)+\lambda [(n-l) f(j) w(j)-(n-j) f(j+1) w(j+1)]
\label{eq:part2reduced}
\ee
for $j+l> n$ and $j,l\in\intwithzero{n}$, where we have used the same change of coordinates of the proof of \cref{thm:dualpoa}, i.e., $j=a+x=n-b$, $l=b+x=n-a$. Thus, we conclude that \eqref{eq:part1reduced} and \eqref{eq:part2reduced} are sufficient to describe the constraints in \eqref{eq:generalbound}, and the result follows.

\item
	First, observe that for $j=0$, it must be $l\in[n]$. Additionally, note that the second set of constraints (those with $j+l>n$) is empty. The first set of constraints yields $\lambda\ge \frac{w(l)}{l}\frac{1}{f(1)w(1)}$ for $l\in[n]$. Define 
	\[\lambda^\star =\max_{l\in[n]} \frac{w(l)}{l}\frac{1}{f(1)w(1)}\,,
	\] and observe that any feasible $\lambda$ must satisfy $\lambda\ge \lambda^\star$. 
	Second, observe that for $l=0$, it must be $j\in[n]$. Additionally, the second set of constraints (those with $j+l>n$) is empty. The first set of constraints yields $\mu\ge \lambda j f(j)$ for $j\in[n]$.
	
	In the following we show that the most binding constraint amongst all those in \eqref{eq:formulacor1} is of the form $\mu\ge \alpha \lambda +\beta$, with $\alpha\ge 0$ (i.e., the most binding constraint is a straight line in the $(\lambda,\mu)$ plane pointing north-east). Consequently, the best choice of $\lambda$ so as to satisfy the constraints and minimize $\mu$ is to select $\lambda$ as small as possible, i.e., $\lambda= \lambda^\star$. See \cref{fig:lambdamu} for an illustrative plot. 
	
\begin{figure}[h!]
\centering
\includegraphics[width=0.5\textwidth]{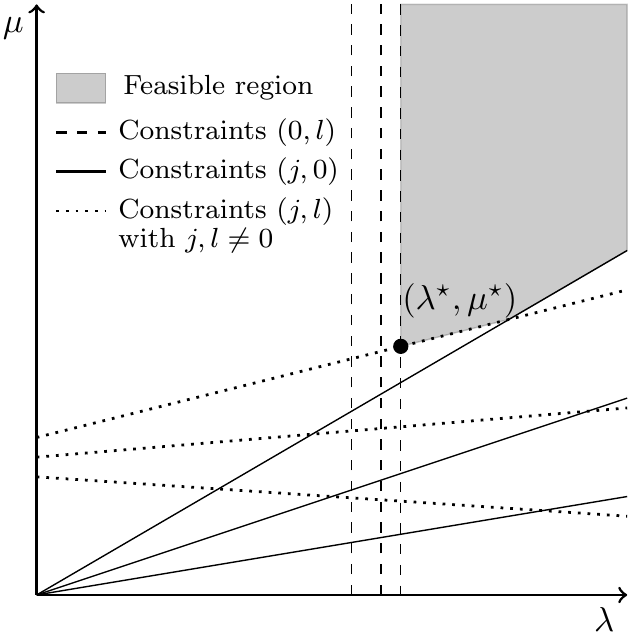}
\caption{Illustration of the three classes of constraints used in the proof of \cref{cor:fwnonincreasingpoadual}.}
\label{fig:lambdamu}
\end{figure}
	
	As shown previously, the constraints with $j=0$ are straight lines parallel to the $\mu$ axis, while the constraints with $l=0$ are straight line of the form $\mu\ge \lambda j f(j)$ (and thus point north-east in the $(\lambda,\mu)$ plane). We are thus left to check the constraints with $j\neq0$ and $l\neq 0$.
	
	To do so, we prove that if one such constraint (identified by the indices $(j,l)$) has negative slope, the constraint identified with $(j,0)$ is more binding. Since the constraint $(j,0)$ is of the form $\mu\ge j\lambda j f(j)$ (and thus has non-negative slope), this will conclude the proof.
	We split the reasoning depending on wether $1\le j+l\le n$ or $j+l> n$ as the constraints in \eqref{eq:formulacor1} have a different expression.
	\begin{itemize}
		\item[-]
		Case of $1\le j+l\le n$: to complete the reasoning, in the following we assume that 
    $jf(j)-lf(j+1)\frac{w(j+1)}{w(j)}<0$, and show that the constraint $(j,0)$ is more binding, i.e., that
	\[
	\lambda jf(j)\ge \frac{w(l)}{w(j)}+\lambda jf(j)-\lambda lf(j+1)\frac{w(j+1)}{w(j)}\,,
	\]	
which is equivalent to showing 
	\be
	\frac{w(l)}{w(j)}-\lambda lf(j+1)\frac{w(j+1)}{w(j)}\le 0\,.
	\label{eq:toshow1}
	\ee
	Since $jf(j)-lf(j+1)\frac{w(j+1)}{w(j)}<0$ it must be 
	\[
	l>j\frac{f(j)w(j)}{f(j+1)w(j+1)}\ge j\,,
	\]
	by non-increasingness of $f(j)w(j)$. Thus it must be \mbox{$l\ge j+1$}.
	Consequently, by non-increasingness of $f(j)w(j)$ it is $w(l)\le {f(j+1)w(j+1)}/{f(l)}$ and we can bound the left hand side of \eqref{eq:toshow1} as 
	\[
	\begin{split}
	&\frac{w(l)}{w(j)}-\lambda lf(j+1)\frac{w(j+1)}{w(j)}\\
	&\le\frac{f(j+1)w(j+1)}{w(j)f(l)}-\lambda l\frac{f(j+1)w(j+1)}{w(j)}\\
	&=\frac{f(j+1)w(j+1)}{f(j)w(j)}\left(\frac{1}{f(l)}-\lambda l \right) f(j)\,.
	\end{split}
	\]
 The claim \eqref{eq:toshow1} is shown upon noticing that $f(l)\ge \frac{1}{l}\min_{l\in[n]} \frac{l}{w(l)}f(1)w(1)=\frac{1}{l\lambda^\star }$ (by assumption), and thus 
	\[
	\frac{f(j+1)w(j+1)}{f(j)w(j)}\left(\frac{1}{f(l)}-\lambda l \right)f(j)\le \frac{f(j+1)w(j+1)}{f(j)w(j)}\left(\lambda^\star-\lambda  \right)l f(j)\le 0\,,
	\] 
	since we have already shown that $\lambda\ge \lambda^\star$ for every feasible~$\lambda$.
\item[-]
 Case of $j+l> n$: to complete the proof we proceed in a similar fashion to what seen in the previous case. In particular, we assume that $(n-l) f(j)-(n-j) f(j+1) \frac{w(j+1)}{w(j)}<0$, and show that the constraints $(j,0)$ is more binding, i.e., that 
	\be
	\frac{w(l)}{w(j)}+\lambda(n-l-j) f(j) -\lambda(n-j) f(j+1) \frac{w(j+1)}{w(j)})\le0\,.
	\label{eq:toshow2}
	\ee
	Since $(n-l) f(j)-(n-j) f(j+1) \frac{w(j+1)}{w(j)}<0$, it must be 
	\[
	n-j>(n-l)\frac{f(j)w(j)}{f(j+1)w(j+1)}\ge n-l
	\]
	by non-increasingness of $f(j)w(j)$. Thus it must be \mbox{$l\ge j+1$}.
	Consequently, by non-increasingness of $f(j)w(j)$ we can bound the left hand side of \eqref{eq:toshow2} as 
	\[
	\begin{split}
	&\frac{w(l)}{w(j)}+\lambda(n-l-j) f(j) -\lambda(n-j) f(j+1) \frac{w(j+1)}{w(j)}\\
	&\le \frac{f(j+1)w(j+1)}{w(j)f(l)}+\lambda(n-l-j) f(j) -\lambda(n-j) f(j+1) \frac{w(j+1)}{w(j)}\\
	&=\frac{f(j+1)w(j+1)}{w(j)f(j)}\left(\frac{1}{f(l)}-\lambda(n-j)\right)f(j)+\lambda(n-l-j) f(j)\\
	&\le \frac{f(j+1)w(j+1)}{w(j)f(j)}\left(\frac{1}{f(l)}-\lambda(n-j)+\lambda(n-l-j) \right)f(j)\\
	&\le  \frac{f(j+1)w(j+1)}{w(j)f(j)}\left(\frac{1}{f(l)}-\lambda l\right)f(j)\le 0\,,
	\end{split}
	\]
where the chain of inequality is proven similarly to what done in the case of $1\le j+l\le n$, using the non-decreasingness of $f(j)w(j)$ and the fact that $f(l)\ge\frac{1}{l\lambda^\star}$ by assumption.
\end{itemize}
\end{enumerate}
\end{proof}
\subsection{Proofs of the results presented in \cref{sec:optimalutilitydesign}}
\subsubsection*{Proof of \cref{thm:optimizepoa}}
\begin{proof}
We first observe that the problem $\argmax_{f\in\mc{F}} \poa(f)$ is well posed, in the sense that the supremum $\sup_{f\in\mc{F}}\poa(f)$ is attained for some $f\in\mc{F}$. A proof of this is reported in the following \cref{lem:finite}.

The (well posed) problem $\argmax_{f\in\mc{F}} \poa(f)$ is equivalent to finding the distribution rule that minimizes $W^\star$ given in \cref{thm:dualpoa}, i.e.,
\[
\begin{split}
	f^\star &\in \argmin_{f\in\mathcal{F}}\min_{\lambda\in\mb{R}_{\ge0},\,\mu\in\mb{R}}~ \mu \\[0.1cm]
	&\,\text{s.t.} ~\ones[\{b+x\ge1\}]w(b+x)
	- \mu \ones[\{a+x\ge 1\}]w(a+x)+\\
	&\quad~~+\lambda[af(a+x)w(a+x)-bf(a+x+1)w(a+x+1)]
	\le 0\qquad \forall (a,x,b)\in\Ir\,.
	\end{split}
\]
The previous program is non linear, but the decision variables $\lambda$ and $f$ always appear multiplied together. Thus, we define $\tilde f(j)\coloneqq\lambda f(j)$ for all $j\in \intwithzero{n+1}$ and observe that the constraint obtained in \eqref{eq:generalbound} for $(a,x,b)=(0,0,1)$ gives $
\tilde f(1)=\lambda f(1)\ge 1$, which also implies $\lambda \ge 1/f(1)>0$ since $f(1)>0$ (by assumption of $f\in\mc{F}$). Folding the $\min$ operators gives
\be
	\begin{split}
	&\tilde{f}^\star \in \argmin_{\substack{\tilde f\in\mb{R}^n_{\ge0} \\ \tilde f(1)\ge 1},\,\mu\in\mb{R}}~ \mu \\[0.1cm]
	&\,\text{s.t.} ~~\ones[\{b+x\ge1\}]w(b+x)
	- \mu \ones[\{a+x\ge 1\}]w(a+x)+\\
	&\quad~~~+a\tilde f(a+x)w(a+x)-b\tilde f(a+x+1)w(a+x+1)
	\le 0\qquad \forall (a,x,b)\in\Ir\,.
	\label{eq:generalbounddualproof}
	\end{split}
\ee
Finally, observe that $\tilde f^\star$ is feasible for the original program, since $\tilde f^\star \in\mc{F}$.
Additionally, we note that $\tilde f^\star$ and $f^\star$ give the same price of anarchy (since $\tilde f(j)=\lambda f(j)$, $\lambda>0$ and the equilibrium conditions are invariant to rescaling). Thus $\tilde{f}^\star$ solving \eqref{eq:generalbounddualproof} must be optimal. The optimal price of anarchy value follows. 
\end{proof}
\begin{lemma}
\label{lem:finite}
The supremum $\sup_{f\in\mc{F}} \poa(f)$ is attained in $\mc{F}$.	
\end{lemma}
\begin{proof}
Recall that $\mc{F}$ is defined as follows 
\[
\mc{F}\coloneqq\{f:[n]\rightarrow \mb{R}_{\ge0}~\text{s.t.}~f(1)\ge1,~f(j)\ge 0~\forall j\in[n]\}.
\]
To conclude, we show that any distribution $f^\star$ achieving a performance equal to \[\sup_{f\in\mc{F}}\poa(f)\] is bounded (i.e., all the components are bounded), so that it must be $f^\star\in\mc{F}$. To do so, consider a fixed distribution $f\in\mc{F}$, and construct from it $f_M$. The distribution $f_M$ is defined as follows: $f_M (j)=M$, with $M\in\mb{R}_{\ge0}$ for some fixed $j \in [n]$, while it exactly matches $f$ for the remaining components. In the following we show that there exists $\hat M\in\mb{R}_{\ge0}$ such that $\poa(f_M)<\poa(f)$ for all $M\ge \hat M$. Thus $f_M$ can not attain $\sup_{f\in\mc{F}}\poa(f)$ for $M\ge\hat M$ as the corresponding $f$ would give a better price of anarchy. Repeating this reasoning for any $f\in\mc{F}$, one concludes that the distribution rule achieving $\sup_{f\in\mc{F}}\poa(f)$ must be bounded along the $j$-th component. Repeating the reasoning for all possible $j\in[n]$, one obtains the claim.

To conclude we are left to show that $\exists \hat M\in\mb{R}_{\ge0}$ such that $\poa(f_M)<\poa(f)$ for all $M\ge \hat M$. To do so, observe that the price of anarchy of $f\in\mc{F}$ can be computed as $\poa(f)=1/W^\star$, where $W^\star$ is the solution to the primal problem in \eqref{eq:primalvalue}. Since the decision variables of \eqref{eq:primalvalue} live in a compact space (and the primal is feasible, see the footnote in the proof of \cref{thm:dualpoa}), we have $W^\star <+\infty$ and so $\poa(f)>0$, i.e., $\poa(f)$ is bounded away from zero. On the other hand, thanks to \cref{thm:dualpoa}, the price of anarchy of $f_M$ can be computed for any $M$ as $\poa(f_M)=1/W_M^\star$, where 
\[
\begin{split}
W_M^\star &= \min_{\lambda\in\mb{R}_{\ge0},\,\mu\in\mb{R}}~ \mu \\[0.1cm]
&\,\text{s.t.} ~~\ones[\{b+x\ge1\}]w(b+x)
- \mu \ones[\{a+x\ge 1\}]w(a+x)+\\
&\quad~~~+\lambda[af_M(a+x)w(a+x)-bf_M(a+x+1)w(a+x+1)]
\le 0\qquad \forall (a,x,b)\in\Ir
\end{split}
\]
First, observe that for any feasible $\lambda$, it must be ${\lambda\ge \frac{1}{f_M(1)}}$, else the constraints obtained form the previous linear program with $a=x=0$, $b=1$ would be infeasible. Further, consider the constraints with $b=0$, $x=0$, $a=j\ge1$. They amount to
\[
\mu\ge \lambda j f_M(j)\ge \frac{j}{f_M(1)} f_M(j) = \frac{jM}{f_M(1)}\,,
\]
where $f_M(1)>0$ by \hyperlink{sas}{Standing Assumptions} and $f_M\in\mc{F}$.
It follows that 
\[
\poa(f_M)=\frac{1}{W_M^\star} \le \frac{f_M(1)}{jM}\,.
\]
Thus, it is possible to make $\poa(f_M)$ arbitrarily close to zero, by selecting $M$ sufficiently large.
It follows that $\exists \hat M\in\mb{R}_{\ge0}$ such that 
$\poa(f_M)<\poa(f)$ for all $M\ge \hat M$, since $\poa(f)$ is bounded away from zero, as previously argued. This concludes the proof.
\end{proof}

\chapter[Submodular, supermodular, covering problems]{Submodular, supermodular,\\ covering problems}
\label{ch:p2subsupercov}
In the previous chapter we have addressed the problem of characterizing and optimizing the price of anarchy as a function of the chosen utilities. In this chapter we specialize the general result of \cref{thm:dualpoa,,thm:optimizepoa} to the case when $W$ is monotone submodular, supermodular, or a coverage function. We show how previously fragmented results from other authors can now be obtained as special case of the more general \cref{thm:dualpoa}. 

Relative to the submodular case, in \cref{sec:submod} we give an explicit expression for the price of anarchy (\cref{thm:wconcavefwdecreas}), and apply the result to obtain the efficiency of the Shapley value and marginal contribution distribution rule (\cref{cor:SVandMC}). This is, to the best of our knowledge, the first exact characterization of the price of anarchy in the submodular settings, and the first exact characterization of the performance associated to the Shapley value and marginal contribution distribution rule. Additionally, we show how the distribution rule designed maximizing the price of anarchy outperforms the very recent $1-c/e$ approximation of \cite{sviridenko2017optimal}, relative to submodular maximization problems.

In \cref{sec:cover} we consider the special case of \ac{mmc} problems (see \cref{sec:problemstatement} for their definition) and obtain a tight expression (\cref{thm:poageneralcovering}) for the price of anarchy relying solely on the \hyperlink{sas}{Standing Assumptions}. We further show how the expression subsumes previous results obtained under the additional assumptions therein required (\cref{cor:backtogair}). The distribution rule designed to maximize the price of anarchy achieves a $1-1/e$ approximation.

In \cref{sec:supermod} we consider the case when $W$ is supermodular and obtain an explicit expression for the price of anarchy (\cref{thm:convex}), extending previous results. Finally, we show that the Shapley value distribution rule is optimal, but observe that the utility design approach provides very poor approximation guarantees limitedly to this case.

Throughout this chapter we assume that the \hyperlink{sas}{Standing Assumptions} introduced in \cref{ch:p2utilitydesign} continue to hold.
All the proofs are reported in the Appendix (\cref{sec:proofsp2-part2}). The results presented in this chapter have been published in \cite{part1Paccagnan2018,,part2Paccagnan2018}.
\section{The case of submodular welfare function}
\label{sec:submod}
In this section we focus on the case when the welfare basis function $w$ is non-decreasing and concave (in the discrete sense). This results in the welfare function in \eqref{eq:welfaredef} being monotone submodular. Submodular functions model problems with diminishing returns and are used to describe a wide range of engineering applications such as satellite assignment problems \cite{qu2015distributed}, Adwords for e-commerce \cite{devanur2012online}, and combinatorial auctions \cite{lehmann2006combinatorial}, among others. For the considered class of problems, we show (\cref{thm:wconcavefwdecreas}) that characterizing the price of anarchy reduces to computing the maximum between $n(n+1)/2\sim\mc{O}(n^2)$ numbers. Using this result, we give an explicit expression of the price of anarchy for the well known Shapley value and marginal contribution distribution rule (\cref{cor:SVandMC}). We then show how to design $f$ so as to maximize the performance measured by $\poa(f)$. 
Finally, we compare our performance certificates with existing approximation results.

We begin by formally introducing two distribution rules that have attracted the researchers' attention due to their simple interpretation and to their special properties: the Shapley value distribution rule and the marginal contribution distribution rule \cite{von2013optimal}.
\begin{definition}
\label{def:svmc}
	The Shapley value and marginal contribution distribution rules are identified with $\fsv$ an $\fmc$, respectively.
	For $j\in[n]$, they are given by
	\begin{align}
			\fsv(j) &= \frac{1}{j}\,, \label{eq:SVdef} \\
			\fmc(j) &= 1-\frac{w(j-1)}{w(j)} \label{eq:MCdef}\,.
	\end{align}		
\end{definition}
Observe that the Shapley value distribution rule is the only distribution rule for which
the sum of all the players utility exactly matches the total welfare. 
The marginal contribution distribution rule takes its name from the observation that \eqref{eq:utilities} reduces to
\[
\begin{split}
 u_i(a) & =\sum_{r\in a_i}v_r w(|a|_r)\fmc(|a|_r) \\
 &= \sum_{r\in a_i}v_r (w(|a|_r)-w(|a|_r-1))
 = W(a)-W(\emptyset,a_{-i})\,,
\end{split}
 \]
i.e., player's $i$ utility function represent its marginal contribution to the total welfare, that is the difference between $W(a)$ and the welfare generated when player $i$ is removed from the game.

\begin{assumption}
\label{ass:sub}
Throughout this section we assume that the function $w$ is non-decreasing and concave, in the following sense 
\[
\begin{split}
w(j+1)&\ge w(j)\,,\\
w(j+1)-w(j)&\le w(j)-w(j-1)\quad \forall j\in [n-1]\,.
\end{split}
\]	
Further we assume that $w(1)=1$.
\end{assumption}
The requirement $w(1)=1$ is without loss of generality. Indeed, If $w(1)\neq 1$, it is possible to normalize its value and reduce to the case $w(1)=1$ (since $w(1)>0$ by \hyperlink{sas}{Standing Assumptions}).
As a consequence of \cref{ass:sub}, the function $W(a)$ is monotone and submodular, i.e., it satisfies the following:\\[0.3cm]
\emph{Monotonicity:}
\[
\forall ~ a,b\in\mc{A} \text{~s.t.~} a_i\subseteq b_i ~ \forall i\in\N \implies W(a)\le W(b)\,.
\]     
\emph{Submodularity:}
\[
	\begin{split}
	&\forall ~ a,b\in\mc{A} ~\text{~s.t.~} a_i\subseteq b_i ~ \forall i\in\N,\\
	&\forall~c \in {2^{\mc{R} n}} \text{~s.t.~} a_i' \coloneqq a_i\cup c_i\in\mc{A}_i,~ b_i' \coloneqq b_i\cup c_i \in\mc{A}_i~ \forall i\in\N,\\
	&\qquad\quad \implies W(a')-W(a)\ge W(b')- W(b)\,.
	\end{split}
\]

While \vref{thm:dualpoa} gives a general answer on how to determine the price of anarchy, it is possible to exploit the additional properties given by \cref{ass:sub} to obtain an explicit expression of $\poa(f)$. 
\begin{theorem}[$\poa$ for submodular welfare]
\label{thm:wconcavefwdecreas}
Consider $f$ a distribution rule such that $f(j)w(j)$ is non increasing and $f(j)\ge \fmc(j)$ for all $j\in[n]$. Then, $\poa(f)=1/W^\star$,
\be
\label{eq:poasubmodular}
	\begin{split}
	W^\star = \max_{l\le j \in [n]}\biggl\{&
	\frac{w(l)}{w(j)}+\min(j,n-l)f(j)-\min(l,n-j)f(j+1)\frac{w(j+1)}{w(j)}\biggl
	\},
	\end{split}
\ee
or equivalently
\be
	\label{eq:poasubmodularlp}
	\begin{split}
	W^\star=&\min_{\mu\in\mb{R}} \mu  \\[0.1cm]
	&\,\text{s.t.}~ \,\mu w(j)\ge w(l)+j f(j) w(j) - l f(j + 1) w(j + 1)\\
	&~\hspace*{25mm} \forall j,l\in\intwithzero{n}~{s.t.}~ j\ge l ~~\text{and}~~ 1\le j+l\le n,\\[0.15cm]
	& \qquad \mu w(j) \ge  w(l) + (n - l) f(j) w(j) - (n - j) f(j + 1) w(j + 1)\\
	&~\hspace*{25mm} \forall j,l\in\intwithzero{n}~{s.t.}~ j\ge l ~~\text{and}~~~~~~~ j+l\ge n\,.
	\end{split}
\ee
\end{theorem}
The proof amounts to showing that $\lambda$ appearing in \vref{cor:fwnonincreasingpoadual} can be computed a priori, and takes the value $\lambda^\star=1$.
The requirements on $f(j)w(j)$ being non increasing and $f(j)\ge\fmc(j)$ might seem restrictive at first. Nevertheless, similar assumptions where made in \cite{marden2014generalized,gairing2009covering}  relative to a simpler class of problems.
Additionally, the Shapley value and marginal contribution distribution rules (and many others) satisfy these assumptions. Thus, a direct application of  \cref{thm:wconcavefwdecreas} returns the exact price of anarchy of $\fsv$ and $\fmc$, as detailed next.
\begin{corollary}[Tight $\poa$ for $\fsv$ and $\fmc$]
\label{cor:SVandMC} 
\begin{enumerate}
\item[]
	\item The $\poa$ for the Shapley value distribution rule is $\poa({\fsv})=1/W_{\rm SV}^\star$, where
\be
\label{eq:poafsv}
	W^\star_{\rm SV} = \max_{l\le j \in [n]} \biggl\{ 
	\frac{w(l)}{w(j)}+\min(j,n-l)\frac{1}{j}
	-\min(l,n-j)\frac{w(j+1)}{(j+1)w(j)}\biggl
	\}.	
\ee
\item
The $\poa$ for the marginal contribution distribution rule is $\poa({\fmc})=1/W^\star_{\rm MC}$, where 
\be
\label{eq:poafmc}
		W_{\rm MC}^\star = 1 + \max_{j \in [n]}\biggl\{ \frac{1}{w(j)}\min(j,n - j)[2w(j) - w(j - 1) - w(j + 1)] \biggl\}
\ee
\end{enumerate} 
\end{corollary}
The previous Corollary shows that the price of anarchy of the Shapley value and marginal contribution distribution rule can be computed as the maximum of $n(n+1)/2$ and $n$ numbers, respectively. 
\begin{remark}[Connection with \cite{marden2014generalized}]
\label{rmk:Rough}
The quantity \eqref{eq:poafsv} can be equivalently written as 
\be
\label{eq:poafsvremanaged}
{W}^\star_{\rm SV} = 1+\max_{l\le j\in [n]}\biggl\{
\frac{w(l)}{w(j)}-\frac{1}{j}[\max\{j+l-n,0\}+\min\{l,n-j\}\beta(j)]
\biggr\}	,
\ee
\[
\text{where}\quad \beta(j)\coloneqq\frac{j}{j+1}\frac{w(j+1)}{w(j)}\,.
\]
The previous expression partially matches the result in \cite[Thm. 6]{marden2014generalized}, where the authors used a different approach to obtain a bound on the price of anarchy for the larger class of coarse correlated equilibria, but limitedly to $\fsv$ and singleton problems. More precisely, \cite[Thm. 6]{marden2014generalized} provides a bound of the price of anarchy relative to $\fsv$, as the minimum between two expression. While their first expression exactly matches \eqref{eq:poafsvremanaged}, the second one is not present here. 
Nevertheless, it is possible to show that such additional expression is redundant, as the first one is always the most constraining.\footnote{This statement is not formally shown here, in the interest of space. Its proof amounts to showing that the second expression appearing in \cite[Thm. 6]{marden2014generalized} is always upper bounded by \eqref{eq:poafsvremanaged}, thanks to the concavity of $w$.}
This allows us to conclude that the bound obtained in \cite[Thm. 6]{marden2014generalized} precisely matches the one in \eqref{eq:poafsvremanaged}. Additionally, since our result is provably tight for the class of Nash equilibria, and the result in  \cite{marden2014generalized} provides a lower bound for \ac{cce}, such bound is tight as well (in the set of \ac{cce}) and the worst performing coarse correlated equilibrium is, simply, a pure Nash equilibrium.
\end{remark}
 
For the submodular welfare case considered here, it is still possible to determine the distribution rule $f^\star$ that maximizes $\poa(f)$ as the solution of a tractable linear program either directly employing the more general result in \vref{thm:optimizepoa}  or using the following linear program derived from \eqref{eq:poasubmodularlp}, which additionally constrains the admissible distributions $f$ to satisfy $f(j)\ge \fmc(j)$ and $f(j)w(j)$ to be non increasing,
\be
	\label{eq:optimalfconcave}
	\begin{split}
	&f^\star \in \argmin_{f\in{\mc{F}_s},\, \mu\in\mb{R}} \mu  \\[0.1cm]
	&\,\text{s.t.}~ \,\mu w(j) \ge  w(l)+j f(j) w(j) - l f(j + 1) w(j + 1)\\
	&~\hspace*{25mm} \forall j,l\in\intwithzero{n}~{s.t.}~ j\ge l ~~\text{and}~~ 1\le j+l\le n,\\[0.15cm]
	& \qquad \mu w(j) \ge  w(l) + (n - l) f(j) w(j) - (n - j) f(j + 1) w(j + 1)\\
	&~\hspace*{25mm} \forall j,l\in\intwithzero{n}~{s.t.}~ j\ge l ~~\text{and}~~~~~~~ j+l\ge n\,,
	\end{split}
\ee
	where 
	$\mc{F}_s = \{f \in\mc{F}\,|\,f(j) \ge \fmc(j),~ f(j+1)w(j+1) \le  f(j)w(j)~\forall j\in[n]\}$. Extensive numerical simulations have shown that both these approaches return the same optimal value, so that the additional constraints $f\in\mc{F}_s$ required in \eqref{eq:optimalfconcave} do not rule out the optimal distribution derived solving the linear program in \vref{thm:optimizepoa}. This statement can be formally proved, for example, by showing that any distribution rule satisfying the \ac{kkt} system of the \ac{lp} in \eqref{eq:optimalfconcave} is also a solution to the \ac{kkt} system of the \ac{lp} in \cref{thm:optimizepoa}. We do not further purse this direction here.

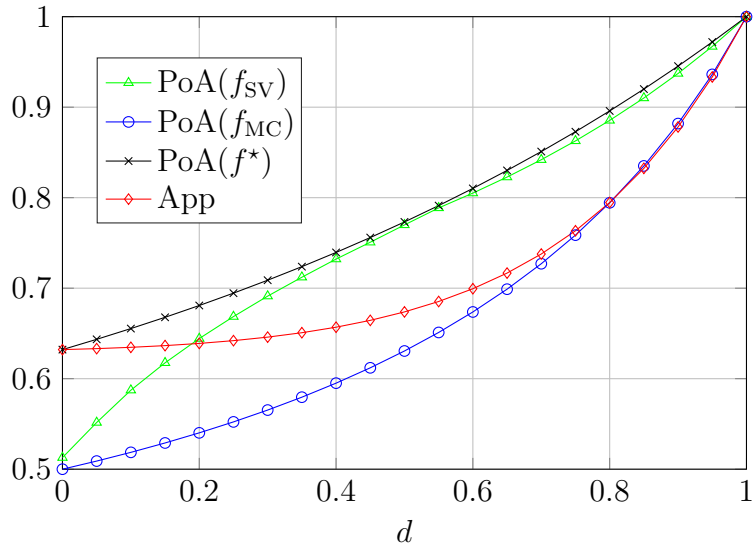
\begin{figure}[ht!] 
\begin{center}
\setlength\figureheight{6cm} 
\setlength\figurewidth{9cm} 
%
%
\begin{tikzpicture}

\begin{axis}[%
width=\figurewidth,
height=\figureheight,,
at={(1.011in,0.642in)},
scale only axis,
xmin=0,
xmax=1,
xlabel={$d$},
ymin=0.5,
ymax=1,
ytick={0.5, 0.6, 0.7, 0.8, 0.9 ,1},
xtick={0, 0.2, 0.4, 0.6, 0.8, 1},
grid=both,
legend style={at={(0.05,0.91)}, anchor=north west, draw=white!15!black},
legend cell align={left}
]

\addplot[color=green,mark=triangle,mark options={solid}]
  table[row sep=crcr]{%
0	0.512820512820513\\
0.05	0.551579892207097\\
0.1	0.587291533788722\\
0.15	0.617541516463445\\
0.2	0.644361679553719\\
0.25	0.668613433410656\\
0.3	0.691233805537062\\
0.35	0.71209950483051\\
0.4	0.732154497890012\\
0.45	0.750757656686259\\
0.5	0.76990681212978\\
0.55	0.788670183689739\\
0.6	0.805061118646457\\
0.65	0.822763630350338\\
0.7	0.841929713797954\\
0.75	0.862735663283182\\
0.8	0.885387134211467\\
0.85	0.910125525400614\\
0.9	0.937236102023765\\
0.95	0.967058439409503\\
1	1\\
};
\addlegendentry{$\poa(\fsv)$};
\addplot [color=blue,mark=o,mark options={solid}]
  table[row sep=crcr]{%
0	0.5\\
0.05	0.508974473014022\\
0.1	0.518611263028954\\
0.15	0.528980031386739\\
0.2	0.540160487999998\\
0.25	0.55224427281837\\
0.3	0.565337276567052\\
0.35	0.579562525251956\\
0.4	0.595063794950633\\
0.45	0.612010182095219\\
0.5	0.630601937481871\\
0.55	0.651077991369887\\
0.6	0.673725770588819\\
0.65	0.698894165659515\\
0.7	0.727010893750765\\
0.75	0.758606100120772\\
0.8	0.794344979667087\\
0.85	0.835073714094604\\
0.9	0.881885528675272\\
0.95	0.936217958793979\\
1	1\\
};
\addlegendentry{$\poa(\fmc)$\hspace*{-1mm}};

\addplot[color=black,solid,mark=x,mark options={solid}]
  table[row sep=crcr]{%
0	0.632120555132191\\
0.05	0.643482913435981\\
0.1	0.655375384527601\\
0.15	0.667825056206542\\
0.2	0.680860546884625\\
0.25	0.694512001740175\\
0.3	0.708811156686717\\
0.35	0.723791477212326\\
0.4	0.739487634182099\\
0.45	0.755936488911006\\
0.5	0.773180596528451\\
0.55	0.791262194488063\\
0.6	0.810212760997551\\
0.65	0.830092814964609\\
0.7	0.850955470654624\\
0.75	0.872854275447214\\
0.8	0.895845961995925\\
0.85	0.919990663553415\\
0.9	0.945352121634846\\
0.95	0.971997956707326\\
1	1\\
};
\addlegendentry{$\poa(f^\star)$};

\addplot[color=red,solid,mark=diamond,mark options={solid}]
  table[row sep=crcr]{%
0      0.632120558828558\\
0.05   0.633215096779674\\
0.10   0.634660102977066\\
0.15   0.636539744118561\\
0.20   0.638956154576726\\
0.25   0.642033032461418\\
0.30   0.645919926631880\\
0.35   0.650797343733750\\
0.40   0.656882827957390\\
0.45   0.664438194090168\\
0.50   0.673778127323984\\
0.55   0.685280402070890\\
0.60   0.699398017781229\\
0.65   0.716673603681631\\
0.70   0.737756507901600\\
0.75   0.763423061338450\\
0.80   0.794600594819830\\
0.85   0.832395892009683\\
0.90   0.878128882826179\\
0.95   0.933372526139581\\
1.00   1.000000000000000\\
};
\addlegendentry{${\rm App}$};

\end{axis}
\end{tikzpicture}%
\caption{Comparison between the  approximation ratio \eqref{eq:approx} and the price of anarchy of the optimal distribution rule $f^\star$ (determined as the solution of the \ac{lp} in \cref{thm:optimizepoa}),
Shapley value $\fsv$ and marginal contribution $\fmc$ distribution rules.
The problems considered features $|N|\le n= 20$ agents and a welfare basis of the form $w(j)=j^d$ with $d\in[0,1]$ represented over the $x$-axis.
} 
\label{fig:comparisonp2}
\end{center}
\end{figure}

\cref{fig:comparisonp2} compares the price of anarchy (and thus the approximation ratio of any algorithm capable of computing a Nash equilibrium) of the Shapley value, marginal contribution and optimal distribution rule $f^\star$, in the case when $w(j)=j^d$ with $d\in[0, 1]$, $|N|\le 20$. They have been computed using respectively \eqref{eq:poafsv}, \eqref{eq:poafmc}, where $f^\star$ has been determined as the solution to the \ac{lp} appearing in \cref{thm:optimizepoa}.
For values of $d\in[0.5, 1]$ the Shapley value distribution rule performs close to the optimal, but its performance degrades for $d\in[0, 0.5]$ and for $d=0$ it reaches the lower bound of $1/2$, as predicted for the class of valid utility games defined in \cite[Thm. 5]{vetta2002nash}. The marginal contribution rule instead, performs the worst amongst the considered distribution rules. While $f^\star$ will always perform better or equal than any other distribution, it is unclear if, and to what extent, $\fsv$ outperforms $\fmc$ in the general settings. The expressions in \eqref{eq:poafsv} and \eqref{eq:poafmc} can nevertheless be used to provide an answer to this question.

\subsection[Improved approximation and comparison with existing result]{Improved approximation and comparison with\\ existing result}
\label{subsec:approxcomparison}
In this section we compare the approximation guarantees offered by the utility design approach with other recent result in the maximization of submodular functions.

A monotone submodular maximization problem is defined as follows. We are given a set $X$, and a collection of subsets $S\subseteq 2^{X}$. Given a monotone and submodular set function $g:2^X\rightarrow \mb{R}_{\ge0}$, the objective is to find a set $s\in S$ maximizing $g$.
 If the collection of subsets $S$ is a matroid, we term the problem a monotone submodular maximization problem subject to matroid constraints. For the latter class of problems, the best approximation ratio achievable in polynomial time has been very recently shown to be \cite{sviridenko2017optimal} 
\be
1-\frac{c}{e}
\label{eq:1-ce}\,,
\ee
where $c$ represents the curvature of the welfare function and $e$ the Euler's number. The curvature is formally defined as \cite{conforti1984submodular}
\be
\label{eq:curvature}
c\coloneqq 1 - \min_{\substack{e\in X\\ g(\{e\})-g(\emptyset) \neq 0}}
\frac{g(X)-g(X\setminus \{e\})}{g(\{e\})-g(\emptyset)}
\ee
 For this class of problems, no polynomial time algorithm can do better than \eqref{eq:1-ce} on all instances, even if the matroid is the uniform matroid, i.e., in the case of cardinality constraints \cite{sviridenko2017optimal}.
The \ac{gmmc} problems studied here differs from the problem of maximizing a submodular function subject to matroid constraints, in that we are given not one, but $n$ collections of sets. Thus, to compare the approximation results, in the following we restrict to \ac{gmmc} problems where $\mc{A}_i=\mc{A}_j=\bar{\mc{A}}\subseteq 2^{\mc{R}}$. We allow for some set to appear multiple times in $\bar{\mc{A}}$ so as to cover the case when different agents select the same set. The objective is to select $n$ subsets from $\bar{\mc{A}}$ so as to maximize $W$ as defined in \eqref{eq:welfaredef}. The problem can be transformed in a monotone submodular maximization problem subject to cardinality constraints. To do so, let us enumerate all the subsets as in $\bar{\mc{A}}=\{A_1,\dots,A_k\}$. We set $X\coloneqq[k]$, $S=2^{[k]}$ and identify with $s=(s_1,\dots,s_l)\in [k]^l$ an element of $S$ (note that $l\le k$). We define $g:2^X \rightarrow  \mb{R}_{\ge 0}$ for any $s\in S$ as
\be
\label{eq:fdefinition}
g(s)\coloneqq \sum_{r\in(\cup _j A_{s_j})} v_r w(|s|_r)\,,
\ee
where $|s|_r=|\{i~\text{s.t.}~r\in A_{s_i}\}|$. Selecting $n$ subsets ($n\le k$) from $\bar{\mc{A}}$  to maximize $W(a)$ is then equivalent to solving 
\be
\label{eq:monotonesub}
\max_{s\in S,\, |s|\le n} g(s)\,.
\ee
The problem in \eqref{eq:monotonesub} belongs to the class of monotone submodular maximization subject to cardinality constraints. Indeed, $g(s)$ is monotone and submodular due to \cref{ass:submodd}. Additionally the number of elements in $s$ is constrained to be less or equal to $n$. Thus, the approximation ratio \eqref{eq:1-ce} holds for \eqref{eq:monotonesub}. The curvature can be determined using \eqref{eq:curvature}, and amounts to $c=1+w(n-1)-w(n)$.\footnote{This value of the curvature constitutes the worst case value amongst all possible \ac{gmmc} problems introduced in \cref{sec:problemstatement}. The reason to consider the worst case curvature over all problem instances (i.e., over all possible choices of $\bar{\mc{A}}$ and $\{v_r\}_{r\in\mc{R}}$) is that we wish to compare the approximation ratio \eqref{eq:1-ce}  with the game theoretic approximation, and the latter gives a certificate over \emph{all possible instances}.} 
In \cref{fig:comparisonp2} we plot the approximation ratio \eqref{eq:1-ce} for the class of problems considered here, with the choice of $w(j)=j^d$, i.e., we plot (red curve) the quantity
\be
{\rm App}=1-\frac{1+w(n-1)-w(n)}{e}\,,
\label{eq:approx}
\ee
for $d\in[0, 1]$.
 We observe that the optimal distribution rule $f^\star$ outperforms \eqref{eq:approx} for different values of $d$, so that, when there exists an algorithm capable of computing a Nash equilibrium in polynomial time (see \cref{prop:poly}), the approach presented here gives improved guarantees compared to \eqref{eq:1-ce}. 
\begin{remark}
It is important to note that this is not in contradiction with the inapproximability result presented in \cite{sviridenko2017optimal}, as we are not solving a general submodular maximization problem, but the welfare function in \eqref{eq:fdefinition} has a \emph{special form}. 	
\end{remark}

\section{Covering problems}
\label{sec:cover}
In this section we specialize the previous results to the case of multiagent weighted maximum coverage (\ac{mmc}) problems introduced in \cref{sec:p2relatedwork}, a generalization of weighted maximum coverage problems. In a \ac{mmc} problem we are given a ground set of elements $\mc{R}$ and $n$ collections of subsets of the ground sets: $\mc{A}_i$ for $i\in \N$.
The goal is to select $n$ subsets, one from each collection, so as to maximize the \emph{total value} of covered elements.
The corresponding welfare is
\[
W(a)=\sum_{r\in\cup_{i\in N} a_i} v_r\,,
\]
which is obtained with the choice of $w(j)=1$ for all $j$ in \eqref{eq:welfaredef} and \eqref{eq:utilities}. \ac{mmc} problems are a subclass of \ac{gmmc} submodular problems (they satisfy \cref{ass:sub}), 
and are used to model engineering problems such as vehicle-target assignment \cite{arslan2007autonomous} and sensor allocation problems \cite{marden2008distributed}. Due to their importance in the applications, we treat their study separately. 

Relative to \ac{mmc} problems, we provide a general expression for the price of anarchy as a function of $f$ (\cref{thm:poageneralcovering}) and show how this reduces to the results obtained in \cite{gairing2009covering, paccagnan2017arxiv}, under the additional assumptions therein required.
 
\begin{theorem}[$\poa$ for multiagent maximum coverage]
\label{thm:poageneralcovering}
Consider \ac{mmc} problems, i.e., fix $w(j)=1$ $\forall j \in [n]$. The price of anarchy is
$\poa(f)=1/ W^\star$ where
\be
\label{eq:Wstarsetcoveringg}
W^\star = 1+\max_{j\in [n-1]}\{(j+1)f(j+1)-1, jf(j)-f(j+1),jf(j+1)\},
\ee
or equivalently
\be
\label{eq:WstarsetcoveringgLP}
\begin{split}
W^\star =& \min_{\mu\in\mb{R}}\mu\\
&\text{s.t.}~ \mu \ge (j+1)f(j+1)	\\
&\hspace*{7mm}\mu \ge 1 + jf(j)- f(j+1)\\
&\hspace*{7mm}\mu \ge 1+jf(j+1)\qquad\qquad\qquad\forall j\in[n-1].
\end{split}
\ee
\end{theorem}
The previous theorem gives a simple and explicit way to compute the price of anarchy \eqref{eq:poadef} as the maximum between $3(n-1)$ numbers. %
Observe that no assumptions are required other than the \hyperlink{sas}{Standing Assumptions}. \cref{thm:poageneralcovering} thus extends the previous bounds derived in \cite{gairing2009covering, paccagnan2017arxiv}. In the latter works, the authors required the distribution rules to be non increasing and sub budget-balanced\mbox{, i.e., $jf(j)\le 1$ for all $j\in [n]$.}

The next corollary shows how the result in the previous theorem matches the results in \cite{gairing2009covering, paccagnan2017arxiv}, simply requiring $f$ to be non increasing (this is a less restrictive assumption than what asked for in \cite{gairing2009covering, paccagnan2017arxiv}).

\begin{corollary}
\label{cor:backtogair}
Consider $f$ a non increasing distribution rule.
The value of \eqref{eq:Wstarsetcoveringg} is given by
\be
\label{eq:Wstarcoveringgair}
W^\star = 1 +\max_{j\in [n-1]}\{jf(j)-f(j+1),(n-1)f(n)\}\,.
\ee	
\end{corollary}
In \cite[Thm. 2]{gairing2009covering} the author provides a bound matching the expression in \eqref{eq:Wstarcoveringgair}. Tightness of the previous bound is shown in \cite[Thm. 1]{paccagnan2017arxiv}. Additionally, \cite[Eq. 5]{gairing2009covering} also determines the distribution rule maximizing the price of anarchy \eqref{eq:Wstarcoveringgair}. The optimal distribution, denoted with $\fgar$, has already been introduced in \eqref{eq:fgar} and is reported in the following for completeness
\be
\fgar(j)=(j-1)!\frac
{\frac{1}{(n-1)(n-1)!} +\sum_{i=j}^{n-1}\frac{1}{i!}}
{\frac{1}{(n-1)(n-1)!} +\sum_{i=1}^{n-1}\frac{1}{i!}},\quad j\in[n]\,.
\label{eq:fstargair}
\ee
In all the above mentioned results the feasible set of distribution rules is limited to $jf(j)\le 1$ and $f$ non increasing. Using the result provided here in \cref{thm:poageneralcovering} it is possible to determine the optimal distribution  without imposing these additional constraints on $f$ by solving the following \ac{lp} derived from \eqref{eq:WstarsetcoveringgLP}
\be
\label{eq:optimalfcovering}
\begin{split}
& \argmin_{f\in\mc{F},\,\mu\in\mb{R}}\mu\\
&\text{s.t.}~ \mu \ge (j+1)f(j+1)	\\
&\hspace*{7mm}\mu \ge 1 + jf(j)- f(j+1)\\
&\hspace*{7mm}\mu \ge 1+jf(j+1)\qquad\qquad\qquad\forall j\in[n-1].
\end{split}
\ee
Numerical simulations have shown that the optimal distribution rule obtained optimizing \eqref{eq:Wstarsetcoveringg} precisely matches the one derived in \cite{gairing2009covering}, so that removing the additional assumption required therein does not improve the best achievable price of anarchy.\footnote{This statement con be proved, by showing that the distribution $\fgar$ solves the \ac{kkt} system of~\eqref{eq:optimalfcovering}.}
\begin{remark}[Matching the $1-1/e$ of \cite{nemhauser1978analysis}]
\label{rmk:1-1/ecomparison}
Relative to \ac{mmc} problems, \cite{gairing2009covering} explicitly determines the value of the price of anarchy for the optimal distribution $\fgar$. It's value amounts to (see \vref{thm:smoothnottight})
\be
\poa(\fgar)=1-\frac{1}{\frac{1}{(n-1)(n-1)!}+\sum_{i=0}^n \frac{1}{i!}}~~\xrightarrow[]{n \to \infty}~ 1-\frac{1}{e}\,.
\ee
This shows that for \ac{mmc} problems (a generalization of weighted maximum coverage problems) one can obtain the same approximation guarantee achievable for weighted maximum coverage problems and first shown in \cite{nemhauser1978analysis}.
\end{remark}

\section{The case of supermodular welfare function}
\label{sec:supermod}
In this section we consider welfare basis functions that are non-decreasing and convex, resulting in a monotone and supermodular total welfare $W(a)$. Applications featuring this property include clustering and image segmentation \cite{stobbe2010efficient}, power allocation in multiuser networks \cite{yassin2017centralized}. 
In the following we explicitly characterize the price of anarchy for the class of supermodular resource allocation problems as a function of $f$ (\cref{thm:convex}), extending \cite{jensen2018, phillips2017design}. Additionally, we show that the Shapley value distribution rule maximizes this measure of efficiency (recovering the result in \cite{jensen2018, phillips2017design}), but \emph{is not the only one}. 
\begin{assumption}
\label{ass:submodd}
Throughout this section we assume that $f(1)=w(1)=1$ and that $w$ is a non-decreasing and convex function, i.e.,
\[
\begin{split}
w(j+1)&\ge w(j)\,,\\
w(j+1)-w(j)&\ge w(j)-w(j-1)\quad \forall j\in [n-1]\,.
\end{split}
\]	
\end{assumption}

\begin{theorem}[$ \poa $ for supermodular welfare]
\label{thm:convex}
Consider a distribution rule $f$ such that $f(j)w(j)\ge 1$ $\forall j \in[n]$. It holds 
\[
\poa{(f)}= \frac{n}{w(n)} \frac{1}{\max_{j\in[n]} j f(j)}\,.
\]
Additionally, $\fsv$ is optimal amongst $f\in\mc{F}$ and achieves
\[
\poa{({\fsv})}= \frac{n}{w(n)}\,.
\]
\end{theorem}
Observe that the Shapley value and all the distribution rules for which $jf(j)\ge 1$ satisfy the conditions of \cref{thm:convex}. Indeed $f(j)w(j) \ge jf(j)\ge 1$ by convexity and \hyperlink{sas}{Standing Assumptions}. 
Further note that the Shapley value distribution rule is \emph{not} the unique maximizer of $\poa(f)$. Indeed, all the distribution rules with $1/w(j)\le f(j)\le 1/j$ are optimal, as the previous theorem applies and they achieve a price of anarchy of $n/w(n)$ since it is $\max_{j\in[n]}jf(j)=1$  (due to $f(1)=1$ and  $f(j)\le 1/j$).
\cref{fig:comparisonsuper} compares the price of anarchy of the Shapley value, marginal contribution and optimal distribution rule, in the case when $w(j)=j^d$ with $d\in[1, 2]$, $|N|\le 20$. First, we observe that any optimal distribution rule and $\fsv$ give the same performance, as predicted from the previous theorem. Additionally, we observe that the quality of the approximation quickly degrades as the welfare basis $w$ gets steeper ($d$ gets larger). This is due to the fact that if $w(n)$ grows much faster than $n$, the quantity $n/w(n)$ quickly decreases.

\begin{figure}[ht!] 
\begin{center}
\setlength\figureheight{6cm} 
\setlength\figurewidth{9cm} 
%
%
\begin{tikzpicture}

\begin{axis}[%
width=\figurewidth,
height=\figureheight,,
at={(1.011in,0.642in)},
scale only axis,
xmin=1,
xmax=2,
xlabel={$d$},
xtick={1, 1.2, 1.4, 1.6, 1.8, 2},
ymin=0,
ymax=1,
grid=both,
legend style={at={(0.95,0.91)}, anchor=north east, draw=white!15!black},
legend cell align={left}
]

\addplot[color=green,mark=triangle,mark options={solid}]
  table[row sep=crcr]{%
1	1\\
1.05	0.860891659331735\\
1.1	0.741134449106947\\
1.15	0.638036465679592\\
1.2	0.549280271653059\\
1.25	0.472870804501588\\
1.3	0.407090531536904\\
1.35	0.350460843193043\\
1.4	0.301708816827258\\
1.45	0.259738603953433\\
1.5	0.223606797749979\\
1.55	0.192501227152835\\
1.6	0.165722700866999\\
1.65	0.142669290938328\\
1.7	0.122822802611579\\
1.75	0.105737126344056\\
1.8	0.091028210151304\\
1.85	0.078365426883154\\
1.9	0.0674641423836782\\
1.95	0.0580793174820771\\
2	0.05\\
};
\addlegendentry{$\poa(\fsv)$};
\addplot [color=blue,mark=o,mark options={solid}]
  table[row sep=crcr]{%
1	1\\
1.05	0.820939648857225\\
1.1	0.675473223027811\\
1.15	0.556933108559613\\
1.2	0.46006526711858\\
1.25	0.380706470829292\\
1.3	0.315541335661298\\
1.35	0.261917758663104\\
1.4	0.217705571084728\\
1.45	0.181187487420064\\
1.5	0.150974407561143\\
1.55	0.125939233377645\\
1.6	0.105164866104713\\
1.65	0.0879031394436211\\
1.7	0.0735422385807872\\
1.75	0.0615807418696769\\
1.8	0.0516068582576294\\
1.85	0.043281760749507\\
1.9	0.0363261634019941\\
1.95	0.0305094773888512\\
2	0.0256410256410256\\};
\addlegendentry{$\poa(\fmc)$};

\addplot[color=black,solid,mark=x,mark options={solid}]
  table[row sep=crcr]{%
1	1\\
1.05	0.860891659331735\\
1.1	0.741134449106947\\
1.15	0.638036465679592\\
1.2	0.549280271653059\\
1.25	0.472870804501588\\
1.3	0.407090531536904\\
1.35	0.350460843193043\\
1.4	0.301708816827258\\
1.45	0.259738603953433\\
1.5	0.223606797749979\\
1.55	0.192501227152835\\
1.6	0.165722700866999\\
1.65	0.142669290938328\\
1.7	0.122822802611579\\
1.75	0.105737126344056\\
1.8	0.091028210151304\\
1.85	0.078365426883154\\
1.9	0.0674641423836782\\
1.95	0.0580793174820771\\
2	0.05\\
};
\addlegendentry{$\poa(f^\star)$};

\end{axis}
\end{tikzpicture}%
\caption{Price of anarchy comparison between 
the optimal distribution rule $f^\star$ determined as the solution of the \ac{lp} in \cref{thm:optimizepoa},
Shapley value $\fsv$ and marginal contribution $\fmc$ distribution rules.
The problems considered features $|N|\le 20$ agents and a welfare basis of the form $w(j)=j^d$ with $d\in[1,2]$ represented over the $x$-axis.
} 
\label{fig:comparisonsuper}
\end{center}
\end{figure}
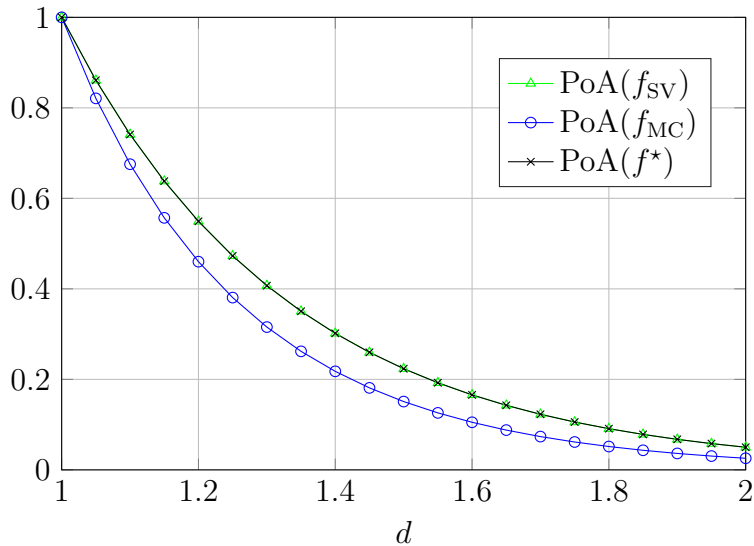
\section{Appendix}
\label{sec:proofsp2-part2}

\subsection{Proofs of the results presented in \cref{sec:submod}}
\subsubsection*{Proof of  \cref{thm:wconcavefwdecreas}}
\label{appendixb}
\begin{proof}
Observe that the value of $W^\star$ in \eqref{eq:poasubmodular} can be equivalently reformulated as in the following program, upon observing that for $j+l\le n$ it holds $\min(j,n-l)=j$ and $\min(l,n-j)=l$, while for $j+l> n$ it holds $\min(j,n-l)=n-l$ and $\min(l,n-j)=n-j$,
\[
	\begin{split}
	W^\star &= \min_{\mu\in\mb{R}}~ \mu  \\[0.1cm]
	&\,\text{s.t.}~ \,\mu w(j)\ge w(l)+j f(j) w(j)-l f(j+1) w(j+1)\\	&~\hspace*{25mm} \forall j,l\in\intwithzero{n}~{s.t.}~ j\ge l ~\text{and}~ 1\le j+l\le n,\\[0.15cm]
	&\qquad \mu w(j)\ge w(l)+(n-l) f(j) w(j)-(n-j) f(j+1) w(j+1)\\
	&~\hspace*{25mm} \forall j,l\in\intwithzero{n}~{s.t.}~ j\ge l ~\text{and}~~~~~~\, j+l> n\,.
	\end{split}
\]
In the following we prove that the latter program follows from \cref{cor:fwnonincreasingpoadual} by showing that only the constraints with $l\le j$ are required, and that the decision variable $\lambda$ in \eqref{eq:formulacor1} takes the value $\lambda^\star=1$.
First, notice that $f(j)w(j)$ is assumed to be non increasing, and so $W^\star$ can be correctly computed using \cref{cor:fwnonincreasingpoadual}.	
For $j=0$, the constraints in \eqref{eq:formulacor1} read as 
\[\lambda\ge\frac{w(l)}{l}\quad\forall \,l\in [n]\,,
\]	
and the most binding amounts to $\lambda\ge 1$, due the to concavity of $w$. 
For $j\neq 0$, we intend to show that the constraints with $l>j$ appearing in \eqref{eq:formulacor1} are not required since
those with $j=l$ are more binding. 
The following figure explains this more clearly.
\begin{figure}[h!]
\begin{center}
\includegraphics[scale=1.3]{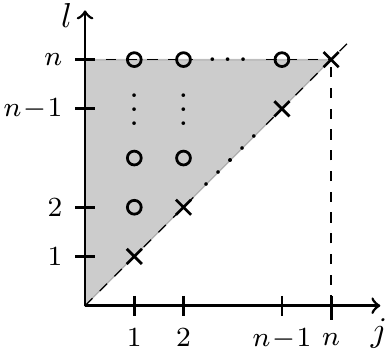}	\caption{The proof amounts to showing that for any constraint identified with the indices $(j,l)$ and $l>j$ (circles), the constraint $(j,j)$ is more binding (crosses).}
\end{center}	
\end{figure}

To do so, we divide the discussion in two cases: $l+j\le n$ and $l+j>n$.\\
{\bf Case 1.}
When $1\le j+l\le n$ we want to show that for any $l>j$ and $\lambda\ge 1$
\[
\begin{split}
&1+\lambda \frac{j}{w(j)}[f(j)w(j)-f(j+1)w(j+1)]\ge\\
& \frac{w(l)}{w(j)}+\lambda \biggl[\frac{j}{w(j)}f(j)w(j)-\frac{l}{w(j)}f(j+1)w(j+1)\biggr]\,,
\end{split}
\]
where the left hand side is obtained setting $l=j$.
This is equivalent to showing
\begin{equation}
\label{eq:first}
w(l)-w(j)+\lambda (j-l)f(j+1)w(j+1)\le 0\,.
\end{equation}
By concavity of $w$ and  $l>j$, one observes that 
\[\begin{split}
w(l)&\le w(j+1)+(w(j+1)-w(j))(l-j-1)=w(j)+(w(j+1)-w(j))(l-j)
\end{split}
\]
and since $l-j>0$, $w(j+1)-w(j)\ge 0$, $\lambda\ge 1$, it holds
\be
w(l)\le w(j)+\lambda (w(j+1)-w(j))(l-j).
\label{eq:ineqproofT4}
\ee
Using inequality \eqref{eq:ineqproofT4}, one can show that \eqref{eq:first} has to hold
\[
\begin{split}
&w(l)-w(j)+\lambda (j-l)f(j+1)w(j+1)\\
&\le w(j)+\lambda (w(j+1)-w(j))(l-j)-w(j)+\lambda (j-l)f(j+1)w(j+1)\
\\
&=\lambda(l-j)(w(j+1)-w(j)-f(j+1)w(j+1))\le 0\,,	
\end{split}
\]
where the last inequality holds because $f(j+1)w(j+1)\ge w(j+1)-w(j)$ (by assumption) and $l>j$. Observe that the previous inequality is never evaluated for $j=n$, as there is no $l\in[n]$ with $l>j=n$.\\
{\bf Case 2.}
We now consider the case $j+l>n$. Here we intend to prove that  for any $l>j$ and $\lambda \ge 1$
\[\begin{split}
&1+\lambda \frac{n-j}{w(j)}[f(j)w(j)-f(j+1)w(j+1)]\ge\\
& \frac{w(l)}{w(j)}+\lambda \biggl[\frac{n-l}{w(j)}f(j)w(j)-\frac{n-j}{w(j)}f(j+1)w(j+1)\biggr],
\end{split}
\]
where the left hand side is obtained setting $l=j$.
The latter is equivalent to
\[
w(l)-w(j)+\lambda(j-l)f(j)w(j)\le 0.
\]
Similarly to \eqref{eq:ineqproofT4}, one can show that
\[
w(l)\le w(j)+\lambda (w(j)-w(j-1))(l-j),
\]
 and get the desired result as follows
\[\begin{split}
&w(l)-w(j)+\lambda(j-l)f(j)w(j)\\
&\le
w(j)+\lambda (w(j)-w(j-1))(l-j)-w(j)+\lambda(j-l)f(j)w(j)\\
&=\lambda(l-j)(w(j)-w(j-1)-f(j)w(j))\le 0\,,
\end{split}
\]
where the last inequality holds because $f(j)w(j)\ge w(j)-w(j-1)$ \mbox{(by assumption) and $l>j$.}

The two cases just discussed showed that $W^\star$ in \eqref{eq:formulacor1} can be computed as
\[
	\begin{split}
	W^\star &= \min_{\lambda\in\mb{R}_{\ge0},\,\mu\in\mb{R}}~ \mu  \\[0.1cm]
	&\,\text{s.t.}~ \,\mu w(j)\ge w(l)+\lambda[j f(j) w(j)-l f(j+1) w(j+1)]\\
	&~\hspace*{25mm} \forall j,l\in\intwithzero{n}~{s.t.}~ j\ge l ~\text{and}~ 1\le j+l\le n,\\[0.15cm]
	&\qquad \mu w(j)\ge w(l)+\lambda[(n-l) f(j) w(j)-(n-j) f(j+1) w(j+1)]\\
	&~\hspace*{25mm} \forall j,l\in\intwithzero{n}~{s.t.}~ j\ge l ~\text{and}~~~~~~\, j+l\ge n,
	\end{split}
\]
Every constraint appearing in the previous program is indexed by $(j,l)$ and can be compactly written as $\mu w(j)\ge b_{jl}+ a_{jl}\lambda$, upon defining $b_{jl}\coloneqq w(l)$ and consequently 
\[
a_{jl}\coloneqq
\begin{cases}
j f(j) w(j)-l f(j+1) w(j+1)\hspace*{18mm} 1\le j+l\le n,\\
(n-l) f(j) w(j)-(n-j) f(j+1) w(j+1)\quad j+l\ge n.
\end{cases}
\]
Consequently $W^\star$ can be computed as
\[
	\begin{split}
	W^\star &= \min_{\lambda\in\mb{R}_{\ge0},\,\mu\in\mb{R}}~ \mu  \\[0.1cm]
	&\,\text{s.t.}~ \,\mu w(j)\ge b_{jl}+ a_{jl}\lambda\qquad\forall j,l\in\intwithzero{n},~\text{s.t.}~j\ge l,~j+l\ge 1\,.
	\end{split}
\]
As previously seen, for $j=0$ the most binding constraint is $\lambda\ge 1$. Observe that, when $j\ge 1$ and $j\ge l$, it holds $a_{jl} \ge 0$. 
Indeed, since $f(j) w(j)$ is non increasing, for $1\le j+l \le n$ one has $a_{jl}=j f(j) w(j)-l f(j+1) w(j+1)\ge (j-l)f(j)w(j)\ge 0$. Similarly for $j+l\ge n$. 
Thus, the optimal choice is to pick $\lambda$ as small \mbox{as possible, i.e., $\lambda^\star=1$.}
\end{proof}
\subsubsection*{Proof of \cref{cor:SVandMC}}
\begin{proof}
	The proof is an application of \cref{thm:wconcavefwdecreas}.
	\begin{enumerate}
	\item Observe that $\fsv$ satisfies the assumptions of \cref{thm:wconcavefwdecreas} in that $f(j)w(j)={w(j)}/{j}$ is non increasing (due to concavity of $w$) and $\fsv(j)={1}/{j}\ge 1-{w(j)}/{w(j-1)}$ $\iff$ ${w(j-1)+j(w(j)-w(j-1))}\ge 0$ (due to positivity and non-decreasingness of $w$). Hence the result of \cref{thm:wconcavefwdecreas} applies and substituting $f(j)=1/j$ gives $W^\star_{\rm SV}$ as in the claim.
	\item
	Observe that $\fmc$ satisfies the assumption of \cref{thm:wconcavefwdecreas} in that $f(j)w(j)=w(j)-w(j-1)$ is non increasing (due to concavity of $w$) and $f(j)=1-\frac{w(j-1)}{w(j)}$ so the second condition is satisfied too.
	
	We conclude by proving that the constraints indexed with $l< j\in[n]$ are not needed and it is enough to consider $j=l\in[n]$, so that $W^\star_{\rm MC}$ is as given \eqref{eq:poafmc}. To do so, we show that for any constraint with $l<j$ the constraint with $l=j$ is more binding.
	
	For $l<j$ and $j+l\le n$ we want to prove that
	\[
\begin{split}
&1+\lambda \frac{j}{w(j)}[f(j)w(j)-f(j+1)w(j+1)]\ge\\
& \frac{w(l)}{w(j)}+\lambda \biggl[\frac{j}{w(j)}f(j)w(j)-\frac{l}{w(j)}f(j+1)w(j+1)\biggr]\,,
\end{split}
\]
where the left hand side is obtained setting $l=j$.
The previous is equivalent to 
	\[
	w(l)-w(j)+\lambda (j-l)f(j+1)w(j+1)\le 0\,,
	\]
	and since $f(j+1)w(j+1)=w(j+1)-w(j)$, it reduces to
	\be
	w(l)-w(j)+ (j-l)(w(j+1)-w(j))\le 0\,.
	\label{eq:proofformula2}
	\ee
	By concavity of $w$ and $l<j$, it holds that $w(j)\ge w(l)+(j-l)(w(j+1)-w(j))$ and thus \eqref{eq:proofformula2} follows.

	In the case of $l<j$ and $j+l> n$ we intend to show
	\[\begin{split}
&1+\lambda \frac{n-j}{w(j)}[f(j)w(j)-f(j+1)w(j+1)]\ge\\
& \frac{w(l)}{w(j)}+\lambda \biggl[\frac{n-l}{w(j)}f(j)w(j)-\frac{n-j}{w(j)}f(j+1)w(j+1)\biggr],
\end{split}
\]
which reduces to
	\[
	w(l)-w(j)+(j-l)(w(j)-w(j-1))\,.
	\]
	The latter follows by concavity of $w$.
	Hence, the price of anarchy of $\fmc$ is governed by $W^\star$ as in \cref{thm:wconcavefwdecreas}, where we set $f=\fmc$ and fix $j=l$. This gives the following expression
\[
	W^\star_{\rm MC} = 1 + \max_{j\in [n]}\left\{\min(j,n-j)\left[\fmc(j)-\fmc(j+1)\frac{w(j+1)}{w(j)}\right]\right
	\}\,,
\]
	which reduced to the expression for $W^\star_{\rm MC}$ in the claim, upon substituting $\fmc$ with its definition.
\end{enumerate}
\end{proof}
\subsection{Proofs of the results presented in \cref{sec:cover}}
\subsubsection*{Proof of \cref{thm:poageneralcovering}}
\begin{proof} The proof is a specialization of the general result obtained in \cref{thm:dualpoa} to the case of set covering problems.
We divide the study in three distinct cases, as in the following
\[ 
\cdue: \begin{cases}
a+x=0\\
	b+x\neq 0
\end{cases}
\quad
\ctre: \begin{cases}
a+x\neq0\\
	b+x= 0
\end{cases}
\quad
\cquattro: \begin{cases}
a+x\neq0\\
	b+x\neq 0
\end{cases}
\]
In case $\cdue$ it must be $a=x=0$, $b\neq 0$ and the constraints read as 
\[
\lambda\ge \frac{1}{b}\,.
\]
The most binding one is obtained for $b=1$, i.e., it suffices to have $\lambda\ge 1$ in order to guarantee $\lambda\ge 1/b$. In case $\ctre$ it must be $b=x=0$, $a\neq 0$. The constraints read as 
\[
\mu\ge \lambda af(a)\quad\forall~a\in[n].\]
In case $\cquattro$, since $a+x\neq 0 $ and $b+x\neq 0$, the constraints become 
\[
\mu\ge 1+\lambda[af(a+x)-bf(a+x+1)]\,.
\]
If $x=0$, then $a,\,b>0$ and the previous inequality reads
\[
\mu\ge 1+\lambda[af(a)-bf(a+1)]\quad a+b\in[n],.
\]
The most constraining inequality is obtained for $b$ taking the smallest possible value, that is $b=1$. Thus $0< a\le n-1$. Consequently when $x=0$, it suffices to have 
\[
\mu\ge 1+\lambda[af(a)-f(a+1)]\quad\forall a\in[n-1]\,.
\]
If $x\neq0$, the most binding constraint is obtained for $b=0$. In such case, $0< a+x\le n$  and the constraints read as 
\[
\mu\ge 1+\lambda af(a+x)\quad\forall a\in[n]\,.
\]
For ease of readability, we introduce the variable $j\coloneqq a+x$ and use $j$ and $x$ instead of $a$ and $x$. With this new system of indices the feasible region becomes $0< j \le n$ and $j-x\ge0$, $x> 0$. The latter set of  constraints read as
\[
\mu\ge 1+\lambda (j-x)f(j)
\]
and the most binding is trivially obtained for $x=1$, reducing the previous to
\[
\mu\ge 1+\lambda (j-1)f(j)\quad\forall ~j\in[n]\,.
\]

This guarantees that the program in \eqref{eq:generalbound} is equivalent to
\[
\begin{split}
W^\star=&\min_{\lambda\in\mb{R}_{\ge0},\,\mu\in\mb{R}}~ \mu \\
&\,\text{s.t.}\quad\!\lambda \ge {1}\\
& \qquad~ \mu \ge \lambda jf(j)\quad j\in[n]\\
&\qquad~\mu\ge1+ \lambda(jf(j)-f(j+1))\quad j\in[n-1]\\
&~\qquad \mu\ge1+ \lambda (j-1)f(j)\quad j\in[n]\,.
\end{split}
\]
Amongst the last three sets of constraints, the tightest constraint always features a positive coefficient multiplying $\lambda$. Indeed the only term multiplying $\lambda$ that could take negative values is $jf(j)-f(j+1)$, but every time this is negative, the constraints $\mu\ge1+ \lambda (j-1)f(j)$ are tighter. It follows that the solution consists in picking $\lambda$ as small as possible, that is in choosing $\lambda^\star = 1$. The program becomes 
\[
\begin{split}
W^\star=&\min_{\mu\in\mb{R}}~ \mu \\
&\,\text{s.t.}\quad\! \mu \ge  jf(j)\quad j\in[n]\\
&\qquad~\mu\ge1+ jf(j)-f(j+1) \quad j\in[n-1]\\
&~\qquad \mu\ge1+ (j-1)f(j)\quad j\in[n]\,.
\end{split}
\]
We conclude with a little of cosmetics: the first and third set of inequalities run over $j\in[n]$, while the second one has $j\in[n-1]$. Observe that the first and the third condition evaluated at $j=1$ read both as $\mu\ge 1$. This condition is implied by the last set of condition with $j=2$, indeed it reads as $\mu\ge 1+f(2)\ge1$ since we assumed $f$ non negative. Thus the first and third conditions can be reduced to $j\in\{2,\dots, n\}$. Shifting the indices down by one, we get 
\[
\begin{split}
W^\star=&\min_{\mu\in\mb{R}}~ \mu \\
&\,\text{s.t.}\quad\! \mu \ge  (j+1)f(j+1)\quad j\in[n-1]\\
&\qquad~\mu\ge1+ jf(j)-f(j+1) \quad j\in[n-1]\\
&~\qquad \mu\ge1+ jf(j+1)\quad j\in[n-1]\,,
\end{split}
\]
from which we get the analytic expression in \eqref{eq:Wstarsetcoveringg}, i.e.,
\[
W^\star = 1+\max_{j\in [n-1]}\{(j+1)f(j+1)-1, jf(j)-f(j+1),jf(j+1)\}\,.
\]
\end{proof}
\subsubsection*{Proof of \cref{cor:backtogair}}
\begin{proof}
Thanks to \cref{thm:poageneralcovering}, the value $W^\star$ and consequently the price of anarchy can be computed as
\[
W^\star = \max_{j\in [n-1]}\{(j+1)f(j+1),\,1+ jf(j)-f(j+1),\,1+jf(j+1)\}\,.\]
We will show that when $f$ is non-increasing, fewer constraints are required, producing exactly \eqref{eq:Wstarcoveringgair}.

First observe that $f$ being non-increasing implies $(j+1)f(j+1)= f(j+1)+j f(j+1)\le f(1)+jf(j+1)= 1+ j f(j+1)$, so that the first set of conditions is implied by the third. Hence
\[
W^\star = 1+\max_{j\in [n-1]}\{jf(j)-f(j+1),\,jf(j+1)\}\,.\]
We now verify that the first set of remaining conditions implies all the conditions in the second set, but not the last one:
\[
\mu\ge 1+jf(j)-f(j+1)\ge 1+jf(j)-f(j)=1+(j-1)f(j)\,,
\]
$\forall j\in[n-1]$ that is, all conditions $\mu \ge j f(j+1)$ are satisfied for $j\in [n-2]$. 
Thus, it suffices to require $\mu-1\ge  jf(j)-f(j+1)$ and $\mu-1\ge (n-1)f(n) $ for all $j\in [n]$ and the result in \eqref{eq:Wstarcoveringgair} follows.
\end{proof}

\subsection{Proofs of the results presented in \cref{sec:supermod}}
\subsubsection*{Proof of \cref{thm:convex}}
\begin{proof}
The proof is a specialization of the general result obtain in \cref{thm:dualpoa}. We divide the study in the same three cases used for the proof of \cref{thm:poageneralcovering}. 

\noindent In case $\cdue$, the constraints read as 
\[
w(b)-\lambda b\le 0 \iff \lambda \ge \frac{w(b)}{b},
\] 
the most constraining of which is given for $b=n$ as $w(b)$ is convex. Thus it must be 
\[
\lambda\ge \frac{w(n)}{n}\,.\]

\noindent In case $\ctre$, the constraints read as
\[
\lambda a f(a)w(a)\le \mu w(a) \iff \mu\ge \lambda a f(a).
\]
In case $\cquattro$, the constraints read as
\[
\mu\ge \frac{w(b+x)}{w(a+x)}+\lambda \biggl[a f(a+x)-bf(a+x+1)\frac{w(a+x+1)}{w(a+x)} \biggr ]\,.
\]
In order to conclude, we will show that the constraints obtained from $\cdue$ and $\ctre$ imply all the conditions stemming from $\cquattro$.
To do so observe that 
\[
\begin{split}
  &\frac{w(b+x)}{w(a+x)}+\lambda \biggl[af(a+x)-bf(a+x+1)\frac{w(a+x+1)}{w(a+x)}\biggr ] \\
= &\frac{1}{w(a+x)}\biggl[w(b+x)-\lambda bw(a+x+1)f(a+x+1)\biggr ]+\lambda af(a+x)\\
\le &\frac{1}{w(a+x)}\biggl[\lambda (b+x)-\lambda b\cdot f(1)w(1)\biggr ] +\lambda af(a+x)=\frac{1}{w(a+x)}x\lambda +\lambda af(a+x)\\
 \le &\lambda x f(a+x) + \lambda a f(a+x) = \lambda (a+x)f(a+x)
\end{split}
\]
From first to second line is rearrangement. From second to third is due to $f(a+x+1)w(a+x+1)\ge w(1)f(1)=1$ and to $w(b+x)\le \frac{w(n)}{n}(b+x)\le \lambda (b+x)$  where the first inequality holds because of convexity of $w$ and the second inequality follows from $\cdue$, i.e., from $\lambda \ge \frac{w(n)}{n}$. From third to fourth is rearrangement. From fourth to fifth is due to $w(a+x)f(a+x)\ge f(1)w(1)=1~\implies f(a+x)\ge \frac{ f(1)w(1)}{w(a+x)}$.

 The previous series of inequalities have demonstrated that if $\mu\ge \lambda a f(a)$ as required by condition $\ctre$, and if $\lambda\ge  \frac{w(n)}{n}$ as required by condition $\cdue$, then $\mu \ge \lambda (a+x)f(a+x)\ge\frac{w(b+x)}{w(a+x)}+\lambda \biggl[a f(a+x)-bf(a+x+1)\frac{w(a+x+1)}{w(a+x)} \biggr ]$, i.e., conditions $\cquattro$ are all satisfied.
 
It follows that $W^\star$ and consequently the price of anarchy is easily obtained as 
\be
\begin{split}
&W^\star = \min_{\lambda\in\mb{R}_{\ge0},\,\mu\in\mb{R}}~ \mu\\
&\text{s.t.~}\mu\ge  \lambda j f(j)\quad\forall j \in[n]\\
&~~~~\,\lambda\ge \frac{w(n)}{n}\,.
\end{split}
\ee
The solution is given by $\lambda^\star = \frac{w(n)}{n}$, $\mu^\star = \lambda^\star \max_{j\in [n]}j f(j)$, which gives a price of anarchy of 
\[
\poa{(f)}= \frac{n}{w(n)} \frac{1}{\max_{j\in[n]} j \cdot f(j)}\,.
\]
Amongst all the distribution rules satisfying $f(j)w(j)\ge 1$, the distribution $\fsv$ is optimal. This follows from the fact that $\max_{j\in[n]} j \cdot \fsv(j)=1$ is the smallest achievable value since $f(1)=1$. 
To conclude that $\fsv$ is optimal not only over all distributions with $f(j)w(j)\ge 1$ but also over all distributions $f\in\mc{F}$ it suffices to observe that \cite[Lem. 7.2]{jensen2018} constructs an instance showing that $n/w(n)$ is the best attainable price of anarchy independently of what $f$ is used. 
\end{proof}
\chapter{Applications}
\label{ch:p2applications}
As set forward in the introduction of \cref{ch:p2introduction}, our objective was to obtain efficient and distributed algorithms for the solution of \ac{gmmc} problems.
We decided to follow a game theoretic approach and studied the utility design problem in \cref{ch:p2utilitydesign} and \cref{ch:p2subsupercov}. More precisely we have developed a general theory to compute and optimize the price of anarchy as a function of the chosen utility functions. 
In the following we \emph{do not} tackle the algorithm design component (the second component of the game design approach of \cref{fig:gamedesign}), as there are readily available algorithms capable of determining a Nash equilibrium in a distributed fashion (see \cref{sec:potentialgames} for the best-response algorithm, and its complexity).
In this chapter we demonstrate the applicability of our results to the vehicle target allocation problem (\cref{sec:vehicletarget}), and to the problem of distributed caching in mobile networks (\cref{sec:distributedcaching}). We provide thorough simulation results and show the theoretical and numerical advantages of our approach. The results presented in this chapter have been published in \cite{part2Paccagnan2018}.

\section{The vehicle target allocation problem}
\label{sec:vehicletarget}

In this section we consider the vehicle target assignment problem introduced in \cite{murphey2000target} and studied, e.g., in \cite{arslan2007autonomous, marden2014generalized}. 
We are given a finite set of targets $\mc{R}$, and for each target $r\in\mc{R}$ its relative importance $v_r\ge 0$. Additionally, we are given a finite set of vehicles $N=\{1,\dots,n\}$, and for each vehicle a set of feasible target assignments $\mc{A}_i\in 2^{\mc{R}}$. The goal is to distributedly compute a feasible allocation $a\in\mc{A}$ so as to maximize the joint probability of successfully destroying the selected targets, \mbox{expressed as} 
\[
W(a)=\sum_{r\in\cup_{i\in N} a_i}v_r (1-(1-p)^{|a|_r}),
\]
where $(1-(1-p)^{|a|_r})$ is the probability that $|a|_r$ vehicles eliminate the target $v_r$ and the scalar quantity $0<p\le 1$ is a parameter representing the probability that a vehicle will successfully destroy a target. In the forthcoming presentation, it is assumed that the success probability $p$ is the same for all vehicles, else one would have to define a different $p_i$ for every vehicle $i\in N$. Observe that the welfare considered here has the form \eqref{eq:welfaredef} with welfare basis $(1-(1-p)^{|a|_r})$. We normalize this quantity (without affecting the problem's solution) so that $w(1)=1$ and thus define 
\be
w(j)=\frac{ 1-(1-p)^{j}}{1-(1-p)}\,.
\label{eq:wassignment}
\ee
Observe that \eqref{eq:wassignment} satisfies the \hyperlink{sas}{Standing Assumptions}, and \vref{ass:submodd} in that $w(j)>0$ and $w(j)$ is increasing and concave. Thus, it is possible to compute the performance of any set of utility functions of the form \eqref{eq:utilities} using \cref{thm:dualpoa}, and to further determine the optimal distribution rule $f^\star\in\mc{F}$ \mbox{by solving a corresponding linear program.} 

\cref{fig:comparison_sensoralloc} shows the achievable approximation ratios for the Shapley value, marginal contribution, optimal distribution, as well as the approximation bound in \eqref{eq:approx}. We observe that the optimal distribution rule significantly outperforms all the others as well as the bound \eqref{eq:approx} for non trivial values of $p$. For the extreme case of $p=1$, $f^\star$ matches \eqref{eq:approx}, while for small $p$ all the design methodologies offer a similarly high performance guarantee.  \cref{fig:optf_assignment} shows the distribution rules $\fsv$, $\fmc$ and $f^\star$ for the choice of $p=0.5$.

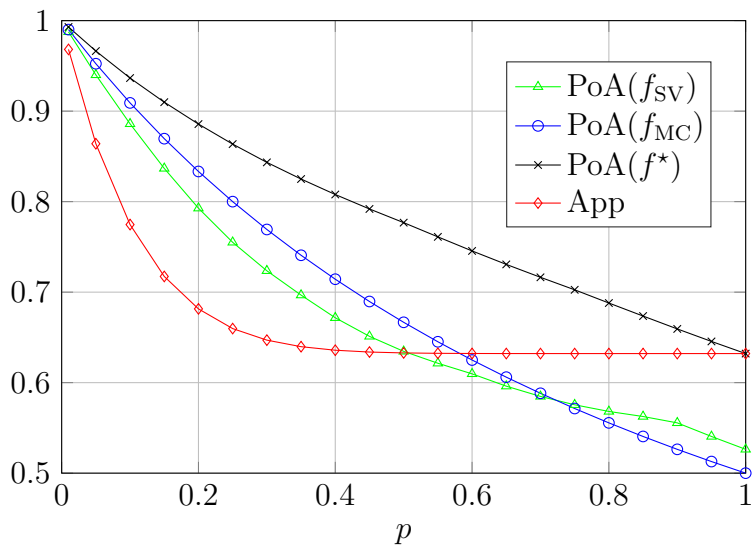
\begin{figure}[ht!] 
\begin{center}
\setlength\figureheight{6cm} 
\setlength\figurewidth{9cm} 
%
%
\begin{tikzpicture}

\begin{axis}[%
width=\figurewidth,
height=\figureheight,,
at={(1.011in,0.642in)},
scale only axis,
xmin=0,
xmax=1,
xlabel={$p$},
ymin=0.5,
ymax=1,
ytick={0.5, 0.6, 0.7, 0.8, 0.9 ,1},
xtick={0, 0.2, 0.4, 0.6, 0.8 ,1},
grid=both,
legend style={at={(0.95,0.91)}, anchor=north east, draw=white!15!black},
legend cell align={left}
]
\addplot[color=green,mark=triangle,mark options={solid}]
  table[row sep=crcr]{%
0.01 0.987595526822133\\
0.05	0.94005506248334\\
0.1	0.885945929385683\\
0.15	0.836703243672895\\
0.2	0.792757768782388\\
0.25	0.755066591317223\\
0.3	0.723594624162156\\
0.35	0.696731173009963\\
0.4	0.671695703060633\\
0.45	0.651151084727883\\
0.5	0.634576018857391\\
0.55	0.621432749556367\\
0.6	0.609723347378961\\
0.65	0.596116584115093\\
0.7	0.584793483986921\\
0.75	0.575539272188375\\
0.8	0.568181786446279\\
0.85	0.562587902575993\\
0.9	0.555555555305556\\
0.95	0.540540540540053\\
1	0.526315789473684\\
};
\addlegendentry{$\poa(\fsv)$};
\addplot [color=blue,mark=o,mark options={solid}]
  table[row sep=crcr]{%
0.01    0.990099009900991\\
0.05	0.952380952380951\\
0.1	0.909090909090909\\
0.15	0.869565217391305\\
0.2	0.833333333333333\\
0.25	0.8\\
0.3	0.769230769230769\\
0.35	0.740740740740741\\
0.4	0.714285714285714\\
0.45	0.689655172413793\\
0.5	0.666666666666667\\
0.55	0.645161290322581\\
0.6	0.625\\
0.65	0.606060606060606\\
0.7	0.588235294117647\\
0.75	0.571428571428571\\
0.8	0.555555555555556\\
0.85	0.54054054054054\\
0.9	0.526315789473684\\
0.95	0.512820512820513\\
1	0.5\\
};
\addlegendentry{$\poa(\fmc)$\hspace*{-1mm}};

\addplot[color=black,solid,mark=x,mark options={solid}]
  table[row sep=crcr]{%
0.01  0.992953665000470\\
0.05	0.966389234043331\\
0.1	0.936548216525508\\
0.15	0.909899317621329\\
0.2	0.885696582179183\\
0.25	0.8635731105029\\
0.3	0.84338976834284\\
0.35	0.824982401255014\\
0.4	0.807881209929548\\
0.45	0.791907322734352\\
0.5	0.776788977492372\\
0.55	0.76108086993257\\
0.6	0.745541787689185\\
0.65	0.730655009452087\\
0.7	0.71639163972289\\
0.75	0.702577389606777\\
0.8	0.687968462507212\\
0.85	0.673586478787689\\
0.9	0.65936662214631\\
0.95	0.645535765279937\\
1	0.632120559276055\\
};
\addlegendentry{$\poa(f^\star)$};

\addplot[color=red,solid,mark=diamond,mark options={solid}]
  table[row sep=crcr]{%
  0.01 0.968184773333313\\
0.05	0.863976359476678\\
0.1	0.774644591820245\\
0.15	0.717327671593058\\
0.2	0.681496501600498\\
0.25	0.659742671257519\\
0.3	0.64696582122097\\
0.35	0.639740045748953\\
0.4	0.635827936001333\\
0.45	0.633814778513864\\
0.5	0.632839073362096\\
0.55	0.632398926080495\\
0.6	0.632216996216784\\
0.65	0.632149553481668\\
0.7	0.632127799799598\\
0.75	0.632121962177256\\
0.8	0.632120747182832\\
0.85	0.632120572971079\\
0.9	0.632120559196437\\
0.95	0.632120558829276\\
1	0.632120558828558\\
};
\addlegendentry{${\rm App}$};

\end{axis}
\end{tikzpicture}%
\caption{Price of anarchy and approximation ratio comparison between 
the optimal distribution rule $f^\star$, the Shapley value distribution rule $\fsv$, the marginal contribution distribution rule $\fmc$, and \eqref{eq:approx}.
The problems considered feature $|N|\le n= 10$ vehicles and  $w(j)=\frac{1-(1-p)^j}{1-(1-p)}$ with $0<p\le 1$ represented over the $x$-axis.
} 
\label{fig:comparison_sensoralloc}
\end{center}
\end{figure}

\begin{figure}[h!] 
\begin{center}
\hspace*{-5mm} 
\setlength\figureheight{6cm} 
\setlength\figurewidth{9cm} 
%
%
\definecolor{mycolor1}{rgb}{0.85000,0.32500,0.09800}%
\begin{tikzpicture}

\begin{axis}[%
width=\figurewidth,
height=\figureheight,
at={(1.011111in,0.641667in)},
scale only axis,
xmin=1,
xmax=10,
xlabel={$j$},
ymin=0,
ymax=1,
ylabel={$f(j)$},
ylabel style = {rotate=-90},
grid = both,
ticklabel style = {font=\large},
legend style={at={(0.95,0.91)}, anchor=north east, draw=white!15!black},
legend cell align={left}
]
\addplot [color=green,mark=triangle,mark options={solid}]
  table[row sep=crcr]{%
1	1.000000000000000\\
2   0.500000000000000\\
3   0.333333333333333\\
4   0.250000000000000\\
5   0.200000000000000\\
6   0.166666666666667\\
7   0.142857142857143\\
8   0.125000000000000\\
9   0.111111111111111\\
10  0.100000000000000\\
};
\addlegendentry{~$\fsv$};
\addplot [color=blue,mark=o,mark options={solid}]
  table[row sep=crcr]{%
 1  1\\
 2  0.333333333333333\\
 3  0.142857142857143\\
 4  0.066666666666667\\
 5  0.032258064516129\\
 6  0.015873015873016\\
 7  0.007874015748031\\
 8  0.003921568627451\\
 9  0.001956947162427\\
 10  0.000977517106549\\
};
\addlegendentry{~$\fmc$};

\addplot [color=black,solid,mark=x,mark options={solid}]
  table[row sep=crcr]{%
1	1\\
2	0.475099390803924\\
3	0.284077657634457\\
4	0.196944960917654\\
5	0.145368837037017\\
6	0.11242636416\\
7	0.0913073915141783\\
8	0.0775913856320315\\
9	0.0669787482107186\\
10	0.058155748584454\\
};
\addlegendentry{~$f^\star$};
\end{axis}
\end{tikzpicture}%
\caption{Distributions $\fsv$, $\fmc$ and optimal distribution $f^\star$ obtained solving the \ac{lp} in \cref{thm:optimizepoa} for the specific choice of $w(j)$ in \eqref{eq:wassignment} with $|N|\le n=10$ and $p=0.5$.} 
\label{fig:optf_assignment}
\end{center}
\end{figure}
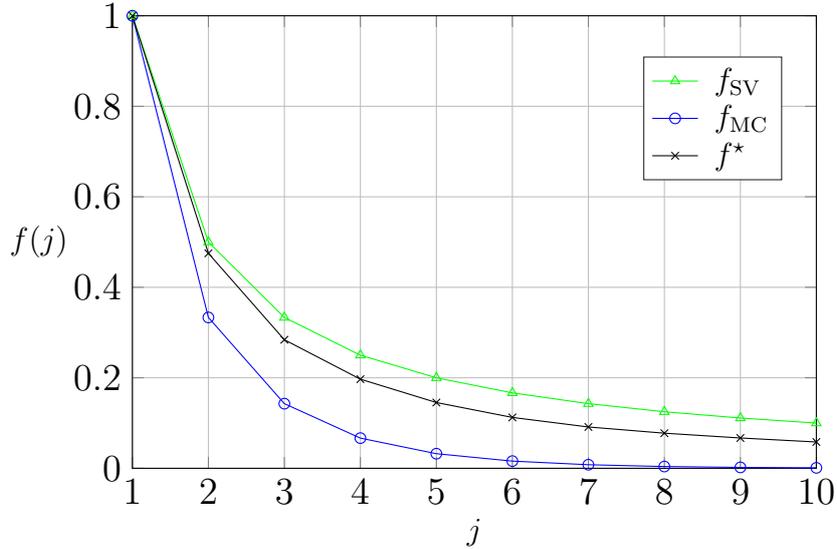

In both \cref{fig:comparison_sensoralloc,,fig:optf_assignment} we have set the number of agents to be relatively small\footnote{{Similar trends and conclusions can be obtained with larger values of $n$.}}, i.e., $|N|\le n =10$. This choice was purely made so as to perform an exhaustive search simulation in order to test the provided bounds displayed in \cref{fig:comparison_sensoralloc}. More specifically, we considered $10^5$ random instances of the vehicle target assignment problem. Each instance features $n=10$ agents, $n+1$ resources and fixed $p=0.8$. Each agent is equipped with an action set with only two allocations, whose elements are singletons, i.e., $|a_i|=1$. We believe this is not restrictive in assessing the performance, as the structure of some worst case instances is of this form \cite{paccagnan2017arxiv}. 

Observe that any constraint set $\mc{A}_i$ where feasible allocations are singletons is the bases of a uniform matroid of rank one, see \vref{example:matroids}. Further note that computing a single best response is a polynomial operation in the number of resources. Thus, the best response algorithm will converge polynomially to a Nash equilibrium (see \cref{prop:poly}) and so the performance guarantees offered by $\poa$ are \emph{easy} to achieve. 

 The structure of the constraints sets $\mc{A}_i$ and the values of the resources are randomly generated, the latter with uniform distribution in the interval $[0,1]$. For this class of problems considered, the theoretical worst case performance is $\poa(\fsv) \approx 0.568 $, $\poa(\fmc)\approx  0.556$, $\poa(f^\star)\approx 0.688$ (see \cref{fig:comparison_sensoralloc} with $p=0.8$).
For each instance $G$ generated, we performed an exhaustive search so as to compute the welfare at the worst equilibrium $\min_{a\in\nashe{G}}W(a)$ and the value $W(\aopt)$. The ratio between these quantities (their empirical cumulative distribution) is plotted across the $10^5$ samples in \cref{fig:CDF_comparison}, for $\fsv$, $\fmc$, $f^\star$. In the same figure the vertical dashed lines represent the theoretical bound on the price of anarchy, while the markers represent the worst case performance occurred during the simulations.

\begin{figure}[ht!] 
\begin{center}
\setlength\figureheight{6cm} 
\setlength\figurewidth{9cm} 
%
%
\begin{tikzpicture}

\begin{axis}[%
width=\figurewidth,
height=\figureheight,,
at={(1.011in,0.642in)},
scale only axis,
xmin=0.54,
xmax=1,
xlabel={$\frac{\min_{a\in\nashe{G}}W(a)}{W(\aopt)}$},
ymin=-0.05,
ymax=1,
unbounded coords=jump,
ytick={0, 0.25, 0.5, 0.75, 1},
ylabel={CDF},
ylabel style={rotate=-90},
grid=both,
legend style={at={(0.93,0.92)}, anchor=north east, draw=white!15!black},
legend cell align={left}
]
\addplot[color=green, dashed , line width = 1pt, forget plot]
table[row sep=crcr]
  {0.568 	-0.1\\
   0.568 	1\\};
   \addplot[color=blue, dashed , line width = 1pt, forget plot]
table[row sep=crcr]
  {0.556 	-0.1\\
   0.556 	1\\};
   \addplot[color=black, dashed , line width = 1pt, forget plot]
table[row sep=crcr]
  {0.688 	-0.1\\
   0.688 	1\\};
%
\addplot[color=green, mark = triangle, mark size = 4pt, mark options={solid}, line width = 1pt]
  table[row sep=crcr]
  {0.74459	1.25e-05\\};
\addlegendentry{$\fsv$\hspace*{-1mm}};
\addplot [color=blue, mark = o, mark size = 3pt, line width = 1pt]
  table[row sep=crcr]
  {%
0.71459	2.5e-05\\};
\addlegendentry{$\fmc$\hspace*{-1mm}};
\addplot[color=black, mark = x, mark size = 4pt, line width = 1pt]
  table[row sep=crcr]{%
0.802	3.75e-05\\};
\addlegendentry{$f^\star$};


\addplot [color=blue, line width = 1pt, forget plot]
  table[row sep=crcr]
  {%
0.71459	2.5e-05\\
0.74718	2.5e-05\\
0.74977	2.5e-05\\
0.75236	2.5e-05\\
0.75495	2.5e-05\\
0.75754	2.5e-05\\
0.76013	2.5e-05\\
0.76272	2.5e-05\\
0.76531	3.75e-05\\
0.7679	6.25e-05\\
0.77049	7.5e-05\\
0.77308	8.75e-05\\
0.77567	0.0001\\
0.77826	0.0001\\
0.78085	0.0001125\\
0.78344	0.000175\\
0.78603	0.000225\\
0.78862	0.0002625\\
0.79121	0.000325\\
0.7938	0.000375\\
0.79639	0.000425\\
0.79898	0.0004625\\
0.80157	0.0005\\
0.80416	0.0006125\\
0.80675	0.0007\\
0.80934	0.000875\\
0.81193	0.0009625\\
0.81452	0.0010625\\
0.81711	0.001225\\
0.8197	0.0014625\\
0.82229	0.001625\\
0.82488	0.001875\\
0.82747	0.0021\\
0.83006	0.0024875\\
0.83265	0.0029\\
0.83524	0.0034125\\
0.83783	0.0039375\\
0.84042	0.00455\\
0.84301	0.0050375\\
0.8456	0.0058\\
0.84819	0.006675\\
0.85078	0.0078125\\
0.85337	0.00915\\
0.85596	0.01055\\
0.85855	0.011775\\
0.86114	0.0136625\\
0.86373	0.0153125\\
0.86632	0.0172375\\
0.86891	0.0192875\\
0.8715	0.0215125\\
0.87409	0.0242\\
0.87668	0.0273875\\
0.87927	0.0306\\
0.88186	0.0342\\
0.88445	0.0379625\\
0.88704	0.0424\\
0.88963	0.0465625\\
0.89222	0.0516625\\
0.89481	0.0563375\\
0.8974	0.0619625\\
0.89999	0.067875\\
0.90258	0.0747875\\
0.90517	0.082175\\
0.90776	0.0900375\\
0.91035	0.0976\\
0.91294	0.1063875\\
0.91553	0.11485\\
0.91812	0.1245875\\
0.92071	0.1340625\\
0.9233	0.1443875\\
0.92589	0.1547125\\
0.92848	0.1658\\
0.93107	0.177275\\
0.93366	0.1895125\\
0.93625	0.2012125\\
0.93884	0.213525\\
0.94143	0.2259875\\
0.94402	0.23905\\
0.94661	0.2526625\\
0.9492	0.2658375\\
0.95179	0.2797\\
0.95438	0.29335\\
0.95697	0.3081875\\
0.95956	0.3231875\\
0.96215	0.3397\\
0.96474	0.3573875\\
0.96733	0.3775625\\
0.96992	0.400475\\
0.97251	0.4269375\\
0.9751	0.4575625\\
0.97769	0.493025\\
0.98028	0.5322125\\
0.98287	0.575875\\
0.98546	0.6218625\\
0.98805	0.670175\\
0.99064	0.7180875\\
0.99323	0.765075\\
0.99582	0.8113375\\
0.99841	0.853225\\
1.001	1\\
};

\addplot[color=green,mark options={solid}, line width = 1pt, forget plot]
  table[row sep=crcr]
  {%
0.74459	1.25e-05\\
0.80008	1.25e-05\\
0.80212	3.75e-05\\
0.80416	3.75e-05\\
0.8062	6.25e-05\\
0.80824	0.0001125\\
0.81028	0.0001375\\
0.81232	0.00015\\
0.81436	0.0001875\\
0.8164	0.0002375\\
0.81844	0.000275\\
0.82048	0.00035\\
0.82252	0.00045\\
0.82456	0.0005125\\
0.8266	0.0006375\\
0.82864	0.0008125\\
0.83068	0.0009625\\
0.83272	0.00115\\
0.83476	0.0012875\\
0.8368	0.001475\\
0.83884	0.00165\\
0.84088	0.0019125\\
0.84292	0.002275\\
0.84496	0.002675\\
0.847	0.003125\\
0.84904	0.0036125\\
0.85108	0.0041625\\
0.85312	0.0048625\\
0.85516	0.00545\\
0.8572	0.0060875\\
0.85924	0.007025\\
0.86128	0.0080125\\
0.86332	0.00915\\
0.86536	0.0102375\\
0.8674	0.0114375\\
0.86944	0.0127125\\
0.87148	0.0141875\\
0.87352	0.015875\\
0.87556	0.0175625\\
0.8776	0.019575\\
0.87964	0.0217625\\
0.88168	0.0241875\\
0.88372	0.0266\\
0.88576	0.029375\\
0.8878	0.032625\\
0.88984	0.03595\\
0.89188	0.0397125\\
0.89392	0.0439125\\
0.89596	0.048225\\
0.898	0.0526375\\
0.90004	0.05815\\
0.90208	0.0635\\
0.90412	0.0689625\\
0.90616	0.0755\\
0.9082	0.0820125\\
0.91024	0.0895375\\
0.91228	0.097525\\
0.91432	0.1049125\\
0.91636	0.1134625\\
0.9184	0.1226125\\
0.92044	0.1315375\\
0.92248	0.1414125\\
0.92452	0.1520875\\
0.92656	0.1637875\\
0.9286	0.1751125\\
0.93064	0.187525\\
0.93268	0.200175\\
0.93472	0.213325\\
0.93676	0.2272625\\
0.9388	0.2422125\\
0.94084	0.2589\\
0.94288	0.27505\\
0.94492	0.291975\\
0.94696	0.309425\\
0.949	0.3276125\\
0.95104	0.34575\\
0.95308	0.364675\\
0.95512	0.384075\\
0.95716	0.4046125\\
0.9592	0.424275\\
0.96124	0.4445125\\
0.96328	0.4653\\
0.96532	0.4869875\\
0.96736	0.5073125\\
0.9694	0.528825\\
0.97144	0.5506875\\
0.97348	0.572625\\
0.97552	0.5948875\\
0.97756	0.6176625\\
0.9796	0.6416875\\
0.98164	0.6652\\
0.98368	0.6896125\\
0.98572	0.7141375\\
0.98776	0.7370625\\
0.9898	0.7595625\\
0.99184	0.781275\\
0.99388	0.80335\\
0.99592	0.825325\\
0.99796	0.848475\\
1	1\\
};
\addplot[color=black, line width = 1pt, forget plot]
  table[row sep=crcr]{%
0.802	3.75e-05\\
0.804	3.75e-05\\
0.806	3.75e-05\\
0.808	5e-05\\
0.81	6.25e-05\\
0.812	6.25e-05\\
0.814	6.25e-05\\
0.816	7.5e-05\\
0.818	7.5e-05\\
0.82	8.75e-05\\
0.822	0.0001125\\
0.824	0.000125\\
0.826	0.00015\\
0.828	0.0001875\\
0.83	0.00025\\
0.832	0.0003\\
0.834	0.000375\\
0.836	0.0004125\\
0.838	0.0005\\
0.84	0.0005625\\
0.842	0.0007125\\
0.844	0.0008875\\
0.846	0.0010625\\
0.848	0.001225\\
0.85	0.0013625\\
0.852	0.0016125\\
0.854	0.001975\\
0.856	0.00225\\
0.858	0.0026875\\
0.86	0.003275\\
0.862	0.0037625\\
0.864	0.00435\\
0.866	0.005\\
0.868	0.0057875\\
0.87	0.00645\\
0.872	0.007375\\
0.874	0.0083125\\
0.876	0.0093375\\
0.878	0.0104125\\
0.88	0.0116375\\
0.882	0.013\\
0.884	0.0146\\
0.886	0.0162\\
0.888	0.0180375\\
0.89	0.0203625\\
0.892	0.0227375\\
0.894	0.025225\\
0.896	0.0280625\\
0.898	0.0312\\
0.9	0.0347\\
0.902	0.038425\\
0.904	0.0422375\\
0.906	0.04655\\
0.908	0.0511875\\
0.91	0.056275\\
0.912	0.062225\\
0.914	0.0674875\\
0.916	0.0734625\\
0.918	0.080025\\
0.92	0.08695\\
0.922	0.0945125\\
0.924	0.1031\\
0.926	0.112\\
0.928	0.1215875\\
0.93	0.1314875\\
0.932	0.1424125\\
0.934	0.1535125\\
0.936	0.16525\\
0.938	0.1771625\\
0.94	0.1905625\\
0.942	0.2050125\\
0.944	0.2193625\\
0.946	0.23455\\
0.948	0.250725\\
0.95	0.267775\\
0.952	0.2849\\
0.954	0.3032625\\
0.956	0.32275\\
0.958	0.34365\\
0.96	0.36425\\
0.962	0.3854375\\
0.964	0.406925\\
0.966	0.4290375\\
0.968	0.4509375\\
0.97	0.47375\\
0.972	0.4972\\
0.974	0.519975\\
0.976	0.5430375\\
0.978	0.568075\\
0.98	0.594075\\
0.982	0.6195125\\
0.984	0.6453875\\
0.986	0.67195\\
0.988	0.6982375\\
0.99	0.7241625\\
0.992	0.74945\\
0.994	0.7751125\\
0.996	0.8014125\\
0.998	0.827725\\
1	1\\
};
\end{axis}
\draw[-stealth, line width=1pt] (6.1 ,2.5) -- (5.5, 2.5);
\node at (6.8, 2.5) {\footnotesize $\poa(f^\star)$};
\node at (4.2, 2.5) {\footnotesize $\poa(\fmc)$};
\draw[-stealth, line width=1pt] (3.4 ,2.5) -- (2.9, 2.5);
\node at (4.2, 3) {\footnotesize $\poa(\fsv)$};
\draw[-stealth, line width=1pt] (3.4 ,3) -- (3.12, 3);
\end{tikzpicture}%
\caption{Cumulative distribution of the ratio $\min_{a\in\nashe{G}}W(a)/W(\aopt)$ for 
$\fsv$, $\fmc$, $f^\star$ across $10^5$ samples. The dashed lines represent the theoretical value of $\poa(\fsv)$, $\poa(\fmc)$, $\poa(f^\star)$ while the corresponding markers identify the worst case performance encountered during the simulations.
} 
\label{fig:CDF_comparison}
\end{center}
\end{figure}
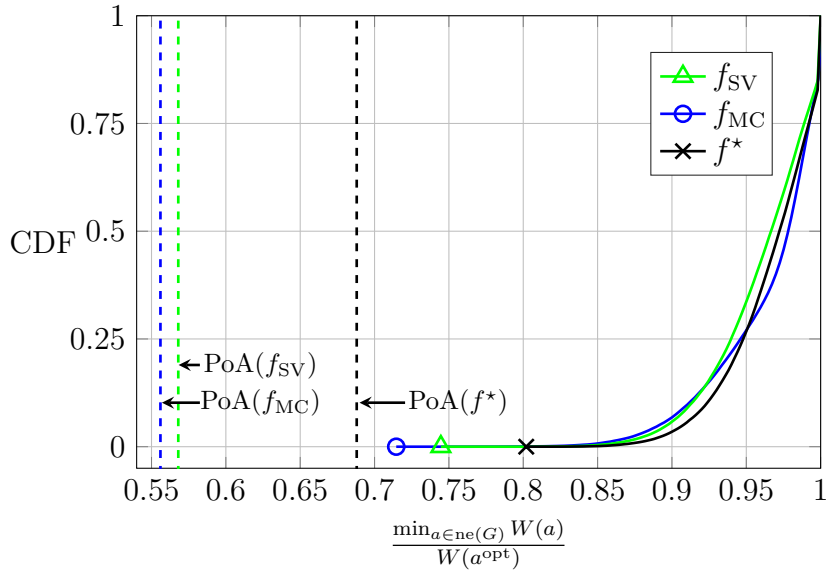

First, we observe that no instance has performed worse than the corresponding price of anarchy, as predicted by \vref{thm:dualpoa}. Second, we note that the worst case performance encountered in the simulation is circa $15\%$ better than the true worst case instance.\footnote{Recall that our results in \cref{thm:primalpoa,,thm:dualpoa} are provably tight: there exists at least one instance achieving exactly an efficiency equal to the price of anarchy.}
Further, the optimal distribution $f^\star$ has outperformed the others also in the simulations. Its worst case performance is indeed superior to the others (markers in \cref{fig:CDF_comparison}). Additionally, the cumulative distribution of $f^\star$ lies below the cumulative distributions of $\fsv$ and $\fmc$ (for abscissas smaller than $0.95$). This means that, for any given approximation ratio $r\in[0, 0.95]$, there is a smaller fraction of problems on which $f^\star$ performs worse or equal to $r$, compared to $\fsv$ and $\fmc$. Observe that this is not obvious a priori, as $f^\star$ is designed to maximize the worst case performance and not the average performance. 
\section{Distributed caching}
\label{sec:distributedcaching}
In this section we consider the problem of distributed data caching introduced in \cite{goemans2006market} as a technique to reduce peak traffic in mobile data networks. In order to alleviate the growing radio congestion caused by the recent surge of mobile data traffic, the latter work suggested to store popular and spectrum intensive items (such as movies or songs) in geographically distributed stations. 
The approach has the advantage of bringing the content closer to the customer, and to  avoid recurring transmission of large quantities of data.  
Similar offloading techniques, aiming at minimizing the peak traffic demand by storing popular items at local cells, have been recently proposed and studied in the context of modern 5G mobile networks \cite{andrews2013seven,de2017competitive}. The fundamental question we seek to answer in this section is how to  geographically distribute the popular items across the nodes of a network so as to maximize the total number of queries fulfilled.
In the following we borrow the model introduced in \cite{goemans2006market} and show how the utility design approach presented here yields improved theoretical and practical performances.

We consider a rectangular grid with $n_x\times n_y$ bins and a finite set $\mc{R}$ of data items. For each item $r\in\mc{R}$, we are given its query rate $q_r\ge0$ as well as its position in the grid $O_r$ and a radius $\rho_r$. A circle of radius $\rho_r$ centered in $O_r$ represents the region where the item $r$ is requested. Additionally we consider a set of geographically distributed nodes $N$ (the local cells), where each node $i\in N$ is assigned to a position in the grid $P_i$. A node is assigned a set of feasible allocations $\mc{A}_i$ according to the following rules:
\begin{itemize}
	\item[i)] $\mc{A}_i\subseteq 2^{\mc{R}_i}$, where $
	\mc{R}_i\coloneqq\{r\in\mc{R}~\text{s.t.}~||O_r- P_i||_2\le \rho_r\}$.
	That is, $r\in\mc{R}_i$ if the (euclidean) distance between the position of node $i$ and item $r$ is smaller equal to $\rho_r$.
	\item[ii)] $|\mc{A}_i|\le k_i$, for some integer $k_i\ge 1$.
\end{itemize}
In other words, node $i$ can include the resource $r$ in his allocation $a_i$ only if he is in the region where the item $r$ is requested (first rule), while we limit the number of stored items to $k_i$ for reasons of physical storage (second rule).\footnote{Similarly to what discussed for the application in \cref{sec:vehicletarget}, it is possible to reduce the problem to the case where $\mc{A}_i$ are the bases of a matroid $\mc{M}_i$, so that  \cref{prop:poly} applies here too. Once more computing the best response is a polynomial task (it amounts to sorting $q_r w(|a|_r)f(|a|_r)$ and picking the $k_i$ first items). Thus the best-response dynamics introduced in \cref{alg:BR} \mbox{converges in polynomial time.}} The situation is exemplified in \cref{fig:checkered}.
\begin{figure}[t]
\begin{center}
\includegraphics[scale=0.95]{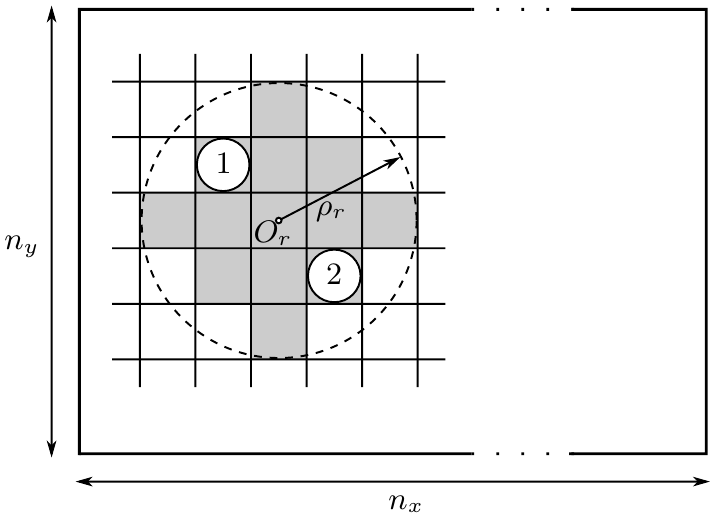}
\caption{The nodes $1$ and $2$ can include the item $r$ in any allocation, i.e., $r\in\mc{R}_1$ and $r\in\mc{R}_2$ since the distance from nodes $1$ and $2$ to $O_r$ is less than $\rho_r$.}
\label{fig:checkered}
\end{center}
\end{figure}

The objective is to select a feasible allocation for every node so as to jointly maximize the total amount of queries fulfilled
\[
\max_{a\in\mc{A}}\sum_{r\in\cup_{i\in N} a_i}q_r\,.
\]
In order to obtain a distributed algorithm, \cite{goemans2006market} proposes a game theoretic approach where each agent is given a Shapley value utility function, i.e., they assign to agents utilities of the form \eqref{eq:utilities}, where $f(j)=\fsv(j)=1/j$. 

In the following we compare the results of numerical simulations obtained using $\fsv$ or the optimal distribution $f^\star = \fgar$ defined in \eqref{eq:fgar}. The following parameters are employed.  
We choose $n_x=n_y=800$, $|N|=100$, $|\mc{R}|=1000$. The nodes and the data items are uniformly randomly placed in the grid. The query rate of data items is chosen according to a power law (Zipf distribution) $q_r=1/r^\alpha$ for $r\in [1000]$.\footnote{Typical query rate curves has been shown to follow this distribution, with $\alpha \in [0.6, 0.9]$, see \cite{breslau1999web}.} The radii of interests are set to be identical for all items $\rho_r=\rho=200$. We let $\alpha$ vary in $[0.7, 0.9]$. We consider $10^5$ instances of such problem, and for every instance compute a Nash equilibrium by means of the best response algorithm. Given the size of the problem, it is not possible to compute the optimal allocation and thus the price of anarchy. As a surrogate for the latter we use the ratio $W(\ae)/W_{\rm tot}$, where $\ae$ is the Nash equilibrium determined by the algorithm and
\[
W_{\rm tot}\coloneqq\sum_{r\in\mc{R}} q_r\]
 is the sum of all the query rates and thus is an upper bound for $W(\aopt)$. Observe that $W_{\rm tot}$ is a constant for all the simulations with fixed $\alpha$, indeed  $W_{\rm tot}= \sum_{r\le 1000}\frac{1}{r^\alpha}$ and thus serves as a mere scaling factor. The theoretical price of anarchy is $\poa(\fsv)=0.5$ (tight also when the query rates are Zipf distributed \cite{goemans2006market}) and $\poa(f^\star)= 1-1/e \approx 0.632$, see \vref{thm:smoothnottight}. 
 \cref{fig:Boxplot} compares the quantity $W(\ae)/W_{\rm tot}$ for the choice of $\fsv$ and $f^\star$, across different values of $\alpha$. First we observe that the worst cases encountered in the simulations are at least $10\%$ better than the theoretical counterparts. Further, for each fixed value of $\alpha$, there is a good separation between the performance of $\fsv$ and $f^\star$, in favor of the latter. This holds true, not only in the worst case sense (markers in \cref{fig:Boxplot}), but also on average. As $\alpha$ increases from $0.6$ to $0.9$, the worst case performance seems to degrade for both $\fsv$ and $f^\star$. Nevertheless, since we are using $W(\ae)/W_{\rm tot}$ as a surrogate for the true price of anarchy, it is unclear if the previous conclusion also holds for $W(\ae)/W(\aopt)$.
\cref{fig:PDF_comparison_content} presents a more detailed comparison between $\fsv$ and $f^\star$ for a fixed value of $\alpha=0.7$ over all the $10^5$ instances. Relative to this case, \cref{fig:BR_count} describes the (distribution of) number of best response rounds required for the algorithm to converge. 
Quick convergence is achieved, with a number of best response rounds equal to $11$ in the worst case.
Observe that in every best response round all players have a chance to update their decision variable, so that a total number of $n_{\rm BR}$ rounds amounts to $n\cdotshort n_{\rm BR}$ individual best responses.
 \begin{figure}[ht!] 
\begin{center}
\hspace*{-5mm} 
\setlength\figureheight{4.3cm} 
\setlength\figurewidth{9cm} 
\input{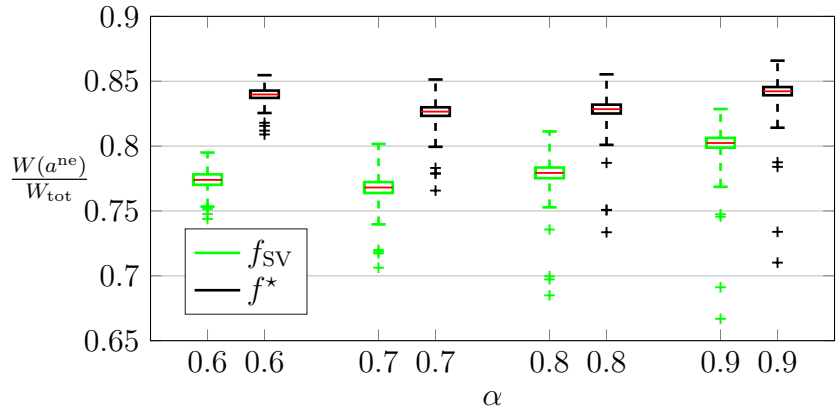}
\vspace*{-3mm}
\caption{Box plot comparing the performance of the best response algorithm on $10^5$ instances for the choice of distributions $\fsv$ and $f^\star$, across different values of $\alpha$. On each plot, the median is represented with a red line, and the corresponding box contains the 25th and 75th percentiles. The (four) worst cases are represented with crosses.} 
\label{fig:Boxplot}
\end{center}
\end{figure}
\begin{figure}[ht!] 
\vspace*{-4mm}
\begin{center}
\hspace*{-5mm} 
\setlength\figureheight{4.3cm} 
\setlength\figurewidth{9cm} 
%
%
\begin{tikzpicture}

\begin{axis}[%
width=\figurewidth,
height=\figureheight,,
at={(1.011in,0.642in)},
scale only axis,
xmin=0.7,
xmax=0.86,
xlabel={$W(\ae)/W_{\rm tot}$},
ymin=-300,
ymax=8000,
unbounded coords=jump,
scaled y ticks=base 10:-3,
ylabel={Count},
ylabel style={rotate=-90},
grid=both,
legend style={at={(0.05,0.91)}, anchor=north west, draw=white!15!black},
legend cell align={left}
]
\addplot[color=green, mark = triangle, mark size = 4pt, mark options={solid}, line width = 1pt]
  table[row sep=crcr]
  {0.706274043136323	0\\};
\addlegendentry{$\fsv$\hspace*{-1mm}};
%
%
\addplot[color=black, mark = x, mark size = 4pt, line width = 1pt]
  table[row sep=crcr]{%
0.765721086795839	0\\};
\addlegendentry{$f^\star$};

\addplot[ybar interval, fill=green, fill opacity=0.6, draw=black, area legend] table[row sep=crcr] {%
x	y\\
0.7056	1\\
0.706561	0\\
0.707522	0\\
0.708483	0\\
0.709444	0\\
0.710405	0\\
0.711366	0\\
0.712327	0\\
0.713288	0\\
0.714249	0\\
0.71521	0\\
0.716171	0\\
0.717132	1\\
0.718093	1\\
0.719054	1\\
0.720015	0\\
0.720976	0\\
0.721937	0\\
0.722898	0\\
0.723859	0\\
0.72482	0\\
0.725781	0\\
0.726742	0\\
0.727703	0\\
0.728664	0\\
0.729625	0\\
0.730586	0\\
0.731547	0\\
0.732508	0\\
0.733469	0\\
0.73443	0\\
0.735391	0\\
0.736352	0\\
0.737313	0\\
0.738274	0\\
0.739235	1\\
0.740196	0\\
0.741157	2\\
0.742118	0\\
0.743079	2\\
0.74404	1\\
0.745001	2\\
0.745962	5\\
0.746923	4\\
0.747884	18\\
0.748845	25\\
0.749806	49\\
0.750767	65\\
0.751728	117\\
0.752689	244\\
0.75365	439\\
0.754611	576\\
0.755572	632\\
0.756533	1081\\
0.757494	1588\\
0.758455	2313\\
0.759416	2802\\
0.760377	2760\\
0.761338	3551\\
0.762299	4660\\
0.76326	5513\\
0.764221	6133\\
0.765182	5555\\
0.766143	5651\\
0.767104	6105\\
0.768065	6289\\
0.769026	6288\\
0.769987	5789\\
0.770948	4872\\
0.771909	4449\\
0.77287	3998\\
0.773831	3595\\
0.774792	3164\\
0.775753	2585\\
0.776714	2019\\
0.777675	1593\\
0.778636	1355\\
0.779597	1052\\
0.780558	818\\
0.781519	591\\
0.78248	450\\
0.783441	323\\
0.784402	248\\
0.785363	194\\
0.786324	147\\
0.787285	92\\
0.788246	63\\
0.789207	42\\
0.790168	35\\
0.791129	13\\
0.79209	11\\
0.793051	9\\
0.794012	6\\
0.794973	1\\
0.795934	1\\
0.796895	3\\
0.797856	3\\
0.798817	1\\
0.799778	2\\
0.800739	1\\
0.8017	1\\
};
\addplot[ybar interval, fill=black, fill opacity=0.6, draw=black, area legend] table[row sep=crcr] {%
x	y\\
0.7656	1\\
0.766458	0\\
0.767316	0\\
0.768174	0\\
0.769032	0\\
0.76989	0\\
0.770748	0\\
0.771606	0\\
0.772464	0\\
0.773322	0\\
0.77418	0\\
0.775038	0\\
0.775896	0\\
0.776754	0\\
0.777612	0\\
0.77847	2\\
0.779328	0\\
0.780186	0\\
0.781044	0\\
0.781902	0\\
0.78276	1\\
0.783618	0\\
0.784476	0\\
0.785334	0\\
0.786192	0\\
0.78705	0\\
0.787908	0\\
0.788766	0\\
0.789624	0\\
0.790482	0\\
0.79134	0\\
0.792198	0\\
0.793056	0\\
0.793914	0\\
0.794772	0\\
0.79563	0\\
0.796488	1\\
0.797346	0\\
0.798204	0\\
0.799062	1\\
0.79992	1\\
0.800778	0\\
0.801636	0\\
0.802494	0\\
0.803352	1\\
0.80421	1\\
0.805068	2\\
0.805926	0\\
0.806784	3\\
0.807642	1\\
0.8085	5\\
0.809358	8\\
0.810216	8\\
0.811074	17\\
0.811932	40\\
0.81279	88\\
0.813648	129\\
0.814506	249\\
0.815364	506\\
0.816222	815\\
0.81708	1342\\
0.817938	1820\\
0.818796	2119\\
0.819654	2837\\
0.820512	3648\\
0.82137	4672\\
0.822228	5632\\
0.823086	6254\\
0.823944	6115\\
0.824802	6473\\
0.82566	6677\\
0.826518	6932\\
0.827376	6906\\
0.828234	6255\\
0.829092	5598\\
0.82995	4796\\
0.830808	3974\\
0.831666	3503\\
0.832524	3019\\
0.833382	2413\\
0.83424	1929\\
0.835098	1439\\
0.835956	1062\\
0.836814	752\\
0.837672	593\\
0.83853	431\\
0.839388	290\\
0.840246	206\\
0.841104	149\\
0.841962	101\\
0.84282	66\\
0.843678	36\\
0.844536	38\\
0.845394	17\\
0.846252	9\\
0.84711	3\\
0.847968	6\\
0.848826	2\\
0.849684	1\\
0.850542	5\\
0.8514	5\\
};
\end{axis}
\end{tikzpicture}%
\vspace*{-3mm}
\caption{Distribution of $W(\ae)/W_{\rm tot}$ on $10^5$ instances for fixed $\alpha=0.7$.} 
\label{fig:PDF_comparison_content}
\end{center}
\end{figure}
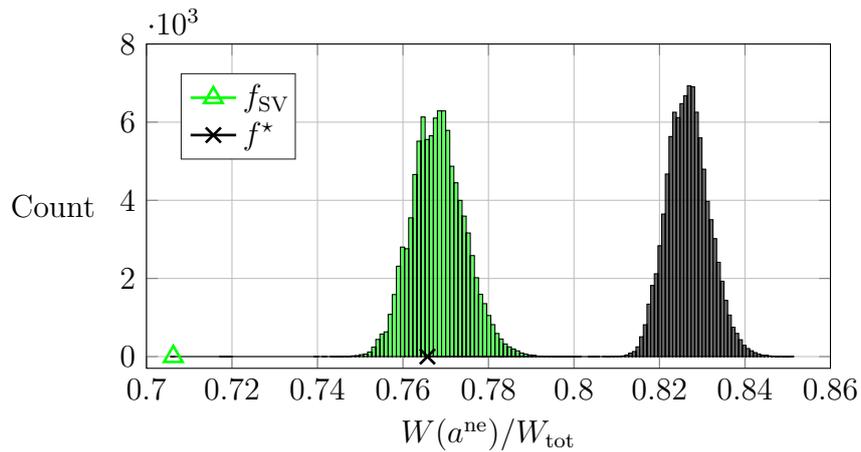
\begin{figure}[H] 
\vspace*{-4mm}
\begin{center}
\hspace*{-5mm} 
\setlength\figureheight{4.3cm} 
\setlength\figurewidth{9cm} 
%
%
\begin{tikzpicture}

\begin{axis}[%
width=\figurewidth,
height=\figureheight,,
at={(1.011in,0.642in)},
scale only axis,
xmin=1,
xmax=12,
xlabel={Number of best response rounds},
ymin=-3000,
ymax=60000,
unbounded coords=jump,
ylabel={Count},
ylabel style={rotate=-90},
grid=both,
legend style={at={(0.05,0.91)}, anchor=north west, draw=white!15!black},
legend cell align={left}
]
\addplot[color=green, line width = 1pt]
  table[row sep=crcr]
  {0.706274043136323	0\\};
\addlegendentry{$\fsv$\hspace*{-1mm}};
%
%
\addplot[color=black, line width = 1pt]
  table[row sep=crcr]{%
0.765721086795839	0\\};
\addlegendentry{$f^\star$};

\addplot[ybar interval, fill=green, fill opacity=0.6, draw=black, area legend] table[row sep=crcr] {%
x	y\\
3.5	5304\\
4.5	43853\\
5.5	38200\\
6.5	10781\\
7.5	1673\\
8.5	168\\
9.5	19\\
10.5	2\\
11.5	2\\
};
\addplot[ybar interval, fill=black, fill opacity=0.6, draw=black, area legend] table[row sep=crcr] {%
x	y\\
2.5	9\\
3.5	9715\\
4.5	50367\\
5.5	31468\\
6.5	7295\\
7.5	1018\\
8.5	119\\
9.5	8\\
10.5	1\\
11.5	1\\
};
\end{axis}
\end{tikzpicture}%
\vspace*{-3mm}
\caption{Distribution of the number of best response rounds required for convergence on $10^5$ instances, $\alpha=0.7$.} 
\label{fig:BR_count}
\end{center}
\end{figure}
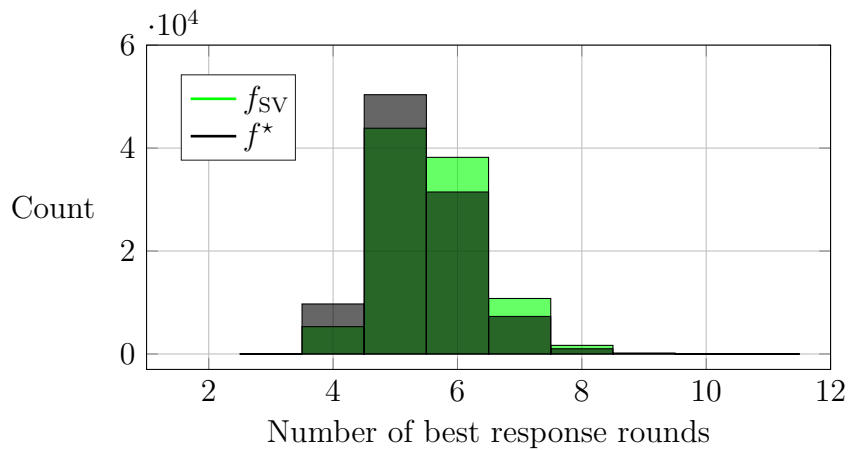

\part{Conclusion}
\chapter{Conclusions and outlook}
\section{Part I: strategic agents}
In the first part of the thesis we considered large scale systems composed of self-interested agents and modeled their strategic interaction using the language of game theory. Motivated by the special structure arising in different real-world applications, we focused on average aggregative games, i.e., games where the cost function of each agent depends solely on his decision and on the average population strategy. The setup considered allows for multidimensional decision variables, heterogenous private constraints, and global constraints coupling the decision variables of the entire population. 

Our research agenda was aimed at i) understanding the performance degradation due to selfish decision making, and ii) designing scalable algorithms to guide agents towards an equilibrium configuration. Towards these goals, we first exploited the theory of variational inequalities to reduce both the Nash equilibrium problem and the Wardrop equilibrium to common ground. This allowed to study the efficiency of a Nash equilibrium allocation through the analysis of the corresponding Wardrop equilibrium counterpart. In this respect, we provided conditions on the agents' cost functions that either guarantee the efficiency of the equilibria, or provide meaningful bounds on the efficiency loss. We concluded \cref{part:1} proposing two decentralized schemes to coordinate the agents towards a Nash or Wardrop equilibrium and discussed under which conditions their convergence is guaranteed. Our findings have been tested on a coordination problem arising in the charging of electric vehicles and on a selfish routing model used in road traffic network.

\subsection{Further research directions}
\subsubsection*{Non average aggregative games}
In \cref{part:1} of this thesis we focused on \emph{average} aggregative games. While this class of games has recently attracted the attention of the researchers, we believe that many problems within the general framework of (non average) aggregative games are still open. As an example, the problem of designing distributed algorithms for network aggregative games has been considered only very recently.\footnote{In a network aggregative game each agent is represented with a node on a graph, while his cost function is influenced by his decision and by a linear combination of the decision \mbox{variables of his neighbours.}} More broadly, it is unclear to what extent the aggregative structure helps in providing results such as existence and uniqueness of the equilibria under \emph{weaker assumptions} than what usually imposed on non aggregative games. There are few works addressing this question and their results are limited in their scope, for example to scalar valued aggregator functions \cite{Jensen2010}.

\subsubsection*{Uncertain games and receding-horizon implementations}
Within the framework studied in this thesis, we focused on the case of deterministic games. Nevertheless, there has been recent interest both in the areas of optimization and equilibrium theory to incorporate the effect of uncertainty. This desire stems from the observation that a large portion of nowadays decision making happens in face of uncertainty. As a concrete example, consider that of a car driver on a road network. While his goal might entail reaching the desired destination as swiftly as possible, his decisions are based on uncertain knowledge of the congestion he will encounter further ahead on the network. In this respect one can envision at least two \mbox{future research directions.} 

First, one could consider stochastic aggregative games where the aggregate function is subject to common uncertainty. The fundamental question one needs to ask is what it means to be an equilibrium configuration. In the simplest scenario, one can think of an equilibrium as a stable configuration of the game constructed with the \emph{expected} costs. Most of the results presented in connection with the variational reformulation of \cref{ch:p1equilibriaVIandSMON} hold with minor modifications, and one could use algorithms derived from the theory of stochastic variational inequalities \mbox{to compute one such equilibrium \cite{yousefian2017smoothing,rockafellar2017stochastic}.} 

As second research direction, one could consider receding-horizon implementations of the schemes proposed here. 
While some of the applications presented in this thesis were of dynamic nature (e.g., the charging coordination for a fleet of electric vehicles), we have been able to model them as single-stage decision problems. This has been possible due to the exact knowledge of the agents' dynamics. As this is hardly the case in a real world scenario, one might consider receding-horizon implementations of the single-stage problems considered here. This research direction follows the same spirit with which model predictive control is used in uncertain dynamic optimization problems \cite{garcia1989model}. 

\subsubsection*{Non monotone games}
Most of the results derived in the first part of the thesis were based on the assumptions of Lipschitzianity and monotonicity of the variational inequality operator (or variations thereof such as strong monotonicity, or co-coercivity see \cref{sec:operatorsprop}). In this regard, a long term research goal is that of weakening the monotonicity assumption. While this direction would have great impact (there are many situations in which the monotonicity property is not satisfied), there seem to be a fundamental roadblock that needs to be resolved or circumvented before embarking on this route. Indeed, as we have seen in \cref{ch:p1mathpreliminaries}, game theory is a generalization of single agent decision making and hence contains the field of optimization as a special case. Thus, the study of non monotone variational inequalities requires a better understanding of non convex optimization first. 
While there has been a recent surge of interest in other classes of continuous functions that produce tractable optimization problems (e.g., continuous submodular functions), we feel that this direction is currently underdeveloped. 

\section{Part II: programmable machines}
In the second part of the thesis we studied a class of combinatorial resource allocation problems arising in various applications connected to multiagent systems and machine learning. More precisely, we considered a setup where a large number of cooperative agents need to select a subset of resources from a common set, with the objective of jointly maximizing a given welfare function. In the considered setup, the welfare function was assumed to be additive over the resources and to be anonymous with respect to the agent identities. An example of problem satisfying these requirements is the well-known and studied weighted maximum coverage.

Since the class of problems investigated is computationally intractable ($\npclass$-hard), our goal was to derive distributed algorithms that run in polynomial time and achieve near-optimal performances. 
We approached the problem from a game-theoretic perspective and aimed at assigning a local utility function to each agent so that their selfish maximization recovers a large portion of the desired system level objective. Towards this goal, we presented a novel framework for the characterization of the equilibrium efficiency (price of anarchy). More precisely, for a given set of utilities, we showed that the problem of computing the worst-case equilibrium efficiency can be posed as a tractable linear problem. This result might be of independent interest to the community concerned with the study of the price of anarchy. We then leveraged the linear programming reformulation to resolve the question previously posed, i.e., to design local utilities that maximize such performance metric. The importance of this results stems from the observation that any algorithm capable of computing a Nash equilibrium would naturally inherit an approximation ratio matching the corresponding equilibrium efficiency.  Surprisingly, the optimal price of anarchy (the price of anarchy achieved by optimally designed utility functions)  matches or outperforms the guarantees available for many commonly used algorithms. We validate our results with two applications: the vehicle-target assignment problem and a coverage problem arising in distributed caching for mobile networks.

\subsection{Further research directions}
\subsubsection*{Different equilibrium notion}
As discussed in the introduction of \cref{ch:p2introduction}, the game design approach for the approximate solution of an optimization problem amounts to the design of three elements: equilibrium concept, agents' utilities and corresponding learning algorithm. While all the efficiency results presented in this thesis are limited to the notion of pure Nash equilibrium, one might be interested in using a different equilibrium concept. As a matter of fact, the choice of pure Nash equilibria originated from the fact that their efficiency is the highest possible. Unfortunately, pure Nash equilibria are intractable to compute in general (see \cref{fig:hierarchy} for the tradeoff between complexity and efficiency). The way we resolved this issue was by assuming that $\{\mc{A}_i\}_{i=1}^\N$ are the sets of bases for a matroid, so that the best-response algorithm converges in a polynomial number of steps (\cref{prop:poly}). Instead, coarse correlated equilibria are tractable to compute in general. Thus, an interesting research direction is to understand whether the efficiency bounds obtained for pure Nash Equilibria extend to coarse correlated equilibria. Nevertheless, the performance guarantees offered by coarse correlated equilibria are in expected value, and one would have to understand how to derandomize the corresponding solution efficiently (if at all possible). 

\subsubsection{Non-anonymous agents}
The results derived in this thesis are relative to welfare functions of the form \eqref{eq:welfaredef}
\[
W(a)=\sum_{r\in\cup a_i} v_r w(|a|_r)\,.
\]
We observe that the key ingredient that allowed to reduce the computation of the price of anarchy to a tractable linear program is the \emph{indistinguishability} of the agents (also called anonymity in the following), see the proof of \cref{thm:primalpoa}. Formally, the agents are anonymous if any allocation $a=(a_1,\dots,a_n)$ and any other allocation obtained as a permutation of the former have the same welfare. While it is very much unclear if and how to extend the current results to the case of non-anonymous agents, we remark that this will greatly expand the number of applications that could \mbox{benefit from this approach.}  

\subsubsection*{The tradeoff between anarchy and stability}
Throughout \cref{part:2} of this thesis, we assessed the quality of an algorithm with its worst case performance over a set of instances. This is a common approach to study the performance of an algorithm as it gives a bound that requires no information on the distribution of inputs and holds instance by instance. Nevertheless, an interesting and underdeveloped question is whether optimizing the worst-case performance comes at the cost of other performance metrics. In relation to the problem studied in this thesis, a different and more optimistic metric to quantify the equilibrium efficiency is known as price of stability. With the same notation previously used, the price of stability can be defined as 
\[
{\rm PoS}(f) \coloneqq \inf_{G\in \mc{G}_f}\biggl(\frac{\max_{a\in \nashe{G}} W(a)}{\max_{a\in\mc{A}} W(a)}\biggr)\,.
\]
Informally, the price of stability bounds the performance of the best equilibrium over all the possible instances in the set $\mc{G}_f$. While preliminary results have shown that there is a fundamental tradeoff between the price of anarchy and the price of stability in specific classes of problems \cite{paccagnan2017arxiv,filos2018pareto}, this research direction warrants further exploration as it would provide an additional guiding principle in the design of efficient algorithms.  

\backmatter
\addcontentsline{toc}{part}{Bibliography}
\printbibliography

\end{document}